\documentclass{amsart}
\usepackage{color}
\usepackage{amsmath}
\usepackage{kotex}
\usepackage{amssymb}
\usepackage{amsthm}
\usepackage{amscd}
\usepackage{bm}
\usepackage{eucal} 
\usepackage{enumerate}
\usepackage{tikz}
\usepackage{tikz-cd}
\usetikzlibrary{positioning, fit, calc}   
\usepackage{shuffle}
\usepackage[utf8]{inputenc}
\usepackage[all]{xy}
\allowdisplaybreaks

\usepackage{hyperref}
\hypersetup{
  colorlinks   = true,          
  urlcolor     = blue,          
  linkcolor    = blue,          
  citecolor   = red             
}

\theoremstyle{plain}
\newtheorem{theorem}{Theorem}[section]
\newtheorem{conjecture}[theorem]{Conjecture}
\newtheorem{lemma}[theorem]{Lemma}
\newtheorem{proposition}[theorem]{Proposition}
\newtheorem{corollary}[theorem]{Corollary}

\newtheorem{theoremx}{Theorem}

\theoremstyle{definition}
\newtheorem{definition}[theorem]{Definition}
\newtheorem{remark}[theorem]{Remark}

\numberwithin{equation}{section}

\newcommand\fantome[1]{}

\def\bN{\mathbb N}
\def\bR{\mathbb R}

\def\bZ{\mathbb Z}

\def\bQ{\mathbb Q}

\def\fa{\mathfrak{a}}
\def\fb{\mathfrak{b}} 
\def\fc{\mathfrak{c}} 
\def\fu{\mathfrak{u}} 
\def\fv{\mathfrak{v}} 

\def\Fq{\mathbb F_q}

\DeclareMathOperator{\Si}{Si}

\DeclareMathOperator{\Li}{Li}

\DeclareMathOperator{\Id}{Id}

\newcommand{\F}{\mathbb{F}}
\newcommand{\C}{\mathbb{C}}

\newcommand{\fs}{\mathfrak{s}}

\newcommand{\N}{\ensuremath \mathbb{N}}

\DeclareMathOperator{\depth}{depth}

\author[B.-H. Im]{Bo-Hae Im}
\address{
Dept. of Mathematical Sciences, KAIST,
291 Daehak-ro, Yuseong-gu,
Daejeon 34141, South Korea
}
\email{bhim@kaist.ac.kr}

\author[H. Kim]{Hojin Kim}
\address{
Dept. of Mathematical Sciences, KAIST,
291 Daehak-ro, Yuseong-gu,
Daejeon 34141, South Korea
}
\email{hojinkim@kaist.ac.kr}

\author[K. N. Le]{Khac Nhuan Le}
\address{
Normandie Université,
Université de Caen Normandie - CNRS,
Laboratoire de Mathématiques Nicolas Oresme (LMNO), UMR 6139,
14000 Caen, France.
}
\email{khac-nhuan.le@unicaen.fr}

\author[T. Ngo Dac]{Tuan Ngo Dac}
\address{
Normandie Université,
Université de Caen Normandie - CNRS,
Laboratoire de Mathématiques Nicolas Oresme (LMNO), UMR 6139,
14000 Caen, France.
}
\email{tuan.ngodac@unicaen.fr}

\author[L. H. Pham]{Lan Huong Pham}
\address{
Institute of Mathematics, Vietnam Academy of Science and Technology, 18 Hoang Quoc Viet, 10307 Hanoi, Viet Nam
}
\email{plhuong@math.ac.vn}

\title[Hopf algebras and MZV's]{Hopf algebras and multiple zeta values in positive characteristic}



\date{\today}

\begin{document}

\begin{abstract}
Multiples zeta values (MZV's for short) in positive characteristic were introduced by Thakur as analogues of classical multiple zeta values of Euler. In this paper we give a systematic study of algebraic structures of MZV's in positive characteristic. We construct both the stuffle algebra and the shuffle algebra of these MZV's and equip them with algebra and Hopf algebra structures. In particular, we completely solve a problem suggested by Deligne and Thakur \cite{Del17} in 2017 and establish Shi's conjectures \cite{Shi18}. The construction of the stuffle algebra is based on our recent work \cite{IKLNDP22}. 
\end{abstract}

\maketitle

\tableofcontents


\section{Introduction} \label{sec: Introduction}

\subsection{Classical multiple zeta values} ${}$\par

Let $\mathbb N=\{1,2,\dots\}$ be the set of positive integers and $\mathbb Z^{\geq 0}=\{0,1,2,\dots\}$ be the set of non-negative integers. The multiple zeta values (MZV's for short) are real positive numbers that were studied notably by Euler in the eighteenth century. They are given by the following convergent series
	\[ \zeta(n_1,\dots,n_r)=\sum_{0<k_1<\dots<k_r} \frac{1}{k_1^{n_1} \dots k_r^{n_r}} \]
where $n_i$ are positive integers with $n_i \geq 1$ and $n_r \geq 2$. Here $r$ is called the depth and $w := n_1+\dots+n_r$ is called the weight of the presentation $\zeta(n_1,\dots,n_r)$. 
These numbers generalize the zeta values 
	\[ \zeta(n)=\sum_{k>0} \frac{1}{k^n}, \quad \text{where } n \in \mathbb{N} \text{ and }  n \geq 2, \]
which were studied well before Riemann studied them as a function $\zeta(s)$ of a complex variable $s$, and its links with the distribution of primes. The simple zeta values, and more generally the multiple zeta values, still keep many secrets and play the role of fundamental constants. They are ubiquitous in many fields of mathematics and physics, in particular through the Feynman integrals which govern the interactions between elementary particles or through the Drinfeld associators coming from quantum groups and knot theory. 

The even zeta values are well understood. In fact, Euler proved in 1735 that, when $n$ is even, $\zeta(n)$ is a rational multiple of $\pi^n$. Thanks to Lindemann’s proof of the transcendence of $\pi$, it follows that all these numbers are transcendental. However, the odd zeta values are much more mysterious. Indeed, a folklore conjecture states
\begin{conjecture}
The numbers $\pi, \zeta(3), \zeta(5), \dots$ are all algebraically independent over $\bQ$.
\end{conjecture}

To our knowledge, we know nothing about the transcendence of odd zeta values. Concerning the irrationality of these numbers, Apéry \cite{Ape79} showed that $\zeta(3)$ is irrational, and Ball-Rivoal \cite{BR01} proved that there are infinitely many irrational numbers among the remaining odd zeta values (see \cite{Riv00,Riv02,Zud01} for related works). 

The product of two multiple zeta values is a linear combination, with integral coefficients, of multiple zeta values. For instance, Euler proved the following identity 
	\[ \zeta(m) \zeta(n)=\zeta(m,n)+\zeta(n,m)+\zeta(m+n) \]
for all integers $m,n \geq 2$. It follows that the $\bQ$-vector space $\mathcal Z$ spanned by all MZV's has an algebra structure. One can argue that the main goal of this theory is to understand all $\mathbb Q$-linear relations among MZV's. Unlike the algebra generated by simple zeta values, there are lots of linear relations among MZV's that endow $\mathcal Z$ with a rich combinatorial structure. One systematic way to produce linear relations among MZV's is to use the so-called extended double shuffle relations introduced by Ihara-Kaneko-Zagier \cite{IKZ06}. To do so, one first defines Hoffman's algebra~$\frak h$ and its subalgebras $\frak h^0 \subset \frak h^1 \subset \frak h$ with respect to the following algebra structures. Next, one constructs two algebra structures as particular cases of quasi-product algebras introduced by Hoffman \cite{Hof00}: the stuffle algebra $(\frak h^1,*)$ and the shuffle algebra $(\frak h,\shuffle)$. By regularization \cite[\S 2]{IKZ06}, there exist zeta maps which are $\bQ$-algebra homomorphisms
	\[ \zeta_*:(\frak h^1,*) \to \mathcal Z, \]
and 
	\[ \zeta_\shuffle:(\frak h,\shuffle) \to \mathcal Z, \]
which give rise to a generalization of the stuffle product and the shuffle product. The extended double shuffle relations are obtained by ``comparing'' the stuffle and shuffle products on $\frak h^1$ (see \cite[Theorem 2]{IKZ06} for a precise statement). Further, Ihara, Kaneko and Zagier formulated the following influential conjecture (see \cite[Conjecture 1]{IKZ06}): 
\begin{conjecture}[Ihara-Kaneko-Zagier's conjecture] \label{conj: IKZ}
The extended double shuffle relations exhaust all $\bQ$-linear relations among MZV's. 
\end{conjecture}
In particular, it implies Goncharov's conjecture which states that all $\mathbb Q$-linear relations among MZV's can be derived from those among MZV's of the same weight. 

Surprisingly, if $\mathcal{Z}_k$ denotes the $\mathbb Q$-vector space spanned by MZV's of weight $k$ for $k \in \mathbb{N}$, Zagier \cite{Zag94} and Hoffman \cite{Hof97} were able to predict the dimension and an explicit basis for  $\mathcal{Z}_k$.

\begin{conjecture}[Zagier's conjecture] \label{conj: Zagier}
We define a Fibonacci-like sequence of integers $d_k$ as follows. Letting $d_0=1, d_1=0$ and $d_2=1$ we define $d_k=d_{k-2}+d_{k-3}$ for $k \geq 3$. Then for $k \in \N$ we have 
	\[ \dim_{\mathbb Q} \mathcal Z_k = d_k. \]
\end{conjecture}

\begin{conjecture}[Hoffman's conjecture] \label{conj: Hoffman}
The $\mathbb Q$-vector space $\mathcal Z_k$ is generated by the basis consisting of MZV's of weight $k$ of the form $\zeta(n_1,\dots,n_r)$ with $n_i \in \{2,3\}$.
\end{conjecture}

The algebraic part of these conjectures which concerns upper bounds for $\dim_{\mathbb Q} \mathcal Z_k$ was solved by Terasoma \cite{Ter02}, Deligne-Goncharov \cite{DG05} and Brown \cite{Bro12} using the theory of mixed Tate motives. 

\begin{theorem}[Deligne-Goncharov, Terasoma] \label{thm: Zagier}
For $k \in \N$ we have $\dim_{\mathbb Q} \mathcal Z_k \leq d_k$.
\end{theorem}

\begin{theorem}[Brown] \label{thm: Hoffman}
The $\mathbb Q$-vector space $\mathcal Z_k$ is generated by MZV's of weight $k$ of the form $\zeta(n_1,\dots,n_r)$ with $n_i \in \{2,3\}$.
\end{theorem}

The proofs of this theorem use by a crucial manner different Hopf algebra structures of Hoffman's algebra $\frak h$ as described above. We mention that the transcendental part which concerns lower bounds for $\dim_{\mathbb Q} \mathcal Z_k$ is completely open. We refer the reader to \cite{BGF,Del13,Zag94} for more details and more exhaustive references.

\subsection{Multiple zeta values in positive characteristic} ${}$\par

There is a well-known analogy between number fields and function fields (see \cite{Iwa69,MW83,Wei39}). Inspired by Euler’s work on multiple zeta values and that of Carlitz \cite{Car35} on zeta values in positive characteristic, Thakur \cite{Tha17} introduced multiple zeta values attached to the affine line over a finite field. We now need to introduce some notations. Let $A=\Fq[\theta]$ be the polynomial ring in the variable $\theta$ over a finite field $\Fq$ of $q$ elements of characteristic $p>0$. We denote by $A_+$ the set of monic polynomials in $A$.  Let $K=\Fq(\theta)$ be the fraction field of $A$ equipped with the rational point $\infty$. Let $K_\infty$ be the completion of $K$ at $\infty$. We denote by $v_\infty$ the discrete valuation on $K$ corresponding to the place $\infty$ normalized such that $v_\infty(\theta)=-1$, and by $\lvert\cdot\rvert_\infty= q^{-v_\infty}$ the associated absolute value on $K$. 

In \cite{Car35} Carlitz introduced the Carlitz zeta values $\zeta_A(n)$ for $n \in \N$ given by
	\[ \zeta_A(n) := \sum_{a \in A_+} \frac{1}{a^n} \in K_\infty \]
which are analogues of classical special zeta values in the function field setting.  For any tuple of positive integers $\mathfrak s=(s_1,\ldots,s_r) \in \N^r$, Thakur \cite{Tha04} defined the characteristic $p$ multiple zeta value (MZV for short) $\zeta_A(\fs)$ or $\zeta_A(s_1,\ldots,s_r)$ by
\begin{equation*}
\zeta_A(\fs):=\sum \frac{1}{a_1^{s_1} \ldots a_r^{s_r}} \in K_\infty
\end{equation*}
where the sum runs through the set of tuples $(a_1,\ldots,a_r) \in A_+^r$ with $\deg a_1>\cdots>\deg a_r$. We call $r$ the depth of $\zeta_A(\fs)$ and $w(\fs): 
=s_1+\dots+s_r$ the weight of $\zeta_A(\fs)$. We note that Carlitz zeta values are exactly depth one MZV's. Thakur \cite{Tha09a} showed that all the MZV's do not vanish. We refer the reader to \cite{AT90,AT09,GP21,Gos96,LRT14,LRT21,Pel12,Tha04,Tha09,Tha10,Tha17,Tha20,Yu91} for more details about these objects.

Thakur proved that the product of two MZV's is a $K$-linear combination of MZV's and we call it the shuffle product in positive characteristic. As in the classical setting, the main goal of the theory is to understand all linear relations over $K$ among MZV's. Analogues of Zagier-Hoffman's conjectures in positive characteristic were formulated by Thakur in \cite[\S 8]{Tha17} and by Todd in \cite{Tod18}. In 2021, the fourth author \cite{ND21} solved these conjectures in the case of small weights. While the algebraic part uses tools introduced by Chen \cite{Che15}, Thakur \cite{Tha10,Tha17} and Todd \cite{Tod18}, the transcendental part uses the theory of $t$-motives and dual motives of Anderson \cite{And86,BP20,HJ20} and a powerful transcendental tool called the Anderson-Brownawell-Papanikolas criterion in \cite{ABP04} (see \cite{Pap08,Cha14,CPY19} for further development). Then the authors \cite{IKLNDP22} have developed a completely new approach and been able to solve these conjectures for all weights. More precisely, we prove (see \cite[Theorem B]{IKLNDP22}):

\begin{theorem}[Zagier's conjecture in positive characteristic]
For $w \in \N$ we denote by $\mathcal Z_w$ the $K$-vector space spanned by the MZV's of weight $w$. Letting 
\begin{align*}
d(w)=\begin{cases}
1 & \text{ if } w=0, \\
2^{w-1} & \text{ if } 1 \leq w \leq q-1, \\
2^{w-1}-1 & \text{ if } w=q,
\end{cases}
\end{align*}
we put $d(w)=\sum_{i=1}^q d(w-i)$ for $w>q$. Then for any $w \in \N$, we have
	\[ \dim_K \mathcal Z_w = d(w). \]
\end{theorem}

\begin{theorem}[Hoffman's conjecture in positive characteristic] 
We keep the above notation. A $K$-basis for $\mathcal Z_w$ is given by  $\mathcal T_w$ consisting of $\zeta_A(s_1,\ldots,s_r)$ of weight $w$ with  $s_i \leq q$ for $1 \leq i <r$, and $s_r<q$.
\end{theorem}

We mention that in {\it loc. cit.} we also extended these results to the setting of alternating multiple zeta values introduced by Harada \cite{Har21}. We note that the classical alternating multiple zeta values have been studied by Broadhurst, Deligne–Goncharov, Hoffman, Kaneko–Tsumura and many others due to many connections in different contexts. We refer the reader to\cite{Cha21,Del10,Har21,Hof19,Zha16} for some references. 

As pointed out by one of the referees of \cite{ND21}, we do not know any algebraic structures of MZV's in positive characteristic (see \cite[Remark 2.2, Part 1]{ND21}). As mentioned above, unlike the classical setting, the proofs of the above theorems are based on new ingredients: some operations introduced by Todd \cite{Tod18} and the fourth author \cite{ND21} as well as a transcendence criterion of Anderson-Brownawell-Papanikolas \cite{ABP04}. 

\subsection{Main results} ${}$\par

In this manuscript we present a systematic study of algebraic structures of MZV's in positive characteristic. This paper grew from an attempt to answer the question in \cite[Remark 2.2, Part 1]{ND21} raised by one of the referees of {\it loc. cit.}. It turned out that in a private letter to Thakur in 2017, Deligne \cite{Del17} went further and suggested the existence of a Hopf algebra structure of MZV's in positive characteristic. Subsequently, the composition space which plays the role of Hoffman's algebra in our context was suggested by Shuji Yamamoto \cite{Tha17}, and Shi in \cite{Shi18} formulated a conjectural Hopf algebra structure for this composition space (see \cite[Conjectures 3.2.2 and 3.2.11]{Shi18}).

In this paper we succeed in constructing both the Hopf stuffle algebra and the Hopf shuffle algebra in positive characteristic. In particular, we completely solve all the aforementioned questions and conjectures of Deligne, Thakur and Shi in the previous paragraph. Our approach is based on various tools of analytic, algebraic and combinatorial nature.

Let us give now more precise statements of our results. 

\subsubsection{Composition space} \label{section: Composition space} ${}$\par

We introduce the composition space $\frak C$ suggested by Shuji Yamamoto (see \cite[\S 5.2]{Tha17}) which plays the role of the Hoffman algebra $\frak h$ in our context. Let $\Sigma = \{x_n\}_{n \in \mathbb{N}}$ be a countable set equipped with the weight $w(x_n)=n$. The set $\Sigma$ will be called an alphabet and its elements will be called letters. A word over the alphabet $\Sigma$ is a finite string of letters. In particular, the empty word will be denoted by $1$. The depth $\depth(\fa)$ of a word $\fa$ is the number of letters in the string of $\fa$, so that $\depth(1) = 0$. The weight of a word is the sum of the weights of its letter and we put $w(1)=0$. Let $\langle \Sigma \rangle$ denote the set of all words over $\Sigma$. We endow $\langle \Sigma \rangle$ with the concatenation product defined by the following formula: 
\begin{equation*}
    x_{i_1} \dotsc x_{i_n} \cdot x_{j_1} \dotsc x_{j_m} = x_{i_1} \dotsc x_{i_n} x_{j_1} \dotsc x_{j_m}.
\end{equation*} 
Let $\mathfrak{C}=\Fq \langle \Sigma \rangle$ be the free $\Fq$-vector space with basis $\langle \Sigma \rangle$. The concatenation product extends to $\mathfrak{C}$ by linearity. For a letter $x_a \in \Sigma$ and an element $\fa \in \mathfrak{C}$, we write simply $x_a\fa$ instead of $x_a \cdot \fa$. For each nonempty word $\fa \in \langle \Sigma \rangle$, we can write $\fa = x_a \fa_-$ where $x_a$ is the first letter of $\fa$ and $\fa_-$ is the word obtained from $\fa$ by removing $x_a$.

\subsubsection{Shuffle algebra and shuffle map} ${}$\par

We define the unit $u:\Fq \to \frak C$ by sending $1$ to the empty word $1$. Next we define recursively two products on $\mathfrak{C}$ as $\Fq$-bilinear maps
\begin{align*}
   \diamond \colon \mathfrak{C} \times \mathfrak{C} \longrightarrow \mathfrak{C} \quad \text{and} \quad 
   \shuffle \colon \mathfrak{C} \times \mathfrak{C} \longrightarrow \mathfrak{C}
\end{align*}
by setting $1 \diamond \mathfrak{a} = \mathfrak{a} \diamond 1 = \mathfrak{a}, 1 \shuffle  \mathfrak{a} = \mathfrak{a} \shuffle  1 = \mathfrak{a}$ and
\begin{align*}
    \fa \diamond \fb &= \ x_{a + b}(\fa_- \shuffle  \fb_-) + \sum\limits_{i+j = a + b} \Delta^j_{a,b} x_i(x_j \shuffle  (\fa_- \shuffle  \fb_-)),\\
    \fa \shuffle  \fb &= \ x_{a}(\fa_- \shuffle  \fb) + x_{b}(\fa \shuffle  \fb_-) + \fa \diamond \fb,
\end{align*}
for any words $\fa,\fb \in \langle \Sigma \rangle$, 
Here the coefficients $\Delta^i_{a,b}$ are given by
\begin{align*}
\Delta^i_{a,b}=\begin{cases} (-1)^{a-1} {i-1 \choose a-1}+(-1)^{b-1} {i-1 \choose b-1} & \quad \text{if } (q-1) \mid i \text{ and } 0<i<a+b, \\
0 & \quad \text{otherwise}.
\end{cases}
\end{align*}
We call $\diamond$ the diamond product and $\shuffle $ the shuffle product.

Our first result gives an affirmative answer to both questions in \cite[Remark 2.2, Part 1]{ND21} and  \cite[Conjectures 3.2.2 and 3.2.11]{Shi18}. It reads as follows (see Theorems \ref{theorem: algebras} and \ref{thm: shuffle map}):
\begin{theoremx} \label{thm: A}
The spaces $(\mathfrak{C}, \diamond)$ and $(\mathfrak{C}, \shuffle )$ are commutative $\Fq$-algebras.  Further, for all words $\fa, \fb \in \frak C$ we have
\begin{align*}
\zeta_A(\fa\shuffle \fb) =\zeta_A(\fa) \, \zeta_A(\fb).
\end{align*}

If we denote by $\mathcal Z$ the $K$-vector space spanned by MZV's, then the homomorphism of $K$-algebras
\begin{align*}
Z_\shuffle:\frak C \otimes_{\Fq} K &\to \mathcal Z \\
\fa &\mapsto \zeta_A(\fa)
\end{align*}
is called the shuffle map in positive characteristic.
\end{theoremx}

To prove Theorem \ref{thm: A}, we are reduced to prove that the diamond and shuffle products are associative. It turns out that this claim is very hard to prove since a direct check involves complicated combinatorics as already noticed by Shi \cite[\S 3.1]{Shi18}. Our method uses analytic tools and consists of unpacking the coefficients $\Delta^i_{a,b}$ involved in the definition of the diamond product. Then we use the uniqueness of partial fraction decomposition to prove the desired associativity.  

\subsubsection{Shuffle Hopf algebra} ${}$\par

We also define recursively a product on $\mathfrak{C}$ as a $\Fq$-bilinear map
\begin{align*}
   \triangleright \colon \mathfrak{C} \times \mathfrak{C} \longrightarrow \mathfrak{C}
\end{align*}
by setting $1 \triangleright \mathfrak{a} = \mathfrak{a} \triangleright 1 = \mathfrak{a}$ and
\begin{align*}
    \fa \triangleright \fb &= \ x_{a}(\fa_- \shuffle  \fb)
\end{align*}
for any words $\fa,\fb \in \langle \Sigma \rangle$. We call $\triangleright$ the triangle product. We stress that the triangle product is neither commutative nor associative. Inspired by the work of Shi \cite[\S 3.2.3]{Shi18} we define a coproduct
	\[ \Delta: \frak C \to \frak C \otimes \frak C. \]
using $\triangleright$ rather than the concatenation on recursive steps for words with depth $>1$ (see \S \ref{sec: coproduct}). The counit $\epsilon:\frak C \to \Fq$ is defined as follows: $\epsilon(1)=1$ and $\epsilon(\fu)=0$ otherwise. 

We note that for quasi-shuffle algebras introduced by Hoffman \cite{Hof00} and their generalization, the coproduct is roughly speaking the deconcatenation. The coproduct~$\Delta$ defined as above is completely different from the deconcatenation and involves complicated combinatorics. We refer the reader to Proposition \ref{proposition: Hopf deltas depth 1 <= q^2} and Appendix~\ref{sec: numerical experiments} for numerical calculations of $\Delta$.

Our second result shows that this construction gives rise to a Hopf algebra structure of the shuffle algebra (see Theorem \ref{thm: Hopf algebra for shuffle product}). 
\begin{theoremx}  \label{thm: B}
The connected graded bialgebra $(\frak C,\shuffle,u,\Delta,\epsilon)$ is a connected graded Hopf algebra of finite type over $\Fq$.
\end{theoremx}

The proof of Theorem \ref{thm: B} is of algebraic nature by exploiting key properties among the diamond, shuffle and triangle products.

\subsubsection{Comparison with Shi's construction} ${}$\par

Next we study the coproduct $\Delta$ for letters in detail and prove some key properties in Proposition \ref{prop: formula delta xn}. As an immediate consequence, we deduce that the coproduct~$\Delta$ coincides with the coproduct introduced by Shi in \cite[\S 3.2.3]{Shi18}. Thus we settle Conjecture 3.2.11 of \cite{Shi18} (see Theorem \ref{thm: comparison with Shi's coproduct}).

\begin{theoremx} \label{thm: C}
Conjecture 3.2.11 in \cite{Shi18} holds.
\end{theoremx}

We note that to convince ourselves its validity, we made intensive numerical calculations (see Proposition \ref{proposition: Hopf deltas depth 1 <= q^2} and the Appendix \ref{sec: numerical experiments}). The proof of Theorem~\ref{thm: C} is based on combinatorial techniques.

\subsubsection{Stuffle algebra and stuffle Hopf algebra} ${}$\par

The stuffle algebra is easier to define. We introduce the stuffle product in the same way as that of $(\frak h^1,*)$ as above. The $*$ product 
\begin{align*} 
   * \colon \mathfrak{C} \times \mathfrak{C} \longrightarrow \mathfrak{C}
\end{align*}
is given by setting $1 *  \mathfrak{a} = \mathfrak{a} * 1 = \mathfrak{a}$ and
\begin{align*}
    \fa *  \fb &= \ x_{a}(\fa_- *  \fb) + x_{b}(\fa * \fb_-) + x_{a+b} (\fa_- * \fb_-)
\end{align*}
for any words $\fa,\fb \in \langle \Sigma \rangle$. We call $*$ the \textit{stuffle product} and see that $(\frak C,*)$ is a commutative $\Fq$-algebra.

We now define a coproduct $\Delta_*:\frak C  \to \frak C \otimes \frak C$ and a counit $\epsilon:\frak C \to \Fq$ by
	\[ \Delta_*(w)= \sum_{uv=w} u \otimes v \]
and 
\begin{align*}
\epsilon(w)=\begin{cases} 1 \quad \text{if } w=0, \\ 0 \quad \text{otherwise}, \end{cases}
\end{align*}
for any words $w \in \langle \Sigma \rangle$. 

We deduce from the work of Hoffman \cite{Hof00} that
\begin{theoremx}  \label{thm: D}
The stuffle algebra $(\frak C,*,u,\Delta_*,\epsilon)$ is a connected graded Hopf algebra of finite type over $\Fq$.
\end{theoremx}

\subsubsection{Stuffle map} ${}$\par

Finally, using our previous works \cite{IKLNDP22,ND21} we know that there is a connection between MZV's of Thakur and Carlitz's multiple polylogarithms. Thus we are able to construct a homomorphism of $K$-algebras called the stuffle map (see \S \ref{sec: stuffle map}).
\begin{theoremx}  \label{thm: E}
Recall that $\mathcal Z$ is the $K$-vector space spanned by MZV's. Then there exists a  homomorphism of $K$-algebras
	\[ Z_*:\frak C \otimes_{\Fq} K \to \mathcal Z \] 
called the stuffle map in positive characteristic.
\end{theoremx}

\subsection{Organization of the paper} ${}$\par

We briefly explain the plan of the manuscript. 
\begin{itemize}
\item In \S \ref{sec: Hopf algebras} we recall the definition and basic facts of different notions of algebras (algebras, coalgebras, bialgebras and Hopf algebras) that will be used in this paper. 

\item In \S \ref{sec: classical MZV's} we present the stuffle algebra and the shuffle algebra of MZV's in the classical setting. We recall different zeta maps and Hopf algebra structures associated to these algebras as well as important conjectures and theorems concerning these objects.

\item In \S \ref{sec: Thakur MZV's} we introduce the notion of MZV's in positive characteristic and define the composition space that is an analogue of Hoffman's algebra in this context.

\item In \S \ref{sec: shuffle algebra 0} and \S \ref{sec: algebra structure} we define the shuffle product $\shuffle$ and the shuffle map (see Theorem \ref{thm: shuffle map}). Using partial fraction decompositions we prove that the composition space equipped with the shuffle map is a commutative $\Fq$-algebra (see Theorem \ref{theorem: algebras}).

\item In \S \ref{sec: Hopf algebra structure} we define the coproduct $\Delta$ and the counit on the shuffle algebra. We prove that these give a Hopf algebra structure of the shuffle algebra (see Theorem \ref{thm: Hopf algebra for shuffle product}). In \S \ref{sec: depth one} we study the coproduct for words of depth one in detail and deduce that the coproduct $\Delta$ coincide with that introduced by Shi (see Theorem \ref{thm: comparison with Shi's coproduct}). Explicit formulas for the coproduct of such words are given in many cases in \S \ref{sec: explicit formula for small weights} and the Appendix \ref{sec: numerical experiments}.

\item Finally, in \S \ref{sec: stuffle algebra} we introduce the stuffle product $*$ and get the algebra and the Hopf algebra structures of the stuffle product as well as the stuffle map in our context (see Theorem  \ref{thm: Hopf algebra for stuffle product}). The key ingredient is a connection between MZV's and Carlitz multiple polylogarithms as explained in \cite{ND21} and \cite{IKLNDP22}.
\end{itemize}

\subsection*{Acknowledgments}  

The fourth author (T. ND.) would like to thank Dinesh Thakur for sharing a private letter of Deligne and for helpful discussions. 

The first named author (B.-H. Im) was supported by the National Research Foundation of Korea (NRF) grant no.~2020R1A2B5B01001835 funded by the Korea government (MSIT). Two of the authors (KN. L. and T. ND.) were partially supported by the Excellence Research Chair ``$L$-functions in positive characteristic and applications'' financed by the Normandy Region.  T. ND. and LH. P. were par-
tially supported by the Vietnam Academy of Science and Technology (VAST) under
grant no.~CTTH00.02/23-24 ``Arithmetic and Geometry of schemes over function fields and applications''.


\section{Review of Hopf algebras} \label{sec: Hopf algebras}

We briefly review the notion of Hopf algebras and follow closely the presentation of \cite[\S 3.2]{BGF}.  Throughout this section, we let $k$ denote a ground field. Unless otherwise specified, all tensor products will be assumed to be over $k$. 

\subsection{Hopf algebras} ${}$\par

For each $k$-vector space $H$, we denote by $\iota \colon H \otimes H \rightarrow H \otimes H $ the transposition map given by $x \otimes y \mapsto y \otimes x$.
\begin{definition}
    An \textit{algebra} over $k$ is a triple $(H, m,u)$ consisting of a $k$-vector space $H$ together with $k$-linear maps  $m \colon H \otimes H \rightarrow H$ called the \textit{multiplication} and $u \colon k \rightarrow H$ called the \textit{unit} such that the following diagrams are commutative:
   \begin{enumerate}[$(1)$]
       \item  associativity
    \begin{center}
        \begin{tikzcd}
H\otimes H \otimes H \arrow[r, "m \otimes \text{id}"] \arrow[d, " \text{id} \otimes m "'] & H \otimes H \arrow[d, "m"] \\
H \otimes H \arrow[r, "m"]                                                                     & H                              
\end{tikzcd}
    \end{center}
\item unitary
    \begin{center}
    \begin{tikzcd}
H \otimes k \arrow[rd] \arrow[r, "\text{id} \otimes u"] & H \otimes H \arrow[d, "m"] & k \otimes H \arrow[ld] \arrow[l, "u \otimes \text{id}"'] \\
                                                           & H                               &                                                 \end{tikzcd}
    \end{center}
    where the diagonal arrows are canonical isomorphisms. 
   \end{enumerate}  
    The algebra is said to be \textit{commutative} if the following diagram is commutative:
    \begin{center}
        \begin{tikzcd}
H \otimes H \arrow[rd, "m"'] \arrow[rr, "\iota"] &   & H \otimes H \arrow[ld, "m"] \\
                                                  & H &                                 
\end{tikzcd}
    \end{center}
\end{definition}
Turning all arrows around, one obtains the definition of coalgebras over $k$.
\begin{definition}
    A \textit{coalgebra} over $k$ is a triple $(H, \Delta,\epsilon)$ consisting of a $k$-vector space $H$ together with $k$-linear maps  $\Delta \colon H  \rightarrow H \otimes H$ called the \textit{coproduct} and $\epsilon \colon H \rightarrow k$ called the \textit{counit} such that the following diagrams are commutative:
\begin{enumerate}[$(1)$]
    \item  coassociativity
    \begin{center}
       \begin{tikzcd}
H \otimes H\otimes H                              & \otimes H \arrow[l, "\Delta\otimes \text{id} "'] \\
H \otimes H \arrow[u, "\text{id} \otimes \Delta"] & {H} \arrow[u, "\Delta"'] \arrow[l, "\Delta"']  
\end{tikzcd}
    \end{center}
 \item counitary
    \begin{center}
    \begin{tikzcd}
H \otimes k & H \otimes H \arrow[l, "\text{id} \otimes \epsilon "'] \arrow[r, "\epsilon \otimes \text{id}"] & k \otimes H \\
            & H \arrow[u, "\Delta"'] \arrow[ru] \arrow[lu]                                                  &            
\end{tikzcd}
    \end{center}
    where the diagonal arrows are canonical isomorphisms. 
    \end{enumerate}
    The coalgebra is said to be \textit{cocommutative} if the following diagram is commutative:
    \begin{center}
        \begin{tikzcd}
H \otimes H \arrow[rr, "\iota"] &                                              & H \otimes H \\
                            & H \arrow[ru, "\Delta"'] \arrow[lu, "\Delta"] &            
\end{tikzcd}
    \end{center}
\end{definition}

\begin{definition}
    A \textit{bialgebra} over $k$ is a tuple $(H,m, u, \Delta,\epsilon)$ consisting of an algebra $(H,m, u)$ over $k$  and a coalgebra $(H,\Delta,\epsilon)$ over $k$   which are compatible, i.e., the following diagrams are commutative:
\begin{enumerate}[$(1)$]
    \item product and coproduct
    \begin{center}
      \begin{tikzcd}
H \otimes H \arrow[d, "\Delta \otimes \Delta"'] \arrow[r, "m"]                  & H \arrow[r, "\Delta"] & H \otimes H                                                          \\
H \otimes H \otimes H \otimes H  \arrow[rr, "\text{id} \otimes \iota \otimes \text{id}"] &                       & H \otimes H \otimes H \otimes H  \arrow[u, "m \otimes m"']
\end{tikzcd}
    \end{center}
 \item unit and coproduct
    \begin{center}
    \begin{tikzcd}
H \arrow[r, "\Delta"]         & H \otimes H                                 \\
k \arrow[u, "u"] \arrow[r] & k \otimes k \arrow[u, "u \otimes u"']
\end{tikzcd}
    \end{center}
\item counit and product
    \begin{center}
    \begin{tikzcd}
H \arrow[d, "\epsilon"'] & H \otimes H \arrow[d, "\epsilon \otimes \epsilon"] \arrow[l, "m"'] \\
k                        & k \otimes k \arrow[l]                                                  
\end{tikzcd}
    \end{center}
  \item unit and counit
\begin{center}
    \begin{tikzcd}
k \arrow[rr, "\text{id}"] \arrow[rd, "u"'] &                           & k \\
                                              & H \arrow[ru, "\epsilon"'] &  
\end{tikzcd}
\end{center}
where the bottom arrows in the second diagram and the third diagram are canonical isomorphisms. 
\end{enumerate} 
\end{definition}

\begin{definition}
    A \textit{Hopf algebra} over $k$ is a bialgebra $(H,m, u, \Delta,\epsilon)$ over $k$ together with a $k$-linear map $S\colon H \rightarrow H$ called \textit{antipode} such that the following diagram is commutative:
    \begin{center}
        \begin{tikzcd}
                                                                    & H \otimes H \arrow[rr, "S \otimes \text{id}"] &                      & H \otimes H \arrow[rd, "m"]   &   \\
H \arrow[rr, "\epsilon"] \arrow[ru, "\Delta"] \arrow[rd, "\Delta"'] &                                               & k \arrow[rr, "u"] &                                    & H \\
                                                                    & H \otimes H \arrow[rr, "\text{id}\otimes S "] &                      & H \otimes H \arrow[ru, "m"'] &  
\end{tikzcd}
    \end{center}
\end{definition}

We note that a bialgebra does not always admit an antipode (see \cite[Exercise 3.83]{BGF} for an example).

\subsection{Graded Hopf algebras} ${}$\par

In this section we introduce the notion of connected graded bialgebras which will be useful. We will see later that every connected graded bialgebra has an antipode, and hence a Hopf algebra structure. 

\begin{definition} \label{defn: graded Hopf algebra}
\text{ }
\begin{enumerate}[$(1)$]
    \item   A bialgebra $(H,m, u, \Delta,\epsilon)$ over $k$ is said to be \textit{graded} if one can write $H$ as a direct sum of $k$-vectors subspaces
    \begin{equation*}
        H = \bigoplus \limits_{n = 0}^{\infty}H_n,
    \end{equation*}
    such that for all integers $r,s\geq 0$, we have
    \begin{equation*}
        m(H_r \otimes H_s) \subseteq H_{r + s} \quad \text{and} \quad \Delta(H_r) \subseteq \bigoplus \limits_{i + j = r} H_i \otimes H_j.
    \end{equation*}
    A graded bialgebra is said to be \textit{connected} if $H_0 = k$.
\item A \textit{graded Hopf algebra} is a Hopf algebra $H$ whose the underlying bialgebra is graded and the antipode $S$ satisfies $S(H_n) \subseteq H_n$.

\item A graded Hopf algebra is said to be \textit{connected} if $H_0 = k$.

\item A graded Hopf algebra is said to be {\it of finite type} if $H_n$ is a $k$-vector space of finite dimension.
\end{enumerate} 
\end{definition}

 The following proposition shows that a connected graded bialgebra automatically admits an antipode, thus it is always a Hopf algebra. It is given as an exercise in \cite[\S 11.2]{Swe69} (see also \cite[Exercise 3.84]{BGF}).
\begin{proposition} \label{prop: graded Hopf algebras}
    Let $(H,m, u, \Delta,\epsilon)$ be a connected graded bialgebra over $k$.
\begin{enumerate}[$(1)$]
\item For each element $x \in H_n$ with $n \geq 1$, we have
    \begin{equation*}
         \Delta(x) = 1 \otimes x + x \otimes 1 + \sum x_{(1)} \otimes x_{(2)},
    \end{equation*}
    where $\sum   x_{(1)} \otimes x_{(2)} \in \bigoplus \limits_{\substack{i,j > 0 \\ i + j = n}} H_i \otimes H_j$. Moreover, the counit $\epsilon$ vanishes on~$H_n$ for all $n \geq 1$.
\item  We continue the notation as in (1) and define recursively a $k$-linear map $S \colon H \rightarrow H$ given by
    \begin{align*}
S(x) = \begin{cases}  x & \quad \text{if } x \in H_0, \\
- x - \sum  m(S(x_{(1)}) \otimes x_{(2)}) & \quad \text{if } x \in H_n \text{ with } n \geq 1.
\end{cases}
\end{align*}
Then $H$ is a graded Hopf algebra whose antipode is $S$.
\end{enumerate}
\end{proposition}

\begin{proof}
See \cite[Lemma 2.1]{Ehr96}.
\end{proof}


\section{Classical multiple zeta values} \label{sec: classical MZV's}

In this section we review classical multiple zeta values studied by Euler in the late eighteenth century. In \S \ref{sec: quasi-shuffle algebras} we recall the theory of quasi-shuffle algebras introduced by Hoffman in 2000 and give basic facts such as the associated Hopf algebra structure. The stuffle algebra and the shuffle one defined in \S \ref{sec: Hoffman algebra} and \S \ref{sec: stuffle and shuffle algebras} are examples of this class. We explain their connection with MZV's and the regularization of Ihara-Kaneko-Zagier in \S \ref{sec: MZV's} and \S \ref{sec: regularization}.

\subsection{Quasi-shuffle algebras}  ${}$\par \label{sec: quasi-shuffle algebras}

We review the notion of quasi-shuffle product introduced by Hoffman \cite{Hof00}. 
Let $\Sigma = \{x_i\}_{i \in \mathbb{N}}$ be a countable set. To each letter $x_i$ we associate a weight $w(x_i) \in \bN$ and we suppose that for any $n \in \bN$ the set $\Sigma_n$ of letters of weight $n$ is finite.  In this context, we follow the notations in Section \ref{section: Composition space}.

Let $\langle \Sigma \rangle$ denote the set of all words over $\Sigma$. 
We denote by $\bQ \langle \Sigma \rangle$ (resp. $\bQ \Sigma$) the $\bQ$-vector space with $\langle \Sigma \rangle$ (resp. $\Sigma$) as a basis. The concatenation product extends to $\bQ \langle \Sigma \rangle$ by linearity so that $\bQ \langle \Sigma \rangle$ is a graded algebra with respect to weight.  

We set $\bar \Sigma=\Sigma \cup \{0\}$. Let $\diamond:\bar \Sigma \times \bar \Sigma \to \bar \Sigma$ be a commutative and associative product which preserves the grading. It means that this map satisfies the following properties: for all $a,b,c \in \Sigma$,
\begin{itemize}
\item $a \diamond 0=0$.
\item $a \diamond b=b \diamond a$.
\item $(a \diamond b) \diamond c=a \diamond (b \diamond c)$.
\item Either $a \diamond b=0$ or $w(a \diamond b)=w(a)+w(b)$.
\end{itemize}

We define a new product $*_\diamond$ on $\bQ \langle \Sigma \rangle$ recursively by setting $1 *_\diamond \fu=\fu *_\diamond 1=\fu$, and
	\[ a\fu *_\diamond b\fv=a(\fu *_\diamond b\fv)+b(a\fu *_\diamond \fv)+(a \diamond b)(\fu *_\diamond \fv)\]
for all letter $a,b \in \Sigma$ and all words $\fu,\fv \in \bQ \langle \Sigma \rangle$. This product is called the quasi-shuffle product associated to $\diamond$. Hoffman \cite[Theorem 2.1]{Hof00} showed that the vector space $\bQ \langle \Sigma \rangle$ equipped with the product $*_\diamond$ is a commutative $\bQ$-algebra.

We now define a coproduct $\Delta:\bQ \langle \Sigma \rangle  \to \bQ \langle \Sigma \rangle \otimes \bQ \langle \Sigma \rangle$ and a counit $\epsilon:\bQ \langle \Sigma \rangle \to \bQ$ by 
	\[ \Delta(\fu)= \sum_{\fa \fb=\fu} \fa \otimes \fb \]
and 
\begin{align*}
\epsilon(\fu)=\begin{cases} 1 \quad \text{if } \fu=1, \\ 0 \quad \text{otherwise}, \end{cases}
\end{align*}
for all words $\fu \in \langle \Sigma \rangle$. Hoffman \cite[Theorem 3.1]{Hof00} showed that $\bQ \langle \Sigma \rangle$ equipped with the $*_{\diamond}$-multiplication and $\Delta$-comultiplication is a bialgebra. Since both $*_{\diamond}$ and $\Delta$ respect the grading, Proposition \ref{prop: graded Hopf algebras} implies

\begin{theorem} \label{theorem: Hopf algebras for qp}
The algebra $\bQ \langle \Sigma \rangle$ with the $*_{\diamond}$-multiplication and $\Delta$-comultiplication is a graded Hopf algebra. Further, it is connected and of finite type.
\end{theorem}

Moreover, the antipode $S:\bQ \langle \Sigma \rangle \to \bQ \langle \Sigma \rangle$ is given explicitly in \cite[Theorem~3.2]{Hof00}: for any word $\fu=x_{i_1} \dots x_{i_n}$ we have 
	\[ S(\fu)= \sum_{(j_1,\dots,j_k)}  (-1)^k x_{i_1} \dots x_{i_{j_1}}*_{\diamond}x_{i_{j_1+1}} \dots x_{i_{j_1+j_2}}*_{\diamond} \dots *_{\diamond} x_{i_{j_1+\dots+j_{k-1}+1}} \dots x_{i_{j_1+\dots+j_k}}\]
where the sum runs through the set of all partitions $(j_1,\dots,j_k)$ of $n$.

For recent developments on quasi-shuffle products, we refer the reader to \cite{Hof15,Hof20,HI17,IKOO11}.

\subsection{Multiple zeta values of Euler} ${}$\par \label{sec: MZV's}

We now illustrate two examples of quasi-product that are related to the classical MZVs. Recall that multiple zeta values of Euler (MZV's for short) are real positive numbers given by the following convergent series
	\[ \zeta(n_1,\dots,n_r)=\sum_{0<k_1<\dots<k_r} \frac{1}{k_1^{n_1} \dots k_r^{n_r}}, \quad \text{where } n_i \geq 1, n_r \geq 2. \]
Here $r$ is called the depth and $w=n_1+\dots+n_r$ is called the weight of the presentation $\zeta(n_1,\dots,n_r)$. When $r=1$, we recover the special zeta values $\zeta(n)$ for $n \geq 2$ of the Riemann zeta function. It was studied by Euler in the eighteenth century and have been studied intensively especially in the last three decades by mathematicians and physicists. We refer the reader to  \cite{BGF, Zag94} for more details. As mentioned in these references, the main goal of this theory is to understand all $\mathbb Q$-linear relations among MZV's. We note that precise conjectures formulated by Zagier \cite{Zag94} and Hoffman \cite{Hof97} give the dimension and an explicit basis for the $\mathbb Q$-vector space $\mathcal{Z}_k$ spanned by MZV's of weight $k$ for $k \in \mathbb{N}$. The algebraic part of these conjectures was completely settled by Brown \cite{Bro12}, Deligne-Goncharov \cite{DG05} and Terosoma \cite{Ter02}. 

\subsection{The Hoffman algebra, stuffle product and shuffle product} ${}$\par \label{sec: Hoffman algebra}

In this section we take $k=\bQ$. Let $X$ be the alphabet with two letters $x_0, x_1$ with weight 1, that means $w(x_0)=w(x_1)=1$. We denote $\frak h=\bQ\langle X \rangle$ and call it the Hoffman algebra. A word in the alphabet $X$ is said to be positive if it is of the form $x_1 \fu$ and is said to be admissible if it is of the form $x_1 \fu x_0$. We denote by $\frak h^1$ (resp. $\frak h^0$) the subspace of $\frak h$ spanned by positive words (resp. admissible words).

For all $i \in \bN$ we put $z_i=x_1x_0^{i-1}$. Then $w(z_i)=i$. Let $Z$ be the alphabet with letters $\{z_i\}_{i \in \mathbb{N}}$. Then $\frak h^1=\bQ \langle Z \rangle$. We now equip the alphabet $Z$ with the commutative and associative product $\diamond:Z \times Z \to Z$ given by 
	\[ z_i \diamond z_j=z_{i+j} \] 
for all $i, j \in \bN$. The associated quasi-product on $\frak h^1=\bQ \langle Z \rangle$ will be denoted by $*$ and called the stuffle product. A word in $\frak h^1$ is called admissible if it can be expressed as $z_{s_1} \dots z_{s_\ell}$ with $s_\ell>1$. We note that $\frak h^0$ is the subspace generated by admissible words in $\frak h^1$ and that $(\frak h^0,*)$ is a subalgebra of $(\frak h^1,*)$. Further, the harmonic product on MZV's gives rise to a homomorphism of $\bQ$-algebras 
	\[ \zeta_*:\frak h^0 \to \bR \]
which sends an admissible word $z_{s_1} \dots z_{s_\ell}$ to the associated zeta value $\zeta(s_1,\dots,s_r)$, that means 
	\[ \zeta_*(\fu*\fv)=\zeta_*(\fu) \zeta_*(\fv) \]
for all words $\fu,\fv \in \frak h^0$. This map is called the stuffle zeta map.

We now recall the shuffle algebra. We endow $X$ with the trivial product $\diamond:X \times X \to X$ given by 
	\[ a \diamond b=0 \] 
for all $a,b \in X$. The associated quasi-product on $\frak h=\bQ \langle X \rangle$ will be denoted by $\shuffle$ and called the shuffle product. We see that $(\frak h^0,\shuffle)$ and $(\frak h^1,\shuffle)$ are subalgebras of $(\frak h,\shuffle)$. The shuffle product on MZV's defines a homomorphism of $\bQ$-algebras 
	\[ \zeta_\shuffle:\frak h^0 \to \bR \]
which sends an admissible word $z_{s_1} \dots z_{s_\ell}$ to the associated zeta value $\zeta(s_1,\dots,s_r)$, that means 
	\[ \zeta_\shuffle(\fu \shuffle \fv)=\zeta_\shuffle(\fu) \zeta_\shuffle(\fv) \]
for all words $w,v \in \frak h^0$. This map is called the shuffle zeta map.

Using these zeta maps yield the so-called double shuffle relations in the convergent case: for all words $\fu,\fv \in \frak h^0$,
	\[ \quad \zeta_*(\fu * \fv)=\zeta_\shuffle(\fu \shuffle \fv). \]

\subsection{Regularized zeta maps} ${}$\par \label{sec: regularization}

Following Ihara-Kaneko-Zagier \cite{IKZ06}, we note that the homomorphism of $(\frak h^0,*)$-algebras $\varphi_*:\frak h^0[T] \to \frak h^1$ which sends $T$ to $z_1$ is an isomorphism. Further, the following homomorphisms of $(\frak h^0,\shuffle)$-algebras 
\begin{align*}
& \varphi_{\shuffle}:\frak h^0[T] \to \frak h^1, \quad T \mapsto x_1, \\
& \varphi_{\shuffle}:\frak h^0[T,U] \to \frak h, \quad T \mapsto x_1, \, U \mapsto x_0,
\end{align*}
are isomorphisms.

Now we define the stuffle regularized zeta map 
\begin{equation} \label{eq: stuffle zeta map}
\zeta_*:\frak h^1 \to \bR
\end{equation}
as the composition 
	\[ \frak h^1 \to \frak h^0[T] \to \bR[T] \to \bR \]
where the first map is $\varphi_*^{-1}$, the second map is induced by the stuffle zeta map and the last one is the evaluation at $T=0$. Similarly, we define the shuffle regularized zeta map  
\begin{equation} \label{eq: shuffle zeta map}
\zeta_\shuffle:\frak h \to \bR
\end{equation}
as the composition 
	\[ \frak h \to \frak h^0[T,U] \to \bR[T,U] \to \bR \]
where the first map is $\varphi_{\shuffle}^{-1}$, the second map is induced by the shuffle zeta map and the last one is the evaluation at $T=U=0$.

We mention that Ihara, Kaneko and Zagier \cite{IKZ06} use the restriction of these maps on $\frak h^1$ to extend the previous double shuffle relations among MZV's. Further, they conjectured that these extended double shuffle relations exhaust all linear relations among MZV's (see Conjecture \ref{conj: IKZ}). As mentioned in the Introduction, other important conjectures in this theory are those of Zagier  and Hoffman which give precise dimension and a basis for the $\mathbb{Q}$-vector space spanned MZV's of fixed weight (see Conjectures \ref{conj: Zagier} and \ref{conj: Hoffman}).

Recall that the algebraic part of these conjectures was solved by Terasoma \cite{Ter02}, Deligne-Goncharov \cite{DG05} and Brown \cite{Bro12} using the theory of mixed Tate motives (see Theorems \ref{thm: Zagier} and \ref{thm: Hoffman}). The proofs of this theorem use by a crucial manner different Hopf algebra structures of Hoffman's algebra $\frak h$ as described above. We mention that the transcendental part which concerns lower bounds for $\dim_{\mathbb Q} \mathcal Z_k$ is completely open. We refer the reader to \cite{BGF,Del13,Zag94} for more details and more exhaustive references.

\subsection{Stuffle Hopf algebra and shuffle Hopf algebra}  ${}$\par \label{sec: stuffle and shuffle algebras}

By the work of Hoffman \cite{Hof00} the above algebras can be endowed with a richer structure, i.e., that of Hopf algebras. In fact, as a direct consequence of Theorem~\ref{theorem: Hopf algebras for qp}, we get two Hopf algebras for classical MZV's. 

\subsubsection{} The first graded Hopf algebra 
	\[ H_*=(\frak h^1,*) \]
comes from the stuffle product. We note that it is related to the algebra of quasi-symmetric functions over $k$ (see \cite{Ehr96,Hof15}). For some applications of Hopf algebra structure, we refer the reader to \cite{Hof15} (see also \cite{KXY21}).

\subsubsection{} The second graded Hopf algebra 
	\[ H_\shuffle=(\frak h,\shuffle) \] 
is the shuffle algebra (see \cite{Reu93}). Explicitly, 
\begin{itemize}
\item $\frak h=\bQ\langle x_0,x_1 \rangle$.

\item The coproduct is given by the shuffle product $\shuffle$.

\item The unit is given by the empty word $1$.

\item The coproduct $\Delta:\frak h  \to \frak h \otimes \frak h$ is given by the deconcatenation
	\[ \Delta(\fu)= \sum_{\fa \fb=\fu} \fa \otimes \fb \]
for any words $\fu \in \frak h$. 

\item The counit $\epsilon:\frak h \to \bQ$ is given by
\begin{align*}
\epsilon(\fu)=\begin{cases} 1 \quad \text{if } \fu=1, \\ 0 \quad \text{otherwise}. \end{cases}
\end{align*}

\item The antipode $S:\frak h \to \frak h$ is given by
	\[ S(x_{i_1} \dots x_{i_n})=(-1)^n x_{i_n} \dots x_{i_1}. \]
\end{itemize}

This Hopf algebra and its motivic version introduced by Goncharov \cite{Gon05} lie in the heart of the works of Brown \cite{Bro12}, Deligne-Goncharov \cite{DG05} and Terosoma \cite{Ter02} (see also \cite{BGF}).


\section{Multiple zeta values in positive characteristic} \label{sec: Thakur MZV's}

In \S \ref{sec: MZV's of Thakur} we recall the notion of multiple zeta values of Thakur and present the main goal and results of this theory. Then we define the composition space which plays the role of Hoffman's algebra in the function field setting (see \S \ref{sec: composition space}).

\subsection{Multiple zeta values in positive characteristic} ${}$\par \label{sec: MZV's of Thakur}

By analogy between number fields and function fields, Carlitz \cite{Car35} introduced zeta values in positive characteristic which are studied extensively in the last three decades. More recently, Thakur \cite{Tha04} generalized the work of Carlitz and defined analogues of multiple zeta values in positive characteristic. We recall some notations in the Introduction. The ring $A=\Fq[\theta]$ is the polynomial ring in the variable $\theta$ over a finite field $\Fq$ of $q$ elements of characteristic $p>0$. We recall that $A_+$ denotes the set of monic polynomials in $A$ and $K=\Fq(\theta)$ is the fraction field of $A$ equipped with the rational point $\infty$. Then $K_\infty$ is the completion of $K$ at $\infty$ and $\C_\infty$ is the completion of a fixed algebraic closure $\overline K$ of $K$ at $\infty$. We denote by $v_\infty$ the discrete valuation on $K$ corresponding to the place $\infty$ normalized such that $v_\infty(\theta)=-1$, and by $\lvert\cdot\rvert_\infty= q^{-v_\infty}$ the associated absolute value on $K$.  The unique valuation of~$\mathbb C_\infty$ which extends $v_\infty$ will still be denoted by $v_\infty$. 

For any tuple of positive integers $\mathfrak s=(s_1,\ldots,s_r) \in \N^r$, Thakur \cite{Tha04} defined the characteristic $p$ multiple zeta value (MZV for short) $\zeta_A(\fs)$ or $\zeta_A(s_1,\ldots,s_r)$ by
\begin{equation*}
\zeta_A(\fs):=\sum \frac{1}{a_1^{s_1} \ldots a_r^{s_r}} \in K_\infty
\end{equation*}
where the sum runs through the set of tuples $(a_1,\ldots,a_r) \in A_+^r$ with $\deg a_1>\cdots>\deg a_r$. We call $r$ the depth of $\zeta_A(\fs)$ and $w(\fs)=s_1+\dots+s_r$ the weight of $\zeta_A(\fs)$. We note that Carlitz zeta values are exactly depth one MZV's. Thakur \cite{Tha09a} showed that all the MZV's do not vanish. As in the classical setting, the main goal of the theory is to understand all linear relations over $K$ among MZV's. In fact, analogues of Zagier-Hoffman's conjectures in positive characteristic were formulated by Thakur in \cite[\S 8]{Tha17} and by Todd in \cite{Tod18} that we recall below. 

For $w \in \N$ we denote by $\mathcal Z_w$ the $K$-vector space spanned by the MZV's of weight $w$. We denote by $\mathcal T_w$ the set of $\zeta_A(\fs)$ where $\fs=(s_1,\ldots,s_r) \in \N^r$ of weight $w$ with $1\leq s_i\leq q$ for $1\leq i\leq r-1$ and $s_r<q$.

\begin{conjecture}[Zagier's conjecture in positive characteristic] \label{conj: dimension}
Letting 
\begin{align*}
d(w)=\begin{cases}
1 & \text{ if } w=0, \\
2^{w-1} & \text{ if } 1 \leq w \leq q-1, \\
2^{w-1}-1 & \text{ if } w=q,
\end{cases}
\end{align*}
we put $d(w)=\sum_{i=1}^q d(w-i)$ for $w>q$. Then for any $w \in \N$, we have
	\[ \dim_K \mathcal Z_w = d(w). \]
\end{conjecture}

\begin{conjecture}[Hoffman's conjecture in positive characteristic] \label{conj: basis}
A $K$-basis for $\mathcal Z_w$ is given by  $\mathcal T_w$ consisting of $\zeta_A(s_1,\ldots,s_r)$ of weight $w$ with  $s_i \leq q$ for $1 \leq i <r$, and $s_r<q$.
\end{conjecture}

These conjectures have been completely solved by the works of \cite{ND21} and \cite{IKLNDP22}. 

\subsection{The composition space} ${}$\par \label{sec: composition space}

In this section, we recall the notion of the composition space $\frak C$ as mentioned in the  Introduction. Let $\Sigma = \{x_n\}_{n \in \mathbb{N}}$ be a countable set equipped with the weight $w(x_n)=n$. The elements of $\Sigma$ will be called \textit{letters}. A \textit{word} over $\Sigma$ is a finite string of letters. In particular, the empty word will be denoted by $1$. Let $\langle \Sigma \rangle$ denote the set of all words over $\Sigma$. We endow $\langle \Sigma \rangle$ with the \textit{concatenation product} defined by the following formula: 
\begin{equation*}
    x_{i_1} \dotsc x_{i_n} \cdot x_{j_1} \dotsc x_{j_m} = x_{i_1} \dotsc x_{i_n} x_{j_1} \dotsc x_{j_m}.
\end{equation*} 
The composition space $\mathfrak{C}$ is the free $\Fq$-vector space $\Fq \langle \Sigma \rangle$ with basis $\langle \Sigma \rangle$. The concatenation product extends to $\mathfrak{C}$ by linearity.


\section{Shuffle algebra and shuffle map in positive characteristic} \label{sec: shuffle algebra 0}

This section aims to introduce the notion of the shuffle product of MZV's in positive characteristic. Then we define different products related to the associated shuffle algebra (see \S \ref{sec: shuffle algebra pos char}) and also the shuffle map (see Theorem \ref{thm: shuffle map}). 

\subsection{Shuffle product for power sums} ${}$\par

For $d \in \mathbb Z$ we introduce 
\begin{equation*} 
S_d(\fs):=\sum \frac{1}{a_1^{s_1} \ldots a_r^{s_r}} \in K_\infty
\end{equation*}
where the sum runs through the set of tuples $(a_1,\ldots,a_r) \in A_+^r$ with $d=\deg a_1>\ldots>\deg a_r$. Further, we define
\begin{equation*} 
S_{<d}(\fs):=\sum \frac{1}{a_1^{s_1} \ldots a_r^{s_r}} \in K_\infty
\end{equation*}
where the sum is over $(a_1,\ldots,a_r) \in A_+^r$ with $d>\deg a_1>\ldots>\deg a_r$. Thus
\begin{align*}
S_{<d}(\fs) =\sum_{i=0}^{d-1} S_i(\fs), \quad S_{d}(\fs) =S_d(s_1) S_{<d}(\fs_-)=S_d(s_1) S_{<d}(s_2,\dots,s_r).
\end{align*}
Here by convention we define empty sums to be $0$ and empty products to be $1$. In particular, $S_{<d}$ of the empty tuple is equal to $1$.

We briefly recall some results of Thakur concerning the shuffle product for power sums in \cite{Tha10} (see also \cite[\S 5.2]{Tha17}). Thakur first proved (see \cite[Theorems 1 and 2]{Tha10}) that for all $a,b \in \mathbb N$, there exist $\Delta^i_{a,b} \in \mathbb F_p$ for $0 < i <a+b$ such that for all $d \in \mathbb Z$,
\begin{equation} \label{eq:product depth1}
S_d(a)S_d(b)=S_d(a+b)+\sum_{0 < i <a+b} \Delta^i_{a,b} S_d(a+b-i,i).
\end{equation}
Shortly after, Chen \cite{Che15} gave explicit formulas for the coefficients $\Delta^i_{a,b}$ and proved
\begin{align*}
\Delta^i_{a,b}=\begin{cases} (-1)^{a-1} {i-1 \choose a-1}+(-1)^{b-1} {i-1 \choose b-1} & \quad \text{if } (q-1) \mid i \text{ and } 0<i<a+b, \\
0 & \quad \text{otherwise}.
\end{cases}
\end{align*}
Here we recall that for integers $a, b$ with $b \geq 0
$, 
\begin{equation*}
    \binom{a}{b} = \dfrac{a(a-1) \dotsc (a-b+1)}{b!}.
\end{equation*}
It should be remarked that $\binom{a}{b} = 0$ if $b > a \geq 0$.

\begin{proposition}[Chen] \label{prop: chen1}
Let $r,s$ be positive integers. For all $d \in \mathbb{Z}$, we have 
\begin{equation*}
    S_d(r)S_d(s) = S_d(r+s) + \sum_{\substack{i,j \in \N, \\ i+j=r+s}} \Delta^{j}_{r,s}S_d(i,j).
\end{equation*}
\end{proposition}

\begin{proposition} \label{prop: chen2}
Let $r,s$ be positive integers. For all $d \in \mathbb{Z}$, we have 
\begin{equation*}
    S_{<d}(r)S_{<d}(s) =  S_{<d}(r+s) + S_{<d}(r,s) + S_{<d}(s,r) +  \sum_{\substack{i,j \in \N, \\ i+j=r+s}} \Delta^{j}_{r,s}S_{<d}(i,j).
\end{equation*}
\end{proposition}

\begin{proof}
Using Proposition \ref{prop: chen1}, we have 
\begin{align*}
    S_{<d}(r)S_{<d}(s) &= \sum \limits_{k<d} S_k(r)S_k(s) + \sum \limits_{k<d} S_k(r)S_{<k}(s) + \sum \limits_{k<d} S_k(s)S_{<k}(r)\\
    &= \sum \limits_{k<d} \Bigg(S_k(r+s) + \sum_{\substack{i,j \in \N, \\ i+j=r+s}} \Delta^{j}_{r,s}S_k(i,j)\Bigg) + \sum \limits_{k<d} S_k(r,s) + \sum \limits_{k<d} S_k(s,r)\\ 
    &= S_{<d}(r+s) +  \sum_{\substack{i,j \in \N, \\ i+j=r+s}} \Delta^{j}_{r,s}S_{<d}(i,j) + S_{<d}(r,s) + S_{<d}(s,r). 
\end{align*}
This proves the proposition.
\end{proof}

\subsection{The shuffle algebra in positive characteristic} ${}$\par \label{sec: shuffle algebra pos char}

Recall that the composition space is defined as in \S \ref{section: Composition space}. We define the unit $u:\Fq \to \frak C$ by sending $1$ to the empty word $1$. Next we define recursively two products on $\mathfrak{C}$ as $\Fq$-bilinear maps
\begin{align*}
   \diamond \colon \mathfrak{C} \times \mathfrak{C} \longrightarrow \mathfrak{C} \quad \text{and} \quad 
   \shuffle \colon \mathfrak{C} \times \mathfrak{C} \longrightarrow \mathfrak{C}
\end{align*}
by setting $1 \diamond \mathfrak{a} = \mathfrak{a} \diamond 1 = \mathfrak{a}, 1 \shuffle  \mathfrak{a} = \mathfrak{a} \shuffle  1 = \mathfrak{a}$ and
\begin{align*}
    \fa \diamond \fb &= \ x_{a + b}(\fa_- \shuffle  \fb_-) + \sum\limits_{i+j = a + b} \Delta^j_{a,b} x_i(x_j \shuffle  (\fa_- \shuffle  \fb_-)),\\
    \fa \shuffle  \fb &= \ x_{a}(\fa_- \shuffle  \fb) + x_{b}(\fa \shuffle  \fb_-) + \fa \diamond \fb
\end{align*}
for any words $\fa,\fb \in \langle \Sigma \rangle$. We call $\diamond$ the \textit{diamond product} and $\shuffle $ the \textit{shuffle product}.

\begin{proposition} \label{prop: commutative}
 The diamond product and the shuffle product are commutative.
\end{proposition}

\begin{proof}
Let $\fa, \fb \in \langle \Sigma \rangle$ be two arbitrary words. It is suffices to show that 
\begin{equation} \label{eq: commutative}
    \fa \diamond \fb = \fb \diamond \fa \quad \text{and} \quad \fa \shuffle  \fb = \fb \shuffle  \fa .
\end{equation}
We proceed the proof by induction on $\depth(\fa) + \depth(\fb)$. If one of $\fa$ or $\fb$ is empty word, then \eqref{eq: commutative} holds trivially. We assume that \eqref{eq: commutative} holds when $\depth(\fa) + \depth(\fb) < n$ with $n \in \N$ and $n \geq 2$. We need to show that \eqref{eq: commutative} holds when $\depth(\fa) + \depth(\fb) = n$.
\par Indeed, we have 
\begin{align*}
    \fa \diamond \fb &= \ x_{a + b}(\fa_-\shuffle  \fb_-) + \sum\limits_{i+j = a + b} \Delta^j_{a,b} x_i(x_j \shuffle  (\fa_- \shuffle  \fb_-)),\\
    \fb \diamond \fa &= \ x_{b + a}(\fb_- \shuffle  \fa_-) + \sum\limits_{i+j = b + a} \Delta^j_{b,a} x_i(x_j \shuffle  (\fb_- \shuffle  \fa_-)).
\end{align*}
It follows from the induction hypothesis that $\fa_-\shuffle \fb_- = \fb_- \shuffle \fa_-$, hence $\fa \diamond \fb = \fb \diamond \fa$. On the other hand, we have 
\begin{align*}
    \fa \shuffle  \fb &= \ x_{a}(\fa_- \shuffle  \fb) + x_{b}(\fa \shuffle  \fb_-) + \fa \diamond \fb,\\
    \fb \shuffle  \fa &= \ x_{b}(\fb_- \shuffle  \fa) + x_{a}(\fb \shuffle  \fa_-) + \fb \diamond \fa.
\end{align*}
It follows from the induction hypothesis and the above arguments that $\fa_- \shuffle  \fb = \fb \shuffle  \fa_-, \fa \shuffle  \fb_- = \fb_- \shuffle  \fa$ and $\fa \diamond \fb = \fb \diamond \fa$, hence $\fa \shuffle  \fb = \fb \shuffle  \fa$. This proves the proposition.
\end{proof}

We next define recursively a product on $\mathfrak{C}$ as a $\Fq$-bilinear map
\begin{align*}
   \triangleright \colon \mathfrak{C} \times \mathfrak{C} \longrightarrow \mathfrak{C}
\end{align*}
by setting $1 \triangleright \mathfrak{a} = \mathfrak{a} \triangleright 1 = \mathfrak{a}$ and
\begin{align*}
    \fa \triangleright \fb &= \ x_{a}(\fa_- \shuffle  \fb)
\end{align*}
for any words $\fa,\fb \in \langle \Sigma \rangle$. We call $\triangleright$ the triangle product. We stress that the triangle product is neither commutative nor associative, as one verifies at once. 

\begin{lemma} \label{lem: triangle formulas}
For all words $\fa,\fb \in \langle \Sigma \rangle$, we have
\begin{enumerate}
    \item $\fa \diamond \fb = (x_a \diamond x_b) \triangleright (\fa_- \shuffle  \fb_-)$ 
    \item $\fa \shuffle  \fb = \fa \triangleright \fb + \fb \triangleright \fa + \fa \diamond \fb$.
\end{enumerate}
\end{lemma}

\begin{proof}
 We have
\begin{align*}
    \fa \diamond \fb 
    &= \ x_{a+b}(\fa_- \shuffle  \fb_-) + \sum \limits_{i+j = a + b} \Delta^j_{a,b} x_i (x_j \shuffle  (\fa_- \shuffle  \fb_-))\\
    &= \ x_{a+b} \triangleright (\fa_- \shuffle  \fb_-) + \sum \limits_{i+j = a + b} \Delta^j_{a,b} x_ix_j \triangleright (\fa_- \shuffle  \fb_-)\\
    &= \ (x_{a+b} + \sum \limits_{i+j = a + b} \Delta^j_{a,b} x_ix_j) \triangleright (\fa_- \shuffle  \fb_-)\\
    &= \ (x_a \diamond x_b) \triangleright (\fa_- \shuffle  \fb_-).
\end{align*}
This proves part $(1)$. Part $(2)$ is straightforward from the commutativity of the shuffle product. We finish the proof.
\end{proof}

\subsection{The shuffle map in positive characteristic} ${}$\par

For all $d \in \bZ$, we define two $\mathbb{F}_q$-linear maps
\begin{equation*}
    S_{<d} \colon \frak C \rightarrow K_{\infty} \quad \text{and} \quad \zeta_A \colon \frak C \rightarrow K_{\infty},
\end{equation*}
which map the empty word to the element $1 \in K_{\infty}$, and map any word $x_{s_1} \dotsc x_{s_r}$ to $S_{<d}(s_1, \dotsc, s_r)$ and $\zeta_A(s_1, \dotsc, s_r)$, respectively. The main result of this section reads as follows:
\begin{theorem} \label{thm: shuffle map}
For all words $\fa, \fb \in \frak C$ and for all $d \in \bZ$ we have
\begin{align*}
S_{<d}(\fa\shuffle \fb) &=S_{<d}(\fa) \, S_{<d}(\fb), \\
\zeta_A(\fa\shuffle \fb) &=\zeta_A(\fa) \, \zeta_A(\fb).
\end{align*}
\end{theorem}

\begin{proof}
See \cite[Theorem 3.1.4]{Shi18}.
\end{proof}

We denote by $\mathcal{Z}_w$ (resp. $\mathcal{Z}$) the $K$-vector space spanned by MZV's of weight $w$ (resp. by MZV's). Then the $K$-linear map \[ Z_\shuffle:\frak C \otimes_{\Fq} K \to \mathcal Z, \] which sends a word $\fa \in \frak C$ to $\zeta_A(\fa)$, is a homomorphism of $K$-algebras, and is called the shuffle map in positive characteristic.


\section{Algebra structure of the shuffle algebra} \label{sec: algebra structure}

The main goal of this section is to prove that the composition space $\frak C$ equipped with the shuffle product $\shuffle $ given by Thakur is an algebra (see Theorem \ref{theorem: algebras}). The key point is to show the associativity of the shuffle product which consequently solves \cite[Conjecture 3.2.2]{Shi18}. In fact, the associativity property of $(\frak C,\shuffle )$ turned out to be very hard to prove (see \cite[Remark 2.2, Part 1]{ND21}) as pointed out one of the referees of \cite{ND21}. Our method is of algebraic nature. It consists of unpacking the nature of the coefficients $\Delta^i_{a,b}$ appearing in the shuffle product of Thakur by using partial fractional decompositions. 

We mention that the associativity could follow from a transcendental approach. As mentioned in \cite[\S 3.2.1]{Shi18}, it follows from a conjecture of Thakur about the $\Fq$-linear independence of MZV's (see \cite[3.2.3]{Shi18}). 

\subsection{Expansions for $S_d$ of  depth one} ${}$\par

\subsubsection{} Let $r,s, t$ be positive integers. We first expand $(S_d(s) S_d(t))S_d(r)$. We have 
\begin{align*}
    &(S_d(r)S_d(s))S_d(t)\\
    &= \left(S_d(r+s) + \sum_{i+j=r+s} \Delta^j_{r,s} S_d(i,j)\right)S_d(t)\\
    &= S_d(r+s)S_d(t) + \sum_{i+j=r+s} \Delta^j_{r,s} S_d(i,j) S_d(t)\\
    &= \left(S_d(r+s+t) + \sum_{i+j=r+s+t} \Delta^j_{r+s,t} S_d(i,j) \right) \\
    &+ \sum_{i+j=r+s} \Delta^j_{r,s} \left(S_d(i+t,j)+\sum_{i_1+j_1=i+t} \Delta^{j_1}_{i,t} S_d(i_1) S_{<d}(j_1) S_{<d}(j)\right).
\end{align*}
The sum of the terms of depth $2$ in the above expansion is the following:
\begin{align} \label{eq: 1}
    \sum_{i+j=r+s+t} \Delta^j_{r+s,t} S_d(i,j)+\sum_{i+j=r+s} \Delta^j_{r,s} \left(S_d(i+t,j)+\sum_{i_1+j_1=i+t} \Delta^{j_1}_{i,t} S_d(i_1,j_1+j)\right)
\end{align}
The sum of the terms of depth $3$ in the above expansion is the following:
\begin{align*}
	\sum_{i+j=r+s} \Delta^j_{r,s} \sum_{i_1+j_1=i+t} \Delta^{j_1}_{i,t} \left(S_d(i_1,j_1,j)+S_d(i_1,j,j_1) + \sum \limits_{i_2 + j_2 =  j+j_1 } \Delta_{j,j_1}^{j_2} S_d(i_1,i_2,j_2)\right).
\end{align*}

\subsubsection{} We next expand $S_d(r) (S_d(s)S_d(t))$.
\begin{align*}
    &S_d(r)(S_d(s)S_d(t))\\
    &= S_d(r)\left(S_d(s+t) + \sum_{i+j=s+t} \Delta^j_{s,t} S_d(i,j)\right) \\
    &= S_d(r)S_d(s+t) + \sum_{i+j=s+t} \Delta^j_{s,t} S_d(r)S_d(i,j)\\
    &= \left(S_d(r+s+t) + \sum_{i+j=r+s+t} \Delta^j_{r,s+t} S_d(i,j) \right) \\
    &+ \sum_{i+j=s+t} \Delta^j_{s,t} \left(S_d(i+r,j)+\sum_{i_1+j_1=i+r} \Delta^{j_1}_{i,r} S_d(i_1) S_{<d}(j_1) S_{<d}(j)\right).
\end{align*}
The sum of the terms of depth $2$ in the above expansion is the following:
\begin{align} \label{eq: 2}
	\sum_{i+j=r+s+t} \Delta^j_{r,s+t} S_d(i,j)+\sum_{i+j=s+t} \Delta^j_{s,t} \left(S_d(i+r,j)+\sum_{i_1+j_1=i+r} \Delta^{j_1}_{i,r} S_d(i_1,j_1+j)\right).
\end{align}
The sum of the terms of depth $3$ in the above expansion is the following:
\begin{align*}
	\sum_{i+j=s+t} \Delta^j_{s,t} \sum_{i_1+j_1=i+r} \Delta^{j_1}_{i,r} \left(S_d(i_1,j_1,j)+S_d(i_1,j,j_1) + \sum \limits_{i_2 + j_2 =  j+j_1 } \Delta_{j,j_1}^{j_2} S_d(i_1,i_2,j_2)\right).
\end{align*}


\subsubsection{Partial fraction decompositions} ${}$\par

In this section, we will give some partial fraction decompositions which will be used in the next section. To simplify the notations, for $j, r \in \mathbb{N}$, we set
\begin{align*}
\nabla^j_r :&= (-1)^{r} {j-1 \choose r-1}.
\end{align*}

\begin{lemma} \label{lem: basic indentity}
Let $r , s$ be positive integers. The following equality of rational function holds:
\begin{equation*}
    \frac{1}{X^r Y^s}=\sum_{\substack{i,j \in \N, \\ i+j=r+s}} \left[\frac{{j-1 \choose s-1}}{(X+Y)^j X^i}+\frac{{j-1 \choose r-1}}{(X+Y)^j Y^i}\right].
\end{equation*}
\end{lemma}

\begin{proof}
See \cite[Lemma 1.49]{BGF}.
\end{proof}

\begin{lemma} \label{lem: decomposition}
Let $R$ be a ring, and let $a, b$ be elements in $R$ such that $a \ne b$. Then for all positive integers $r , s$, we have the following partial fraction decomposition:
\begin{equation*} 
\frac{1}{(X+a)^r (X+b)^s}=\sum_{\substack{i,j \in \N, \\ i+j=r+s}} \left[\frac{\nabla^j_s}{(a-b)^j (X+a)^i}+\frac{\nabla^j_r}{(b-a)^j (X+b)^i}\right].
\end{equation*}
\end{lemma}

\begin{proof}
The result follows immediately from Lemma \ref{lem: basic indentity} by taking $X = X + a$ and $Y = - (X + b)$.
\end{proof}

It should be remarked that Chen's formula is based on the following identity.

\begin{corollary} \label{cor: identity}
Let $a, b$ be elements in $A$ such that $a \ne b$. Then for all positive integers $r , s$, we have 
\begin{equation*} 
\frac{1}{a^r b^s}=\sum_{\substack{i,j \in \N, \\ i+j=r+s}} \left[\frac{\nabla^j_s}{(a-b)^j a^i}+\frac{\nabla^j_r}{(b-a)^j b^i}\right].
\end{equation*}
\end{corollary}

\begin{proof}
The result is straightforward from Lemma \ref{lem: decomposition} when $X = 0$ and $R = A$.
\end{proof}

\par Let $r,s,t$ be positive integers. Let $R$ be a ring, and let $u, v$ be elements in $R$ such that $u \ne 0, v \ne 0 $ and $u \ne v$. Consider the following fractional function 
\begin{equation*}
    P(A) = \frac{1}{A^{r}(A+u)^s(A+v)^t}
\end{equation*}
Set $B= A + u$ and $C = A+v$. Using Lemma \ref{lem: decomposition}, we give the partial fraction decomposition of $P(A)$ in two ways. First, we expand $P(A)$ from the left to the right as follows:
\begin{align} \label{eq: PFD L-R}
   \frac{1}{A^{r}B^s} \cdot \frac{1}{C^t} &=\sum \limits_{i+j = r+s} \left(\frac{\nabla^j_s}{(A-B)^jA^i} + \frac{\nabla^j_r}{(B-A)^jB^i}\right)\frac{1}{C^t}\\ \notag
    &= \sum \limits_{i+j = r+s} \sum \limits_{i_1+j_1 = i + t}\frac{\nabla^j_s}{(A-B)^j} \frac{\nabla^{j_1}_t}{(A-C)^{j_1}}\frac{1}{A^{i_1}}\\ \notag
    &+ \sum \limits_{i+j = r+s} \sum \limits_{i_1+j_1 = i + t}\frac{\nabla^j_s}{(A-B)^j} \frac{\nabla^{j_1}_i}{(C-A)^{j_1}}\frac{1}{C^{i_1}}\\\notag
    &+ \sum \limits_{i+j = r+s} \sum \limits_{i_1+j_1 = i + t}\frac{\nabla^j_r}{(B-A)^j} \frac{\nabla^{j_1}_t}{(B-C)^{j_1}}\frac{1}{B^{i_1}}\\\notag
    &+ \sum \limits_{i+j = r+s} \sum \limits_{i_1+j_1 = i + t}\frac{\nabla^j_r}{(B-A)^j} \frac{\nabla^{j_1}_i}{(C-B)^{j_1}}\frac{1}{C^{i_1}}\notag.
\end{align}
Next we expand $P(A)$ from the right to the left as follows:
\begin{align} \label{eq: PFD R-L}
  \frac{1}{A^r} \cdot \frac{1}{B^{s}C^t} &=  \frac{1}{A^r}  \sum \limits_{i+j = s + t} \left(\frac{\nabla^j_t}{(B-C)^jB^i} + \frac{\nabla^j_s}{(C-B)^jC^i}\right)\\\notag
    &=  \sum \limits_{i+j = s + t}\sum \limits_{i_1+j_1 = i + r} \frac{\nabla^j_t}{(B-C)^j}\frac{\nabla^{j_1}_i}{(A-B)^{j_1}}\frac{1}{A^{i_1}}\\\notag
    &+ \sum \limits_{i+j = s + t}\sum \limits_{i_1+j_1 = i + r} \frac{\nabla^j_t}{(B-C)^j}\frac{\nabla^{j_1}_r}{(B-A)^{j_1}}\frac{1}{B^{i_1}}\\\notag
    &+\sum \limits_{i+j = s + t}\sum \limits_{i_1+j_1 = i + r}\frac{\nabla^j_s}{(C-B)^j}\frac{\nabla^{j_1}_i}{(A-C)^{j_1}}\frac{1}{A^{i_1}}\\\notag
     &+\sum \limits_{i+j = s + t}\sum \limits_{i_1+j_1 = i + r}\frac{\nabla^j_s}{(C-B)^j}\frac{\nabla^{j_1}_r}{(C-A)^{j_1}}\frac{1}{C^{i_1}}.
\end{align}

\subsection{Associativity for $S_d$ of depth one}

\subsubsection{Main results} ${}$\par

To simplify the notations, for $j,k \in \mathbb{N}$, we set
\begin{align*}
\delta_j :&= \sum \limits_{\lambda \in \Fq^\times} \frac{1}{\lambda^j} = \begin{cases}  -1 & \quad \text{if } (q - 1) \mid j , \\
0 & \quad \text{otherwise,} 
\end{cases}\\
    \delta_{j,k} :&= \sum \limits_{\lambda,\mu \in \Fq^\times; \lambda \ne \mu} \frac{1}{\lambda^j \mu ^k}.
\end{align*}
The following formulas will be used frequently later:
\begin{align}
    \Delta_{r,s}^j &= \delta_j(\nabla^j_r + \nabla^j_s)\label{eq: formula 1},\\
    \delta_j &= \delta_j(-1)^j\label{eq: formula 2},\\
    \delta_j\delta_k &= \delta_{j+k} + \delta_{j,k}.\label{eq: formula 3}
\end{align}
Consider all the cases of tuples $(a,b,c) \in  A_+(d) \times A_+(d) \times A_+(d)$, we set
\begin{align*}
    M_0 &= \{(a,b,c) \in A_+(d) \times A_+(d) \times A_+(d) : a = b = c \},\\
    M_1 &= \{(a,b,c) \in A_+(d) \times A_+(d) \times A_+(d) : a = b \ne c \},\\
    M_2 &= \{(a,b,c) \in A_+(d) \times A_+(d) \times A_+(d) : a \ne b = c \},\\
    M_3 &= \{(a,b,c) \in A_+(d) \times A_+(d) \times A_+(d) : c = a \ne b \},\\
    M_4 &= \{(a,b,c) \in A_+(d) \times A_+(d) \times A_+(d) : a \ne b, b \ne c, c \ne a \}.
\end{align*}
The last set $M_4$ can be decomposed by the following partition:
\begin{equation*}
    M_4 = N_0 \sqcup N_1 \sqcup N_2 \sqcup N_3 \sqcup N_4.
\end{equation*}
Here 
\begin{itemize}
    \item the set $N_0$ consists of tuples $(a,b,c) \in M_4$ where $b = a + \lambda f, c = a + \mu f$ with $\lambda, \mu \in \Fq^\times$ such that $\lambda \ne \mu$ and $f \in A_+$ such that $d > \deg f$;
    \item the set $N_1$ consists of tuples $(a,b,c) \in M_4$ where $b = a + \lambda f, c = a + \mu u$ with $\lambda, \mu \in \Fq^\times$  and $f,u \in A_+$ such that $d > \deg f > \deg u$;
    \item the set $N_2$ consists of tuples $(a,b,c) \in M_4$ where $b = a + \lambda u, c = a + \mu f$ with $\lambda, \mu \in \Fq^\times$  and $f,u \in A_+$ such that $d > \deg f > \deg u$;
    \item the set $N_3$ consists of tuples $(a,b,c) \in M_4$ where $b = a + \lambda f, c = a + \lambda f + \mu u$ with $\lambda, \mu \in \Fq^\times$  and $f,u \in A_+$ such that $d > \deg f > \deg u$;
    \item the set $N_4$ consists of tuples $(a,b,c) \in M_4$ where $b = a + \lambda f, c = a + \mu f + \eta u$ with $\lambda, \mu, \eta \in \Fq^\times$ such that $\lambda \ne \mu$  and $f,u \in A_+$ such that $d > \deg f > \deg u$.
\end{itemize}

\begin{proposition} \label{prop: associativity}
Let $r,s , t$ be positive integers. The expansions using Corollary~\ref{cor: identity} of 
\begin{equation*}
    \sum \frac{1}{a^rb^s} \cdot \frac{1}{c^t} \quad \text{and} \quad \sum \frac{1}{a^r} \cdot \frac{1}{b^s c^ t}
\end{equation*}
yield the same expression in terms of power sums, where $(a,b,c)$ ranges over all tuples in the sets $M_0,M_1,M_2,M_3,N_0,N_1,N_2,N_3,N_4$. Moreover,
\begin{enumerate}[$(1)$]
\item the expansion of $\sum \frac{1}{a^rb^s} \cdot \frac{1}{c^t}$ (respectively, $\sum \frac{1}{a^r} \cdot \frac{1}{b^s c^ t}$), where $(a,b,c)$ ranges over all tuples in $M_0$, yields the term of depth $1$ in the expression of $(S_d(r)S_d(s))S_d(t)$ (respectively, $S_d(r)(S_d(s)S_d(t))$);
\item the expansion of $\sum \frac{1}{a^rb^s} \cdot \frac{1}{c^t}$ (respectively, $\sum \frac{1}{a^r} \cdot \frac{1}{b^s c^ t}$), where $(a,b,c)$ ranges over all tuples in $M_1 \sqcup M_2\sqcup M_3\sqcup N_0$, yields the terms of depth $2$ in the expression of $(S_d(r)S_d(s))S_d(t)$ (respectively, $S_d(r)(S_d(s)S_d(t))$);
\item  the expansion of $\sum \frac{1}{a^rb^s} \cdot \frac{1}{c^t}$ (respectively, $\sum \frac{1}{a^r} \cdot \frac{1}{b^s c^ t}$), where $(a,b,c)$ ranges over all tuples in $N_1 \sqcup N_2\sqcup N_3\sqcup N_4$, yields the terms of depth $3$ in the expression of $(S_d(r)S_d(s))S_d(t)$ (respectively, $S_d(r)(S_d(s)S_d(t))$).
\end{enumerate}  
\end{proposition}

As a direct consequence, we obtain the following theorem.

\begin{theorem} \label{thm: assoc S_d depth 1}
Let $r,s , t$ be positive integers. For all $d \in \mathbb{N}$, the expansions using Chen's formula of $(S_d(s) S_d(t))S_d(r)$ and $S_d(s) (S_d(t)S_d(r))$ yield the same expression in terms of power sums.
\end{theorem}

The rest of this section is devoted to a proof of Proposition \ref{prop: associativity}.

\subsubsection{Depth 1 terms: proof of Proposition \ref{prop: associativity}, Part 1}  ${}$\par
It is obvious that 
\begin{equation*}
    \sum \limits_{(a,b,c) \in M_0} \frac{1}{a^rb^s} \cdot \frac{1}{c^t} = \sum \limits_{(a,b,c) \in M_0} \frac{1}{a^r} \cdot \frac{1}{b^s c^ t} = S_d(r+s+t),
\end{equation*}
which yields the term of depth $1$ in the expression of $(S_d(r)S_d(s))S_d(t)$ and  $S_d(r)(S_d(s)S_d(t))$.

\subsubsection{Depth 2 terms: proof of Proposition \ref{prop: associativity}, Part 2}  ${}$\par
Consider the following cases:

\medskip
\noindent \textbf{Case 1:} $(a,b,c)$ ranges over all tuples in $M_1$.
\par For each $\lambda \in \Fq^\times$, consider the rational function 
\begin{equation*}
    P_{\lambda}(A,F) = \frac{1}{A^{r+s}(A+\lambda F)^t}.
\end{equation*}
We will deduce the partial fraction decomposition of $  P_{\lambda}(A,F) $ by the following process: we first give the partial fraction decomposition of $  P_{\lambda}(A,F) $ in variable $A$ with coefficients are rational functions of the form $Q(F) \in \Fq(F)$; then we continue to give the partial fraction decomposition of $Q(F)$ in variable $F$ whose coefficients are elements in $\Fq$.
\par Set $C= A + \lambda F$. Using Lemma \ref{lem: decomposition}, we proceed the process in two ways. First, we expand $P_{\lambda}(A,F)$ from the left to the right as follows:
\begin{align} \label{eq: L-R M1}
    \frac{1}{A^{r+s}} \cdot \frac{1}{C^t} &= \sum \limits_{i+j = r + s+ t} \left(\frac{\nabla^j_t}{(A-C)^jA^i} + \frac{\nabla^j_{r+s}}{(C-A)^jC^i} \right)\\ \notag
    &=\sum \limits_{i+j = r + s+ t} \left(\frac{\nabla^j_t}{(-\lambda F)^jA^i} + \frac{\nabla^j_{r+s}}{(\lambda F)^jC^i} \right)\\ \notag
    &= \sum \limits_{i+j = r + s+ t} \frac{\nabla^j_t}{(-\lambda)^j}\frac{1}{ A^iF^j} + \sum \limits_{i+j = r + s+ t}\frac{\nabla^j_{r+s}}{\lambda^j}\frac{1}{C^iF^j}. \notag
\end{align}
Next we expand $P_{\lambda}(A,F)$ from the right to the left as follows:
\begin{align} \label{eq: R-L M1}
   \frac{1}{A^r} \cdot\frac{1}{A^{s}C^t} &= \frac{1}{A^r} \sum \limits_{i+j = s+ t} \left(\frac{\nabla^j_t}{(A-C)^jA^{i}} + \frac{\nabla^j_{s}}{(C-A)^jC^i} \right)\\ \notag
    &=  \sum \limits_{i+j = s+ t} \frac{\nabla^j_t}{(A-C)^j}\frac{1}{A^{i+r}}\\ \notag
    &+ \sum \limits_{i+j = s+ t} \sum \limits_{i_1+j_1 = i+r} \frac{\nabla^j_{s}}{(C-A)^j}\frac{\nabla^{j_1}_{i}}{(A-C)^{j_1}} \frac{1}{A^{i_1}}\\\notag
    &+ \sum \limits_{i+j = s+ t} \sum \limits_{i_1+j_1 = i+r} \frac{\nabla^j_{s}}{(C-A)^j}\frac{\nabla^{j_1}_{r}}{(C-A)^{j_1}}\frac{1}{C^{i_1}}\\\notag
    &=  \sum \limits_{i+j = s+ t} \frac{\nabla^j_t}{(-\lambda F)^j}\frac{1}{A^{i+r}}\\ \notag
    &+ \sum \limits_{i+j = s+ t} \sum \limits_{i_1+j_1 = i+r} \frac{\nabla^j_{s}}{(\lambda F)^j}\frac{\nabla^{j_1}_{i}}{(-\lambda F)^{j_1}} \frac{1}{A^{i_1}}\\ \notag
    &+ \sum \limits_{i+j = s+ t} \sum \limits_{i_1+j_1 = i+r} \frac{\nabla^j_{s}}{(\lambda F)^j}\frac{\nabla^{j_1}_{r}}{(\lambda F)^{j_1}}\frac{1}{C^{i_1}}\\ \notag
    &=  \sum \limits_{i+j = s+ t} \frac{\nabla^j_t}{(-\lambda)^j}\frac{1}{A^{i+r}F^j}\\ \notag
    &+ \sum \limits_{i+j = s+ t} \sum \limits_{i_1+j_1 = i+r}
    \frac{\nabla^j_{s} (-1)^{j_1}\nabla^{j_1}_{i}}{\lambda^{j+j_1}}\frac{1}{A^{i_1}F^{j+j_1}}\\ \notag
    &+ \sum \limits_{i+j = s+ t} \sum \limits_{i_1+j_1 = i+r} \frac{\nabla^j_{s} \nabla^{j_1}_{r}}{\lambda^{j+j_1}}\frac{1}{C^{i_1}F^{j+j_1}}. \notag
\end{align}
For each $(a,b,c)\in M_1$, there exist $\lambda \in \Fq^\times$ and $f \in A_+$ with $d > \deg f$ such that $c = a + \lambda f$. Replacing $A = a, F = f$, one deduces from \eqref{eq: L-R M1} that 
\begin{align*} 
   \sum \limits_{(a,b,c) \in M_1}& \frac{1}{a^rb^s} \cdot \frac{1}{c^t}\\
   &=  \sum\limits_{\substack{a,f \in A_+ \\ d = \deg a > \deg f}} \sum \limits_{\lambda \in \Fq^\times} \frac{1}{a^{r+s}} \cdot \frac{1}{(a+\lambda f)^t}\\
   &= \sum \limits_{i+j = r + s+ t} \delta_j\nabla^j_t \sum\limits_{\substack{a,f \in A_+ \\ d = \deg a > \deg f}}\frac{1}{ a^if^j} + \sum \limits_{i+j = r + s+ t}\delta_j\nabla^j_{r+s}\sum\limits_{\substack{c,f \in A_+ \\ d = \deg c > \deg f}}\frac{1}{c^if^j}\\
   &= \sum \limits_{i+j = r + s+ t} \delta_j\nabla^j_t S_d(i,j) + \sum \limits_{i+j = r + s+ t}\delta_j\nabla^j_{r+s}S_d(i,j)\\
   &= \sum \limits_{i+j = r + s+ t}\Delta^j_{r+s,t}S_d(i,j).
\end{align*}
Similarly, one deduces from \eqref{eq: R-L M1} that
\begin{align*}
    \sum \limits_{(a,b,c) \in M_1} &\frac{1}{a^r} \cdot \frac{1}{b^sc^t} 
    \\&=  \sum\limits_{\substack{a,f \in A_+ \\ d = \deg a > \deg f}} \sum \limits_{\lambda \in \Fq^\times} \frac{1}{a^r} \cdot \frac{1}{a^s(a+\lambda f)^t}\\
    &=  \sum \limits_{i+j = s+ t} \delta_j\nabla^j_t \sum\limits_{\substack{a,f \in A_+ \\ d = \deg a > \deg f}}\frac{1}{a^{i+r}f^j}\\ \notag
  &+ \sum \limits_{i+j = s+ t} \sum \limits_{i_1+j_1 = i+r}\delta_{j+j_1}\nabla^j_{s} (-1)^{j_1}\nabla^{j_1}_{i} \sum\limits_{\substack{a,f \in A_+ \\ d = \deg a > \deg f}}\frac{1}{a^{i_1}f^{j+j_1}}\\ \notag
    &+ \sum \limits_{i+j = s+ t} \sum \limits_{i_1+j_1 = i+r} \delta_{j+j_1}\nabla^j_{s} \nabla^{j_1}_{r}\sum\limits_{\substack{c,f \in A_+ \\ d = \deg c > \deg f}}\frac{1}{c^{i_1}f^{j+j_1}}\\
    &=  \sum \limits_{i+j = s+ t} \delta_j\nabla^j_t S_d(i+r,j)\\ \notag
  &+ \sum \limits_{i+j = s+ t} \sum \limits_{i_1+j_1 = i+r}\delta_{j+j_1}\nabla^j_{s} (-1)^{j_1}\nabla^{j_1}_{i} S_d(i_1,j+j_1)\\ \notag
    &+ \sum \limits_{i+j = s+ t} \sum \limits_{i_1+j_1 = i+r} \delta_{j+j_1}\nabla^j_{s} \nabla^{j_1}_{r}S_d(i_1,j+j_1).
\end{align*}
Since the partial fraction decomposition of $P_{\lambda}(A,F)$ obtained from the process is unique, it follows that the above expansions of 
\begin{equation*}
    \sum \limits_{(a,b,c)\in M_1} \frac{1}{a^rb^s} \cdot \frac{1}{c^t} \quad \text{and} \quad \sum\limits_{(a,b,c)\in M_1} \frac{1}{a^r} \cdot \frac{1}{b^s c^ t}
\end{equation*}
yield the same expression in terms of power sums.

\medskip
\noindent \textbf{Case 2:} $(a,b,c)$ ranges over all tuples in $M_2$.
\par For each $\lambda \in \Fq^\times$, consider the rational function 
\begin{equation*}
    P_{\lambda}(A,F) = \frac{1}{A^{r}(A+\lambda F)^{s+t}}.
\end{equation*}
Set $B = A + \lambda F$. From the same process as in the the case of $M_1$, we expand $P_{\lambda}(A,F)$ in two ways. First, we expand $P_{\lambda}(A,F)$ from the left to the right as follows:
\begin{align} \label{eq: L-R M2}
    \frac{1}{A^rB^s} \cdot \frac{1}{B^t} &= 
    \sum \limits_{i+j = r + s} \left(\frac{\nabla^j_s}{(A-B)^jA^i} + \frac{\nabla^j_r}{(B-A)^jB^i}\right) \frac{1}{B^t}\\ \notag
    &= \sum \limits_{i+j = r + s}\sum \limits_{i_1+j_1 = i + t} \frac{\nabla^j_s}{(A-B)^j}
    \frac{\nabla^{j_1}_t}{(A-B)^{j_1}}\frac{1}{A^{i_1}}\\\notag
    &+ \sum \limits_{i+j = r + s}\sum \limits_{i_1+j_1 = i + t} \frac{\nabla^j_s}{(A-B)^j}
    \frac{\nabla^{j_1}_i}{(B-A)^{j_1}}\frac{1}{B^{i_1}}\\\notag
    &+\sum \limits_{i+j = r + s} \frac{\nabla^j_r}{(B-A)^j}\frac{1}{B^{i+t}}\\\notag
    &= \sum \limits_{i+j = r + s}\sum \limits_{i_1+j_1 = i + t} \frac{\nabla^j_s}{(-\lambda F)^j}
    \frac{\nabla^{j_1}_t}{(-\lambda F)^{j_1}}\frac{1}{A^{i_1}}\\\notag
    &+ \sum \limits_{i+j = r + s}\sum \limits_{i_1+j_1 = i + t} \frac{\nabla^j_s}{(-\lambda F)^j}
    \frac{\nabla^{j_1}_i}{(\lambda F)^{j_1}}\frac{1}{B^{i_1}}\\\notag
    &+\sum \limits_{i+j = r + s} \frac{\nabla^j_r}{(\lambda F)^j}\frac{1}{B^{i+t}}\\\notag
    &= \sum \limits_{i+j = r + s}\sum \limits_{i_1+j_1 = i + t} \frac{\nabla^j_s
    \nabla^{j_1}_t}{(-\lambda)^{j+j_1}}\frac{1}{A^{i_1}F^{j+j_1}}\\\notag
    &+ \sum \limits_{i+j = r + s}\sum \limits_{i_1+j_1 = i + t} \frac{\nabla^j_s
    (-1)^{j_1}\nabla^{j_1}_i}{(-\lambda)^{j+j_1}}\frac{1}{B^{i_1}F^{j+j_1}}\\\notag
    &+\sum \limits_{i+j = r + s} \frac{\nabla^j_r}{\lambda^j}\frac{1}{B^{i+t}F^j}.\notag
\end{align}
Next we expand $P_{\lambda}(A,F)$ from the right to the left as follows:
\begin{align}\label{eq: R-L M2}
    \frac{1}{A^r} \cdot \frac{1}{B^{s+t}} &= \sum \limits_{i+j = r + s + t} \left(\frac{\nabla^j_{s+t}}{(A-B)^jA^i} + \frac{\nabla^j_{r}}{(B-A)^jB^i}\right)\\\notag
    &=\sum \limits_{i+j = r + s + t} \left(\frac{\nabla^j_{s+t}}{(-\lambda F)^jA^i} + \frac{\nabla^j_{r}}{(\lambda F)^jB^i}\right)\\\notag
    &=\sum \limits_{i+j = r + s + t} \frac{\nabla^j_{s+t}}{(-\lambda)^j} \frac{1}{A^iF^j} + \sum \limits_{i+j = r + s + t} \frac{\nabla^j_{r}}{\lambda^j}\frac{1}{B^iF^j}.\notag
\end{align}
For each $(a,b,c)\in M_2$, there exist $\lambda \in \Fq^\times$ and $f \in A_+$ with $d > \deg f$ such that $b = a + \lambda f$. Replacing $A = a, F = f$, one deduces from \eqref{eq: L-R M2} that 
\begin{align*}
    \sum \limits_{(a,b,c) \in M_2} & \frac{1}{a^rb^s} \cdot \frac{1}{c^t}\\
    &=  \sum\limits_{\substack{a,f \in A_+ \\ d = \deg a > \deg f}} \sum \limits_{\lambda \in \Fq^\times} \frac{1}{a^{r}(a+\lambda f)^s} \cdot \frac{1}{(a+\lambda f)^t}\\
    &= \sum \limits_{i+j = r + s}\sum \limits_{i_1+j_1 = i + t} \delta_{j+j_1}\nabla^j_s
    \nabla^{j_1}_t
    \sum\limits_{\substack{a,f \in A_+ \\ d = \deg a > \deg f}}\frac{1}{a^{i_1}f^{j+j_1}}\\\notag
    &+ \sum \limits_{i+j = r + s}\sum \limits_{i_1+j_1 = i + t} \delta_{j+j_1}\nabla^j_s
    (-1)^{j_1}\nabla^{j_1}_i\sum\limits_{\substack{b,f \in A_+ \\ d = \deg b > \deg f}}\frac{1}{b^{i_1}f^{j+j_1}}\\\notag
    &+\sum \limits_{i+j = r + s} \delta_j\nabla^j_r\sum\limits_{\substack{b,f \in A_+ \\ d = \deg b > \deg f}}\frac{1}{b^{i+t}f^j}\\\notag
    &= \sum \limits_{i+j = r + s}\sum \limits_{i_1+j_1 = i + t} \delta_{j+j_1}\nabla^j_s
\nabla^{j_1}_t
    S_d(i_1,j+j_1)\\\notag
    &+ \sum \limits_{i+j = r + s}\sum \limits_{i_1+j_1 = i + t}  \delta_{j+j_1}\nabla^j_s
    (-1)^{j_1}\nabla^{j_1}_i S_d(i_1,j+j_1)\\\notag
    &+\sum \limits_{i+j = r + s} \delta_j\nabla^j_r S_d(i+t,j).
\end{align*}
Similarly, one deduces from \eqref{eq: R-L M2} that
\begin{align*}
    \sum \limits_{(a,b,c) \in M_2} &\frac{1}{a^r} \cdot \frac{1}{b^sc^t}\\
    &=  \sum\limits_{\substack{a,f \in A_+ \\ d = \deg a > \deg f}} \sum \limits_{\lambda \in \Fq^\times} \frac{1}{a^r} \cdot \frac{1}{(a+\lambda f)^{s+t}}\\
    &=\sum \limits_{i+j = r + s + t} \delta_j\nabla^j_{s+t} \sum\limits_{\substack{a,f \in A_+ \\ d = \deg a > \deg f}}\frac{1}{a^if^j} + \sum \limits_{i+j = r + s + t} \delta_j\nabla^j_{r}\sum\limits_{\substack{b,f \in A_+ \\ d = \deg b > \deg f}}\frac{1}{b^if^j}\\
    &=\sum \limits_{i+j = r + s + t} \delta_j\nabla^j_{s+t} S_d(i,j) + \sum \limits_{i+j = r + s + t} \delta_j\nabla^j_{r}S_d(i,j)\\
    &=\sum \limits_{i+j = r + s + t} \Delta^j_{r,s+t} S_d(i,j).\notag
\end{align*}
Since the partial fraction decomposition of $P_{\lambda}(A,F)$ obtained from the process is unique, it follows that the above expansions of 
\begin{equation*}
    \sum \limits_{(a,b,c)\in M_2} \frac{1}{a^rb^s} \cdot \frac{1}{c^t} \quad \text{and} \quad \sum\limits_{(a,b,c)\in M_2} \frac{1}{a^r} \cdot \frac{1}{b^s c^ t}
\end{equation*}
yield the same expression in terms of power sums.

\medskip
\noindent \textbf{Case 3:} $(a,b,c)$ ranges over all tuples in $M_3$.
\par For each $\lambda \in \Fq^\times$, consider the rational function 
\begin{equation*}
    P_{\lambda}(A,F) = \frac{1}{A^{r}(A+\lambda F)^{s}A^t}.
\end{equation*}
Set $B = A + \lambda F$. From the same process as in the the case of $M_1$, we expand $P_{\lambda}(A,F)$ in two ways. First, we expand $P_{\lambda}(A,F)$ from the left to the right as follows:
\begin{align} \label{eq: L-R M3}
    \frac{1}{A^rB^s} \cdot \frac{1}{A^t} &= \sum \limits_{i+j = r+s} \left(\frac{\nabla^j_s}{(A-B)^jA^i} + \frac{\nabla^j_r}{(B-A)^jB^i}\right) \frac{1}{A^t}\\\notag
    &= \sum \limits_{i+j = r+s}\frac{\nabla^j_s}{(A-B)^j}\frac{1}{A^{i+t}}\\\notag
    &+ \sum \limits_{i+j = r+s}\sum \limits_{i_1+j_1 = i+t}\frac{\nabla^j_r}{(B-A)^j}\frac{\nabla^{j_1}_t}{(B-A)^{j_1}}\frac{1}{B^{i_1}}\\\notag
    &+ \sum \limits_{i+j = r+s}\sum \limits_{i_1+j_1 = i+t}\frac{\nabla^j_r}{(B-A)^j}\frac{\nabla^{j_1}_i}{(A-B)^{j_1}}\frac{1}{A^{i_1}}\\\notag
    &= \sum \limits_{i+j = r+s}\frac{\nabla^j_s}{(-\lambda F)^j}\frac{1}{A^{i+t}}\\\notag
    &+ \sum \limits_{i+j = r+s}\sum \limits_{i_1+j_1 = i+t}\frac{\nabla^j_r}{(\lambda F)^j}\frac{\nabla^{j_1}_t}{(\lambda F)^{j_1}}\frac{1}{B^{i_1}}\\\notag
    &+ \sum \limits_{i+j = r+s}\sum \limits_{i_1+j_1 = i+t}\frac{\nabla^j_r}{(\lambda F)^j}\frac{\nabla^{j_1}_i}{(-\lambda F)^{j_1}}\frac{1}{A^{i_1}}\\\notag
     &= \sum \limits_{i+j = r+s}\frac{\nabla^j_s}{(-\lambda)^j}\frac{1}{A^{i+t}F^j}\\\notag
    &+ \sum \limits_{i+j = r+s}\sum \limits_{i_1+j_1 = i+t}\frac{\nabla^j_r\nabla^{j_1}_t}{\lambda^{j+j_1}}\frac{1}{B^{i_1}F^{j+j_1}}\\\notag
    &+ \sum \limits_{i+j = r+s}\sum \limits_{i_1+j_1 = i+t}\frac{\nabla^j_r(-1)^{j_1}\nabla^{j_1}_i}{\lambda^{j+j_1}}\frac{1}{A^{i_1}F^{j+j_1}}.\notag
\end{align}
Next we expand $P_{\lambda}(A,F)$ from the right to the left as follows:
\begin{align}\label{eq: R-L M3}
    \frac{1}{A^r} \cdot \frac{1}{B^{s}A^t} &= \frac{1}{A^r} \sum \limits_{i+j = s+ t} \left(\frac{\nabla^j_t}{(B-A)^jB^i} + \frac{\nabla^j_s}{(A-B)^jA^i}\right)\\\notag
    &= \sum \limits_{i+j = s+ t}\sum \limits_{i_1+j_1 = i+r} \frac{\nabla^j_t}{(B-A)^j}\frac{\nabla^{j_1}_i}{(A-B)^{j_1}}\frac{1}{A^{i_1}}\\\notag
    &+ \sum \limits_{i+j = s+ t}\sum \limits_{i_1+j_1 = i+r} \frac{\nabla^j_t}{(B-A)^j}\frac{\nabla^{j_1}_r}{(B-A)^{j_1}}\frac{1}{B^{i_1}}\\\notag
    &+\sum \limits_{i+j = s+ t}\frac{\nabla^j_s}{(A-B)^j}\frac{1}{A^{i+r}}\\\notag
    &= \sum \limits_{i+j = s+ t}\sum \limits_{i_1+j_1 = i+r} \frac{\nabla^j_t}{(\lambda F)^j}\frac{\nabla^{j_1}_i}{(-\lambda F)^{j_1}}\frac{1}{A^{i_1}}\\\notag
    &+ \sum \limits_{i+j = s+ t}\sum \limits_{i_1+j_1 = i+r} \frac{\nabla^j_t}{(\lambda F)^j}\frac{\nabla^{j_1}_r}{(\lambda F)^{j_1}}\frac{1}{B^{i_1}}\\\notag
    &+\sum \limits_{i+j = s+ t}\frac{\nabla^j_s}{(-\lambda F)^j}\frac{1}{A^{i+r}}\\\notag
    &= \sum \limits_{i+j = s+ t}\sum \limits_{i_1+j_1 = i+r} \frac{\nabla^j_t(-1)^{j_1}\nabla^{j_1}_i}{\lambda^{j+j_1}}\frac{1}{A^{i_1}F^{j+j_1}}\\\notag
    &+ \sum \limits_{i+j = s+ t}\sum \limits_{i_1+j_1 = i+r} \frac{\nabla^j_t\nabla^{j_1}_r}{\lambda^{j+j_1}}\frac{1}{B^{i_1}F^{j+j_1}}\\\notag
    &+\sum \limits_{i+j = s+ t}\frac{\nabla^j_s}{(-\lambda)^j}\frac{1}{A^{i+r}F^j}. \notag
\end{align}
For each $(a,b,c)\in M_3$, there exist $\lambda \in \Fq^\times$ and $f \in A_+$ with $d > \deg f$ such that $b = a + \lambda f$. Replacing $A = a, F = f$, one deduces from \eqref{eq: L-R M3} that 
\begin{align*}
    \sum \limits_{(a,b,c) \in M_3} &\frac{1}{a^rb^s} \cdot \frac{1}{c^t} 
    \\&=  \sum\limits_{\substack{a,f \in A_+ \\ d = \deg a > \deg f}} \sum \limits_{\lambda \in \Fq^\times} \frac{1}{a^{r}(a+\lambda f)^s} \cdot \frac{1}{a^t}\\
    &= \sum \limits_{i+j = r+s}\delta_j\nabla^j_s\sum\limits_{\substack{a,f \in A_+ \\ d = \deg a > \deg f}}\frac{1}{a^{i+t}f^j}\\\notag
    &+ \sum \limits_{i+j = r+s}\sum \limits_{i_1+j_1 = i+t}\delta_{j+j_1}\nabla^j_r\nabla^{j_1}_t\sum\limits_{\substack{b,f \in A_+ \\ d = \deg b > \deg f}}\frac{1}{b^{i_1}f^{j+j_1}}\\\notag
    &+ \sum \limits_{i+j = r+s}\sum \limits_{i_1+j_1 = i+t}\delta_{j+j_1}\nabla^j_r(-1)^{j_1}\nabla^{j_1}_i\sum\limits_{\substack{a,f \in A_+ \\ d = \deg a > \deg f}}\frac{1}{a^{i_1}f^{j+j_1}}.\notag\\
    &= \sum \limits_{i+j = r+s}\delta_j\nabla^j_s S_d(i+t,j)\\\notag
    &+ \sum \limits_{i+j = r+s}\sum \limits_{i_1+j_1 = i+t}\delta_{j+j_1}\nabla^j_r\nabla^{j_1}_tS_d(i_1,j+j_1)\\\notag
    &+ \sum \limits_{i+j = r+s}\sum \limits_{i_1+j_1 = i+t}\delta_{j+j_1}\nabla^j_r(-1)^{j_1}\nabla^{j_1}_iS_d(i_1,j+j_1).\notag
\end{align*}
Similarly, one deduces from \eqref{eq: R-L M3} that
\begin{align*}
    \sum \limits_{(a,b,c) \in M_3} &\frac{1}{a^r} \cdot \frac{1}{b^sc^t}\\
    &=  \sum\limits_{\substack{a,f \in A_+ \\ d = \deg a > \deg f}} \sum \limits_{\lambda \in \Fq^\times} \frac{1}{a^r} \cdot \frac{1}{(a+\lambda f)^{s}a^t}\\
    &= \sum \limits_{i+j = s+ t}\sum \limits_{i_1+j_1 = i+r} \delta_{j+j_1}\nabla^j_t(-1)^{j_1}\nabla^{j_1}_i\sum\limits_{\substack{a,f \in A_+ \\ d = \deg a > \deg f}}\frac{1}{a^{i_1}f^{j+j_1}}\\\notag
    &+ \sum \limits_{i+j = s+ t}\sum \limits_{i_1+j_1 = i+r} \delta_{j+j_1}\nabla^j_t\nabla^{j_1}_r\sum\limits_{\substack{b,f \in A_+ \\ d = \deg b > \deg f}}\frac{1}{b^{i_1}f^{j+j_1}}\\\notag
    &+\sum \limits_{i+j = s+ t}\delta_j\nabla^j_s\sum\limits_{\substack{a,f \in A_+ \\ d = \deg a > \deg f}}\frac{1}{a^{i+r}f^j}. \notag\\
    &= \sum \limits_{i+j = s+ t}\sum \limits_{i_1+j_1 = i+r} \delta_{j+j_1}\nabla^j_t(-1)^{j_1}\nabla^{j_1}_i S_d(i_1,j+j_1)\\\notag
    &+ \sum \limits_{i+j = s+ t}\sum \limits_{i_1+j_1 = i+r} \delta_{j+j_1}\nabla^j_t\nabla^{j_1}_rS_d(i_1,j+j_1)\\\notag
    &+\sum \limits_{i+j = s+ t}\delta_j\nabla^j_s S_d(i+r,j). \notag
\end{align*}

Since the partial fraction decomposition of $P_{\lambda}(A,F)$ obtained from the process is unique, it follows that the above expansions of 
\begin{equation*}
    \sum \limits_{(a,b,c)\in M_3} \frac{1}{a^rb^s} \cdot \frac{1}{c^t} \quad \text{and} \quad \sum\limits_{(a,b,c)\in M_3} \frac{1}{a^r} \cdot \frac{1}{b^s c^ t}
\end{equation*}
yield the same expression in terms of power sums.

\medskip
\noindent \textbf{Case 4:} $(a,b,c)$ ranges over all tuples in $N_0$.
\par For each $\lambda, \mu \in \Fq^\times$ such that $\lambda \ne \mu$, consider the rational function 
\begin{equation*}
    P_{\lambda,\mu}(A,F) = \frac{1}{A^{r}(A+\lambda F)^{s}(A + \mu F)^t}.
\end{equation*}
Set $B = A + \lambda F$ and $C = A + \mu F$. From the same process as in the the case of $M_1$, we expand $P_{\lambda ,\mu}(A,F)$ in two ways. First, we expand $P_{\lambda, \mu}(A,F)$ from the left to the right. From \eqref{eq: PFD L-R}, we have;
\begin{align} \label{eq: L-R N0}
    \frac{1}{A^rB^s} \cdot \frac{1}{C^t} 
    &= \sum \limits_{i+j = r+s} \sum \limits_{i_1+j_1 = i + t}\frac{\nabla^j_s}{(-\lambda F)^j} \frac{\nabla^{j_1}_t}{(-\mu F)^{j_1}}\frac{1}{A^{i_1}}\\ \notag
    &+ \sum \limits_{i+j = r+s} \sum \limits_{i_1+j_1 = i + t}\frac{\nabla^j_s}{(-\lambda F)^j} \frac{\nabla^{j_1}_i}{(\mu F)^{j_1}}\frac{1}{C^{i_1}}\\\notag
    &+ \sum \limits_{i+j = r+s} \sum \limits_{i_1+j_1 = i + t}\frac{\nabla^j_r}{(\lambda F)^j} \frac{\nabla^{j_1}_t}{[(\lambda - \mu)F]^{j_1}}\frac{1}{B^{i_1}}\\\notag
    &+ \sum \limits_{i+j = r+s} \sum \limits_{i_1+j_1 = i + t}\frac{\nabla^j_r}{(\lambda F)^j} \frac{\nabla^{j_1}_i}{[( \mu-\lambda)F]^{j_1}}\frac{1}{C^{i_1}}\\\notag
    &= \sum \limits_{i+j = r+s} \sum \limits_{i_1+j_1 = i + t}\frac{\nabla^j_s\nabla^{j_1}_t}{(-\lambda)^j(-\mu)^{j_1}} \frac{1}{A^{i_1}F^{j+j_1}}\\ \notag
    &+ \sum \limits_{i+j = r+s} \sum \limits_{i_1+j_1 = i + t}\frac{\nabla^j_s(-1 )^{j_1}\nabla^{j_1}_i}{(-\lambda)^j(-\mu)^{j_1}} \frac{1}{C^{i_1}F^{j+j_1}}\\\notag
    &+ \sum \limits_{i+j = r+s} \sum \limits_{i_1+j_1 = i + t}\frac{\nabla^j_r\nabla^{j_1}_t}{\lambda^j(\lambda - \mu)^{j_1}} \frac{1}{B^{i_1}F^{j+j_1}}\\\notag
    &+ \sum \limits_{i+j = r+s} \sum \limits_{i_1+j_1 = i + t}\frac{\nabla^j_r(-1)^{j_1}\nabla^{j_1}_i}{\lambda^j(\lambda- \mu)^{j_1}} \frac{1}{C^{i_1}F^{j+j_1}}\notag.
\end{align}
Next we expand $P_{\lambda,\mu}(A,F)$ from the right to the left. From \eqref{eq: PFD R-L}, we have;
\begin{align}\label{eq: R-L N0}
    \frac{1}{A^r} \cdot \frac{1}{B^{s}C^t} 
    &=  \sum \limits_{i+j = s + t}\sum \limits_{i_1+j_1 = i + r} \frac{\nabla^j_t}{[(\lambda - \mu)F]^j}\frac{\nabla^{j_1}_i}{(-\lambda F)^{j_1}}\frac{1}{A^{i_1}}\\\notag
    &+ \sum \limits_{i+j = s + t}\sum \limits_{i_1+j_1 = i + r} \frac{\nabla^j_t}{[(\lambda - \mu)F]^j}\frac{\nabla^{j_1}_r}{(\lambda F)^{j_1}}\frac{1}{B^{i_1}}\\\notag
    &+\sum \limits_{i+j = s + t}\sum \limits_{i_1+j_1 = i + r}\frac{\nabla^j_s}{[( \mu - \lambda)F]^j}\frac{\nabla^{j_1}_i}{(-\mu F)^{j_1}}\frac{1}{A^{i_1}}\\\notag
     &+\sum \limits_{i+j = s + t}\sum \limits_{i_1+j_1 = i + r}\frac{\nabla^j_s}{[(\mu - \lambda)F]^j}\frac{\nabla^{j_1}_r}{(\mu F)^{j_1}}\frac{1}{C^{i_1}}\\\notag
     &=  \sum \limits_{i+j = s + t}\sum \limits_{i_1+j_1 = i + r} \frac{\nabla^j_t(-1)^{j_1}\nabla^{j_1}_i}{(\lambda - \mu)^j\lambda^{j_1}}\frac{1}{A^{i_1}F^{j+j_1}}\\\notag
    &+ \sum \limits_{i+j = s + t}\sum \limits_{i_1+j_1 = i + r} \frac{\nabla^j_t\nabla^{j_1}_r}{(\lambda - \mu)^j\lambda^{j_1}}\frac{1}{B^{i_1}F^{j+j_1}}\\\notag
    &+\sum \limits_{i+j = s + t}\sum \limits_{i_1+j_1 = i + r}\frac{\nabla^j_s(-1)^{j_1}\nabla^{j_1}_i}{(\mu - \lambda)^j\mu ^{j_1}}\frac{1}{A^{i_1}F^{j+j_1}}\\\notag
     &+\sum \limits_{i+j = s + t}\sum \limits_{i_1+j_1 = i + r}\frac{\nabla^j_s\nabla^{j_1}_r}{(\mu - \lambda)^j\mu ^{j_1}}\frac{1}{C^{i_1}F^{j+j_1}}.\notag
\end{align}
For each $(a,b,c)\in N_0$, we have $b = a + \lambda f, c = a + \mu f$ with $\lambda, \mu \in \Fq^\times$ such that $\lambda \ne \mu$ and $f \in A_+$ such that $d > \deg f$. Replacing $A = a, F = f$, one deduces from \eqref{eq: L-R N0} that 
\begin{align*}
    \sum \limits_{(a,b,c) \in N_0}& \frac{1}{a^rb^s} \cdot \frac{1}{c^t} \\
    &= \sum\limits_{\substack{a,f \in A_+ \\ d = \deg a > \deg f}} \sum \limits_{\lambda,\mu \in \Fq^\times; \lambda \ne \mu} \frac{1}{a^r(a+\lambda f)^s} \cdot \frac{1}{(a+\mu f)^t}\\
    &= \sum \limits_{i+j = r+s} \sum \limits_{i_1+j_1 = i + t}\delta_{j,j_1}\nabla^j_s\nabla^{j_1}_t \sum\limits_{\substack{a,f \in A_+ \\ d = \deg a > \deg f}}\frac{1}{a^{i_1}f^{j+j_1}}\\ \notag
    &+ \sum \limits_{i+j = r+s} \sum \limits_{i_1+j_1 = i + t}\delta_{j,j_1}\nabla^j_s(-1 )^{j_1}\nabla^{j_1}_i \sum\limits_{\substack{c,f \in A_+ \\ d = \deg c > \deg f}}\frac{1}{c^{i_1}f^{j+j_1}}\\\notag
    &+ \sum \limits_{i+j = r+s} \sum \limits_{i_1+j_1 = i + t}\delta_{j,j_1}\nabla^j_r\nabla^{j_1}_t \sum\limits_{\substack{b,f \in A_+ \\ d = \deg b > \deg f}}\frac{1}{b^{i_1}f^{j+j_1}}\\\notag
    &+ \sum \limits_{i+j = r+s} \sum \limits_{i_1+j_1 = i + t}\delta_{j,j_1}\nabla^j_r(-1)^{j_1}\nabla^{j_1}_i \sum\limits_{\substack{c,f \in A_+ \\ d = \deg c > \deg f}}\frac{1}{c^{i_1}f^{j+j_1}}\\\notag
    &= \sum \limits_{i+j = r+s} \sum \limits_{i_1+j_1 = i + t}\delta_{j,j_1}\nabla^j_s\nabla^{j_1}_t S_d(i_1,j+j_1)\\ \notag
    &+ \sum \limits_{i+j = r+s} \sum \limits_{i_1+j_1 = i + t}\delta_{j,j_1}\nabla^j_s(-1 )^{j_1}\nabla^{j_1}_i S_d(i_1,j+j_1)\\\notag
    &+ \sum \limits_{i+j = r+s} \sum \limits_{i_1+j_1 = i + t}\delta_{j,j_1}\nabla^j_r\nabla^{j_1}_t S_d(i_1,j+j_1)\\\notag
    &+ \sum \limits_{i+j = r+s} \sum \limits_{i_1+j_1 = i + t}\delta_{j,j_1}\nabla^j_r(-1)^{j_1}\nabla^{j_1}_i S_d(i_1,j+j_1).\notag
\end{align*}
Similarly, one deduces from \eqref{eq: R-L N0} that
\begin{align*}
    \sum \limits_{(a,b,c) \in N_0}& \frac{1}{a^r} \cdot \frac{1}{b^sc^t} \\&= \sum\limits_{\substack{a,f \in A_+ \\ d = \deg a > \deg f}} \sum \limits_{\lambda,\mu \in \Fq^\times; \lambda \ne \mu} \frac{1}{a^r} \cdot \frac{1}{(a+\lambda f)^s(a+\mu f)^t}\\
    &=  \sum \limits_{i+j = s + t}\sum \limits_{i_1+j_1 = i + r} \delta_{j,j_1}\nabla^j_t(-1)^{j_1}\nabla^{j_1}_i\sum\limits_{\substack{a,f \in A_+ \\ d = \deg a > \deg f}}\frac{1}{a^{i_1}f^{j+j_1}}\\\notag
    &+ \sum \limits_{i+j = s + t}\sum \limits_{i_1+j_1 = i + r} \delta_{j,j_1}\nabla^j_t\nabla^{j_1}_r\sum\limits_{\substack{b,f \in A_+ \\ d = \deg b > \deg f}}\frac{1}{b^{i_1}f^{j+j_1}}\\\notag
    &+\sum \limits_{i+j = s + t}\sum \limits_{i_1+j_1 = i + r}\delta_{j,j_1}\nabla^j_s(-1)^{j_1}\nabla^{j_1}_i\sum\limits_{\substack{a,f \in A_+ \\ d = \deg a > \deg f}}\frac{1}{a^{i_1}f^{j+j_1}}\\\notag
     &+\sum \limits_{i+j = s + t}\sum \limits_{i_1+j_1 = i + r}\delta_{j,j_1}\nabla^j_s\nabla^{j_1}_r\sum\limits_{\substack{c,f \in A_+ \\ d = \deg c > \deg f}}\frac{1}{c^{i_1}f^{j+j_1}}\\\notag
      &=  \sum \limits_{i+j = s + t}\sum \limits_{i_1+j_1 = i + r} \delta_{j,j_1}\nabla^j_t(-1)^{j_1}\nabla^{j_1}_i S_d(i_1,j+j_1)\\\notag
    &+ \sum \limits_{i+j = s + t}\sum \limits_{i_1+j_1 = i + r} \delta_{j,j_1}\nabla^j_t\nabla^{j_1}_r S_d(i_1,j+j_1)\\\notag
    &+\sum \limits_{i+j = s + t}\sum \limits_{i_1+j_1 = i + r}\delta_{j,j_1}\nabla^j_s(-1)^{j_1}\nabla^{j_1}_i S_d(i_1,j+j_1)\\\notag
     &+\sum \limits_{i+j = s + t}\sum \limits_{i_1+j_1 = i + r}\delta_{j,j_1}\nabla^j_s\nabla^{j_1}_r S_d(i_1,j+j_1).\notag
\end{align*}
Since the partial fraction decomposition of $P_{\lambda,\mu}(A,F)$ obtained from the process is unique, it follows that the above expansions of
\begin{equation*}
    \sum \limits_{(a,b,c)\in N_0} \frac{1}{a^rb^s} \cdot \frac{1}{c^t} \quad \text{and} \quad \sum\limits_{(a,b,c)\in N_0} \frac{1}{a^r} \cdot \frac{1}{b^s c^ t}
\end{equation*}
yield the same expression in terms of power sums.\\
\par Using Formulas \eqref{eq: formula 1}, \eqref{eq: formula 2} and \eqref{eq: formula 3}, one verifies easily that 
\begin{align*}
    &\sum \limits_{(a,b,c)\in M_1} \frac{1}{a^rb^s} \cdot \frac{1}{c^t} + \sum \limits_{(a,b,c)\in M_2} \frac{1}{a^rb^s} \cdot \frac{1}{c^t} + \sum \limits_{(a,b,c)\in M_3} \frac{1}{a^rb^s} \cdot \frac{1}{c^t}+ \sum \limits_{(a,b,c)\in N_0} \frac{1}{a^rb^s} \cdot \frac{1}{c^t}\\
    &=  \sum_{i+j=r+s+t} \Delta^j_{r+s,t} S_d(i,j)+\sum_{i+j=r+s} \Delta^j_{r,s} \Bigg(S_d(i+t,j)+\sum_{i_1+j_1=i+t} \Delta^{j_1}_{i,t} S_d(i_1,j_1+j)\Bigg),
\end{align*}
which is the sum of the terms of depth $2$ in the expression of $(S_d(r)S_d(s))S_d(t)$. Similarly, one verifies easily that 
\begin{align*}
    &\sum \limits_{(a,b,c)\in M_1} \frac{1}{a^r} \cdot \frac{1}{b^s c^ t} + \sum \limits_{(a,b,c)\in M_2} \frac{1}{a^r} \cdot \frac{1}{b^s c^ t} + \sum \limits_{(a,b,c)\in M_3} \frac{1}{a^r} \cdot \frac{1}{b^s c^ t}+ \sum \limits_{(a,b,c)\in N_0} \frac{1}{a^r} \cdot \frac{1}{b^s c^ t}\\
    &=  \sum_{i+j=r+s+t} \Delta^j_{r,s+t} S_d(i,j)+\sum_{i+j=s+t} \Delta^j_{s,t} \Bigg(S_d(i+r,j)+\sum_{i_1+j_1=i+r} \Delta^{j_1}_{i,r} S_d(i_1,j_1+j)\Bigg),
\end{align*}
which is the sum the terms of depth $2$ in the expression of $S_d(r)(S_d(s)S_d(t))$.

\subsubsection{Depth 3 terms: proof of Proposition \ref{prop: associativity}, Part 3} ${}$\par

Consider the following cases:

\medskip
\noindent \textbf{Case 1:} $(a,b,c)$ ranges over all tuples in $N_1$.
\par For $\lambda,\mu \in \Fq^\times$, consider the rational function 
\begin{equation*}
    P_{\lambda, \mu }(A,F,U) = \frac{1}{A^{r}(A+\lambda F)^s(A+\mu U)^t}.
\end{equation*}
We will deduce the partial fraction decomposition of $  P_{\lambda, \mu }(A,F,U) $ by the following process: we first give the partial fraction decomposition of $  P_{\lambda, \mu }(A,F,U) $ in variable $A$ with coefficients are rational functions of the form $Q(F,U) \in \Fq(F,U)$; then we continue to give the partial fraction decomposition of $Q(F,U)$ in variable $F$ with coefficients are rational functions of the form $R(U) \in \Fq(U)$; finally, we  give the partial fraction decomposition of $R(U)$ in variable $U$ with coefficients are elements in $\Fq$.
\par Set $B= A + \lambda F, C = A + \mu U$ and $G = F + \mu' U$ where $\mu' = - \frac{\mu}{\lambda}$, so that $B-C = \lambda G$. Using Lemma \ref{lem: decomposition}, we proceed the process in two ways. First, we expand $P_{\lambda, \mu }(A,F,U)$ from the left to the right. From \eqref{eq: PFD L-R}, we have;
\begin{align} \label{eq: L-R N1}
    \frac{1}{A^{r}B^s} \cdot \frac{1}{C^t}
    &= \sum \limits_{i+j = r+s} \sum \limits_{i_1+j_1 = i + t}\frac{\nabla^j_s}{(-\lambda F)^j} \frac{\nabla^{j_1}_t}{(-\mu U)^{j_1}}\frac{1}{A^{i_1}}\\ \notag
    &+ \sum \limits_{i+j = r+s} \sum \limits_{i_1+j_1 = i + t}\frac{\nabla^j_s}{(-\lambda F)^j} \frac{\nabla^{j_1}_i}{(\mu U)^{j_1}}\frac{1}{C^{i_1}}\\\notag
    &+ \sum \limits_{i+j = r+s} \sum \limits_{i_1+j_1 = i + t}\frac{\nabla^j_r}{(\lambda F)^j} \frac{\nabla^{j_1}_t}{(\lambda G)^{j_1}}\frac{1}{B^{i_1}}\\\notag
    &+ \sum \limits_{i+j = r+s} \sum \limits_{i_1+j_1 = i + t}\frac{\nabla^j_r}{(\lambda F)^j} \frac{\nabla^{j_1}_i}{(-\lambda G)^{j_1}}\frac{1}{C^{i_1}}.\notag
\end{align}
Here, 
\begin{align*}
    \frac{\nabla^j_s}{(-\lambda F)^j} \frac{\nabla^{j_1}_t}{(-\mu U)^{j_1}}\frac{1}{A^{i_1}} &=\frac{\nabla^j_s\nabla^{j_1}_t}{(-\lambda)^j(-\mu)^{j_1}} \frac{1}{A^{i_1}F^jU^{j_1}},\\
    \frac{\nabla^j_s}{(-\lambda F)^j} \frac{\nabla^{j_1}_i}{(\mu U)^{j_1}}\frac{1}{C^{i_1}}&=\frac{\nabla^j_s\nabla^{j_1}_i}{(-\lambda)^j\mu^{j_1}} \frac{1}{C^{i_1}F^jU^{j_1}}
\end{align*}
and 
\begin{align*}
    \frac{\nabla^j_r}{(\lambda F)^j} \frac{\nabla^{j_1}_t}{(\lambda G)^{j_1}}\frac{1}{B^{i_1}} 
    &= \frac{\nabla^j_r\nabla^{j_1}_t}{\lambda^{j+j_1}} \sum \limits_{i_2 + j_2 = j + j_1}\Bigg(\frac{\nabla^{j_2}_{j_1}}{(F-G)^{j_2}F^{i_2}} + \frac{\nabla^{j_2}_{j}}{(G-F)^{j_2}G^{i_2}}\Bigg)\frac{1}{B^{i_1}}\\ 
    &= \frac{\nabla^j_r\nabla^{j_1}_t}{\lambda^{j+j_1}} \sum \limits_{i_2 + j_2 = j + j_1}\Bigg(\frac{\nabla^{j_2}_{j_1}}{(-\mu' U)^{j_2}F^{i_2}} + \frac{\nabla^{j_2}_{j}}{(\mu' U)^{j_2}G^{i_2}}\Bigg)\frac{1}{B^{i_1}}\\ 
    &= \frac{\nabla^j_r\nabla^{j_1}_t}{\lambda^{j+j_1}} \sum \limits_{i_2 + j_2 = j + j_1}\frac{\nabla^{j_2}_{j_1}}{(-\mu' )^{j_2}}\frac{1}{B^{i_1}F^{i_2}U^{j_2}}\\ 
    &+ 
    \frac{\nabla^j_r\nabla^{j_1}_t}{\lambda^{j+j_1}} \sum \limits_{i_2 + j_2 = j + j_1}\frac{\nabla^{j_2}_{j}}{(\mu')^{j_2}}\frac{1}{B^{i_1}G^{i_2}U^{j_2}},\\ 
    \frac{\nabla^j_r}{(\lambda F)^j} \frac{\nabla^{j_1}_i}{(-\lambda G)^{j_1}}\frac{1}{C^{i_1}} &=   \frac{(-1)^j\nabla^j_r\nabla^{j_1}_i}{(-\lambda )^{j+j_1}}\sum \limits_{i_2 + j_2 = j + j_1}\Bigg(\frac{\nabla^{j_2}_{j_1}}{(F-G)^{j_2}F^{i_2}} + \frac{\nabla^{j_2}_{j}}{(G-F)^{j_2}G^{i_2}}\Bigg)\frac{1}{C^{i_1}}\\
    &=   \frac{(-1)^j\nabla^j_r\nabla^{j_1}_i}{(-\lambda )^{j+j_1}}\sum \limits_{i_2 + j_2 = j + j_1}\Bigg(\frac{\nabla^{j_2}_{j_1}}{(-\mu' U)^{j_2}F^{i_2}} + \frac{\nabla^{j_2}_{j}}{(\mu' U)^{j_2}G^{i_2}}\Bigg)\frac{1}{C^{i_1}}\\
     &= \frac{(-1)^j\nabla^j_r\nabla^{j_1}_i}{(-\lambda )^{j+j_1}}\sum \limits_{i_2 + j_2 = j + j_1}\frac{\nabla^{j_2}_{j_1}}{(-\mu' )^{j_2}}\frac{1}{C^{i_1}F^{i_2}U^{j_2}}\\
     &+
    \frac{(-1)^j\nabla^j_r\nabla^{j_1}_i}{(-\lambda )^{j+j_1}} \sum \limits_{i_2 + j_2 = j + j_1}\frac{\nabla^{j_2}_{j}}{(\mu')^{j_2}}\frac{1}{C^{i_1}G^{i_2}U^{j_2}}.
\end{align*}
Next we expand $P_{\lambda, \mu }(A,F,U)$ from the right to the left. From \eqref{eq: PFD R-L}, we have;
\begin{align} \label{eq: R-L N1}
   \frac{1}{A^r} \cdot \frac{1}{B^{s}C^t} 
    &=  \sum \limits_{i+j = s + t}\sum \limits_{i_1+j_1 = i + r} \frac{\nabla^j_t}{(\lambda G)^j}\frac{\nabla^{j_1}_i}{(-\lambda F)^{j_1}}\frac{1}{A^{i_1}}\\\notag
    &+ \sum \limits_{i+j = s + t}\sum \limits_{i_1+j_1 = i + r} \frac{\nabla^j_t}{(\lambda G)^j}\frac{\nabla^{j_1}_r}{(\lambda F)^{j_1}}\frac{1}{B^{i_1}}\\\notag
    &+\sum \limits_{i+j = s + t}\sum \limits_{i_1+j_1 = i + r}\frac{\nabla^j_s}{(-\lambda G)^j}\frac{\nabla^{j_1}_i}{(-\mu U)^{j_1}}\frac{1}{A^{i_1}}\\\notag
     &+\sum \limits_{i+j = s + t}\sum \limits_{i_1+j_1 = i + r}\frac{\nabla^j_s}{(-\lambda G)^j}\frac{\nabla^{j_1}_r}{(\mu U)^{j_1}}\frac{1}{C^{i_1}}.\notag
\end{align}
Here,
\begin{align*}
    \frac{\nabla^j_t}{(\lambda G)^j}\frac{\nabla^{j_1}_i}{(-\lambda F)^{j_1}}\frac{1}{A^{i_1}}
    &=\frac{(-1)^j\nabla^j_t\nabla^{j_1}_i}{(-\lambda)^{j+j_1}} \sum \limits_{i_2 + j_2 = j + j_1}\Bigg(\frac{\nabla^{j_2}_{j_1}}{(G-F)^{j_2}G^{i_2}} + \frac{\nabla^{j_2}_{j}}{(F-G)^{j_2}F^{i_2}}\Bigg)\frac{1}{A^{i_1}}\\
    &=\frac{(-1)^j\nabla^j_t\nabla^{j_1}_i}{(-\lambda)^{j+j_1}} \sum \limits_{i_2 + j_2 = j + j_1}\Bigg(\frac{\nabla^{j_2}_{j_1}}{(\mu' U)^{j_2}G^{i_2}} + \frac{\nabla^{j_2}_{j}}{(-\mu' U)^{j_2}F^{i_2}}\Bigg)\frac{1}{A^{i_1}}\\
    &=\frac{(-1)^j\nabla^j_t\nabla^{j_1}_i}{(-\lambda)^{j+j_1}} \sum \limits_{i_2 + j_2 = j + j_1}\frac{\nabla^{j_2}_{j_1}}{(\mu')^{j_2}} \frac{1}{A^{i_1}G^{i_2}U^{j_2}}\\
    &+ 
    \frac{(-1)^j\nabla^j_t\nabla^{j_1}_i}{(-\lambda)^{j+j_1}} \sum \limits_{i_2 + j_2 = j + j_1}\frac{\nabla^{j_2}_{j}}{(-\mu')^{j_2}}\frac{1}{A^{i_1}F^{i_2}U^{j_2}},\\
    \frac{\nabla^j_t}{(\lambda G)^j}\frac{\nabla^{j_1}_r}{(\lambda  F)^{j_1}}\frac{1}{B^{i_1}}  &=\frac{\nabla^j_t\nabla^{j_1}_r}{\lambda ^{j+j_1}} \sum \limits_{i_2 + j_2 = j + j_1}\Bigg(\frac{\nabla^{j_2}_{j_1}}{(G-F)^{j_2}G^{i_2}} + \frac{\nabla^{j_2}_{j}}{(F-G)^{j_2}F^{i_2}}\Bigg)\frac{1}{B^{i_1}}\\
    &=\frac{\nabla^j_t\nabla^{j_1}_r}{\lambda ^{j+j_1}} \sum \limits_{i_2 + j_2 = j + j_1}\Bigg(\frac{\nabla^{j_2}_{j_1}}{(\mu' U)^{j_2}G^{i_2}} + \frac{\nabla^{j_2}_{j}}{(-\mu' U)^{j_2}F^{i_2}}\Bigg)\frac{1}{B^{i_1}}\\
    &=\frac{\nabla^j_t\nabla^{j_1}_r}{\lambda ^{j+j_1}} \sum \limits_{i_2 + j_2 = j + j_1}\frac{\nabla^{j_2}_{j_1}}{(\mu')^{j_2}} \frac{1}{B^{i_1}G^{i_2}U^{j_2}}\\
    &+ 
    \frac{\nabla^j_t\nabla^{j_1}_r}{\lambda ^{j+j_1}} \sum \limits_{i_2 + j_2 = j + j_1}\frac{\nabla^{j_2}_{j}}{(-\mu')^{j_2}}\frac{1}{B^{i_1}F^{i_2}U^{j_2}}.
\end{align*}
and 
\begin{align*}
    \frac{\nabla^j_s}{(-\lambda G)^j}\frac{\nabla^{j_1}_i}{(-\mu U)^{j_1}}\frac{1}{A^{i_1}} &=\frac{\nabla^j_s\nabla^{j_1}_i}{(-\lambda )^j(-\mu)^{j_1}}\frac{1}{A^{i_1}G^jU^{j_1}},\\
    \frac{\nabla^j_s}{(-\lambda G)^j}\frac{\nabla^{j_1}_r}{(\mu U)^{j_1}}\frac{1}{C^{i_1}}&=\frac{\nabla^j_s\nabla^{j_1}_r}{(-\lambda)^j\mu^{j_1}}\frac{1}{C^{i_1}G^jU^{j_1}}.
\end{align*}
For each $(a,b,c)\in N_1$, we have $b = a + \lambda f, c = a + \mu u$ with $\lambda, \mu \in \Fq^\times$  and $f,u \in A_+$ such that $d > \deg f > \deg u$. Set $g = f + \mu' u$ where $\mu' = - \frac{\mu}{\lambda}$. Replacing $A = a, F = f,U = u$, one deduces from \eqref{eq: L-R N1} that 
\begin{align*} 
   \sum \limits_{(a,b,c) \in N_1}& \frac{1}{a^rb^s} \cdot \frac{1}{c^t} \\
   &=  \sum\limits_{\substack{a,f,u \in A_+ \\ d = \deg a > \deg f> \deg u}} \sum \limits_{\lambda, \mu \in \Fq^\times} \frac{1}{a^{r}(a+\lambda  f)^s} \cdot \frac{1}{(a+\mu u)^t}\\
   &= \sum \limits_{i+j = r+s} \sum \limits_{i_1+j_1 = i + t}\delta_j\nabla^j_s\delta_{j_1}\nabla^{j_1}_t \sum\limits_{\substack{a,f,u \in A_+ \\ d = \deg a > \deg f> \deg u}}\frac{1}{a^{i_1}f^ju^{j_1}}\\ 
    &+ \sum \limits_{i+j = r+s} \sum \limits_{i_1+j_1 = i + t}\delta_j\nabla^j_s\delta_{j_1}\nabla^{j_1}_i \sum\limits_{\substack{c,f,u \in A_+ \\ d = \deg c > \deg f> \deg u}}\frac{1}{c^{i_1}f^ju^{j_1}}\\
    &+ \sum \limits_{i+j = r+s} \sum \limits_{i_1+j_1 = i + t}\delta_{j+j_1}\nabla^j_r\nabla^{j_1}_t \sum \limits_{i_2 + j_2 = j + j_1}\delta_{j_2}\nabla^{j_2}_{j_1}\sum\limits_{\substack{b,f,u \in A_+ \\ d = \deg b > \deg f> \deg u}}\frac{1}{b^{i_1}f^{i_2}u^{j_2}}\\
    &+ \sum \limits_{i+j = r+s} \sum \limits_{i_1+j_1 = i + t}\delta_{j+j_1}\nabla^j_r\nabla^{j_1}_t \sum \limits_{i_2 + j_2 = j + j_1}\delta_{j_2}\nabla^{j_2}_{j}\sum\limits_{\substack{b,g,u \in A_+ \\ d = \deg b > \deg g> \deg u}}\frac{1}{b^{i_1}g^{i_2}u^{j_2}}\\
    &+ \sum \limits_{i+j = r+s} \sum \limits_{i_1+j_1 = i + t}\delta_{j+j_1}(-1)^j\nabla^j_r\nabla^{j_1}_i\sum \limits_{i_2 + j_2 = j + j_1}\delta_{j_2}\nabla^{j_2}_{j_1}\sum\limits_{\substack{c,f,u \in A_+ \\ d = \deg c > \deg f> \deg u}}\frac{1}{c^{i_1}f^{i_2}u^{j_2}}\\
    &+ \sum \limits_{i+j = r+s} \sum \limits_{i_1+j_1 = i + t}\delta_{j+j_1}(-1)^j\nabla^j_r\nabla^{j_1}_i \sum \limits_{i_2 + j_2 = j + j_1}\delta_{j_2}\nabla^{j_2}_{j}\sum\limits_{\substack{c,g,u \in A_+ \\ d = \deg c > \deg g> \deg u}}\frac{1}{c^{i_1}g^{i_2}u^{j_2}}\\
    &= \sum \limits_{i+j = r+s} \sum \limits_{i_1+j_1 = i + t}\delta_j\nabla^j_s\delta_{j_1}\nabla^{j_1}_t S_d(i_1,j,j_1)\\
    &+ \sum \limits_{i+j = r+s} \sum \limits_{i_1+j_1 = i + t}\delta_j\nabla^j_s\delta_{j_1}\nabla^{j_1}_i S_d(i_1,j,j_1)\\
    &+ \sum \limits_{i+j = r+s} \sum \limits_{i_1+j_1 = i + t}\delta_{j+j_1}\nabla^j_r\nabla^{j_1}_t \sum \limits_{i_2 + j_2 = j + j_1}\delta_{j_2}\nabla^{j_2}_{j_1}S_d(i_1,i_2,j_2)\\
    &+ \sum \limits_{i+j = r+s} \sum \limits_{i_1+j_1 = i + t}\delta_{j+j_1}\nabla^j_r\nabla^{j_1}_t \sum \limits_{i_2 + j_2 = j + j_1}\delta_{j_2}\nabla^{j_2}_{j}S_d(i_1,i_2,j_2)\\
    &+ \sum \limits_{i+j = r+s} \sum \limits_{i_1+j_1 = i + t}\delta_{j+j_1}(-1)^j\nabla^j_r\nabla^{j_1}_i\sum \limits_{i_2 + j_2 = j + j_1}\delta_{j_2}\nabla^{j_2}_{j_1}S_d(i_1,i_2,j_2)\\
    &+ \sum \limits_{i+j = r+s} \sum \limits_{i_1+j_1 = i + t}\delta_{j+j_1}(-1)^j\nabla^j_r\nabla^{j_1}_i \sum \limits_{i_2 + j_2 = j + j_1}\delta_{j_2}\nabla^{j_2}_{j}S_d(i_1,i_2,j_2)\\
    &= \sum \limits_{i+j = r+s} \sum \limits_{i_1+j_1 = i + t}\delta_j\nabla^j_s\Delta^{j_1}_{i,t} S_d(i_1,j,j_1)\\
    &+ \sum \limits_{i+j = r+s} \sum \limits_{i_1+j_1 = i + t}\delta_{j+j_1}\nabla^j_r\nabla^{j_1}_t \sum \limits_{i_2 + j_2 = j + j_1}\Delta^{j_2}_{j,j_1}S_d(i_1,i_2,j_2)\\
    &+ \sum \limits_{i+j = r+s} \sum \limits_{i_1+j_1 = i + t}\delta_{j+j_1}(-1)^j\nabla^j_r\nabla^{j_1}_i\sum \limits_{i_2 + j_2 = j + j_1}\Delta^{j_2}_{j,j_1}S_d(i_1,i_2,j_2).
\end{align*}
Similarly, one deduces from \eqref{eq: R-L N1} that
\begin{align*}
    \sum \limits_{(a,b,c) \in N_1}& \frac{1}{a^r} \cdot \frac{1}{b^sc^t}\\
    &=  \sum\limits_{\substack{a,f,u \in A_+ \\ d = \deg a > \deg f> \deg u}} \sum \limits_{\lambda, \mu \in \Fq^\times} \frac{1}{a^r} \cdot \frac{1}{(a+\lambda f)^s(a+\mu u)^t}\\
    &=  \sum \limits_{i+j = s + t}\sum \limits_{i_1+j_1 = i + r} \delta_{j+j_1}(-1)^j\nabla^j_t\nabla^{j_1}_i \sum \limits_{i_2 + j_2 = j + j_1}\delta_{j_2}\nabla^{j_2}_{j_1} \sum\limits_{\substack{a,g,u \in A_+ \\ d = \deg a > \deg g> \deg u}}\frac{1}{a^{i_1}g^{i_2}u^{j_2}}\\
    &+ \sum \limits_{i+j = s + t}\sum \limits_{i_1+j_1 = i + r} \delta_{j+j_1}(-1)^j\nabla^j_t\nabla^{j_1}_i \sum \limits_{i_2 + j_2 = j + j_1}\delta_{j_2}\nabla^{j_2}_{j}\sum\limits_{\substack{a,f,u \in A_+ \\ d = \deg a > \deg f> \deg u}}\frac{1}{a^{i_1}f^{i_2}u^{j_2}}\\
    &+ \sum \limits_{i+j = s + t}\sum \limits_{i_1+j_1 = i + r} \delta_{j+j_1}\nabla^j_t\nabla^{j_1}_r \sum \limits_{i_2 + j_2 = j + j_1}\delta_{j_2}\nabla^{j_2}_{j_1} \sum\limits_{\substack{b,g,u \in A_+ \\ d = \deg b > \deg g> \deg u}}\frac{1}{b^{i_1}g^{i_2}u^{j_2}}\\
    &+ \sum \limits_{i+j = s + t}\sum \limits_{i_1+j_1 = i + r} \delta_{j+j_1}\nabla^j_t\nabla^{j_1}_r  \sum \limits_{i_2 + j_2 = j + j_1}\delta_{j_2}\nabla^{j_2}_{j}\sum\limits_{\substack{b,f,u \in A_+ \\ d = \deg b > \deg f> \deg u}}\frac{1}{b^{i_1}f^{i_2}u^{j_2}}\\
    &+ \sum \limits_{i+j = s + t}\sum \limits_{i_1+j_1 = i + r} \delta_j\nabla^j_s\delta_{j_1}\nabla^{j_1}_i\sum\limits_{\substack{a,g,u \in A_+ \\ d = \deg a > \deg g> \deg u}}\frac{1}{a^{i_1}g^ju^{j_1}}\\
    &+ \sum \limits_{i+j = s + t}\sum \limits_{i_1+j_1 = i + r} \delta_j\nabla^j_s\delta_{j_1}\nabla^{j_1}_r\sum\limits_{\substack{c,g,u \in A_+ \\ d = \deg c > \deg g> \deg u}}\frac{1}{c^{i_1}g^ju^{j_1}}\\
    &=  \sum \limits_{i+j = s + t}\sum \limits_{i_1+j_1 = i + r} \delta_{j+j_1}(-1)^j\nabla^j_t\nabla^{j_1}_i \sum \limits_{i_2 + j_2 = j + j_1}\delta_{j_2}\nabla^{j_2}_{j_1} S_d(i_1,i_2,j_2)\\
    &+ \sum \limits_{i+j = s + t}\sum \limits_{i_1+j_1 = i + r} \delta_{j+j_1}(-1)^j\nabla^j_t\nabla^{j_1}_i \sum \limits_{i_2 + j_2 = j + j_1}\delta_{j_2}\nabla^{j_2}_{j}S_d(i_1,i_2,j_2)\\
    &+ \sum \limits_{i+j = s + t}\sum \limits_{i_1+j_1 = i + r} \delta_{j+j_1}\nabla^j_t\nabla^{j_1}_r \sum \limits_{i_2 + j_2 = j + j_1}\delta_{j_2}\nabla^{j_2}_{j_1} S_d(i_1,i_2,j_2)\\
    &+ \sum \limits_{i+j = s + t}\sum \limits_{i_1+j_1 = i + r} \delta_{j+j_1}\nabla^j_t\nabla^{j_1}_r  \sum \limits_{i_2 + j_2 = j + j_1}\delta_{j_2}\nabla^{j_2}_{j}S_d(i_1,i_2,j_2)\\
    &+ \sum \limits_{i+j = s + t}\sum \limits_{i_1+j_1 = i + r} \delta_j\nabla^j_s\delta_{j_1}\nabla^{j_1}_iS_d(i_1,j,j_1)\\
    &+ \sum \limits_{i+j = s + t}\sum \limits_{i_1+j_1 = i + r} \delta_j\nabla^j_s\delta_{j_1}\nabla^{j_1}_rS_d(i_1,j,j_1)\\
    &= \sum \limits_{i+j = s + t}\sum \limits_{i_1+j_1 = i + r} \delta_{j+j_1}(-1)^j\nabla^j_t\nabla^{j_1}_i \sum \limits_{i_2 + j_2 = j + j_1}\Delta^{j_2}_{j,j_1}S_d(i_1,i_2,j_2)\\
    &+ \sum \limits_{i+j = s + t}\sum \limits_{i_1+j_1 = i + r} \delta_{j+j_1}\nabla^j_t\nabla^{j_1}_r \sum \limits_{i_2 + j_2 = j + j_1}\Delta^{j_2}_{j,j_1} S_d(i_1,i_2,j_2)\\
    &+ \sum \limits_{i+j = s + t}\sum \limits_{i_1+j_1 = i + r} \delta_j\nabla^j_s\Delta^{j_1}_{i,r}S_d(i_1,j,j_1).
\end{align*}
Since the partial fraction decomposition of $P_{\lambda,\mu}(A,F,U)$ obtained from the process is unique, it follows that the above expansions of 
\begin{equation*}
    \sum \limits_{(a,b,c)\in N_1} \frac{1}{a^rb^s} \cdot \frac{1}{c^t} \quad \text{and} \quad \sum\limits_{(a,b,c)\in N_1} \frac{1}{a^r} \cdot \frac{1}{b^s c^ t}
\end{equation*}
yield the same expression in terms of power sums.

\medskip
\noindent \textbf{Case 2:} $(a,b,c)$ ranges over all tuples in $N_2$.
\par For $\lambda,\mu \in \Fq^\times$, consider the rational function 
\begin{equation*}
    P_{\lambda, \mu }(A,F,U) = \frac{1}{A^{r}(A+\lambda U)^s(A+\mu F)^t}.
\end{equation*}
\par Set $B= A + \lambda U, C = A + \mu F$ and $G = F + \lambda' U$ where $\lambda' = - \frac{\lambda}{\mu}$, so that $C-B = \mu G$. From the same process as in the the case of $N_1$, we expand $P_{\lambda, \mu }(A,F,U)$ in two ways. First, we expand $P_{\lambda, \mu }(A,F,U)$ from the left to the right. From \eqref{eq: PFD L-R}, we have;
\begin{align} \label{eq: L-R N2}
    \frac{1}{A^{r}B^s} \cdot \frac{1}{C^t}
    &= \sum \limits_{i+j = r+s} \sum \limits_{i_1+j_1 = i + t}\frac{\nabla^j_s}{(-\lambda U)^j} \frac{\nabla^{j_1}_t}{(-\mu F)^{j_1}}\frac{1}{A^{i_1}}\\ \notag
    &+ \sum \limits_{i+j = r+s} \sum \limits_{i_1+j_1 = i + t}\frac{\nabla^j_s}{(-\lambda U)^j} \frac{\nabla^{j_1}_i}{(\mu F)^{j_1}}\frac{1}{C^{i_1}}\\\notag
    &+ \sum \limits_{i+j = r+s} \sum \limits_{i_1+j_1 = i + t}\frac{\nabla^j_r}{(\lambda U)^j} \frac{\nabla^{j_1}_t}{(-\mu G)^{j_1}}\frac{1}{B^{i_1}}\\\notag
    &+ \sum \limits_{i+j = r+s} \sum \limits_{i_1+j_1 = i + t}\frac{\nabla^j_r}{(\lambda U)^j} \frac{\nabla^{j_1}_i}{(\mu G)^{j_1}}\frac{1}{C^{i_1}}\\\notag
    &= \sum \limits_{i+j = r+s} \sum \limits_{i_1+j_1 = i + t}\frac{\nabla^j_s\nabla^{j_1}_t}{(-\lambda)^j(-\mu)^{j_1}} \frac{1}{A^{i_1}F^{j_1}U^j}\\ \notag
    &+ \sum \limits_{i+j = r+s} \sum \limits_{i_1+j_1 = i + t}\frac{\nabla^j_s\nabla^{j_1}_i}{(-\lambda)^j\mu^{j_1}} \frac{1}{C^{i_1}F^{j_1}U^j}\\\notag
    &+ \sum \limits_{i+j = r+s} \sum \limits_{i_1+j_1 = i + t}\frac{\nabla^j_r\nabla^{j_1}_t}{\lambda^j(-\mu)^{j_1}} \frac{1}{B^{i_1}G^{j_1}U^j}\\\notag
    &+ \sum \limits_{i+j = r+s} \sum \limits_{i_1+j_1 = i + t}\frac{\nabla^j_r\nabla^{j_1}_i}{\lambda^j\mu^{j_1}} \frac{1}{C^{i_1}G^{j_1}U^j}\\\notag
\end{align}
Next we expand $P_{\lambda, \mu }(A,F,U)$ from the right to the left. From \eqref{eq: PFD R-L}, we have;
\begin{align} \label{eq: R-L N2}
   \frac{1}{A^r} \cdot \frac{1}{B^{s}C^t} 
    &=  \sum \limits_{i+j = s + t}\sum \limits_{i_1+j_1 = i + r} \frac{\nabla^j_t}{(-\mu G)^j}\frac{\nabla^{j_1}_i}{(-\lambda U)^{j_1}}\frac{1}{A^{i_1}}\\\notag
    &+ \sum \limits_{i+j = s + t}\sum \limits_{i_1+j_1 = i + r} \frac{\nabla^j_t}{(-\mu G)^j}\frac{\nabla^{j_1}_r}{(\lambda U)^{j_1}}\frac{1}{B^{i_1}}\\\notag
    &+\sum \limits_{i+j = s + t}\sum \limits_{i_1+j_1 = i + r}\frac{\nabla^j_s}{(\mu G)^j}\frac{\nabla^{j_1}_i}{(-\mu F)^{j_1}}\frac{1}{A^{i_1}}\\\notag
     &+\sum \limits_{i+j = s + t}\sum \limits_{i_1+j_1 = i + r}\frac{\nabla^j_s}{(\mu G)^j}\frac{\nabla^{j_1}_r}{(\mu F)^{j_1}}\frac{1}{C^{i_1}}.\notag
\end{align}
Here
\begin{align*}
     \frac{\nabla^j_t}{(-\mu G)^j}\frac{\nabla^{j_1}_i}{(-\lambda U)^{j_1}}\frac{1}{A^{i_1}}
    &= \frac{\nabla^j_t\nabla^{j_1}_i}{(-\mu)^j(-\lambda)^{j_1}}\frac{1}{A^{i_1}G^jU^{j_1}},\\
     \frac{\nabla^j_t}{(-\mu G)^j}\frac{\nabla^{j_1}_r}{(\lambda U)^{j_1}}\frac{1}{B^{i_1}}
    &= \frac{\nabla^j_t\nabla^{j_1}_r}{(-\mu)^j\lambda ^{j_1}}\frac{1}{B^{i_1}G^jU^{j_1}}
\end{align*}
and
\begin{align*}
    \frac{\nabla^j_s}{(\mu G)^j}\frac{\nabla^{j_1}_i}{(-\mu F)^{j_1}}\frac{1}{A^{i_1}} 
    &=\frac{(-1)^j\nabla^j_s\nabla^{j_1}_i}{(-\mu)^{j+j_1}} \sum \limits_{i_2 + j_2 = j + j_1} \Bigg(\frac{\nabla^{j_2}_{j_1}}{(G-F)^{j_2}G^{i_2}} + \frac{\nabla^{j_2}_{j}}{(F-G)^{j_2}F^{i_2}}\Bigg)\frac{1}{A^{i_1}} \\
    &=\frac{(-1)^j\nabla^j_s\nabla^{j_1}_i}{(-\mu)^{j+j_1}} \sum \limits_{i_2 + j_2 = j + j_1} \Bigg(\frac{\nabla^{j_2}_{j_1}}{(\lambda' U)^{j_2}G^{i_2}} + \frac{\nabla^{j_2}_{j}}{(-\lambda' U)^{j_2}F^{i_2}}\Bigg)\frac{1}{A^{i_1}} \\
     &=\frac{(-1)^j\nabla^j_s\nabla^{j_1}_i}{(-\mu)^{j+j_1}} \sum \limits_{i_2 + j_2 = j + j_1} \frac{\nabla^{j_2}_{j_1}}{(\lambda' )^{j_2}}\frac{1}{A^{i_1}G^{i_2}U^{j_2}}\\
     &+ 
     \frac{(-1)^j\nabla^j_s\nabla^{j_1}_i}{(-\mu)^{j+j_1}} \sum \limits_{i_2 + j_2 = j + j_1}\frac{\nabla^{j_2}_{j}}{(-\lambda')^{j_2}}\frac{1}{A^{i_1}F^{i_2}U^{j_2}},\\
     \frac{\nabla^j_s}{(\mu G)^j}\frac{\nabla^{j_1}_r}{(\mu F)^{j_1}}\frac{1}{C^{i_1}}
     &=  \frac{\nabla^j_s\nabla^{j_1}_r}{\mu^{j+j_1}} \sum \limits_{i_2 + j_2 = j + j_1} \Bigg(\frac{\nabla^{j_2}_{j_1}}{(G-F)^{j_2}G^{i_2}} + \frac{\nabla^{j_2}_{j}}{(F-G)^{j_2}F^{i_2}}\Bigg)\frac{1}{C^{i_1}}\\
     &=  \frac{\nabla^j_s\nabla^{j_1}_r}{\mu^{j+j_1}} \sum \limits_{i_2 + j_2 = j + j_1} \Bigg(\frac{\nabla^{j_2}_{j_1}}{(\lambda' U)^{j_2}G^{i_2}} + \frac{\nabla^{j_2}_{j}}{(-\lambda' U)^{j_2}F^{i_2}}\Bigg)\frac{1}{C^{i_1}}\\
     &=  \frac{\nabla^j_s\nabla^{j_1}_r}{\mu^{j+j_1}} \sum \limits_{i_2 + j_2 = j + j_1} \frac{\nabla^{j_2}_{j_1}}{(\lambda' )^{j_2}}\frac{1}{C^{i_1}G^{i_2}U^{j_2}}\\
     &+ 
     \frac{\nabla^j_s\nabla^{j_1}_r}{\mu^{j+j_1}} \sum \limits_{i_2 + j_2 = j + j_1}\frac{\nabla^{j_2}_{j}}{(-\lambda')^{j_2}}\frac{1}{C^{i_1}F^{i_2}U^{j_2}}.
\end{align*}
For each $(a,b,c)\in N_2$, we have $b = a + \lambda u, c = a + \mu f$ with $\lambda, \mu \in \Fq^\times$  and $f,u \in A_+$ such that $d > \deg f > \deg u$. Set $g = f + \lambda' u$ where $\lambda' = - \frac{\lambda}{\mu}$. Replacing $A = a, F = f,U = u$, one deduces from \eqref{eq: L-R N2} that 
\begin{align*} 
   \sum \limits_{(a,b,c) \in N_2}& \frac{1}{a^rb^s} \cdot \frac{1}{c^t} \\
   &=  \sum\limits_{\substack{a,f,u \in A_+ \\ d = \deg a > \deg f> \deg u}} \sum \limits_{\lambda, \mu \in \Fq^\times} \frac{1}{a^{r}(a+\lambda  u)^s} \cdot \frac{1}{(a+\mu f)^t}\\
   &= \sum \limits_{i+j = r+s} \sum \limits_{i_1+j_1 = i + t}\delta_j\nabla^j_s\delta_{j_2}\nabla^{j_1}_t \sum\limits_{\substack{a,f,u \in A_+ \\ d = \deg a > \deg f> \deg u}}\frac{1}{a^{i_1}f^{j_1}u^j}\\ \notag
    &+ \sum \limits_{i+j = r+s} \sum \limits_{i_1+j_1 = i + t}\delta_j\nabla^j_s\delta_{j_1}\nabla^{j_1}_i \sum\limits_{\substack{c,f,u \in A_+ \\ d = \deg c > \deg f> \deg u}}\frac{1}{c^{i_1}f^{j_1}u^j}\\\notag
    &+ \sum \limits_{i+j = r+s} \sum \limits_{i_1+j_1 = i + t}\delta_j\nabla^j_r\delta_{j_1}\nabla^{j_1}_t \sum\limits_{\substack{b,g,u \in A_+ \\ d = \deg b > \deg g> \deg u}}\frac{1}{b^{i_1}g^{j_1}u^j}\\\notag
    &+ \sum \limits_{i+j = r+s} \sum \limits_{i_1+j_1 = i + t}\delta_j\nabla^j_r\delta_{j_1}\nabla^{j_1}_i \sum\limits_{\substack{c,g,u \in A_+ \\ d = \deg c > \deg g> \deg u}}\frac{1}{c^{i_1}g^{j_1}u^j}\\\notag
    &= \sum \limits_{i+j = r+s} \sum \limits_{i_1+j_1 = i + t}\delta_j\nabla^j_s\delta_{j_2}\nabla^{j_1}_t S_d(i_1,j_1,j)\\ \notag
    &+ \sum \limits_{i+j = r+s} \sum \limits_{i_1+j_1 = i + t}\delta_j\nabla^j_s\delta_{j_1}\nabla^{j_1}_i S_d(i_1,j_1,j)\\\notag
    &+ \sum \limits_{i+j = r+s} \sum \limits_{i_1+j_1 = i + t}\delta_j\nabla^j_r\delta_{j_1}\nabla^{j_1}_t S_d(i_1,j_1,j)\\\notag
    &+ \sum \limits_{i+j = r+s} \sum \limits_{i_1+j_1 = i + t}\delta_j\nabla^j_r\delta_{j_1}\nabla^{j_1}_i S_d(i_1,j_1,j)\\\notag
    &= \sum \limits_{i+j = r+s} \sum \limits_{i_1+j_1 = i + t}\Delta^{j}_{r,s} \Delta^{j_1}_{i,t}S_d(i_1,j_1,j).
\end{align*}
Similarly, one deduces from \eqref{eq: R-L N2} that
\begin{align*}
    \sum \limits_{(a,b,c) \in N_2} &\frac{1}{a^r} \cdot \frac{1}{b^sc^t} \\
    &=  \sum\limits_{\substack{a,f,u \in A_+ \\ d = \deg a > \deg f> \deg u}} \sum \limits_{\lambda, \mu \in \Fq^\times} \frac{1}{a^r} \cdot \frac{1}{(a+\lambda u)^s(a+\mu f)^t}\\
    &=  \sum \limits_{i+j = s + t}\sum \limits_{i_1+j_1 = i + r} \delta_{j}\nabla^j_t\delta_{j_1}\nabla^{j_1}_i\sum\limits_{\substack{a,g,u \in A_+ \\ d = \deg a > \deg g> \deg u}}\frac{1}{a^{i_1}g^ju^{j_1}}\\\notag
    &+ \sum \limits_{i+j = s + t}\sum \limits_{i_1+j_1 = i + r} \delta_j\nabla^j_t\delta_{j_1}\nabla^{j_1}_r\sum\limits_{\substack{b,g,u \in A_+ \\ d = \deg b > \deg g> \deg u}}\frac{1}{b^{i_1}g^ju^{j_1}}\\\notag
    &+\sum \limits_{i+j = s + t}\sum \limits_{i_1+j_1 = i + r}\delta_{j+j_1}(-1)^j\nabla^j_s\nabla^{j_1}_i \sum \limits_{i_2 + j_2 = j + j_1} \delta_{j_2}\nabla^{j_2}_{j_1}\sum\limits_{\substack{a,g,u \in A_+ \\ d = \deg a > \deg g> \deg u}}\frac{1}{a^{i_1}g^{i_2}u^{j_2}} \\\notag
     &+\sum \limits_{i+j = s + t}\sum \limits_{i_1+j_1 = i + r}\delta_{j+j_1}(-1)^j\nabla^j_s\nabla^{j_1}_i \sum \limits_{i_2 + j_2 = j + j_1}\delta_{j_2}\nabla^{j_2}_{j}\sum\limits_{\substack{a,f,u \in A_+ \\ d = \deg a > \deg f> \deg u}}\frac{1}{a^{i_1}f^{i_2}u^{j_2}}\\\notag
     &+\sum \limits_{i+j = s + t}\sum \limits_{i_1+j_1 = i + r}\delta_{j+j_1}\nabla^j_s\nabla^{j_1}_r \sum \limits_{i_2 + j_2 = j + j_1} \delta_{j_2}\nabla^{j_2}_{j_1}\sum\limits_{\substack{c,g,u \in A_+ \\ d = \deg c > \deg g> \deg u}}\frac{1}{c^{i_1}g^{i_2}u^{j_2}}\\\notag
     &+\sum \limits_{i+j = s + t}\sum \limits_{i_1+j_1 = i + r}\delta_{j+j_1}\nabla^j_s\nabla^{j_1}_r \sum \limits_{i_2 + j_2 = j + j_1}\delta_{j_2}\nabla^{j_2}_{j}\sum\limits_{\substack{c,f,u \in A_+ \\ d = \deg c > \deg f> \deg u}}\frac{1}{c^{i_1}f^{i_2}u^{j_2}}\\\notag
     &=  \sum \limits_{i+j = s + t}\sum \limits_{i_1+j_1 = i + r} \delta_{j}\nabla^j_t\delta_{j_1}\nabla^{j_1}_iS_d(i_1,j,j_1)\\\notag
    &+ \sum \limits_{i+j = s + t}\sum \limits_{i_1+j_1 = i + r} \delta_j\nabla^j_t\delta_{j_1}\nabla^{j_1}_rS_d(i_1,j,j_1)\\\notag
    &+\sum \limits_{i+j = s + t}\sum \limits_{i_1+j_1 = i + r}\delta_{j+j_1}(-1)^j\nabla^j_s\nabla^{j_1}_i \sum \limits_{i_2 + j_2 = j + j_1} \delta_{j_2}\nabla^{j_2}_{j_1}S_d(i_1,i_2,j_2) \\\notag
     &+\sum \limits_{i+j = s + t}\sum \limits_{i_1+j_1 = i + r}\delta_{j+j_1}(-1)^j\nabla^j_s\nabla^{j_1}_i \sum \limits_{i_2 + j_2 = j + j_1}\delta_{j_2}\nabla^{j_2}_{j}S_d(i_1,i_2,j_2) \\\notag
     &+\sum \limits_{i+j = s + t}\sum \limits_{i_1+j_1 = i + r}\delta_{j+j_1}\nabla^j_s\nabla^{j_1}_r \sum \limits_{i_2 + j_2 = j + j_1} \delta_{j_2}\nabla^{j_2}_{j_1}S_d(i_1,i_2,j_2) \\\notag
     &+\sum \limits_{i+j = s + t}\sum \limits_{i_1+j_1 = i + r}\delta_{j+j_1}\nabla^j_s\nabla^{j_1}_r \sum \limits_{i_2 + j_2 = j + j_1}\delta_{j_2}\nabla^{j_2}_{j}S_d(i_1,i_2,j_2) \\\notag
     &=  \sum \limits_{i+j = s + t}\sum \limits_{i_1+j_1 = i + r} \delta_{j}\nabla^{j}_t\Delta^{j_1}_{i,r}S_d(i_1,j,j_1)\\\notag
     &+\sum \limits_{i+j = s + t}\sum \limits_{i_1+j_1 = i + r}\delta_{j+j_1}(-1)^j\nabla^j_s\nabla^{j_1}_i \sum \limits_{i_2 + j_2 = j + j_1}\Delta^{j_2}_{j,j_1}S_d(i_1,i_2,j_2) \\\notag
     &+\sum \limits_{i+j = s + t}\sum \limits_{i_1+j_1 = i + r}\delta_{j+j_1}\nabla^j_s\nabla^{j_1}_r \sum \limits_{i_2 + j_2 = j + j_1} \Delta^{j_2}_{j,j_1}S_d(i_1,i_2,j_2).\notag
\end{align*}
Since the partial fraction decomposition of $P_{\lambda,\mu}(A,F,U)$ obtained from the process is unique, it follows that the above expansions of 
\begin{equation*}
    \sum \limits_{(a,b,c)\in N_2} \frac{1}{a^rb^s} \cdot \frac{1}{c^t} \quad \text{and} \quad \sum\limits_{(a,b,c)\in N_2} \frac{1}{a^r} \cdot \frac{1}{b^s c^ t}
\end{equation*}
yield the same expression in terms of power sums.

\medskip
\noindent \textbf{Case 3:} $(a,b,c)$ ranges over all tuples in $N_3$.
\par For $\lambda,\mu \in \Fq^\times$, consider the rational function 
\begin{equation*}
    P_{\lambda, \mu }(A,F,U) = \frac{1}{A^{r}(A+\lambda F)^s(A+ \lambda F + \mu U)^t}.
\end{equation*}

\par Set $B= A + \lambda F, C = A + \lambda F + \mu U$ and $G = F + \mu' U$ where $\mu' =  \frac{\mu}{\lambda}$, so that $C-B = \mu U$ and $C-A = \lambda G$. From the same process as in the the case of $N_1$, we expand $P_{\lambda, \mu }(A,F,U)$ in two ways. First, we expand $P_{\lambda, \mu }(A,F,U)$ from the left to the right. From \eqref{eq: PFD L-R}, we have;
\begin{align} \label{eq: L-R N3}
    \frac{1}{A^{r}B^s} \cdot \frac{1}{C^t} 
    &= \sum \limits_{i+j = r+s} \sum \limits_{i_1+j_1 = i + t}\frac{\nabla^j_s}{(-\lambda F)^j} \frac{\nabla^{j_1}_t}{(-\lambda G)^{j_1}}\frac{1}{A^{i_1}}\\ \notag
    &+ \sum \limits_{i+j = r+s} \sum \limits_{i_1+j_1 = i + t}\frac{\nabla^j_s}{(-\lambda F)^j} \frac{\nabla^{j_1}_i}{(\lambda G)^{j_1}}\frac{1}{C^{i_1}}\\\notag
    &+ \sum \limits_{i+j = r+s} \sum \limits_{i_1+j_1 = i + t}\frac{\nabla^j_r}{(\lambda F)^j} \frac{\nabla^{j_1}_t}{(-\mu U)^{j_1}}\frac{1}{B^{i_1}}\\\notag
    &+ \sum \limits_{i+j = r+s} \sum \limits_{i_1+j_1 = i + t}\frac{\nabla^j_r}{(\lambda F)^j} \frac{\nabla^{j_1}_i}{(\mu U)^{j_1}}\frac{1}{C^{i_1}}\notag.
\end{align}
Here,
\begin{align*}
   \frac{\nabla^j_s}{(-\lambda F)^j} \frac{\nabla^{j_1}_t}{(-\lambda G)^{j_1}}\frac{1}{A^{i_1}} 
   &= \frac{\nabla^j_s\nabla^{j_1}_t}{(-\lambda)^{j+j_1}} \sum \limits_{i_2 + j_2 = j + j_1} \Bigg(\frac{\nabla^{j_2}_{j_1}}{(F-G)^{j_2}F^{i_2}} + \frac{\nabla^{j_2}_{j}}{(G-F)^{j_2}G^{i_2}}\Bigg)\frac{1}{A^{i_1}}\\\notag
   &= \frac{\nabla^j_s\nabla^{j_1}_t}{(-\lambda)^{j+j_1}} \sum \limits_{i_2 + j_2 = j + j_1} \Bigg(\frac{\nabla^{j_2}_{j_1}}{(-\mu' U)^{j_2}F^{i_2}} + \frac{\nabla^{j_2}_{j}}{(\mu' U)^{j_2}G^{i_2}}\Bigg)\frac{1}{A^{i_1}}\\\notag
   &= \frac{\nabla^j_s\nabla^{j_1}_t}{(-\lambda)^{j+j_1}} \sum \limits_{i_2 + j_2 = j + j_1} \frac{\nabla^{j_2}_{j_1}}{(-\mu' )^{j_2}}\frac{1}{A^{i_1}F^{i_2}U^{j_2}}\\
   &+ 
   \frac{\nabla^j_s\nabla^{j_1}_t}{(-\lambda)^{j+j_1}} \sum \limits_{i_2 + j_2 = j + j_1} \frac{\nabla^{j_2}_{j}}{(\mu' )^{j_2}}\frac{1}{A^{i_1}G^{i_2}U^{j_2}},\\\notag
    \frac{\nabla^j_s}{(-\lambda F)^j} \frac{\nabla^{j_1}_i}{(\lambda G)^{j_1}}\frac{1}{C^{i_1}} 
    &=  \frac{(-1)^j\nabla^j_s\nabla^{j_1}_i}{\lambda^{j+j_1}}\sum \limits_{i_2 + j_2 = j + j_1} \Bigg(\frac{\nabla^{j_2}_{j_1}}{(F-G)^{j_2}F^{i_2}} + \frac{\nabla^{j_2}_{j}}{(G-F)^{j_2}G^{i_2}}\Bigg) \frac{1}{C^{i_1}} \\\notag
    &=  \frac{(-1)^j\nabla^j_s\nabla^{j_1}_i}{\lambda^{j+j_1}}\sum \limits_{i_2 + j_2 = j + j_1} \Bigg(\frac{\nabla^{j_2}_{j_1}}{(-\mu' U)^{j_2}F^{i_2}} + \frac{\nabla^{j_2}_{j}}{(\mu' U)^{j_2}G^{i_2}}\Bigg) \frac{1}{C^{i_1}} \\\notag
    &=  \frac{(-1)^j\nabla^j_s\nabla^{j_1}_i}{\lambda^{j+j_1}}\sum \limits_{i_2 + j_2 = j + j_1} \frac{\nabla^{j_2}_{j_1}}{(-\mu')^{j_2}} \frac{1}{C^{i_1}F^{i_2}U^{j_2}} \\
    &+ 
    \frac{(-1)^j\nabla^j_s\nabla^{j_1}_i}{\lambda^{j+j_1}}\sum \limits_{i_2 + j_2 = j + j_1}\frac{\nabla^{j_2}_{j}}{(\mu')^{j_2}} \frac{1}{C^{i_1}G^{i_2}U^{j_2}}.\notag
\end{align*}
and 
\begin{align*}
    \frac{\nabla^j_r}{(\lambda F)^j} \frac{\nabla^{j_1}_t}{(-\mu U)^{j_1}}\frac{1}{B^{i_1}} &=\frac{\nabla^j_r\nabla^{j_1}_t}{(\lambda)^j(-\mu )^{j_1}} \frac{1}{B^{i_1}F^jU^{j_1}},\\\notag
    \frac{\nabla^j_r}{(\lambda F)^j} \frac{\nabla^{j_1}_i}{(\mu U)^{j_1}}\frac{1}{C^{i_1}} &= \frac{\nabla^j_r\nabla^{j_1}_i}{\lambda^j\mu ^{j_1}} \frac{1}{C^{i_1}F^jU^{j_1}}\notag.
\end{align*}
Next we expand $P_{\lambda, \mu }(A,F,U)$ from the right to the left. From \eqref{eq: PFD R-L}, we have
\begin{align} \label{eq: R-L N3}
   \frac{1}{A^r} \cdot \frac{1}{B^{s}C^t} 
    &=  \sum \limits_{i+j = s + t}\sum \limits_{i_1+j_1 = i + r} \frac{\nabla^j_t}{(-\mu U)^j}\frac{\nabla^{j_1}_i}{(-\lambda F)^{j_1}}\frac{1}{A^{i_1}}\\\notag
    &+ \sum \limits_{i+j = s + t}\sum \limits_{i_1+j_1 = i + r} \frac{\nabla^j_t}{(-\mu U)^j}\frac{\nabla^{j_1}_r}{(\lambda F)^{j_1}}\frac{1}{B^{i_1}}\\\notag
    &+\sum \limits_{i+j = s + t}\sum \limits_{i_1+j_1 = i + r}\frac{\nabla^j_s}{(\mu U)^j}\frac{\nabla^{j_1}_i}{(-\lambda G)^{j_1}}\frac{1}{A^{i_1}}\\\notag
     &+\sum \limits_{i+j = s + t}\sum \limits_{i_1+j_1 = i + r}\frac{\nabla^j_s}{(\mu U)^j}\frac{\nabla^{j_1}_r}{(\lambda G)^{j_1}}\frac{1}{C^{i_1}}\\\notag
      &=  \sum \limits_{i+j = s + t}\sum \limits_{i_1+j_1 = i + r} \frac{\nabla^j_t\nabla^{j_1}_i}{(-\mu)^j(-\lambda)^{j_1}}\frac{1}{A^{i_1}F^{j_1}U^j}\\\notag
    &+ \sum \limits_{i+j = s + t}\sum \limits_{i_1+j_1 = i + r} \frac{\nabla^j_t\nabla^{j_1}_r}{(-\mu )^j\lambda ^{j_1}}\frac{1}{B^{i_1}F^{j_1}U^j}\\\notag
    &+\sum \limits_{i+j = s + t}\sum \limits_{i_1+j_1 = i + r}\frac{\nabla^j_s\nabla^{j_1}_i}{\mu ^j(-\lambda)^{j_1}}\frac{1}{A^{i_1}G^{j_1}U^j}\\\notag
     &+\sum \limits_{i+j = s + t}\sum \limits_{i_1+j_1 = i + r}\frac{\nabla^j_s\nabla^{j_1}_r}{\mu ^j\lambda ^{j_1}}\frac{1}{C^{i_1}G^{j_1}U^j}.\notag
\end{align}
For each $(a,b,c)\in N_3$, we have $b = a + \lambda f, c = a + \lambda f + \mu u$ with $\lambda, \mu \in \Fq^\times$  and $f,u \in A_+$ such that $d > \deg f > \deg u$. Set $g = f + \mu'u$ where $\mu' = \frac{\mu}{\lambda} $. Replacing $A = a, F = f,U = u$, one deduces from \eqref{eq: L-R N3} that 
\vspace{-0.2cm}
\begin{align*} 
   \sum \limits_{(a,b,c) \in N_3}& \frac{1}{a^rb^s} \cdot \frac{1}{c^t} \\
   &=  \sum\limits_{\substack{a,f,u \in A_+ \\ d = \deg a > \deg f> \deg u}} \sum \limits_{\lambda, \mu \in \Fq^\times} \frac{1}{a^{r}(a+\lambda  f)^s} \cdot \frac{1}{(a+\lambda f + \mu u)^t}\\
   &= \sum \limits_{i+j = r+s} \sum \limits_{i_1+j_1 = i + t}\delta_{j+j_1}\nabla^j_s\nabla^{j_1}_t \sum \limits_{i_2 + j_2 = j + j_1} \delta_{j_2}\nabla^{j_2}_{j_1}\sum\limits_{\substack{a,f,u \in A_+ \\ d = \deg a > \deg f> \deg u}}\frac{1}{a^{i_1}f^{i_2}u^{j_2}}\\ \notag
    &+ \sum \limits_{i+j = r+s} \sum \limits_{i_1+j_1 = i + t}\delta_{j+j_1}\nabla^j_s\nabla^{j_1}_t \sum \limits_{i_2 + j_2 = j + j_1} \delta_{j_2}\nabla^{j_2}_{j}\sum\limits_{\substack{a,g,u \in A_+ \\ d = \deg a > \deg g> \deg u}}\frac{1}{a^{i_1}g^{i_2}u^{j_2}}\\\notag
    &+ \sum \limits_{i+j = r+s} \sum \limits_{i_1+j_1 = i + t} \delta_{j+j_1}(-1)^j\nabla^j_s\nabla^{j_1}_i\sum \limits_{i_2 + j_2 = j + j_1} \delta_{j_2}\nabla^{j_2}_{j_1} \sum\limits_{\substack{c,f,u \in A_+ \\ d = \deg c > \deg f> \deg u}}\frac{1}{c^{i_1}f^{i_2}u^{j_2}}\\\notag
    &+ \sum \limits_{i+j = r+s} \sum \limits_{i_1+j_1 = i + t}\delta_{j+j_1}(-1)^j\nabla^j_s\nabla^{j_1}_i\sum \limits_{i_2 + j_2 = j + j_1}\delta_{j_2}\nabla^{j_2}_{j} \sum\limits_{\substack{c,g,u \in A_+ \\ d = \deg c > \deg g> \deg u}}\frac{1}{c^{i_1}g^{i_2}u^{j_2}}\\\notag
    &+ \sum \limits_{i+j = r+s} \sum \limits_{i_1+j_1 = i + t}\delta_j\nabla^j_r\delta_{j_1}\nabla^{j_1}_t \sum\limits_{\substack{b,f,u \in A_+ \\ d = \deg b > \deg f> \deg u}}\frac{1}{b^{i_1}f^ju^{j_1}}\\\notag
    &+ \sum \limits_{i+j = r+s} \sum \limits_{i_1+j_1 = i + t}\delta_{j}\nabla^j_r\delta_{j_1}\nabla^{j_1}_i \sum\limits_{\substack{c,f,u \in A_+ \\ d = \deg c > \deg f> \deg u}}\frac{1}{c^{i_1}f^ju^{j_1}}\\\notag
    &= \sum \limits_{i+j = r+s} \sum \limits_{i_1+j_1 = i + t}\delta_{j+j_1}\nabla^j_s\nabla^{j_1}_t \sum \limits_{i_2 + j_2 = j + j_1} \delta_{j_2}\nabla^{j_2}_{j_1}S_d(i_1,i_2,j_2)\\ \notag
    &+ \sum \limits_{i+j = r+s} \sum \limits_{i_1+j_1 = i + t}\delta_{j+j_1}\nabla^j_s\nabla^{j_1}_t \sum \limits_{i_2 + j_2 = j + j_1} \delta_{j_2}\nabla^{j_2}_{j}S_d(i_1,i_2,j_2)\\\notag
    &+ \sum \limits_{i+j = r+s} \sum \limits_{i_1+j_1 = i + t} \delta_{j+j_1}(-1)^j\nabla^j_s\nabla^{j_1}_i\sum \limits_{i_2 + j_2 = j + j_1} \delta_{j_2}\nabla^{j_2}_{j_1} S_d(i_1,i_2,j_2)\\\notag
    &+ \sum \limits_{i+j = r+s} \sum \limits_{i_1+j_1 = i + t}\delta_{j+j_1}(-1)^j\nabla^j_s\nabla^{j_1}_i\sum \limits_{i_2 + j_2 = j + j_1}\delta_{j_2}\nabla^{j_2}_{j} S_d(i_1,i_2,j_2)\\\notag
    &+ \sum \limits_{i+j = r+s} \sum \limits_{i_1+j_1 = i + t}\delta_j\nabla^j_r\delta_{j_1}\nabla^{j_1}_t S_d(i_1,j,j_1)\\\notag
    &+ \sum \limits_{i+j = r+s} \sum \limits_{i_1+j_1 = i + t}\delta_{j}\nabla^j_r\delta_{j_1}\nabla^{j_1}_i S_d(i_1,j,j_1)\\\notag
    &= \sum \limits_{i+j = r+s} \sum \limits_{i_1+j_1 = i + t}\delta_{j+j_1}\nabla^j_s\nabla^{j_1}_t \sum \limits_{i_2 + j_2 = j + j_1} \Delta^{j_2}_{j,j_1}S_d(i_1,i_2,j_2)\\\notag
    &+ \sum \limits_{i+j = r+s} \sum \limits_{i_1+j_1 = i + t} \delta_{j+j_1}(-1)^j\nabla^j_s\nabla^{j_1}_i\sum \limits_{i_2 + j_2 = j + j_1} \Delta^{j_2}_{j,j_1} S_d(i_1,i_2,j_2)\\\notag
    &+ \sum \limits_{i+j = r+s} \sum \limits_{i_1+j_1 = i + t}\delta_{j}\nabla^j_r\Delta^{j_1}_{i,t} S_d(i_1,j,j_1).
\end{align*}
Similarly, one deduces from \eqref{eq: R-L N3} that
\vspace{-0.2cm}
\begin{align*}
    \sum \limits_{(a,b,c) \in N_3}& \frac{1}{a^r} \cdot \frac{1}{b^sc^t} \\
    &=  \sum\limits_{\substack{a,f,u \in A_+ \\ d = \deg a > \deg f> \deg u}} \sum \limits_{\lambda, \mu \in \Fq^\times} \frac{1}{a^r} \cdot \frac{1}{(a+\lambda f)^s(a+\lambda f + \mu u)^t}\\
    &=  \sum \limits_{i+j = s + t}\sum \limits_{i_1+j_1 = i + r} \delta_j\nabla^j_t\delta_{j_1}\nabla^{j_1}_i\sum\limits_{\substack{a,f,u \in A_+ \\ d = \deg a > \deg f> \deg u}}\frac{1}{a^{i_1}f^{j_1}u^j}\\\notag
    &+ \sum \limits_{i+j = s + t}\sum \limits_{i_1+j_1 = i + r} \delta_j\nabla^j_t\delta_{j_1}\nabla^{j_1}_r\sum\limits_{\substack{b,f,u \in A_+ \\ d = \deg b > \deg f> \deg u}}\frac{1}{b^{i_1}f^{j_1}u^j}\\\notag
    &+\sum \limits_{i+j = s + t}\sum \limits_{i_1+j_1 = i + r}\delta_j\nabla^j_s\delta_{j_1}\nabla^{j_1}_i\sum\limits_{\substack{a,g,u \in A_+ \\ d = \deg a > \deg g> \deg u}}\frac{1}{a^{i_1}g^{j_1}u^j}\\\notag
     &+\sum \limits_{i+j = s + t}\sum \limits_{i_1+j_1 = i + r}\delta_j\nabla^j_s\delta_{j_1}\nabla^{j_1}_r\sum\limits_{\substack{c,g,u \in A_+ \\ d = \deg c > \deg g> \deg u}}\frac{1}{c^{i_1}g^{j_1}u^j}\\\notag
     &=  \sum \limits_{i+j = s + t}\sum \limits_{i_1+j_1 = i + r} \delta_j\nabla^j_t\delta_{j_1}\nabla^{j_1}_iS_d(i_1,j_1,j)\\\notag
    &+ \sum \limits_{i+j = s + t}\sum \limits_{i_1+j_1 = i + r} \delta_j\nabla^j_t\delta_{j_1}\nabla^{j_1}_rS_d(i_1,j_1,j)\\\notag
    &+\sum \limits_{i+j = s + t}\sum \limits_{i_1+j_1 = i + r}\delta_j\nabla^j_s\delta_{j_1}\nabla^{j_1}_iS_d(i_1,j_1,j)\\\notag
     &+\sum \limits_{i+j = s + t}\sum \limits_{i_1+j_1 = i + r}\delta_j\nabla^j_s\delta_{j_1}\nabla^{j_1}_rS_d(i_1,j_1,j)\\\notag
     &=\sum \limits_{i+j = s + t}\sum \limits_{i_1+j_1 = i + r}\Delta^{j}_{s,t} \Delta^{j_1}_{i,r}S_d(i_1,j_1,j).\notag
\end{align*}
Since the partial fraction decomposition of $P_{\lambda,\mu}(A,F,U)$ obtained from the process is unique, it follows that the above expansions of 
\begin{equation*}
    \sum \limits_{(a,b,c)\in N_3} \frac{1}{a^rb^s} \cdot \frac{1}{c^t} \quad \text{and} \quad \sum\limits_{(a,b,c)\in N_3} \frac{1}{a^r} \cdot \frac{1}{b^s c^ t}
\end{equation*}
yield the same expression in terms of power sums.

\medskip
\noindent \textbf{Case 4:} $(a,b,c)$ ranges over all tuples in $N_4$.
\par For $\lambda,\mu, \eta \in \Fq^\times$ such that $\lambda \ne \mu$, consider the rational function 
\begin{equation*}
    P_{\lambda, \mu , \eta}(A,F,U) = \frac{1}{A^{r}(A+\lambda F)^s(A+\mu F + \eta U)^t}.
\end{equation*}

\par Set $B= A + \lambda F, C = A + \mu F + \eta U, G = F + \eta'U$ where $\eta' = \frac{\eta}{\mu}$ and $H = F + \eta'' U$ where $\eta'' =  \frac{\eta}{\mu - \lambda}$, so that $C-A = \mu G$ and $C-B = (\mu - \lambda)H$. It should be remarked that $\eta' \ne \eta''$ since $\lambda \ne 0$,  From the same process as in the the case of $N_1$, we expand $P_{\lambda, \mu , \eta}(A,F,U)$ in two ways. First, we expand $P_{\lambda, \mu , \eta}(A,F,U)$ from the left to the right. From \eqref{eq: PFD L-R}, we have
\begin{align} \label{eq: L-R N4}
    \frac{1}{A^{r}B^s} \cdot \frac{1}{C^t} 
    &= \sum \limits_{i+j = r+s} \sum \limits_{i_1+j_1 = i + t}\frac{\nabla^j_s}{(-\lambda F)^j} \frac{\nabla^{j_1}_t}{(-\mu G)^{j_1}}\frac{1}{A^{i_1}}\\ \notag
    &+ \sum \limits_{i+j = r+s} \sum \limits_{i_1+j_1 = i + t}\frac{\nabla^j_s}{(-\lambda F)^j} \frac{\nabla^{j_1}_i}{(\mu G)^{j_1}}\frac{1}{C^{i_1}}\\\notag
    &+ \sum \limits_{i+j = r+s} \sum \limits_{i_1+j_1 = i + t}\frac{\nabla^j_r}{(\lambda F)^j} \frac{\nabla^{j_1}_t}{[(\lambda - \mu )H]^{j_1}}\frac{1}{B^{i_1}}\\\notag
    &+ \sum \limits_{i+j = r+s} \sum \limits_{i_1+j_1 = i + t}\frac{\nabla^j_r}{(\lambda F)^j} \frac{\nabla^{j_1}_i}{[(\mu - \lambda )H]^{j_1}}\frac{1}{C^{i_1}}\notag.
\end{align}
Here
\begin{align*}
    \frac{\nabla^j_s}{(-\lambda F)^j} &\frac{\nabla^{j_1}_t}{(-\mu G)^{j_1}}\frac{1}{A^{i_1}} \\
    &=\frac{\nabla^j_s\nabla^{j_1}_t}{(-\lambda)^j(-\mu)^{j_1}} \sum \limits_{i_2 + j_2 = j + j_1} \Bigg(\frac{\nabla^{j_2}_{j_1}}{(F-G)^{j_2}F^{i_2}} + \frac{\nabla^{j_2}_{j}}{(G-F)^{j_2}G^{i_2}}\Bigg)\frac{1}{A^{i_1}}\\ \notag
    &=\frac{\nabla^j_s\nabla^{j_1}_t}{(-\lambda)^j(-\mu)^{j_1}} \sum \limits_{i_2 + j_2 = j + j_1} \Bigg(\frac{\nabla^{j_2}_{j_1}}{(-\eta' U)^{j_2}F^{i_2}} + \frac{\nabla^{j_2}_{j}}{(\eta' U)^{j_2}G^{i_2}}\Bigg)\frac{1}{A^{i_1}}\\ \notag
    &=\frac{\nabla^j_s\nabla^{j_1}_t}{(-\lambda)^j(-\mu)^{j_1}} \sum \limits_{i_2 + j_2 = j + j_1} \frac{\nabla^{j_2}_{j_1}}{(-\eta')^{j_2}}\frac{1}{A^{i_1}F^{i_2}U^{j_2}}\\
    &+ 
    \frac{\nabla^j_s\nabla^{j_1}_t}{(-\lambda)^j(-\mu)^{j_1}} \sum \limits_{i_2 + j_2 = j + j_1}\frac{\nabla^{j_2}_{j}}{(\eta')^{j_2}}\frac{1}{A^{i_1}G^{i_2}U^{j_2}},\\ \notag
    \frac{\nabla^j_s}{(-\lambda F)^j} &\frac{\nabla^{j_1}_i}{(\mu G)^{j_1}}\frac{1}{C^{i_1}} \\
    &=\frac{(-1)^j\nabla^j_s\nabla^{j_1}_i}{\lambda^j\mu^{j_1}} \sum \limits_{i_2 + j_2 = j + j_1} \Bigg(\frac{\nabla^{j_2}_{j_1}}{(F-G)^{j_2}F^{i_2}} + \frac{\nabla^{j_2}_{j}}{(G-F)^{j_2}G^{i_2}}\Bigg)\frac{1}{C^{i_1}}\\\notag
    &=\frac{(-1)^j\nabla^j_s\nabla^{j_1}_i}{\lambda^j\mu^{j_1}} \sum \limits_{i_2 + j_2 = j + j_1} \Bigg(\frac{\nabla^{j_2}_{j_1}}{(-\eta' U)^{j_2}F^{i_2}} + \frac{\nabla^{j_2}_{j}}{(\eta' U)^{j_2}G^{i_2}}\Bigg)\frac{1}{C^{i_1}}\\\notag
    &=\frac{(-1)^j\nabla^j_s\nabla^{j_1}_i}{\lambda^j\mu^{j_1}} \sum \limits_{i_2 + j_2 = j + j_1} \frac{\nabla^{j_2}_{j_1}}{(-\eta')^{j_2}} \frac{1}{C^{i_1}F^{i_2}U^{j_2}}\\
    &+ 
    \frac{(-1)^j\nabla^j_s\nabla^{j_1}_i}{\lambda^j\mu^{j_1}} \sum \limits_{i_2 + j_2 = j + j_1} \frac{\nabla^{j_2}_{j}}{(\eta' )^{j_2}}\frac{1}{C^{i_1}G^{i_2}U^{j_2}},\\\notag
    \frac{\nabla^j_r}{(\lambda F)^j} &\frac{\nabla^{j_1}_t}{[(\lambda - \mu )H]^{j_1}}\frac{1}{B^{i_1}} \\
    &=\frac{\nabla^j_r\nabla^{j_1}_t}{\lambda^j(\lambda - \mu )^{j_1}}\sum \limits_{i_2 + j_2 = j + j_1} \Bigg(\frac{\nabla^{j_2}_{j_1}}{(F-H)^{j_2}F^{i_2}} + \frac{\nabla^{j_2}_{j}}{(H-F)^{j_2}H^{i_2}}\Bigg) \frac{1}{B^{i_1}}\\\notag
    &=\frac{\nabla^j_r\nabla^{j_1}_t}{\lambda^j(\lambda - \mu )^{j_1}}\sum \limits_{i_2 + j_2 = j + j_1} \Bigg(\frac{\nabla^{j_2}_{j_1}}{(-\eta'' U)^{j_2}F^{i_2}} + \frac{\nabla^{j_2}_{j}}{(\eta'' U)^{j_2}G^{i_2}}\Bigg) \frac{1}{B^{i_1}}\\\notag
    &=\frac{\nabla^j_r\nabla^{j_1}_t}{\lambda^j(\lambda - \mu )^{j_1}}\sum \limits_{i_2 + j_2 = j + j_1} \frac{\nabla^{j_2}_{j_1}}{(-\eta'')^{j_2}} \frac{1}{B^{i_1}F^{i_2}U^{j_2}}\\
    &+ 
    \frac{\nabla^j_r\nabla^{j_1}_t}{\lambda^j(\lambda - \mu )^{j_1}}\sum \limits_{i_2 + j_2 = j + j_1}\frac{\nabla^{j_2}_{j}}{(\eta'')^{j_2}} \frac{1}{B^{i_1}G^{i_2}U^{j_2}},\\\notag
    \frac{\nabla^j_r}{(\lambda F)^j}& \frac{\nabla^{j_1}_i}{[(\mu - \lambda )H]^{j_1}}\frac{1}{C^{i_1}} \\
    &=\frac{(-1)^j\nabla^j_r\nabla^{j_1}_i}{(-\lambda )^j(\mu - \lambda)^{j_1}} \sum \limits_{i_2 + j_2 = j + j_1} \Bigg(\frac{\nabla^{j_2}_{j_1}}{(F-H)^{j_2}F^{i_2}} + \frac{\nabla^{j_2}_{j}}{(H-F)^{j_2}H^{i_2}}\Bigg) \frac{1}{C^{i_1}}\\\notag
    &=\frac{(-1)^j\nabla^j_r\nabla^{j_1}_i}{(-\lambda )^j(\mu - \lambda)^{j_1}} \sum \limits_{i_2 + j_2 = j + j_1} \Bigg(\frac{\nabla^{j_2}_{j_1}}{(-\eta'' U)^{j_2}F^{i_2}} + \frac{\nabla^{j_2}_{j}}{(\eta'' U)^{j_2}G^{i_2}}\Bigg) \frac{1}{C^{i_1}}\\\notag
    &=\frac{(-1)^j\nabla^j_r\nabla^{j_1}_i}{(-\lambda )^j(\mu - \lambda)^{j_1}} \sum \limits_{i_2 + j_2 = j + j_1} \frac{\nabla^{j_2}_{j_1}}{(-\eta'' )^{j_2}} \frac{1}{C^{i_1}F^{i_2}U^{j_2}}\\
    &+
    \frac{(-1)^j\nabla^j_r\nabla^{j_1}_i}{(-\lambda )^j(\mu - \lambda)^{j_1}} \sum \limits_{i_2 + j_2 = j + j_1} \frac{\nabla^{j_2}_{j}}{(\eta'' )^{j_2}}\frac{1}{C^{i_1}G^{i_2}U^{j_2}}.\notag
\end{align*}
Next we expand $P_{\lambda, \mu , \eta}(A,F,U)$ from the right to the left. From \eqref{eq: PFD R-L}, we have;
\begin{align} \label{eq: R-L N4}
   \frac{1}{A^r} \cdot \frac{1}{B^{s}C^t} 
    &=  \sum \limits_{i+j = s + t}\sum \limits_{i_1+j_1 = i + r} \frac{\nabla^j_t}{[(\lambda - \mu )H]^j}\frac{\nabla^{j_1}_i}{(-\lambda F)^{j_1}}\frac{1}{A^{i_1}}\\\notag
    &+ \sum \limits_{i+j = s + t}\sum \limits_{i_1+j_1 = i + r} \frac{\nabla^j_t}{[(\lambda - \mu )H]^j}\frac{\nabla^{j_1}_r}{(\lambda F)^{j_1}}\frac{1}{B^{i_1}}\\\notag
    &+\sum \limits_{i+j = s + t}\sum \limits_{i_1+j_1 = i + r}\frac{\nabla^j_s}{[(\mu - \lambda )H]^j}\frac{\nabla^{j_1}_i}{(-\mu G)^{j_1}}\frac{1}{A^{i_1}}\\\notag
     &+\sum \limits_{i+j = s + t}\sum \limits_{i_1+j_1 = i + r}\frac{\nabla^j_s}{[(\mu - \lambda )H]^j}\frac{\nabla^{j_1}_r}{(\mu G)^{j_1}}\frac{1}{C^{i_1}}.\notag
\end{align}
Here
\begin{align*}
    \frac{\nabla^j_t}{[(\lambda - \mu )H]^j}&\frac{\nabla^{j_1}_i}{(-\lambda F)^{j_1}}\frac{1}{A^{i_1}} \\
    &=\frac{(-1)^j\nabla^j_t\nabla^{j_1}_i}{(\mu - \lambda)^j(-\lambda )^{j_1}}\sum \limits_{i_2+j_2 = j + j_1} \Bigg(\frac{\nabla^{j_2}_{j_1}}{(H-F)^{j_2}H^{i_2}} + \frac{\nabla^{j_2}_{j}}{(F-H)^{j_2}F^{i_2}}\Bigg) \frac{1}{A^{i_1}}\\\notag
    &=\frac{(-1)^j\nabla^j_t\nabla^{j_1}_i}{(\mu - \lambda)^j(-\lambda )^{j_1}}\sum \limits_{i_2+j_2 = j + j_1} \Bigg(\frac{\nabla^{j_2}_{j_1}}{(\eta''U)^{j_2}H^{i_2}} + \frac{\nabla^{j_2}_{j}}{(-\eta''U)^{j_2}F^{i_2}}\Bigg) \frac{1}{A^{i_1}}\\\notag
    &=\frac{(-1)^j\nabla^j_t\nabla^{j_1}_i}{(\mu - \lambda)^j(-\lambda )^{j_1}}\sum \limits_{i_2+j_2 = j + j_1} \frac{\nabla^{j_2}_{j_1}}{(\eta'')^{j_2}} \frac{1}{A^{i_1}H^{i_2}U^{j_2}}\\
    &+
    \frac{(-1)^j\nabla^j_t\nabla^{j_1}_i}{(\mu - \lambda)^j(-\lambda )^{j_1}}\sum \limits_{i_2+j_2 = j + j_1}\frac{\nabla^{j_2}_{j}}{(-\eta'')^{j_2}}\frac{1}{A^{i_1}F^{i_2}U^{j_2}},\\\notag
     \frac{\nabla^j_t}{[(\lambda - \mu )H]^j}&\frac{\nabla^{j_1}_r}{(\lambda F)^{j_1}}\frac{1}{B^{i_1}} \\
     &=\frac{\nabla^j_t\nabla^{j_1}_r}{(\lambda - \mu )^j\lambda^{j_1}}\sum \limits_{i_2+j_2 = j + j_1} \Bigg(\frac{\nabla^{j_2}_{j_1}}{(H-F)^{j_2}H^{i_2}} + \frac{\nabla^{j_2}_{j}}{(F-H)^{j_2}F^{i_2}}\Bigg)\frac{1}{B^{i_1}}\\\notag
     &=\frac{\nabla^j_t\nabla^{j_1}_r}{(\lambda - \mu )^j\lambda^{j_1}}\sum \limits_{i_2+j_2 = j + j_1} \Bigg(\frac{\nabla^{j_2}_{j_1}}{(\eta''U)^{j_2}H^{i_2}} + \frac{\nabla^{j_2}_{j}}{(-\eta''U)^{j_2}F^{i_2}}\Bigg)\frac{1}{B^{i_1}}\\\notag
     &=\frac{\nabla^j_t\nabla^{j_1}_r}{(\lambda - \mu )^j\lambda^{j_1}}\sum \limits_{i_2+j_2 = j + j_1} \frac{\nabla^{j_2}_{j_1}}{(\eta'')^{j_2}} \frac{1}{B^{i_1}H^{i_2}U^{j_2}}\\
     &+
     \frac{\nabla^j_t\nabla^{j_1}_r}{(\lambda - \mu )^j\lambda^{j_1}}\sum \limits_{i_2+j_2 = j + j_1}\frac{\nabla^{j_2}_{j}}{(-\eta'')^{j_2}}\frac{1}{B^{i_1}F^{i_2}U^{j_2}},\\\notag
   \frac{\nabla^j_s}{[(\mu - \lambda )H]^j}&\frac{\nabla^{j_1}_i}{(-\mu G)^{j_1}}\frac{1}{A^{i_1}} \\
   &=\frac{(-1)^j\nabla^j_s\nabla^{j_1}_i}{( \lambda - \mu)^j(-\mu )^{j_1}}\sum \limits_{i_2+j_2 = j + j_1} \Bigg(\frac{\nabla^{j_2}_{j_1}}{(H-G)^{j_2}H^{i_2}} + \frac{\nabla^{j_2}_{j}}{(G-H)^{j_2}G^{i_2}}\Bigg)\frac{1}{A^{i_1}}\\\notag
   &=\frac{(-1)^j\nabla^j_s\nabla^{j_1}_i}{( \lambda - \mu)^j(-\mu )^{j_1}}\sum \limits_{i_2+j_2 = j + j_1} \Bigg(\frac{\nabla^{j_2}_{j_1}}{[(\eta''-\eta')U]^{j_2}H^{i_2}} + \frac{\nabla^{j_2}_{j}}{[(\eta'-\eta'')U]^{j_2}G^{i_2}}\Bigg)\frac{1}{A^{i_1}}\\\notag
   &=\frac{(-1)^j\nabla^j_s\nabla^{j_1}_i}{( \lambda - \mu)^j(-\mu )^{j_1}}\sum \limits_{i_2+j_2 = j + j_1} \frac{\nabla^{j_2}_{j_1}}{(\eta''-\eta')^{j_2}}\frac{1}{A^{i_1}H^{i_2}U^{j_2}} \\
   &+ 
   \frac{(-1)^j\nabla^j_s\nabla^{j_1}_i}{( \lambda - \mu)^j(-\mu )^{j_1}}\sum \limits_{i_2+j_2 = j + j_1}\frac{\nabla^{j_2}_{j}}{(\eta'-\eta'')^{j_2}}\frac{1}{A^{i_1}G^{i_2}U^{j_2}},\\\notag
     \frac{\nabla^j_s}{[(\mu - \lambda )H]^j}&\frac{\nabla^{j_1}_r}{(\mu G)^{j_1}}\frac{1}{C^{i_1}} \\
     &=\frac{\nabla^j_s\nabla^{j_1}_r}{(\mu - \lambda )^j\mu ^{j_1}}\sum \limits_{i_2+j_2 = j + j_1} \Bigg(\frac{\nabla^{j_2}_{j_1}}{(H-G)^{j_2}H^{i_2}} + \frac{\nabla^{j_2}_{j}}{(G-H)^{j_2}G^{i_2}}\Bigg)\frac{1}{C^{i_1}}\\\notag
     &=\frac{\nabla^j_s\nabla^{j_1}_r}{(\mu - \lambda )^j\mu ^{j_1}}\sum \limits_{i_2+j_2 = j + j_1} \Bigg(\frac{\nabla^{j_2}_{j_1}}{[(\eta''-\eta')U]^{j_2}H^{i_2}} + \frac{\nabla^{j_2}_{j}}{[(\eta'-\eta'')U]^{j_2}G^{i_2}}\Bigg)\frac{1}{C^{i_1}}\\\notag
      &=\frac{\nabla^j_s\nabla^{j_1}_r}{(\mu - \lambda )^j\mu ^{j_1}}\sum \limits_{i_2+j_2 = j + j_1} \frac{\nabla^{j_2}_{j_1}}{(\eta''-\eta')^{j_2}} \frac{1}{C^{i_1}H^{i_2}U^{j_2}}\\
      &+
      \frac{\nabla^j_s\nabla^{j_1}_r}{(\mu - \lambda )^j\mu ^{j_1}}\sum \limits_{i_2+j_2 = j + j_1}\frac{\nabla^{j_2}_{j}}{(\eta'-\eta'')^{j_2}}\frac{1}{C^{i_1}G^{i_2}U^{j_2}}.\notag
\end{align*}
For each $(a,b,c)\in N_4$, we have $b = a + \lambda f, c = a + \mu f + \eta u$ with $\lambda, \mu, \eta \in \Fq^\times$ such that $\lambda \ne \mu$  and $f,u \in A_+$ such that $d > \deg f > \deg u$. Set $g = f +  \eta'u$ where $\eta' = \frac{\eta}{\mu}$ and $ h = f + \eta'' u$ where $\eta'' = \frac{\eta}{\mu - \lambda}$. Replacing $A = a, F = f,U = u$, one deduces from \eqref{eq: L-R N4} that 
\vspace{-0.2cm}
\begin{align*} 
   \sum \limits_{(a,b,c) \in N_4} & \frac{1}{a^rb^s} \cdot \frac{1}{c^t} \\
   &=  \sum\limits_{\substack{a,f,u \in A_+ \\ d = \deg a > \deg f> \deg u}} \sum \limits_{\substack{\lambda, \mu, \eta \in \Fq^\times \\ \lambda \ne \mu}} \frac{1}{a^{r}(a+\lambda  f)^s} \cdot \frac{1}{(a+\mu f + \eta u)^t}\\
   &= \sum \limits_{i+j = r+s} \sum \limits_{i_1+j_1 = i + t}\delta_{j,j_1}\nabla^j_s\nabla^{j_1}_t \sum \limits_{i_2 + j_2 = j + j_1} \delta_{j_2}\nabla^{j_2}_{j_1}\sum\limits_{\substack{a,f,u \in A_+ \\ d = \deg a > \deg f> \deg u}}\frac{1}{a^{i_1}f^{i_2}u^{j_2}}\\ 
    &+ \sum \limits_{i+j = r+s} \sum \limits_{i_1+j_1 = i + t}\delta_{j,j_1}\nabla^j_s\nabla^{j_1}_t\sum \limits_{i_2 + j_2 = j + j_1}\delta_{j_2}\nabla^{j_2}_{j}\sum\limits_{\substack{a,g,u \in A_+ \\ d = \deg a > \deg g> \deg u}}\frac{1}{a^{i_1}g^{i_2}u^{j_2}}\\
    &+ \sum \limits_{i+j = r+s} \sum \limits_{i_1+j_1 = i + t}\delta_{j,j_1}(-1)^j\nabla^j_s\nabla^{j_1}_i\sum \limits_{i_2 + j_2 = j + j_1} \delta_{j_2}\nabla^{j_2}_{j_1} \sum\limits_{\substack{c,f,u \in A_+ \\ d = \deg c > \deg f> \deg u}}\frac{1}{c^{i_1}f^{i_2}u^{j_2}}\\
    &+ \sum \limits_{i+j = r+s} \sum \limits_{i_1+j_1 = i + t}\delta_{j,j_1}(-1)^j\nabla^j_s\nabla^{j_1}_i \sum \limits_{i_2 + j_2 = j + j_1} \delta_{j_2}\nabla^{j_2}_{j}\sum\limits_{\substack{c,g,u \in A_+ \\ d = \deg c > \deg g> \deg u}}\frac{1}{c^{i_1}g^{i_2}u^{j_2}}\\
    &+ \sum \limits_{i+j = r+s} \sum \limits_{i_1+j_1 = i + t}\delta_{j,j_1}\nabla^j_r\nabla^{j_1}_t\sum \limits_{i_2 + j_2 = j + j_1} \delta_{j_2}\nabla^{j_2}_{j_1} \sum\limits_{\substack{b,f,u \in A_+ \\ d = \deg b > \deg f> \deg u}}\frac{1}{b^{i_1}f^{i_2}u^{j_2}}\\ 
    &+ \sum \limits_{i+j = r+s} \sum \limits_{i_1+j_1 = i + t}\delta_{j,j_1}\nabla^j_r\nabla^{j_1}_t\sum \limits_{i_2 + j_2 = j + j_1}\delta_{j_2}\nabla^{j_2}_{j_1}\nabla^{j_2}_{j} \sum\limits_{\substack{b,g,u \in A_+ \\ d = \deg b > \deg g> \deg u}}\frac{1}{b^{i_1}g^{i_2}u^{j_2}}\\
    &+ \sum \limits_{i+j = r+s} \sum \limits_{i_1+j_1 = i + t}\delta_{j,j_1}(-1)^j\nabla^j_r\nabla^{j_1}_i \sum \limits_{i_2 + j_2 = j + j_1} \delta_{j_2}\nabla^{j_2}_{j_1} \sum\limits_{\substack{c,f,u \in A_+ \\ d = \deg c > \deg f> \deg u}}\frac{1}{c^{i_1}f^{i_2}u^{j_2}}\\
    &+ \sum \limits_{i+j = r+s} \sum \limits_{i_1+j_1 = i + t}\delta_{j,j_1}(-1)^j\nabla^j_r\nabla^{j_1}_i \sum \limits_{i_2 + j_2 = j + j_1} \delta_{j_2}\nabla^{j_2}_{j}\sum\limits_{\substack{c,g,u \in A_+ \\ d = \deg c > \deg g> \deg u}}\frac{1}{c^{i_1}g^{i_2}u^{j_2}}\\
    &= \sum \limits_{i+j = r+s} \sum \limits_{i_1+j_1 = i + t}\delta_{j,j_1}\nabla^j_s\nabla^{j_1}_t \sum \limits_{i_2 + j_2 = j + j_1} \delta_{j_2}\nabla^{j_2}_{j_1}S_d(i_1,i_2,j_2)\\ 
    &+ \sum \limits_{i+j = r+s} \sum \limits_{i_1+j_1 = i + t}\delta_{j,j_1}\nabla^j_s\nabla^{j_1}_t\sum \limits_{i_2 + j_2 = j + j_1}\delta_{j_2}\nabla^{j_2}_{j}S_d(i_1,i_2,j_2)\\
    &+ \sum \limits_{i+j = r+s} \sum \limits_{i_1+j_1 = i + t}\delta_{j,j_1}(-1)^j\nabla^j_s\nabla^{j_1}_i\sum \limits_{i_2 + j_2 = j + j_1} \delta_{j_2}\nabla^{j_2}_{j_1} S_d(i_1,i_2,j_2)\\
    &+ \sum \limits_{i+j = r+s} \sum \limits_{i_1+j_1 = i + t}\delta_{j,j_1}(-1)^j\nabla^j_s\nabla^{j_1}_i \sum \limits_{i_2 + j_2 = j + j_1} \delta_{j_2}\nabla^{j_2}_{j}S_d(i_1,i_2,j_2)\\
    &+ \sum \limits_{i+j = r+s} \sum \limits_{i_1+j_1 = i + t}\delta_{j,j_1}\nabla^j_r\nabla^{j_1}_t\sum \limits_{i_2 + j_2 = j + j_1} \delta_{j_2}\nabla^{j_2}_{j_1} S_d(i_1,i_2,j_2)\\ 
    &+ \sum \limits_{i+j = r+s} \sum \limits_{i_1+j_1 = i + t}\delta_{j,j_1}\nabla^j_r\nabla^{j_1}_t\sum \limits_{i_2 + j_2 = j + j_1}\delta_{j_2}\nabla^{j_2}_{j_1}\nabla^{j_2}_{j} S_d(i_1,i_2,j_2)\\
    &+ \sum \limits_{i+j = r+s} \sum \limits_{i_1+j_1 = i + t}\delta_{j,j_1}(-1)^j\nabla^j_r\nabla^{j_1}_i \sum \limits_{i_2 + j_2 = j + j_1} \delta_{j_2}\nabla^{j_2}_{j_1} S_d(i_1,i_2,j_2)\\
    &+ \sum \limits_{i+j = r+s} \sum \limits_{i_1+j_1 = i + t}\delta_{j,j_1}(-1)^j\nabla^j_r\nabla^{j_1}_i \sum \limits_{i_2 + j_2 = j + j_1} \delta_{j_2}\nabla^{j_2}_{j}S_d(i_1,i_2,j_2)\\
     &= \sum \limits_{i+j = r+s} \sum \limits_{i_1+j_1 = i + t}\delta_{j,j_1}\nabla^j_s\nabla^{j_1}_t \sum \limits_{i_2 + j_2 = j + j_1} \Delta^{j_2}_{j,j_1}S_d(i_1,i_2,j_2)\\ 
    &+ \sum \limits_{i+j = r+s} \sum \limits_{i_1+j_1 = i + t}\delta_{j,j_1}(-1)^j\nabla^j_s\nabla^{j_1}_i\sum \limits_{i_2 + j_2 = j + j_1} \Delta^{j_2}_{j,j_1} S_d(i_1,i_2,j_2)\\
    &+ \sum \limits_{i+j = r+s} \sum \limits_{i_1+j_1 = i + t}\delta_{j,j_1}\nabla^j_r\nabla^{j_1}_t\sum \limits_{i_2 + j_2 = j + j_1} \Delta^{j_2}_{j,j_1} S_d(i_1,i_2,j_2)\\ 
    &+ \sum \limits_{i+j = r+s} \sum \limits_{i_1+j_1 = i + t}\delta_{j,j_1}(-1)^j\nabla^j_r\nabla^{j_1}_i \sum \limits_{i_2 + j_2 = j + j_1} \Delta^{j_2}_{j,j_1} S_d(i_1,i_2,j_2).
\end{align*}
Similarly, one deduces from \eqref{eq: R-L N4} that
\vspace{-0.2cm}
\begin{align*}
    \sum \limits_{(a,b,c) \in N_4}& \frac{1}{a^r} \cdot \frac{1}{b^sc^t} \\
    &=  \sum\limits_{\substack{a,f,u \in A_+ \\ d = \deg a > \deg f> \deg u}} \sum \limits_{\substack{\lambda, \mu, \eta \in \Fq^\times \\ \lambda \ne \mu}} \frac{1}{a^r} \cdot \frac{1}{(a+\lambda f)^s(a+\mu f + \eta u)^t}\\
    &=  \sum \limits_{i+j = s + t}\sum \limits_{i_1+j_1 = i + r} \delta_{j,j_1}(-1)^j\nabla^j_t\nabla^{j_1}_i\sum \limits_{i_2+j_2 = j + j_1} \delta_{j_2}\nabla^{j_2}_{j_1} \sum\limits_{\substack{a,h,u \in A_+ \\ d = \deg a > \deg h> \deg u}}\frac{1}{a^{i_1}h^{i_2}u^{j_2}}\\\notag
    &+ \sum \limits_{i+j = s + t}\sum \limits_{i_1+j_1 = i + r} \delta_{j,j_1}(-1)^j\nabla^j_t\nabla^{j_1}_i\sum \limits_{i_2+j_2 = j + j_1}\delta_{j_2}\nabla^{j_2}_{j}\sum\limits_{\substack{a,f,u \in A_+ \\ d = \deg a > \deg f> \deg u}}\frac{1}{a^{i_1}f^{i_2}u^{j_2}}\\\notag
    &+\sum \limits_{i+j = s + t}\sum \limits_{i_1+j_1 = i + r}\delta_{j,j_1}\nabla^j_t\nabla^{j_1}_r\sum \limits_{i_2+j_2 = j + j_1} \delta_{j_2}\nabla^{j_2}_{j_1} \sum\limits_{\substack{b,h,u \in A_+ \\ d = \deg b > \deg h> \deg u}}\frac{1}{b^{i_1}h^{i_2}u^{j_2}}\\\notag
    &+\sum \limits_{i+j = s + t}\sum \limits_{i_1+j_1 = i + r}\delta_{j,j_1}\nabla^j_t\nabla^{j_1}_r\sum \limits_{i_2+j_2 = j + j_1}\delta_{j_2}\nabla^{j_2}_{j}\sum\limits_{\substack{b,f,u \in A_+ \\ d = \deg b > \deg f> \deg u}}\frac{1}{b^{i_1}f^{i_2}u^{j_2}}\\\notag
    &+  \sum \limits_{i+j = s + t}\sum \limits_{i_1+j_1 = i + r} \delta_{j,j_1}(-1)^j\nabla^j_s\nabla^{j_1}_i\sum \limits_{i_2+j_2 = j + j_1} \delta_{j_2}\nabla^{j_2}_{j_1}\sum\limits_{\substack{a,h,u \in A_+ \\ d = \deg a > \deg h> \deg u}}\frac{1}{a^{i_1}h^{i_2}u^{j_2}}\\\notag
    &+ \sum \limits_{i+j = s + t}\sum \limits_{i_1+j_1 = i + r} \delta_{j,j_1}(-1)^j\nabla^j_s\nabla^{j_1}_i\sum \limits_{i_2+j_2 = j + j_1}\delta_{j_2}\nabla^{j_2}_{j}\sum\limits_{\substack{a,g,u \in A_+ \\ d = \deg a > \deg g> \deg u}}\frac{1}{a^{i_1}g^{i_2}u^{j_2}}\\\notag
    &+\sum \limits_{i+j = s + t}\sum \limits_{i_1+j_1 = i + r}\delta_{j,j_1}\nabla^j_s\nabla^{j_1}_r\sum \limits_{i_2+j_2 = j + j_1} \delta_{j_2}\nabla^{j_2}_{j_1} \sum\limits_{\substack{c,h,u \in A_+ \\ d = \deg c > \deg h> \deg u}}\frac{1}{c^{i_1}h^{i_2}u^{j_2}}\\\notag
     &+\sum \limits_{i+j = s + t}\sum \limits_{i_1+j_1 = i + r}\delta_{j,j_1}\nabla^j_s\nabla^{j_1}_r\sum \limits_{i_2+j_2 = j + j_1}\delta_{j_2}\nabla^{j_2}_{j}\sum\limits_{\substack{c,g,u \in A_+ \\ d = \deg c > \deg g> \deg u}}\frac{1}{c^{i_1}g^{i_2}u^{j_2}}\\\notag
     &=  \sum \limits_{i+j = s + t}\sum \limits_{i_1+j_1 = i + r} \delta_{j,j_1}(-1)^j\nabla^j_t\nabla^{j_1}_i\sum \limits_{i_2+j_2 = j + j_1} \delta_{j_2}\nabla^{j_2}_{j_1} S_d(i_1,i_2,j_2)\\\notag
    &+ \sum \limits_{i+j = s + t}\sum \limits_{i_1+j_1 = i + r} \delta_{j,j_1}(-1)^j\nabla^j_t\nabla^{j_1}_i\sum \limits_{i_2+j_2 = j + j_1}\delta_{j_2}\nabla^{j_2}_{j}S_d(i_1,i_2,j_2)\\\notag
    &+\sum \limits_{i+j = s + t}\sum \limits_{i_1+j_1 = i + r}\delta_{j,j_1}\nabla^j_t\nabla^{j_1}_r\sum \limits_{i_2+j_2 = j + j_1} \delta_{j_2}\nabla^{j_2}_{j_1} S_d(i_1,i_2,j_2)\\\notag
    &+\sum \limits_{i+j = s + t}\sum \limits_{i_1+j_1 = i + r}\delta_{j,j_1}\nabla^j_t\nabla^{j_1}_r\sum \limits_{i_2+j_2 = j + j_1}\delta_{j_2}\nabla^{j_2}_{j}S_d(i_1,i_2,j_2)\\\notag
    &+  \sum \limits_{i+j = s + t}\sum \limits_{i_1+j_1 = i + r} \delta_{j,j_1}(-1)^j\nabla^j_s\nabla^{j_1}_i\sum \limits_{i_2+j_2 = j + j_1} \delta_{j_2}\nabla^{j_2}_{j_1}S_d(i_1,i_2,j_2)\\\notag
    &+ \sum \limits_{i+j = s + t}\sum \limits_{i_1+j_1 = i + r} \delta_{j,j_1}(-1)^j\nabla^j_s\nabla^{j_1}_i\sum \limits_{i_2+j_2 = j + j_1}\delta_{j_2}\nabla^{j_2}_{j}S_d(i_1,i_2,j_2)\\\notag
    &+\sum \limits_{i+j = s + t}\sum \limits_{i_1+j_1 = i + r}\delta_{j,j_1}\nabla^j_s\nabla^{j_1}_r\sum \limits_{i_2+j_2 = j + j_1} \delta_{j_2}\nabla^{j_2}_{j_1} S_d(i_1,i_2,j_2)\\\notag
     &+\sum \limits_{i+j = s + t}\sum \limits_{i_1+j_1 = i + r}\delta_{j,j_1}\nabla^j_s\nabla^{j_1}_r\sum \limits_{i_2+j_2 = j + j_1}\delta_{j_2}\nabla^{j_2}_{j}S_d(i_1,i_2,j_2)\\\notag
     &=  \sum \limits_{i+j = s + t}\sum \limits_{i_1+j_1 = i + r} \delta_{j,j_1}(-1)^j\nabla^j_t\nabla^{j_1}_i\sum \limits_{i_2+j_2 = j + j_1} \Delta^{j_2}_{j,j_1} S_d(i_1,i_2,j_2)\\\notag
    &+\sum \limits_{i+j = s + t}\sum \limits_{i_1+j_1 = i + r}\delta_{j,j_1}\nabla^j_t\nabla^{j_1}_r\sum \limits_{i_2+j_2 = j + j_1} \Delta^{j_2}_{j,j_1} S_d(i_1,i_2,j_2)\\\notag
    &+  \sum \limits_{i+j = s + t}\sum \limits_{i_1+j_1 = i + r} \delta_{j,j_1}(-1)^j\nabla^j_s\nabla^{j_1}_i\sum \limits_{i_2+j_2 = j + j_1} \Delta^{j_2}_{j,j_1}S_d(i_1,i_2,j_2)\\\notag
    &+\sum \limits_{i+j = s + t}\sum \limits_{i_1+j_1 = i + r}\delta_{j,j_1}\nabla^j_s\nabla^{j_1}_r\sum \limits_{i_2+j_2 = j + j_1} \Delta^{j_2}_{j,j_1} S_d(i_1,i_2,j_2).\notag
\end{align*}
Since the partial fraction decomposition of $P_{\lambda, \mu , \eta}(A,F,U)$ obtained from the process is unique, it follows that the above expansions of 
\begin{equation*}
    \sum \limits_{(a,b,c)\in N_4} \frac{1}{a^rb^s} \cdot \frac{1}{c^t} \quad \text{and} \quad \sum\limits_{(a,b,c)\in N_4} \frac{1}{a^r} \cdot \frac{1}{b^s c^ t}
\end{equation*}
yield the same expression in terms of power sums.\\
\par Using Formulas \eqref{eq: formula 1}, \eqref{eq: formula 2} and \eqref{eq: formula 3}, one verifies easily that 
\begin{align*}
    &\sum \limits_{(a,b,c)\in N_1} \frac{1}{a^rb^s} \cdot \frac{1}{c^t} + \sum \limits_{(a,b,c)\in N_2} \frac{1}{a^rb^s} \cdot \frac{1}{c^t} + \sum \limits_{(a,b,c)\in N_3} \frac{1}{a^rb^s} \cdot \frac{1}{c^t}+ \sum \limits_{(a,b,c)\in N_4} \frac{1}{a^rb^s} \cdot \frac{1}{c^t}\\
    &= \sum_{i+j=r+s} \Delta^j_{r,s} \sum_{i_1+j_1=i+t} \Delta^{j_1}_{i,t} \Bigg(S_d(i_1,j_1,j)+S_d(i_1,j,j_1) + \sum \limits_{i_2 + j_2 =  j+j_1 } \Delta_{j,j_1}^{j_2} S_d(i_1,i_2,j_2)\Bigg),
\end{align*}
which is the sum of the terms of depth $3$ in the expression of $(S_d(r)S_d(s))S_d(t)$. Similarly, one verifies easily that 
\begin{align*}
    &\sum \limits_{(a,b,c)\in N_1} \frac{1}{a^r} \cdot \frac{1}{b^s c^ t} + \sum \limits_{(a,b,c)\in N_2} \frac{1}{a^r} \cdot \frac{1}{b^s c^ t} + \sum \limits_{(a,b,c)\in N_3} \frac{1}{a^r} \cdot \frac{1}{b^s c^ t}+ \sum \limits_{(a,b,c)\in N_4} \frac{1}{a^r} \cdot \frac{1}{b^s c^ t}\\
    &=  \sum_{i+j=s+t} \Delta^j_{s,t} \sum_{i_1+j_1=i+r} \Delta^{j_1}_{i,r} \Bigg(S_d(i_1,j_1,j)+S_d(i_1,j,j_1) + \sum \limits_{i_2 + j_2 =  j+j_1 } \Delta_{j,j_1}^{j_2} S_d(i_1,i_2,j_2)\Bigg),
\end{align*}
which is the sum of  the terms of depth $3$ in the expression of $S_d(r)(S_d(s)S_d(t))$. This completes the proof.

\subsection{Expansions for $S_{<d}$ of depth one} \label{sec: 4} ${}$\par

For two positive integers $r,s$ and for all $d \in \mathbb{N}$, we first recall the following formula:
\begin{equation} \label{eq: formula}
    S_{<d}(r)S_{<d}(s) = \sum \limits_{i<d} S_i(r)S_i(s) + S_{<d}(r,s) + S_{<d}(s,r).
\end{equation}
 Let $r,s,t$ be positive integers. In this section, we give the expansions of $(S_{<d}(r)S_{<d}(s))S_{<d}(t)$ and $S_{<d}(r)(S_{<d}(s)S_{<d}(t))$ by using Formula \eqref{eq: formula}.
\subsubsection{} We expand $(S_{<d}(r)S_{<d}(s))S_{<d}(t)$ as follows. We have;
\begin{align*}
   &(S_{<d}(r)S_{<d}(s))S_{<d}(t)\\
   &= \Bigg(\sum \limits_{i<d} S_i(r)S_i(s) + S_{<d}(r,s) + S_{<d}(s,r)\Bigg)S_{<d}(t)\\
   &=\sum \limits_{i<d} S_i(r)S_i(s)S_{<d}(t) + S_{<d}(r,s)S_{<d}(t) + S_{<d}(s,r)S_{<d}(t)\\
   &=\sum \limits_{i<d} S_i(r)S_i(s) \sum \limits_{i<d} S_i(t) + \sum \limits_{i<d} S_i(r,s)\sum \limits_{i<d} S_i(t)+ \sum \limits_{i<d} S_i(s,r)\sum \limits_{i<d} S_i(t)\\
   &=\Bigg(\sum\limits_{i<d} (S_i(r)S_i(s))S_i(t) + \sum\limits_{i<d} (S_i(r)S_i(s))S_{<i}(t) + \sum\limits_{i<d} S_i(t) \sum \limits_{j<i} S_j(r)S_j(s)\Bigg) \\
   &+ \Bigg(\sum\limits_{i<d}S_i(r,s)S_i(t) + \sum\limits_{i<d}S_i(r,s)S_{<i}(t) + \sum\limits_{i<d} S_i(t)S_{<i}(r,s)\Bigg)\\
   &+ \Bigg(\sum\limits_{i<d}S_i(s,r)S_i(t) + \sum\limits_{i<d}S_i(s,r)S_{<i}(t) + \sum\limits_{i<d} S_i(t)S_{<i}(s,r)\Bigg)\\
   &= \sum\limits_{i<d} (S_i(r)S_i(s))S_i(t) + \sum\limits_{i<d} (S_i(r)S_i(s))S_{<i}(t) + \sum\limits_{i<d}S_i(r,s)S_i(t) + \sum\limits_{i<d}S_i(r,s)S_{<i}(t)\\
   &+ \sum\limits_{i<d}S_i(s,r)S_i(t) + \sum\limits_{i<d}S_i(s,r)S_{<i}(t) + \sum\limits_{i<d} S_i(t)\Bigg(\sum \limits_{j<i} S_j(r)S_j(s) + S_{<i}(r,s) + S_{<i}(s,r)\Bigg).
\end{align*}

\subsubsection{} We expand $S_{<d}(r)(S_{<d}(s)S_{<d}(t))$ as follows. We have;
\begin{align*}
    &S_{<d}(r)(S_{<d}(s)S_{<d}(t))\\
    &= S_{<d}(r) \Bigg(\sum \limits_{i<d} S_i(s)S_i(t) + S_{<d}(s,t) + S_{<d}(t,s)\Bigg)\\
    &= S_{<d}(r) \sum \limits_{i<d} S_i(s)S_i(t) + S_{<d}(r)S_{<d}(s,t) + S_{<d}(r)S_{<d}(t,s)\\
    &= \sum\limits_{i<d}S_{i}(r)\sum \limits_{i<d} S_i(s)S_i(t)  + \sum\limits_{i<d}S_{i}(r)\sum \limits_{i<d} S_i(s,t) + \sum\limits_{i<d}S_{i}(r)\sum \limits_{i<d} S_i(t,s)\\
    &= \Bigg(\sum\limits_{i<d}S_{i}(r)(S_i(s)S_i(t)) + \sum\limits_{i<d}S_{i}(r)\sum\limits_{j<i}S_j(s)S_j(t) + \sum\limits_{i<d}(S_i(s)S_i(t))S_{<i}(r)\Bigg)\\
    &+ \Bigg(\sum\limits_{i<d}S_{i}(r)S_i(s,t) + \sum\limits_{i<d}S_{i}(r)S_{<i}(s,t) + \sum\limits_{i<d}S_i(s,t)S_{<i}(r)\Bigg)\\
    &+ \Bigg(\sum\limits_{i<d}S_{i}(r)S_i(t,s) + \sum\limits_{i<d}S_{i}(r)S_{<i}(t,s) + \sum\limits_{i<d}S_i(t,s)S_{<i}(r)\Bigg)\\
    &= \sum\limits_{i<d}S_{i}(r)(S_i(s)S_i(t)) + \sum\limits_{i<d}(S_i(s)S_i(t))S_{<i}(r) + \sum\limits_{i<d}S_{i}(r)S_i(s,t) + \sum\limits_{i<d}S_i(s,t)S_{<i}(r)\\
    &+ \sum\limits_{i<d}S_{i}(r)S_i(t,s) + \sum\limits_{i<d}S_i(t,s)S_{<i}(r) + \sum\limits_{i<d}S_{i}(r) \Bigg(\sum\limits_{j<i}S_j(s)S_j(t) + S_{<i}(s,t) + S_{<i}(t,s)\Bigg).
\end{align*}

\subsection{Associativity for $S_{<d}$ of depth one} \label{sec: 5}
\begin{theorem} \label{thm: assoc S_<d depth one}
Let $r,s , t$ be positive integers. For all $d \in \mathbb{N}$, the expansions using Chen's formula and Formula \eqref{eq: formula} of $(S_{<d}(r)S_{<d}(s))S_{<d}(t)$ and $S_{<d}(r)(S_{<d}(s)S_{<d}(t))$ yield the same expression in terms of $S_{<d}$.
\end{theorem}
\begin{proof}
To prove the desired associativity, we compare the expansions of $(S_{<d}(r)S_{<d}(s))S_{<d}(t)$ and $S_{<d}(r)(S_{<d}(s)S_{<d}(t))$ in Section \ref{sec: 4}. From Theorem \ref{thm: assoc S_d depth 1}, it is obvious that the expansions of 
\begin{equation*}
    \sum\limits_{i<d} (S_i(r)S_i(s))S_i(t) \quad \text{and} \quad \sum\limits_{i<d}S_{i}(r)(S_i(s)S_i(t))
\end{equation*}
yield the same expression in terms of $S_{<d}$. From Formula \eqref{eq: formula}, one verifies easily that the expansions of 
\begin{equation*}
    \sum\limits_{i<d} S_i(t)\Bigg(\sum \limits_{j<i} S_j(r)S_j(s) + S_{<i}(r,s) + S_{<i}(s,r)\Bigg) \quad \text{and} \quad \sum\limits_{i<d}S_i(t,s)S_{<i}(r)
\end{equation*}
yield the same expression in terms of $S_{<d}$. Similarly, one deduces that the expansions of 
\begin{equation*}
    \sum\limits_{i<d}S_i(r,s)S_{<i}(t) \quad \text{and} \quad \sum\limits_{i<d}S_{i}(r) \Bigg(\sum\limits_{j<i}S_j(s)S_j(t) + S_{<i}(s,t) + S_{<i}(t,s)\Bigg)
\end{equation*}
yield the same expression in terms of $S_{<d}$. From a straightforward verification, one deduces that the expansions of 
\begin{equation*}
    \sum\limits_{i<d} (S_i(r)S_i(s))S_{<i}(t), \sum\limits_{i<d}S_i(r,s)S_i(t),\sum\limits_{i<d}S_i(s,r)S_i(t),\sum\limits_{i<d}S_i(s,r)S_{<i}(t)
\end{equation*} 
yield respectively the same expressions in terms of $S_{<d}$ as those of
\begin{equation*}
    \sum\limits_{i<d} S_i(r)S_i(s,t),\sum\limits_{i<d}S_i(r)S_i(t,s),\sum\limits_{i<d}(S_i(s)S_i(t))S_{<i}(r),\sum\limits_{i<d}S_i(s,t)S_{<i}(r).
\end{equation*}
This proves the theorem.
\end{proof}


\subsection{Expansions of arbitrary depth} \label{sec: 6} ${}$\par

We now extend our results for the case of arbitrary depth. Let $\fa=(a_1,\dotsc, a_m)$ and $\fb=(b_1,\dotsc, b_n)$ be two positive tuples. For simplicity, we set $\fa_- = (a_2,\dotsc, a_m)$ and $\fb_- = (b_2,\dotsc, b_n)$. For all $d \in \mathbb{N}$, we recall the following formulas (See \cite{ND21}):
\begin{align}
    S_d(\fa) S_d(\fb) &= (S_d(a_1)S_d(b_1))(S_{<d}(\fa_-)S_{<d}(\fb_-))\label{eq: SdSd}\\
    &=S_d(a_1+b_1)(S_{<d}(\fa_-)S_{<d}(\fb_-)) \notag\\ 
    &+ \sum \limits_{i+j = a_1+b_1}\Delta_{a_1,b_1}^{j}S_d(i)[S_{<d}(j)(S_{<d}(\fa_-)S_{<d}(\fb_-))], \notag\\
    S_d(\fa)S_{<d}(\fb) &= S_{<d}(\fb)S_d(\fa) = S_d(a_1)(S_{<d}(\fa_-)S_{<d}(\fb)),\label{eq: SdS<d}\\
    S_{<d}(\fa) S_{<d}(\fb) &= \sum \limits_{i<d}S_i(\fa)S_i(\fb) + \sum \limits_{i<d}S_i(\fa)S_{<i}(\fb) + \sum \limits_{i<d}S_i(\fb)S_{<i}(\fa).\label{eq: S<dS<d}
\end{align}
Let $\fa,\fb,\fc$ be positive tuples. We first give the expansions of $(S_{d}(\fa)S_{d}(\fb))S_{d}(\fc)$ and $S_{d}(\fa)(S_{d}(\fb)S_{d}(\fc))$ by using Formula \eqref{eq: SdSd}.

\subsubsection{} We expand $(S_{d}(\fa)S_{d}(\fb))S_{d}(\fc)$ as follows. We have;
\begin{align} \label{eq: S_d L-R}
    &(S_{d}(\fa)S_{d}(\fb))S_{d}(\fc)\\
    &= \Bigg(S_d(a_1+b_1)(S_{<d}(\fa_-)S_{<d}(\fb_-))+ \sum \limits_{i+j = a_1+b_1}\Delta_{a_1,b_1}^{j}S_d(i)[S_{<d}(j)(S_{<d}(\fa_-)S_{<d}(\fb_-))]\Bigg)S_{d}(\fc)\notag\\
    &= S_d(a_1+b_1)(S_{<d}(\fa_-)S_{<d}(\fb_-))S_{d}(\fc) + \sum \limits_{i+j = a_1+b_1}\Delta_{a_1,b_1}^{j}S_d(i)[S_{<d}(j)(S_{<d}(\fa_-)S_{<d}(\fb_-))]S_{d}(\fc)\notag\\
    &= S_d(a_1+b_1)S_d(c_1)[(S_{<d}(\fa_-)S_{<d}(\fb_-))S_{<d}(\fc_-)]\notag\\
    &+ \sum \limits_{i+j = a_1+b_1}\Delta_{a_1,b_1}^{j}(S_d(i)S_d(c_1))[S_{<d}(j)(S_{<d}(\fa_-)S_{<d}(\fb_-))]S_{<d}(\fc_-).\notag
\end{align}
\subsubsection{} We expand $S_{d}(\fa)(S_{d}(\fb)S_{d}(\fc))$ as follows. We have;
\begin{align}\label{eq: S_d R-L}
    &S_{d}(\fa)(S_{d}(\fb)S_{d}(\fc))\\
    &= S_{d}(\fa)\Bigg(S_d(b_1+c_1)(S_{<d}(\fb_-)S_{<d}(\fc_-))+ \sum \limits_{i+j = b_1+c_1}\Delta_{b_1,c_1}^{j}S_d(i)[S_{<d}(j)(S_{<d}(\fb_-)S_{<d}(\fc_-))]\Bigg)\notag\\
    &= S_{d}(\fa)S_d(b_1+c_1)(S_{<d}(\fb_-)S_{<d}(\fc_-))+ S_{d}(\fa)\sum \limits_{i+j = b_1+c_1}\Delta_{b_1,c_1}^{j}S_d(i)[S_{<d}(j)(S_{<d}(\fb_-)S_{<d}(\fc_-))]\notag\\
    &= (S_{d}(a_1)S_d(b_1+c_1))[S_{<d}(\fa_-)(S_{<d}(\fb_-)S_{<d}(\fc_-))]\notag\\
    &+ \sum \limits_{i+j = b_1+c_1}\Delta_{b_1,c_1}^{j}(S_d(a_1)S_d(i))S_{<d}(\fa_-)[S_{<d}(j)(S_{<d}(\fb_-)S_{<d}(\fc_-))].\notag
\end{align}

We next give the expansions of $(S_{<d}(\fa)S_{<d}(\fb))S_{<d}(\fc)$ and $S_{<d}(\fa)(S_{<d}(\fb)S_{<d}(\fc))$ by using Formula \eqref{eq: S<dS<d}.
\subsubsection{} We expand $(S_{<d}(\fa)S_{<d}(\fb))S_{<d}(\fc)$ as follows. We have;
\begin{align}\label{eq: S<d L-R}
    &(S_{<d}(\fa)S_{<d}(\fb))S_{<d}(\fc)\\
    &= \Bigg(\sum \limits_{i<d}S_i(\fa)S_i(\fb) + \sum \limits_{i<d}S_i(\fa)S_{<i}(\fb) + \sum \limits_{i<d}S_i(\fb)S_{<i}(\fa)\Bigg)S_{<d}(\fc)\notag\\
    &= \sum \limits_{i<d}S_i(\fa)S_i(\fb)S_{<d}(\fc) + \sum \limits_{i<d}S_i(\fa)S_{<i}(\fb)S_{<d}(\fc) + \sum \limits_{i<d}S_i(\fb)S_{<i}(\fa)S_{<d}(\fc)\notag\\
    &= \sum \limits_{i<d}S_i(\fa)S_i(\fb)\sum\limits_{i<d}S_{i}(\fc) + \sum \limits_{i<d}S_i(\fa)S_{<i}(\fb)\sum\limits_{i<d}S_{i}(\fc) + \sum \limits_{i<d}S_i(\fb)S_{<i}(\fa)\sum\limits_{i<d}S_{i}(\fc)\notag\\
    &= \Bigg(\sum \limits_{i<d}(S_i(\fa)S_i(\fb))S_i(\fc) + \sum \limits_{i<d}(S_i(\fa)S_i(\fb))S_{<i}(\fc) + \sum \limits_{i<d}S_i(\fc)\sum \limits_{j<i}S_j(\fa)S_j(\fb) \Bigg)\notag\\
    &+ \Bigg(\sum \limits_{i<d}(S_i(\fa)S_{<i}(\fb))S_i(\fc) + \sum \limits_{i<d}(S_i(\fa)S_{<i}(\fb))S_{<i}(\fc) + \sum \limits_{i<d}S_i(\fc)\sum \limits_{j<i}S_j(\fa)S_{<j}(\fb)\Bigg) \notag\\
    &+ \Bigg(\sum \limits_{i<d}(S_i(\fb)S_{<i}(\fa))S_i(\fc) + \sum \limits_{i<d}(S_i(\fb)S_{<i}(\fa))S_{<i}(\fc) + \sum \limits_{i<d}S_i(\fc)\sum \limits_{j<i}S_j(\fb)S_{<j}(\fa)\Bigg) \notag\\
    &=\sum \limits_{i<d}(S_i(\fa)S_i(\fb))S_i(\fc) + \sum \limits_{i<d}(S_i(\fa)S_i(\fb))S_{<i}(\fc) + \sum \limits_{i<d}(S_i(\fa)S_{<i}(\fb))S_i(\fc)\notag\\
    & + \sum \limits_{i<d}(S_i(\fa)S_{<i}(\fb))S_{<i}(\fc) + \sum \limits_{i<d}(S_i(\fb)S_{<i}(\fa))S_i(\fc) + \sum \limits_{i<d}(S_i(\fb)S_{<i}(\fa))S_{<i}(\fc) \notag\\
    &+ \sum \limits_{i<d}S_i(\fc)\Bigg(\sum \limits_{j<i}S_j(\fa)S_j(\fb) + \sum \limits_{j<i}S_j(\fa)S_{<j}(\fb) + \sum \limits_{j<i}S_j(\fb)S_{<j}(\fa) \Bigg).\notag
\end{align}

\subsubsection{} We expand $S_{<d}(\fa)(S_{<d}(\fb)S_{<d}(\fc))$ as follows. We have;
\begin{align}\label{eq: S<d R-L}
    &S_{<d}(\fa)(S_{<d}(\fb)S_{<d}(\fc))\\
    &= S_{<d}(\fa)\Bigg(\sum \limits_{i<d}S_i(\fb)S_i(\fc) + \sum \limits_{i<d}S_i(\fb)S_{<i}(\fc) + \sum \limits_{i<d}S_i(\fc)S_{<i}(\fb)\Bigg) \notag\\
    &= S_{<d}(\fa)\sum \limits_{i<d}S_i(\fb)S_i(\fc) + S_{<d}(\fa)\sum \limits_{i<d}S_i(\fb)S_{<i}(\fc) + S_{<d}(\fa)\sum \limits_{i<d}S_i(\fc)S_{<i}(\fb) \notag\\
    &= \sum \limits_{i<d}S_{i}(\fa)\sum \limits_{i<d}S_i(\fb)S_i(\fc) + \sum \limits_{i<d}S_{i}(\fa)\sum \limits_{i<d}S_i(\fb)S_{<i}(\fc) + \sum \limits_{i<d}S_{i}(\fa)\sum \limits_{i<d}S_i(\fc)S_{<i}(\fb) \notag\\
    &= \Bigg(\sum \limits_{i<d}S_{i}(\fa)(S_i(\fb)S_i(\fc)) + \sum \limits_{i<d}S_{i}(\fa)  \sum \limits_{j<i}S_j(\fb)S_j(\fc) +  \sum \limits_{i<d}(S_i(\fb)S_i(\fc))S_{<i}(\fa)\Bigg) \notag\\
    &+ \Bigg(\sum \limits_{i<d}S_{i}(\fa)(S_i(\fb)S_{<i}(\fc)) + \sum \limits_{i<d}S_{i}(\fa)  \sum \limits_{j<i}S_j(\fb)S_{<j}(\fc) +  \sum \limits_{i<d}(S_i(\fb)S_{<i}(\fc))S_{<i}(\fa)\Bigg)\notag\\
    &+ \Bigg(\sum \limits_{i<d}S_{i}(\fa)(S_i(\fc)S_{<i}(\fb)) + \sum \limits_{i<d}S_{i}(\fa)  \sum \limits_{j<i}S_j(\fc)S_{<j}(\fb) +  \sum \limits_{i<d}(S_i(\fc)S_{<i}(\fb))S_{<i}(\fa)\Bigg)\notag\\
    &=\sum \limits_{i<d}S_{i}(\fa)(S_i(\fb)S_i(\fc)) + \sum \limits_{i<d}(S_i(\fb)S_i(\fc))S_{<i}(\fa) + 
    \sum \limits_{i<d}S_{i}(\fa)(S_i(\fb)S_{<i}(\fc))\notag\\
    &+ \sum \limits_{i<d}(S_i(\fb)S_{<i}(\fc))S_{<i}(\fa) + \sum \limits_{i<d}S_{i}(\fa)(S_i(\fc)S_{<i}(\fb)) + \sum \limits_{i<d}(S_i(\fc)S_{<i}(\fb))S_{<i}(\fa) \notag\\
    &+ \sum \limits_{i<d}S_{i}(\fa) \Bigg(\sum \limits_{j<i}S_j(\fb)S_j(\fc) + \sum \limits_{j<i}S_j(\fb)S_{<j}(\fc) + \sum \limits_{j<i}S_j(\fc)S_{<j}(\fb)\Bigg). \notag
\end{align}

\subsection{Associativity of arbitrary depth} \label{sec: 7} ${}$\par

\begin{theorem} \label{thm: assoc}
Let $\fa,\fb,\fc$ be positive tuples. 
\begin{enumerate}[$(1)$]
\item For all $d \in \mathbb{N}$, the expansions using \eqref{eq: SdSd}, \eqref{eq: SdS<d}, \eqref{eq: S<dS<d} of $(S_d(\fa) S_d(\fb))S_d(\fc)$ and $S_d(\fa) (S_d(\fb)S_d(\fc))$ yield the same expression in terms of power sums.
\item  For all $d \in \mathbb{N}$, the expansions using \eqref{eq: SdSd}, \eqref{eq: SdS<d}, \eqref{eq: S<dS<d} of $(S_{<d}(\fa) S_{<d}(\fb))S_{<d}(\fc)$ and $S_{<d}(\fa) (S_{<d}(\fb)S_{<d}(\fc))$ yield the same expression in terms of $S_{<d}$.
\end{enumerate}
\end{theorem}

\begin{proof}
We proceed the proof by induction on $\depth(\fa) + \depth(\fb) + \depth(\fc)$. The base step $\depth(\fa) + \depth(\fb) + \depth(\fc)= 3$, i.e., $\depth(\fa) = \depth(\fb) =  \depth(\fc)= 1$ of Theorem \ref{thm: assoc} follows from Theorem \ref{thm: assoc S_d depth 1} and  Theorem \ref{thm: assoc S_<d depth one}.

Assume that Theorem \ref{thm: assoc} holds when $\depth(\fa) + \depth(\fb) + \depth(\fc) < n$ with $n \in \mathbb{N}$ and $n \geq 4$. We need to show that Theorem \ref{thm: assoc} holds when $\depth(\fa) + \depth(\fb) + \depth(\fc) = n$. 

In order to prove Part (1), we apply the induction hypothesis on Expansions \eqref{eq: S_d L-R} and \eqref{eq: S_d R-L}. We deduce that the expansion of $(S_d(\fa) S_d(\fb))S_d(\fc)$ yields the same expression in terms of power sums as that of
\begin{align*}
    &\Bigg[\Bigg(S_d(a_1 + b_1) + \sum \limits_{i+j = a_1 + b_1} \Delta^j_{a_1,b_1}S_d(i,j)\Bigg)S_d(c_1)\Bigg][(S_{<d}(\fa_-)S_{<d}(\fb_-))S_{<d}(\fc_-)]\\
    &= [(S_d(a_1)S_d(b_1))S_d(c_1)][(S_{<d}(\fa_-)S_{<d}(\fb_-))S_{<d}(\fc_-)],
\end{align*}
and the expansion of $S_d(\fa) (S_d(\fb)S_d(\fc))$ yields the same expression in terms of power sums as that of 
\begin{align*}
    &\Bigg[S_d(a_1)\Bigg(S_d(b_1 + c_1) + \sum \limits_{i+j = b_1 + c_1} \Delta^j_{b_1,c_1}S_d(i,j)\Bigg)\Bigg][S_{<d}(\fa_-)(S_{<d}(\fb_-)S_{<d}(\fc_-))]\\
    &=[S_d(a_1)(S_d(b_1)S_d(c_1))][S_{<d}(\fa_-)(S_{<d}(\fb_-)S_{<d}(\fc_-))].
\end{align*}
Using Theorem \ref{thm: assoc S_d depth 1} and the induction hypothesis again, we conclude that expansions of $(S_d(\fa) S_d(\fb))S_d(\fc)$ and $S_d(\fa) (S_d(\fb)S_d(\fc))$ yield the same expression in terms of power sums.

In order to prove Part (2), we compare Expansions \eqref{eq: S<d L-R} and \eqref{eq: S<d R-L}. From Part (1), it is obvious that the expansions of 
\begin{equation*}
     \sum \limits_{i<d}(S_i(\fa)S_i(\fb))S_i(\fc)\quad \text{and} \quad \sum \limits_{i<d}S_i(\fa)(S_i(\fb)S_i(\fc))
\end{equation*}
yield the same expression in terms of $S_{<d}$.

Note that from Formulas \eqref{eq: SdS<d} and \eqref{eq: S<dS<d}, we have;
\begin{align*}
    &S_i(\fc)\Bigg(\sum \limits_{j<i}S_j(\fa)S_j(\fb) + \sum \limits_{j<i}S_j(\fa)S_{<j}(\fb) + \sum \limits_{j<i}S_j(\fb)S_{<j}(\fa) \Bigg) \\
    &= S_i(\fc)(S_{<i}(\fb)S_{<i}(\fa))\\
    &= S_i(c_1)[S_{<i}(\fc_-)(S_{<i}(\fb)S_{<i}(\fa))]
\end{align*}
and 
\begin{align*}
    (S_i(\fc)S_{<i}(\fb))S_{<i}(\fa)  = S_i(c_1) [(S_{<i}(\fc_-)S_{<i}(\fb))S_{<i}(\fa)].
\end{align*}
One then deduces from the induction hypothesis that the expansions of
\begin{align*}
    \sum \limits_{i<d}S_i(\fc)\Bigg(\sum \limits_{j<i}S_j(\fa)S_j(\fb) + \sum \limits_{j<i}S_j(\fa)S_{<j}(\fb) + \sum \limits_{j<i}S_j(\fb)S_{<j}(\fa) \Bigg)
\end{align*}
and 
\begin{align*}    
    \sum \limits_{i<d} (S_i(\fc)S_{<i}(\fb))S_{<i}(\fa)
\end{align*}
yield the same expression in terms of $S_{<d}$. Similarly, one verifies easily that the expansions of
\begin{align*}
   \sum \limits_{i<d}(S_i(\fa)S_{<i}(\fb))S_{<i}(\fc)
\end{align*}
and 
\begin{align*}   
   \sum \limits_{i<d}S_{i}(\fa) \Bigg(\sum \limits_{j<i}S_j(\fb)S_j(\fc) + \sum \limits_{j<i}S_j(\fb)S_{<j}(\fc) + \sum \limits_{j<i}S_j(\fc)S_{<j}(\fb)\Bigg)
\end{align*}
yield the same expression in terms of $S_{<d}$.

Note that from Formulas \eqref{eq: SdSd} and \eqref{eq: SdS<d}, we have; 
\begin{align*}
    (S_i(\fa)S_i(\fb))S_{<i}(\fc) = (S_i(a_1)S_i(b_1))[(S_{<i}(\fa_-)S_{<i}(\fb_-))S_{<i}(\fc)]
\end{align*}
and
\begin{align*}
    S_i(\fa)(S_i(\fb)S_{<i}(\fc)) = (S_i(a_1)S_i(b_1))[S_{<i}(\fa_-)(S_{<i}(\fb_-)S_{<i}(\fc))].
\end{align*}
One then deduces from the induction hypothesis that the expansions of 
\begin{align*}
    \sum \limits_{i<d}(S_i(\fa)S_i(\fb))S_{<i}(\fc) \quad \text{and} \quad \sum \limits_{i<d}S_i(\fa)(S_i(\fb)S_{<i}(\fc))
\end{align*}
yield the same expression in terms of $S_{<d}$. Similarly, one verifies easily that the expansions of
\begin{align*}
    \sum \limits_{i<d}(S_i(\fa)S_{<i}(\fb))S_i(\fc), \sum \limits_{i<d}(S_i(\fb)S_{<i}(\fa))S_i(\fc), \sum \limits_{i<d}(S_i(\fb)S_{<i}(\fa))S_{<i}(\fc)
\end{align*}
yield respectively the same expressions in terms of $S_{<d}$ as those of 
\begin{align*}
    \sum \limits_{i<d}S_i(\fa)(S_i(\fc)S_{<i}(\fb)), \sum \limits_{i<d}(S_i(\fb)S_i(\fc))S_{<i}(\fa), \sum \limits_{i<d}(S_i(\fb)S_{<i}(\fc))S_{<i}(\fa). 
\end{align*}
From the above arguments, we conclude that the expansions of $(S_{<d}(\fa) S_{<d}(\fb))S_{<d}(\fc)$ and $S_{<d}(\fa) (S_{<d}(\fb)S_{<d}(\fc))$ yield the same expression in terms of $S_{<d}$. This completes the proof.
 \end{proof}

\subsection{Associativity of the shuffle algebra}

\begin{proposition} \label{prop: associative}
  The diamond product and the shuffle product are associative.
\end{proposition}

\begin{proof}
Let $\fa, \fb, \fc \in \langle \Sigma \rangle$ be arbitrary words. It is suffices to show that 
\begin{equation} \label{eq: assocative}
    (\fa \diamond \fb) \diamond \fc = \fa \diamond (\fb \diamond \fc) \quad \text{and} \quad (\fa \shuffle  \fb) \shuffle  \fc = \fa \shuffle  (\fb \shuffle \fc) .
\end{equation}
We proceed the proof by induction on $\depth(\fa) + \depth(\fb) + \depth(\fc)$. If one of $\fa, \fb$ or $\fc$ is empty word, then \eqref{eq: assocative} holds trivially. We assume that \eqref{eq: assocative} holds when $\depth(\fa) + \depth(\fb) + \depth(\fc) < n$ with $n \in \N$ and $n \geq 3$. We need to show that \eqref{eq: assocative} holds when $\depth(\fa) + \depth(\fb) + \depth(\fc) = n$.

We first show that $(\fa \diamond \fb) \diamond \fc = \fa \diamond (\fb \diamond \fc)$. From Lemma \ref{lem: triangle formulas}, we have
\begin{align*}
& (\fa \diamond \fb) \diamond \fc \\
    &=  \Bigg[x_{a + b}(\fa_-\shuffle  \fb_-) + \sum\limits_{i+j = a + b} \Delta^j_{a,b} x_i(x_j \shuffle  (\fa_- \shuffle  \fb_-))\Bigg] \diamond \fc\\
    &= \ x_{a + b}(\fa_-\shuffle  \fb_-) \diamond \fc + \sum\limits_{i+j = a + b} \Delta^j_{a,b} x_i(x_j \shuffle  (\fa_- \shuffle  \fb_-)) \diamond \fc\\
    &= \ (x_{a + b} \diamond x_c) \triangleright ((\fa_-\shuffle  \fb_-) \shuffle  \fc_-) + \sum\limits_{i+j = a + b} \Delta^j_{a,b} (x_i \diamond x_c) \triangleright [(x_j \shuffle  (\fa_- \shuffle  \fb_-)) \shuffle  \fc_-].
\end{align*}
For all $i,j \in \N$ with $i+j = a+b$, it follows from the induction hypothesis that
\begin{align*}
    &(x_i \diamond x_c) \triangleright [(x_j \shuffle  (\fa_- \shuffle  \fb_-)) \shuffle  \fc_-]\\ &= \ (x_i \diamond x_c) \triangleright [x_j \shuffle  ((\fa_- \shuffle  \fb_-) \shuffle  \fc_-)]\\
    &=  \Bigg[x_{i + c} + \sum\limits_{i_1+j_1 = i + c} \Delta^{j_1}_{i,c} x_{i_1}x_{j_1}\Bigg] \triangleright [x_j \shuffle  ((\fa_- \shuffle  \fb_-) \shuffle  \fc_-)]\\
    &= \ x_{i + c}\triangleright [x_j \shuffle  ((\fa_- \shuffle  \fb_-) \shuffle  \fc_-)] + \sum\limits_{i_1+j_1 = i + c} \Delta^{j_1}_{i,c} x_{i_1}x_{j_1}\triangleright [x_j \shuffle  ((\fa_- \shuffle  \fb_-) \shuffle  \fc_-)]\\
    &= \ x_{i + c} [x_j \shuffle  ((\fa_- \shuffle  \fb_-) \shuffle  \fc_-)] + \sum\limits_{i_1+j_1 = i + c} \Delta^{j_1}_{i,c} x_{i_1}[(x_{j_1} \shuffle  x_j) \shuffle  ((\fa_- \shuffle  \fb_-) \shuffle  \fc_-)]\\
    &= \ x_{i + c}x_j \triangleright ((\fa_- \shuffle  \fb_-) \shuffle  \fc_-) + \sum\limits_{i_1+j_1 = i + c} \Delta^{j_1}_{i,c} x_{i_1}(x_{j_1} \shuffle  x_j) \triangleright  ((\fa_- \shuffle  \fb_-) \shuffle  \fc_-)\\
    &= \Bigg[ x_{i + c}x_j  + \sum\limits_{i_1+j_1 = i + c} \Delta^{j_1}_{i,c} x_{i_1}(x_{j_1} \shuffle  x_j)\Bigg] \triangleright  ((\fa_- \shuffle  \fb_-) \shuffle  \fc_-)\\
    &= (x_ix_j \diamond x_c) \triangleright ((\fa_- \shuffle  \fb_-) \shuffle  \fc_-).
\end{align*}
Thus
\begin{align*}
    (\fa \diamond \fb) \diamond \fc &= \ (x_{a + b} \diamond x_c) \triangleright ((\fa_-\shuffle  \fb_-) \shuffle  \fc_-) + \sum\limits_{i+j = a + b} \Delta^j_{a,b} (x_ix_j \diamond x_c) \triangleright ((\fa_- \shuffle  \fb_-) \shuffle  \fc_-)\\
    &= \Bigg[(x_{a + b} \diamond x_c) + \sum\limits_{i+j = a + b} \Delta^j_{a,b} (x_ix_j \diamond x_c) \Bigg]\triangleright ((\fa_- \shuffle  \fb_-) \shuffle  \fc_-)\\
    &= \ ((x_a \diamond x_b) \diamond x_c) \triangleright ((\fa_- \shuffle  \fb_-) \shuffle  \fc_-).
\end{align*}
On the other hand, from Proposition \ref{prop: commutative} and the above arguments, we deduce that 
\begin{align*}
    \fa \diamond (\fb \diamond \fc) &= (\fb \diamond \fc) \diamond \fa \\
    &=  ((x_b \diamond x_c) \diamond x_a) \triangleright ((\fb_- \shuffle  \fc_-) \shuffle  \fa_-) \\
    &= (x_a \diamond (x_b \diamond x_c) ) \triangleright (\fa_- \shuffle   (\fb_- \shuffle  \fc_-)).
\end{align*}
It is straightforward from Theorem \ref{thm: assoc S_d depth 1} that $(x_a \diamond x_b) \diamond x_c = x_a \diamond (x_b \diamond x_c)$. Moreover, it follows from the induction hypothesis that $(\fa_- \shuffle  \fb_-) \shuffle  \fc_- = \fa_- \shuffle   (\fb_- \shuffle  \fc_-)$. We thus conclude that $(\fa \diamond \fb) \diamond \fc = \fa \diamond (\fb \diamond \fc)$.

We next show that $(\fa \shuffle  \fb) \shuffle  \fc = \fa \shuffle  (\fb \shuffle  \fc)$. From Lemma \ref{lem: triangle formulas}, we have;
\begin{align*}
    (\fa \shuffle  \fb) \shuffle  \fc 
    &= \ (\fa \triangleright \fb + \fb \triangleright \fa + \fa \diamond \fb) \shuffle  \fc\\
    &= \ (\fa \triangleright \fb) \shuffle  \fc  + (\fb \triangleright \fa) \shuffle  \fc  + (\fa \diamond \fb) \shuffle  \fc\\
    &= \ ((\fa \triangleright \fb) \triangleright \fc + \fc \triangleright (\fa \triangleright \fb) + (\fa \triangleright \fb) \diamond \fc)\\
    &+  ((\fb \triangleright \fa) \triangleright \fc + \fc \triangleright (\fb \triangleright \fa) + (\fb \triangleright \fa) \diamond \fc)\\
    &+  ((\fa \diamond \fb) \triangleright \fc + \fc \triangleright (\fa \diamond \fb) + (\fa \diamond \fb) \diamond \fc)\\
    &= \ (\fa \triangleright \fb) \triangleright \fc + (\fa \triangleright \fb) \diamond \fc + (\fb \triangleright \fa) \triangleright \fc + (\fb \triangleright \fa) \diamond \fc + (\fa \diamond \fb) \triangleright \fc\\
    &+ (\fc \triangleright (\fa \triangleright \fb) + \fc \triangleright (\fb \triangleright \fa) + \fc \triangleright (\fa \diamond \fb)) + (\fa \diamond \fb) \diamond \fc\\
    &= \ (\fa \triangleright \fb) \triangleright \fc + (\fa \triangleright \fb) \diamond \fc + (\fb \triangleright \fa) \triangleright \fc + (\fb \triangleright \fa) \diamond \fc + (\fa \diamond \fb) \triangleright \fc\\
    &+ \fc \triangleright (\fa \shuffle  \fb)  + (\fa \diamond \fb) \diamond \fc
\end{align*}
and 
\begin{align*}
    \fa \shuffle  (\fb \shuffle  \fc) 
    &= \ \fa \shuffle  (\fb \triangleright \fc + \fc \triangleright \fb + \fb \diamond \fc)\\
    &= \ \fa \shuffle  (\fb \triangleright \fc) + \fa \shuffle  (\fc \triangleright \fb) + \fa \shuffle  (\fb \diamond \fc)\\
    &= \ (\fa \triangleright (\fb \triangleright \fc) + (\fb \triangleright \fc) \triangleright \fa + \fa \diamond (\fb \triangleright \fc))\\
    &+ (\fa \triangleright (\fc \triangleright \fb) + (\fc \triangleright \fb) \triangleright \fa + \fa \diamond (\fc \triangleright \fb)) \\
    &+ (\fa \triangleright (\fb \diamond \fc) + (\fb \diamond \fc) \triangleright \fa + \fa \diamond (\fb \diamond \fc))\\
    &= \ (\fb \triangleright \fc) \triangleright \fa + \fa \diamond (\fb \triangleright \fc) + (\fc \triangleright \fb) \triangleright \fa + \fa \diamond (\fc \triangleright \fb) + (\fb \diamond \fc) \triangleright \fa\\
    &+ (\fa \triangleright (\fb \triangleright \fc) + \fa \triangleright (\fc \triangleright \fb) + \fa \triangleright (\fb \diamond \fc)) + \fa \diamond (\fb \diamond \fc)\\
    &= \ (\fb \triangleright \fc) \triangleright \fa + \fa \diamond (\fb \triangleright \fc) + (\fc \triangleright \fb) \triangleright \fa + \fa \diamond (\fc \triangleright \fb) + (\fb \diamond \fc) \triangleright \fa\\
    &+ \fa \triangleright (\fb \shuffle  \fc) + \fa \diamond (\fb \diamond \fc).
\end{align*}

We now compare the above expansions. We have showed that $(\fa \diamond \fb) \diamond \fc = \fa \diamond (\fb \diamond \fc)$. On the other hand, we have;
\begin{align*}
    \fc \triangleright (\fa \shuffle  \fb) = x_c(\fc_- \shuffle (\fa \shuffle  \fb))
\end{align*}
and
\begin{align*}
    (\fc \triangleright \fb) \triangleright \fa = x_c(\fc_- \shuffle  \fb) \triangleright \fa = x_c((\fc_- \shuffle  \fb) \shuffle  \fa).
\end{align*}
From the induction hypothesis and commutativity of shuffle product, one deduces that $\fc \triangleright (\fa \shuffle  \fb) = (\fc \triangleright \fb) \triangleright \fa$. Similarly, one deduces that $(\fa \triangleright \fb) \triangleright \fc= \fa \triangleright (\fb \shuffle  \fc)$.

We have
\begin{align*}
    (\fa \triangleright \fb) \diamond \fc &= x_a(\fa_- \shuffle  \fb) \diamond \fc \\
    &= (x_a \diamond x_c) \triangleright ((\fa_- \shuffle  \fb) \shuffle  \fc_-)
\end{align*}
and
\begin{align*}
    \fa \diamond (\fc \triangleright \fb) &= \fa \diamond x_c(\fc_- \shuffle \fb)\\
    &= (x_a \diamond x_c) \triangleright (\fa_- \shuffle  (\fc_- \shuffle \fb)).
\end{align*}
From the induction hypothesis and commutativity of shuffle product, one deduces that $(\fa \triangleright \fb) \diamond \fc = \fa \diamond (\fc \triangleright \fb)$.  

We have
\begin{align*}
    (\fb \triangleright \fa) \triangleright \fc &= x_b(\fb_- \shuffle  \fa) \triangleright \fc\\
    &=  x_b((\fb_- \shuffle  \fa) \shuffle  \fc)
\end{align*}
and
\begin{align*}
    (\fb \triangleright \fc) \triangleright \fa &= x_b(\fb_- \shuffle  \fc) \triangleright \fa\\
    &= x_b((\fb_- \shuffle  \fc) \shuffle  \fa).
\end{align*}
From the induction hypothesis and commutativity of shuffle product, one deduces that $ (\fb \triangleright \fa) \triangleright \fc = (\fb \triangleright \fc) \triangleright \fa$.

It follows from the induction hypothesis that
\begin{align*}
(\fa \diamond \fb) \triangleright \fc
    &= \Bigg[x_{a+b}(\fa_- \shuffle  \fb_-) + \sum \limits_{i+j = a + b} \Delta^j_{a,b} x_i (x_j \shuffle  (\fa_- \shuffle  \fb_-))\Bigg] \triangleright \fc\\
    &= \ x_{a+b}(\fa_- \shuffle  \fb_-) \triangleright \fc + \sum \limits_{i+j = a + b} \Delta^j_{a,b} x_i (x_j \shuffle  (\fa_- \shuffle  \fb_-)) \triangleright \fc\\
    &= \ x_{a+b}((\fa_- \shuffle  \fb_-) \shuffle  \fc) + \sum \limits_{i+j = a + b} \Delta^j_{a,b} x_i [(x_j \shuffle  (\fa_- \shuffle  \fb_-)) \shuffle \fc]\\
    &= \ x_{a+b}((\fa_- \shuffle  \fb_-) \shuffle  \fc) + \sum \limits_{i+j = a + b} \Delta^j_{a,b} x_i [x_j \shuffle  ((\fa_- \shuffle  \fb_-) \shuffle \fc)]\\
    &= \ x_{a+b} \triangleright ((\fa_- \shuffle  \fb_-) \shuffle  \fc) + \sum \limits_{i+j = a + b} \Delta^j_{a,b} x_ix_j \triangleright ((\fa_- \shuffle  \fb_-) \shuffle \fc)\\
    &= \ \Bigg[x_{a+b}  + \sum \limits_{i+j = a + b} \Delta^j_{a,b} x_ix_j \Bigg] \triangleright ((\fa_- \shuffle  \fb_-) \shuffle \fc)\\
    &= \ (x_a \diamond x_b) \triangleright ((\fa_- \shuffle  \fb_-) \shuffle \fc)\\
\end{align*}
and 
\begin{align*}
    \fa \diamond (\fb \triangleright \fc) = \fa \diamond x_b(\fb_- \shuffle \fc) = (x_a \diamond x_b) \triangleright (\fa_- \shuffle  (\fb_- \shuffle \fc)).
\end{align*}
From the induction hypothesis, one deduces that $(\fa \diamond \fb) \triangleright \fc = \fa \diamond (\fb \triangleright \fc)$. Similarly, one deduces that $(\fb \triangleright \fa) \diamond \fc = (\fb \diamond \fc) \triangleright \fa$. 

From the above arguments, we conclude that $(\fa \shuffle  \fb) \shuffle  \fc = \fa \shuffle  (\fb \shuffle  \fc)$. This completes the proof.
\end{proof}

As a direct consequence of Proposition \ref{prop: commutative} and Proposition \ref{prop: associative}, we obtain the following result.

\begin{theorem} \label{theorem: algebras}
The spaces $(\mathfrak{C}, \diamond)$ and $(\mathfrak{C}, \shuffle )$ are commutative $\Fq$-algebras. In particular, Conjecture 3.2.2 of \cite{Shi18} holds.
\end{theorem}

The following proposition summarizes several properties of different products $\triangleright$, $\diamond$ and $\shuffle $ that will be useful in the sequel.

\begin{proposition} \label{prop: key properties}
Let $\fa, \fb, \fc \in \langle \Sigma \rangle$ be arbitrary words. Then we have
\begin{enumerate}[$(1)$]
\item $\fa \diamond \fb=\fb \diamond \fa$.
\item $(\fa \diamond \fb) \diamond \fc=\fa \diamond (\fb \diamond \fc)$.
\item $\fa \shuffle  \fb = \fb\shuffle \fa=\fa \triangleright \fb + \fb \triangleright \fa + \fa \diamond \fb$.
\item $(\fa \shuffle  \fb) \shuffle  \fc=\fa \shuffle  (\fb \shuffle  \fc)$.
\end{enumerate}
If we assume further that $\fa, \fb, \fc$ are nonempty words, then 
\begin{enumerate}[$(1)$]
\item $(\fa \triangleright \fb) \triangleright \fc=(\fa \triangleright \fc) \triangleright \fb=\fa \triangleright (\fb\shuffle \fc)$.
\item $(\fa \diamond \fb) \triangleright \fc=\fa \diamond (\fb \triangleright \fc)=(\fa \triangleright \fc) \diamond \fb$.
\end{enumerate}
\end{proposition}


\section{Shuffle Hopf algebra in positive characteristic} \label{sec: Hopf algebra structure}

We first equip the shuffle algebra with a coproduct $\Delta$ and a counit map. The main theorem of this section states that these give a Hopf algebra structure of the shuffle algebra (see Theorem \ref{thm: Hopf algebra for shuffle product}). 

Throughout this section we continue with the notation of the previous section.

\subsection{Coproduct} ${}$\par \label{sec: coproduct}

We first introduce the coproduct 
	\[ \Delta: \frak C \to \frak C \otimes \frak C. \]
We will define it on  $\langle \Sigma \rangle$ by induction on weight and extend by $\Fq$-linearity to $\frak C$. We put
\begin{align*}
\Delta(1)&:=1 \otimes 1, \\
\Delta(x_1)&:=1 \otimes x_1 + x_1 \otimes 1.
\end{align*}
Let $w \in \bN$ and we suppose that we have defined $\Delta(\fv)$ for all words $\fv$ of weight $w(\fv)<w$. We now give a formula for $\Delta(\fu)$ for all words $\fu$ with $w(\fu)=w$. For such a word $\fu$ with $\depth(\fu)>1$, we put $\fu=x_u \fv$ with $w(\fv)<w$. Since $x_u$ and $\fv$ are both of weight less than $w$, we have already defined
\begin{align*}
\Delta(x_u)&:=1 \otimes x_u + \sum a_u \otimes b_u, \\
\Delta(\fv)&:= \sum a_\fv \otimes b_\fv.
\end{align*}
Then we set
\begin{align*}
\Delta(\fu):=1 \otimes \fu + \sum (a_u \triangleright a_\fv) \otimes (b_u \shuffle  b_\fv).
\end{align*}
Our last task is to define $\Delta(x_w)$. We know that
	\[ x_1\shuffle x_{w-1}=x_w+x_1x_{w-1}+x_{w-1}x_1+\sum_{0<j<w} \Delta^j_{1,w-1} x_{w-j} x_j \]
where all the words $x_{w-j} x_j$ have weight $w$ and depth $2$ and all $\Delta^j_{1,w-1}$ belong to~$\Fq$. Therefore, we set
\begin{equation} \label{eq: coproduct}
\Delta(x_w):=\Delta(x_1) \shuffle  \Delta(x_{w-1})-\Delta(x_1x_{w-1})-\Delta(x_{w-1}x_1)-\sum_{0<j<w} \Delta^j_{1,w-1} \Delta(x_{w-j} x_j).
\end{equation}
We note that this definition of coproduct is different from that given as in Shi's thesis (see \cite[\S 3.2.3]{Shi18}).

\begin{lemma} \label{lem: factor 1}
For all words $\fu$, we have
	\[ \Delta(\fu)=1 \otimes \fu+\sum a_\fu \otimes b_\fu \]
where $a_\fu \neq 1$.
\end{lemma}

\begin{proof}
The proof is by induction on the weight $w=w(\fu)$. For $w=0$ and $w=1$ Lemma \ref{lem: factor 1}  immediately holds as 
\begin{align*}
\Delta(1)&:=1 \otimes 1, \\
\Delta(x_1)&:=1 \otimes x_1 + x_1 \otimes 1.
\end{align*}

Let $w \in \bN$ with $w \geq 2$. We suppose that for all words $\fu$ with $w(\fu)<w$, we have
	\[ \Delta(\fu)=1 \otimes \fu+\sum a_\fu \otimes b_\fu \]
where $a_\fu \neq 1$. 

We have to prove that the statement holds for all words $\fu$ with $w(\fu)=w$. In fact, we first consider a word $\fu$ with $w(\fu)=w$ and $\depth(\fu)>1$. We put $\fu=x_u \fv$ with $\depth(\fv) \geq 1$. By the induction hypothesis, we write
	\[ \Delta(x_u)=1 \otimes x_u+\sum a_u \otimes b_u \]
where $a_u \neq 1$. If we put $\Delta(x_\fv)=\sum a_\fv \otimes b_\fv$, then we know that
\begin{align*}
\Delta(\fu):=1 \otimes \fu + \sum (a_u \triangleright a_\fv) \otimes (b_u \shuffle  b_\fv).
\end{align*}
Since $a_u \neq 1$,  $a_u \triangleright a_\fv \neq 1$. Thus Lemma \ref{lem: factor 1} holds for $\fu$. 

To conclude, it suffices to prove that Lemma \ref{lem: factor 1} holds for $x_w$. By the induction hypothesis, we deduce that
\begin{align*}
\Delta(x_1)\shuffle  \Delta(x_{w-1})=1 \otimes (x_1\shuffle x_{w-1}) + \sum a \otimes b, \quad a \neq 1,
\end{align*}
and for all $0<j<w$, 
\begin{align*}
\Delta(x_{w-j} x_j)=1 \otimes x_{w-j} x_j + \sum a_j \otimes b_j, \quad a_j \neq 1.
\end{align*}

Thus by \eqref{eq: coproduct},
\begin{align*}
& \Delta(x_w) \\
&=\Delta(x_1) \shuffle  \Delta(x_{w-1})-\Delta(x_1x_{w-1})-\Delta(x_{w-1}x_1)-\sum_{0<j<w} \Delta^j_{j,w-j} \Delta(x_{w-j} x_j) \\
&= \left(1 \otimes (x_1\shuffle x_{w-1}) + \sum a \otimes b\right)-\left(1 \otimes x_1x_{w-1} + \sum a_{w-1} \otimes b_{w-1}\right) \\
&-\left(1 \otimes x_{w-1}x_1 + \sum a_1 \otimes b_1\right)-\sum_{0<j<w} \Delta^j_{j,w-j} \left(1 \otimes x_{w-j} x_j + \sum a_j \otimes b_j\right) \\
&= 1 \otimes \left(x_1\shuffle x_{w-1}-x_1x_{w-1}-x_{w-1}x_1-\sum_{0<j<w} \Delta^j_{j,w-j} x_{w-j} x_j \right) \\
&+ \sum a \otimes b- \sum a_{w-1} \otimes b_{w-1}- \sum a_1 \otimes b_1-\sum_{0<j<w} \Delta^j_{j,w-j} \sum a_j \otimes b_j \\
&= 1 \otimes x_w+ \sum a \otimes b-\sum a_{w-1} \otimes b_{w-1}-\sum a_1 \otimes b_1-\sum_{0<j<w} \Delta^j_{j,w-j} \sum a_j \otimes b_j.
\end{align*}
The proof is finished.
\end{proof}

\subsection{Compatibility of the coproduct} ${}$\par

In this section we prove the compatibility of the coproduct $\Delta$ given as in the previous section.

\begin{theorem} \label{thm: compatibility}
Let $\fa, \fb \in \langle \Sigma \rangle$. Then we have
	\[ \Delta(\fa\shuffle \fb)=\Delta(\fa)\shuffle \Delta(\fb). \]
\end{theorem}

The rest of this section is devoted to a proof of Theorem \ref{thm: compatibility}. The proof is by induction on the total weight $w=w(\fa)+w(\fb)$.

For $w=0$ and $w=1$ we see that Theorem \ref{thm: compatibility} holds. Let $w \in \bN$ with $w \geq 2$ and we suppose that for all $\fa, \fb \in \langle \Sigma \rangle$ such that $w(\fa)+w(\fb)<w$, we have $\Delta(\fa\shuffle \fb)=\Delta(\fa)\shuffle \Delta(\fb)$.

We now show that for all $\fa, \fb \in \langle \Sigma \rangle$ such that $w(\fa)+w(\fb)=w$, we have
 	\[ \Delta(\fa\shuffle \fb)=\Delta(\fa)\shuffle \Delta(\fb). \]
The proof will be divided into three steps.

\subsubsection{Step 1} 
We first prove the following proposition. 
\begin{proposition} \label{prop: compatibility step 1}
For all words $x_u \fu, x_v \in \langle \Sigma \rangle$ with $u,v \in \bN$, $\depth(\fu) \geq 1$ and $w(x_u \fu)+w(x_v)=w$, we have
 	\[ \Delta(x_u \fu\shuffle x_v)=\Delta(x_u \fu)\shuffle \Delta(x_v). \]
\end{proposition}

\begin{proof}
By definition of the product $\shuffle $ and Lemma \ref{lem: triangle formulas}, we write as 
\begin{align} \label{eq: expansion 1}
x_u \fu \shuffle x_v&=x_v \triangleright (x_u \fu)+(x_u \fu) \triangleright x_v+(x_u \fu) \diamond x_v \\
&=x_vx_u\fu+x_u(\fu \shuffle x_v) +(x_u \diamond x_v) \triangleright \fu \notag \\
&=x_v x_u \fu + x_u(\fu\shuffle x_v)+x_{u+v} \fu+ \sum_{0<j<u+v} \Delta^j_{u,v} x_{u+v-j} (x_j\shuffle \fu). \notag
\end{align}
where the coefficients $\Delta^j_{u,v}$ belong to $\Fq$. Therefore, we get
\begin{align} \label{eq:compatibility 1}
& \Delta(x_u \fu \shuffle x_v)-\Delta(x_u \fu)\shuffle \Delta(x_v) \\
&= \Delta(x_v x_u \fu) + \Delta(x_u(\fu\shuffle x_v)) \notag \\
&+\Delta(x_{u+v} \fu)+ \sum_{0<j<u+v} \Delta^j_{u,v} \Delta(x_{u+v-j} (x_j\shuffle \fu)) \notag \\ 
&-\Delta(x_u \fu)\shuffle \Delta(x_v). \notag
\end{align}

We now analyze each term of the RHS of the above expression. To do so we put
\begin{align*}
\Delta(\fu) &=1 \otimes \fu+\sum a_\fu \otimes b_\fu,
\end{align*}
and for all $j \in \bN$, we simply put
\begin{align*}
\Delta(x_j) &=1 \otimes x_j+\sum a_j \otimes b_j.
\end{align*}
In particular, 
\begin{align*}
\Delta(x_u) &=1 \otimes x_u+\sum a_u \otimes b_u, \\
\Delta(x_v) &=1 \otimes x_v+\sum a_v \otimes b_v.
\end{align*}

\noindent {\bf The first term $\Delta(x_v x_u \fu)$.} \par From the definition of the coproduct $\Delta$ we deduce
\begin{equation} \label{eq: x_u u}
\Delta(x_u \fu)=1 \otimes x_u \fu+\sum a_u \otimes (b_u\shuffle \fu)+\sum (a_u \triangleright a_\fu) \otimes (b_u\shuffle b_\fu).
\end{equation}
Thus
\begin{align} \label{eq:term 1}
\Delta(x_v x_u \fu)&=1 \otimes x_v x_u \fu+ \sum a_v \otimes ((x_u \fu)\shuffle b_v)+\sum (a_v \triangleright a_u) \otimes (b_u\shuffle \fu\shuffle b_v)  \\
&+\sum (a_v \triangleright (a_u \triangleright a_\fu)) \otimes (b_u\shuffle b_\fu\shuffle b_v). \notag
\end{align}

\noindent {\bf The second term $\Delta(x_u(\fu\shuffle x_v))$.} \par Since $w(\fu)+w(x_v)<w$, the induction hypothesis implies
\begin{align*}
\Delta(\fu\shuffle x_v)&=\Delta(\fu)\shuffle \Delta(x_v) \\
&=\left(1 \otimes \fu+\sum a_\fu \otimes b_\fu \right)\shuffle \left(1 \otimes x_v+\sum a_v \otimes b_v \right) \\
&=1 \otimes (\fu\shuffle x_v)+\sum a_\fu \otimes (b_\fu\shuffle x_v) \\
&+\sum a_v \otimes (\fu\shuffle b_v)+\sum (a_\fu\shuffle a_v) \otimes (b_\fu\shuffle b_v).
\end{align*}
It follows that
\begin{align} \label{eq:term 2}
\Delta(x_u(\fu\shuffle x_v))&=1 \otimes (x_u(\fu\shuffle x_v))+\sum a_u \otimes (b_u\shuffle \fu\shuffle x_v) \\
&+\sum (a_u \triangleright a_\fu) \otimes (b_u\shuffle b_\fu\shuffle x_v)+\sum (a_u \triangleright a_v) \otimes (b_u\shuffle \fu\shuffle b_v) \notag \\
&+\sum (a_u \triangleright (a_\fu\shuffle a_v)) \otimes (b_u\shuffle b_\fu\shuffle b_v). \notag
\end{align}

\noindent {\bf The third term $\Delta(x_{u+v} \fu)$.} \par By definition,
\begin{align} \label{eq:term 3}
\Delta(x_{u+v} \fu) &=1 \otimes (x_{u+v} \fu)+ \sum a_{u+v} \otimes (b_{u+v}\shuffle \fu)+\sum (a_{u+v} \triangleright a_\fu) \otimes (b_{u+v}\shuffle b_\fu).
\end{align}

\noindent {\bf The fourth terms $\Delta(x_{u+v-j} (x_j\shuffle \fu))$ for all $0<j<u+v$.} \par 

As $w(x_j)+w(\fu)<w$, by the induction hypothesis,
\begin{align*}
\Delta(x_j\shuffle \fu)&=\Delta(x_j)\shuffle \Delta(\fu) \\
&= \left(1 \otimes x_j+\sum a_j \otimes b_j \right)\shuffle \left(1 \otimes \fu+\sum a_\fu \otimes b_\fu \right) \\
&= 1 \otimes (x_j\shuffle \fu)+\sum a_j \otimes (b_j\shuffle \fu) \\
&+\sum a_\fu \otimes (x_j\shuffle b_\fu)+\sum (a_j\shuffle a_\fu) \otimes (b_j\shuffle b_\fu).
\end{align*}
We then get;
\begin{align} \label{eq:term 4}
\Delta(x_{u+v-j} (x_j\shuffle \fu))&=1 \otimes (x_{u+v-j}(x_j\shuffle \fu))+\sum a_{u+v-j} \otimes (b_{u+v-j}\shuffle x_j\shuffle \fu) \\
&+\sum (a_{u+v-j} \triangleright a_j) \otimes (b_{u+v-j}\shuffle b_j\shuffle \fu) \notag \\
&+\sum (a_{u+v-j} \triangleright a_\fu) \otimes (b_{u+v-j}\shuffle x_j\shuffle b_\fu) \notag \\
&+\sum (a_{u+v-j} \triangleright (a_j\shuffle a_\fu)) \otimes (b_{u+v-j}\shuffle b_j\shuffle b_\fu) \notag.
\end{align}

\noindent {\bf The last term $\Delta(x_u \fu)\shuffle \Delta(x_v)$.} \par 

Recall that $\Delta(x_u \fu)$ is given by \eqref{eq: x_u u}. Thus,
\begin{align} \label{eq:term 5}
& \Delta(x_u \fu)\shuffle \Delta(x_v) \\
&= \left(1 \otimes x_u \fu+\sum a_u \otimes (b_u\shuffle \fu)+\sum (a_u \triangleright a_\fu) \otimes (b_u\shuffle b_\fu) \right)\shuffle \left(1 \otimes x_v+\sum a_v \otimes b_v\right) \notag \\
&= 1 \otimes ((x_u \fu)\shuffle x_v) + \sum a_v \otimes (b_v\shuffle (x_u \fu)) \notag \\
&+ \sum a_u \otimes (b_u\shuffle \fu\shuffle x_v)+\sum (a_u\shuffle a_v) \otimes (b_u\shuffle b_v\shuffle \fu) \notag \\
&+ \sum (a_u \triangleright a_\fu) \otimes (b_u\shuffle b_\fu\shuffle x_v)+ \sum ((a_u \triangleright a_\fu)\shuffle a_v) \otimes (b_u\shuffle b_\fu\shuffle b_v). \notag
\end{align}

Plugging the equations \eqref{eq:term 1}, \eqref{eq:term 2}, \eqref{eq:term 3}, \eqref{eq:term 4}, \eqref{eq:term 5} into \eqref{eq:compatibility 1} yields
\begin{align*}
\Delta(x_u \fu \shuffle x_v)-\Delta(x_u \fu)\shuffle \Delta(x_v) =S_0-S_1-S_2+S_3+S_4.
\end{align*}
Here the sums $S_i$ with $0 \leq i \leq 4$ are given as follows:
\begin{align*}
S_0&= 1 \otimes x_v x_u \fu+1 \otimes (x_u(\fu\shuffle x_v))+1 \otimes (x_{u+v} \fu) \\
&+\sum_{0<j<u+v} \Delta^j_{u,v} 1 \otimes (x_{u+v-j}(x_j\shuffle \fu)) - 1 \otimes ((x_u \fu)\shuffle x_v). \\
S_1&=\sum (a_u\shuffle a_v) \otimes (b_u\shuffle b_v\shuffle \fu)-\sum (a_u \triangleright a_v) \otimes (b_u\shuffle b_v\shuffle \fu) \\
&-\sum (a_v \triangleright a_u) \otimes (b_u\shuffle b_v\shuffle \fu). \\
S_2&=\sum ((a_u \triangleright a_\fu)\shuffle a_v) \otimes (b_u\shuffle b_\fu\shuffle b_v)-\sum (a_u \triangleright (a_\fu\shuffle a_v)) \otimes (b_u\shuffle b_\fu\shuffle b_v) \\
&-\sum (a_v \triangleright (a_u \triangleright a_\fu)) \otimes (b_u\shuffle b_\fu\shuffle b_v). \\
S_3&= \sum a_{u+v} \otimes (b_{u+v}\shuffle \fu)+\sum_{0<j<u+v} \Delta^j_{u,v} \sum a_{u+v-j} \otimes (b_{u+v-j}\shuffle x_j\shuffle \fu) \\
&+ \sum_{0<j<u+v} \Delta^j_{u,v} \sum (a_{u+v-j} \triangleright a_j) \otimes (b_{u+v-j}\shuffle b_j\shuffle \fu). \\
S_4&= \sum (a_{u+v} \triangleright a_\fu) \otimes (b_{u+v}\shuffle b_\fu) \\
&+\sum_{0<j<u+v} \Delta^j_{u,v} \sum (a_{u+v-j} \triangleright a_\fu) \otimes (b_{u+v-j}\shuffle x_j\shuffle b_\fu) \\
&+ \sum_{0<j<u+v} \Delta^j_{u,v} \sum (a_{u+v-j} \triangleright (a_j\shuffle a_\fu)) \otimes (b_{u+v-j}\shuffle b_j\shuffle b_\fu).
\end{align*}

We claim that
\begin{enumerate}
\item $S_0=0$.
\item $S_1-S_3=0$.
\item $S_2-S_4=0$.
\end{enumerate}

We now prove the previous claim. For Part (1), we want to show that $S_0=0$. In fact, it follows immediately from \eqref{eq: expansion 1}, e.g.,
\begin{align*}
x_v x_u \fu+x_u(\fu\shuffle x_v)+x_{u+v} \fu +\sum_{0<j<u+v} \Delta^j_{u,v} x_{u+v-j}(x_j\shuffle \fu) - (x_u \fu)\shuffle x_v=0. 
\end{align*}

For Part (2), we will show that 
\begin{align*}
S_1=S_3=\sum (a_u \diamond a_v)  \otimes (b_u\shuffle b_v\shuffle \fu).
\end{align*}
In fact, we note that
\begin{align*}
S_1&=\sum (a_u\shuffle a_v) \otimes (b_u\shuffle b_v\shuffle \fu)-\sum (a_u \triangleright a_v) \otimes (b_u\shuffle b_v\shuffle \fu) \\
&-\sum (a_v \triangleright a_u) \otimes (b_u\shuffle b_v\shuffle \fu) \\
&= \sum (a_u\shuffle a_v- a_u \triangleright a_v-a_v \triangleright a_u) \otimes (b_u\shuffle b_v\shuffle \fu) \\
&= \sum (a_u \diamond a_v)  \otimes (b_u\shuffle b_v\shuffle \fu).
\end{align*}
Here the last equality follows from Lemma \ref{lem: triangle formulas}, Part (2). 

Next, since $w(x_u)+w(x_v)<w$, by the induction hypothesis we know that 
\begin{equation} \label{eq: x_u x_v}
\Delta(x_u\shuffle x_v)=\Delta(x_u)\shuffle \Delta(x_v)
\end{equation}
The LHS equals 
\begin{align*}
& \Delta(x_u\shuffle x_v) \\
&=\Delta(x_ux_v)+\Delta(x_vx_u)+\Delta(x_{u+v}) +\sum_{0<j<u+v} \Delta^j_{u,v} \Delta(x_{u+v-j} x_j) \\
&=1 \otimes x_ux_v + \sum a_u \otimes (b_u\shuffle x_v)+\sum (a_u \triangleright a_v) \otimes (b_u\shuffle b_v) \\
&+1 \otimes x_vx_u + \sum a_v \otimes (x_u\shuffle b_v)+\sum (a_v \triangleright a_u) \otimes (b_u\shuffle b_v) \\
&+ 1 \otimes x_{u+v} + \sum a_{u+v} \otimes b_{u+v} \\
&+ \sum_{0<j<u+v} \Delta^j_{u,v} 1 \otimes x_{u+v-j} x_j \\
&+ \sum_{0<j<u+v} \Delta^j_{u,v} \sum a_{u+v-j} \otimes (b_{u+v-j}\shuffle x_j) \\
&+ \sum_{0<j<u+v} \Delta^j_{u,v} \sum (a_{u+v-j} \triangleright a_j) \otimes (b_{u+v-j}\shuffle b_j).
\end{align*}
We see that
\begin{align*}
\Delta(x_u)\shuffle \Delta(x_v) &= \left(1 \otimes x_u+\sum a_u \otimes b_u \right)\shuffle \left(1 \otimes x_v+\sum a_v \otimes b_v \right) \\
&= 1 \otimes (x_u\shuffle x_v)+\sum a_u \otimes (b_u\shuffle x_v) \\
&+\sum a_v \otimes (x_u\shuffle b_v)+\sum (a_u\shuffle a_v) \otimes (b_u\shuffle b_v).
\end{align*}
Replacing these equalities into \eqref{eq: x_u x_v} gets
\begin{align*}
& \sum (a_u \triangleright a_v) \otimes (b_u\shuffle b_v) +\sum (a_v \triangleright a_u) \otimes (b_u\shuffle b_v) \\
&+ \sum a_{u+v} \otimes b_{u+v} + \sum_{0<j<u+v} \Delta^j_{u,v} \sum a_{u+v-j} \otimes (b_{u+v-j}\shuffle x_j) \\
&+ \sum_{0<j<u+v} \Delta^j_{u,v} \sum (a_{u+v-j} \triangleright a_j) \otimes (b_{u+v-j}\shuffle b_j) \\
&= \sum (a_u\shuffle a_v) \otimes (b_u\shuffle b_v). 
\end{align*}
Thus we deduce
\begin{align} \label{eq: x_u x_v 2}
& \sum a_{u+v} \otimes b_{u+v} + \sum_{0<j<u+v} \Delta^j_{u,v} \sum a_{u+v-j} \otimes (b_{u+v-j}\shuffle x_j) \\
&+ \sum_{0<j<u+v} \Delta^j_{u,v} \sum (a_{u+v-j} \triangleright a_j) \otimes (b_{u+v-j}\shuffle b_j) \notag \\
&= \sum (a_u\shuffle a_v-a_u \triangleright a_v-a_v \triangleright a_u) \otimes (b_u\shuffle b_v) \notag \\
&= \sum (a_u \diamond a_v) \otimes (b_u\shuffle b_v). \notag
\end{align}

As a direct consequence of \eqref{eq: x_u x_v 2}, we obtain
\begin{align*}
S_3&= \sum a_{u+v} \otimes (b_{u+v}\shuffle \fu)+\sum_{0<j<u+v} \Delta^j_{u,v} \sum a_{u+v-j} \otimes (b_{u+v-j}\shuffle x_j\shuffle \fu) \\
&+ \sum_{0<j<u+v} \Delta^j_{u,v} \sum (a_{u+v-j} \triangleright a_j) \otimes (b_{u+v-j}\shuffle b_j\shuffle \fu). \\
&= \sum (a_u \diamond a_v) \otimes (b_u\shuffle b_v\shuffle u)
\end{align*}
as desired. We conclude that $S_1=S_3$ as claimed.

For Part (3), we will show that
\begin{align*}
S_2=S_4=\sum ((a_u \triangleright a_\fu) \diamond a_v) \otimes (b_u\shuffle b_\fu\shuffle b_v).
\end{align*}
In fact, to prove the equality for $S_2$, we note that
\begin{align*}
(a_u \triangleright a_\fu)\shuffle a_v&=(a_u \triangleright a_\fu) \triangleright a_v+a_v \triangleright (a_u \triangleright a_\fu)+(a_u \triangleright a_\fu) \diamond a_v \\
&=a_u \triangleright (a_\fu \shuffle  a_v)+a_v \triangleright (a_u \triangleright a_\fu)+(a_u \triangleright a_\fu) \diamond a_v.
\end{align*}
The first equality follows from Lemma \ref{lem: triangle formulas} and the second one follows from Proposition \ref{prop: key properties}, Part 5. We then obtain
\begin{align*}
S_2&=\sum ((a_u \triangleright a_\fu)\shuffle a_v) \otimes (b_u\shuffle b_\fu\shuffle b_v)-\sum (a_u \triangleright (a_\fu\shuffle a_v)) \otimes (b_u\shuffle b_\fu\shuffle b_v) \\
&-\sum (a_v \triangleright (a_u \triangleright a_\fu)) \otimes (b_u\shuffle b_\fu\shuffle b_v) \\
&=\sum ((a_u \triangleright a_\fu)\shuffle a_v-a_u \triangleright (a_\fu\shuffle a_v)-a_v \triangleright (a_u \triangleright a_\fu)) \otimes (b_u\shuffle b_\fu\shuffle b_v) \\
&=\sum ((a_u \triangleright a_\fu) \diamond a_v) \otimes (b_u\shuffle b_\fu\shuffle b_v).
\end{align*}

We now consider the term $S_4$. We have
\begin{align*}
S_4&= \sum (a_{u+v} \triangleright a_\fu) \otimes (b_{u+v}\shuffle b_\fu) \\
&+\sum_{0<j<u+v} \Delta^j_{u,v} \sum (a_{u+v-j} \triangleright a_\fu) \otimes (b_{u+v-j}\shuffle x_j\shuffle b_\fu) \\
&+ \sum_{0<j<u+v} \Delta^j_{u,v} \sum (a_{u+v-j} \triangleright (a_j\shuffle a_\fu)) \otimes (b_{u+v-j}\shuffle b_j\shuffle b_\fu) \\
&= \sum (a_{u+v} \triangleright a_\fu) \otimes (b_{u+v}\shuffle b_\fu) \\
&+\sum_{0<j<u+v} \Delta^j_{u,v} \sum (a_{u+v-j} \triangleright a_\fu) \otimes (b_{u+v-j}\shuffle x_j\shuffle b_\fu) \\
&+ \sum_{0<j<u+v} \Delta^j_{u,v} \sum ((a_{u+v-j} \triangleright a_j) \triangleright a_\fu) \otimes (b_{u+v-j}\shuffle b_j\shuffle b_\fu) \\
&= \sum ((a_u \diamond a_v) \triangleright a_\fu) \otimes (b_u\shuffle b_v\shuffle b_\fu) \\
&=\sum ((a_u \triangleright a_\fu) \diamond a_v) \otimes (b_u\shuffle b_\fu\shuffle b_v).
\end{align*}
Here the second and fourth equalities follow from Proposition \ref{prop: key properties}, Part 5 and 6. The third one is a direct consequence of \eqref{eq: x_u x_v 2}.

Thus we have proved that $S_2=S_4$ as claimed.

To conclude, we see immediately that 
\begin{align*}
\Delta(x_u \fu \shuffle x_v)-\Delta(x_u \fu)\shuffle \Delta(x_v) =S_0-S_1-S_2+S_3+S_4=0.
\end{align*}
The proof of Proposition \ref{prop: compatibility step 1} is finished.
\end{proof}

\subsubsection{Step 2} 
Next we generalize Proposition \ref{prop: compatibility step 1} for words of arbitrary depth.
\begin{proposition} \label{prop: compatibility step 2}
We work with the above assumption. Then for all words $x_u \fu, x_v \in \langle \Sigma \rangle$ with $u,v \in \bN$, $\depth(\fu) \geq 1$, $\depth(\fv) \geq 1$ and $w(x_u \fu)+w(x_v \fv)=w$, we have
 	\[ \Delta(x_u \fu\shuffle x_v \fv)=\Delta(x_u \fu)\shuffle \Delta(x_v \fv). \]
\end{proposition}

\begin{proof}
We follow the same strategy as that of the proof of Proposition \ref{prop: compatibility step 1} but the proof is much more involved and complicated. 

By the definition of the product $\shuffle $ and Lemma \ref{lem: triangle formulas}, we obtain
\begin{align} \label{eq: expansion 2}
& x_u \fu \shuffle x_v \fv \\
&=(x_u \fu) \triangleright (x_v \fv)+(x_v \fv) \triangleright (x_u \fu)+(x_u \fu) \diamond (x_v \fv) \notag \\
&=x_v(x_u \fu\shuffle \fv) + x_u(\fu\shuffle (x_v \fv))+(x_u \diamond x_v) \triangleright (\fu\shuffle \fv) \notag \\
&=x_v(x_u \fu\shuffle \fv) + x_u(\fu\shuffle (x_v \fv))+x_{u+v}(\fu\shuffle \fv)+ \sum_{0<j<u+v} \Delta^j_{u,v} x_{u+v-j} (x_j\shuffle \fu\shuffle \fv), \notag
\end{align}
where the coefficients $\Delta^j_{u,v}$ belong to $\Fq$. Therefore, we get
\begin{align} \label{eq:compatibility b}
& \Delta(x_u \fu \shuffle x_v \fv)-\Delta(x_u \fu)\shuffle \Delta(x_v \fv) \\
&= \Delta(x_v(x_u \fu\shuffle \fv)) + \Delta(x_u(\fu\shuffle (x_v \fv))+\Delta(x_{u+v}(\fu\shuffle \fv)) \notag \\
&+ \sum_{0<j<u+v} \Delta^j_{u,v} \Delta(x_{u+v-j} (x_j\shuffle \fu\shuffle \fv)) -\Delta(x_u \fu)\shuffle \Delta(x_v \fv). \notag
\end{align}

We now analyze each term of the RHS of the above expression. To do so we put
\begin{align*}
\Delta(\fu) &=1 \otimes \fu+\sum a_\fu \otimes b_\fu, \\
\Delta(\fv) &=1 \otimes \fv+\sum a_\fv \otimes b_\fv, 
\end{align*}
and for all $j \in \bN$, we simply put
\begin{align*}
\Delta(x_j) &=1 \otimes x_j+\sum a_j \otimes b_j.
\end{align*}
In particular, 
\begin{align*}
\Delta(x_u) &=1 \otimes x_u+\sum a_u \otimes b_u, \\
\Delta(x_v) &=1 \otimes x_v+\sum a_v \otimes b_v.
\end{align*}

\noindent {\bf The first term $\Delta(x_v(x_u \fu\shuffle \fv))$.} \par Recall that $\Delta(x_u \fu)$ is given as in \eqref{eq: x_u u}. As $w(x_u \fu)+w(\fv)<w$, we obtain
\begin{align*}
\Delta(x_u \fu\shuffle \fv)&= \Delta(x_u \fu)\shuffle \Delta(\fv) \\
&= 1 \otimes ((x_u \fu)\shuffle \fv)+\sum a_u \otimes (b_u\shuffle \fu\shuffle \fv)+\sum a_\fv \otimes ((x_u \fu)\shuffle b_\fv) \\
&+\sum (a_u \triangleright a_\fu) \otimes (b_u\shuffle b_\fu\shuffle \fv)+\sum (a_u \shuffle  a_\fv) \otimes (b_u\shuffle \fu\shuffle b_\fv) \\
&+\sum ((a_u \triangleright a_\fu)\shuffle a_\fv) \otimes (b_u\shuffle b_\fu\shuffle b_\fv).
\end{align*}
Thus
\begin{align} \label{eq:term 1b}
\Delta(x_v(x_u \fu\shuffle \fv))&=1 \otimes x_v(x_u \fu\shuffle x_v)+ \sum a_v \otimes (b_v\shuffle (x_u \fu)\shuffle \fv) \\
&+\sum (a_v \triangleright a_u) \otimes (b_u\shuffle b_v\shuffle \fu\shuffle \fv) \notag \\
&+\sum (a_v \triangleright (a_u \triangleright a_\fu)) \otimes (b_u\shuffle b_v\shuffle b_\fu\shuffle \fv) \notag \\
&+\sum (a_v \triangleright a_\fv) \otimes (b_v\shuffle (x_u \fu)\shuffle b_\fv) \notag \\
&+\sum (a_v \triangleright (a_u\shuffle a_\fv)) \otimes (b_u\shuffle b_v\shuffle \fu\shuffle b_\fv) \notag \\
&+\sum (a_v \triangleright ((a_u \triangleright a_\fu)\shuffle a_\fv)) \otimes (b_u\shuffle b_v\shuffle b_\fu\shuffle b_\fv). \notag
\end{align}

\noindent {\bf The second term $\Delta(x_u(\fu\shuffle (x_v \fv))$.} \par Similarly, we get
\begin{align} \label{eq:term 2b}
\Delta(x_u(\fu\shuffle (x_v \fv))&=1 \otimes x_u(\fu\shuffle x_v \fv)+ \sum a_u \otimes (b_u\shuffle \fu\shuffle (x_v \fv)) \\
&+\sum ((a_u \triangleright a_v) \otimes (b_u\shuffle b_v\shuffle \fu\shuffle \fv) \notag \\
&+\sum (a_u \triangleright (a_v \triangleright a_\fv)) \otimes (b_u\shuffle b_v\shuffle \fu\shuffle b_\fv) \notag \\
&+\sum (a_u \triangleright a_\fu) \otimes (b_u\shuffle b_\fu\shuffle (x_v \fv)) \notag \\
&+\sum (a_u \triangleright (a_\fu\shuffle a_v)) \otimes (b_u\shuffle b_\fu\shuffle b_v\shuffle \fv) \notag \\
&+\sum (a_u \triangleright (a_\fu\shuffle (a_v \triangleright a_\fv)) \otimes (b_u\shuffle b_v\shuffle b_\fu\shuffle b_\fv). \notag
\end{align}

\noindent {\bf The third term $\Delta(x_{u+v}(\fu\shuffle \fv))$.} \par 

We put
	\[ \Delta(\fu\shuffle \fv)=1 \otimes (\fu\shuffle \fv)+\sum a_{\fu\shuffle \fv} \otimes b_{\fu\shuffle \fv}. \]
As $w(\fu)+w(\fv)<w$, the induction hypothesis implies that
\begin{align*}
\Delta(\fu\shuffle \fv)&=\Delta(\fu)\shuffle \Delta(\fv) \\
&= \left(1 \otimes \fu+\sum a_\fu \otimes b_\fu \right)\shuffle \left(1 \otimes \fv+\sum a_\fv \otimes b_\fv \right).
\end{align*}
Thus
\begin{align*}
& 1 \otimes (\fu\shuffle \fv)+\sum a_{\fu\shuffle \fv} \otimes b_{\fu\shuffle \fv} \\
&= \left(1 \otimes \fu+\sum a_\fu \otimes b_\fu \right)\shuffle \left(1 \otimes \fv+\sum a_\fv \otimes b_\fv \right) \\
&= 1 \otimes (\fu\shuffle \fv)+\sum a_\fu \otimes (b_\fu\shuffle \fv) \\
&+\sum a_\fv \otimes (\fu\shuffle b_\fv)+\sum (a_\fu\shuffle a_\fv) \otimes (b_\fu\shuffle b_\fv),
\end{align*}
which implies
\begin{equation} \label{eq: fu fv}
\sum a_{\fu\shuffle \fv} \otimes b_{\fu\shuffle \fv}=\sum a_\fu \otimes (b_\fu\shuffle \fv)+\sum a_\fv \otimes (\fu\shuffle b_\fv)+\sum (a_\fu\shuffle a_\fv) \otimes (b_\fu\shuffle b_\fv).
\end{equation}

Finally, we have
\begin{align} \label{eq:term 3b}
\Delta(x_{u+v}(\fu\shuffle \fv)) &=1 \otimes (x_{u+v}(\fu\shuffle \fv))+ \sum a_{u+v} \otimes (b_{u+v}\shuffle \fu\shuffle \fv) \\
&+\sum (a_{u+v} \triangleright a_{\fu\shuffle \fv}) \otimes (b_{u+v}\shuffle b_{\fu\shuffle \fv}). \notag
\end{align}

\noindent {\bf The fourth terms $\Delta(x_{u+v-j} (x_j\shuffle \fu\shuffle \fv))$ for all $0<j<u+v$.}  \par 

As $w(x_j)+w(\fu\shuffle \fv)<w$, by the induction hypothesis, we get
\begin{align*}
& \Delta(x_j\shuffle \fu\shuffle \fv) \\
&=\Delta(x_j)\shuffle \Delta(\fu\shuffle \fv) \\
&= \left(1 \otimes x_j+\sum a_j \otimes b_j \right)\shuffle \left(1 \otimes (\fu\shuffle \fv)+\sum a_{\fu\shuffle \fv} \otimes b_{\fu\shuffle \fv} \right) \\
&= 1 \otimes (x_j\shuffle \fu\shuffle \fv)+\sum a_j \otimes (b_j\shuffle \fu\shuffle \fv) \\
&+\sum a_{\fu\shuffle \fv} \otimes (x_j\shuffle b_{\fu\shuffle \fv})+\sum (a_j\shuffle a_{\fu\shuffle \fv}) \otimes (b_j\shuffle b_{\fu\shuffle \fv}).
\end{align*}
We then get
\begin{align} \label{eq:term 4b}
& \Delta(x_{u+v-j} (x_j\shuffle \fu\shuffle \fv)) \\
&=1 \otimes (x_{u+v-j}(x_j\shuffle \fu\shuffle \fv))+\sum a_{u+v-j} \otimes (b_{u+v-j}\shuffle x_j\shuffle \fu\shuffle \fv) \notag \\
&+\sum (a_{u+v-j} \triangleright a_j) \otimes (b_{u+v-j}\shuffle b_j\shuffle \fu\shuffle \fv) \notag \\
&+\sum (a_{u+v-j} \triangleright a_{\fu\shuffle \fv}) \otimes (b_{u+v-j}\shuffle x_j\shuffle b_{\fu\shuffle \fv}) \notag \\
&+\sum (a_{u+v-j} \triangleright (a_j\shuffle a_{\fu\shuffle \fv})) \otimes (b_{u+v-j}\shuffle b_j\shuffle b_{\fu\shuffle \fv}) \notag.
\end{align}

\noindent {\bf The last term $\Delta(x_u \fu)\shuffle \Delta(x_v \fv)$.} \par 

Recall that $\Delta(x_u \fu)$ is given by \eqref{eq: x_u u}. Similarly, 
\begin{equation*}
\Delta(x_v \fv)=1 \otimes x_v \fv+\sum a_v \otimes (b_v\shuffle \fv)+\sum (a_v \triangleright a_\fv) \otimes (b_v\shuffle b_\fv).
\end{equation*}
Thus
\begin{align} \label{eq:term 5b}
& \Delta(x_u \fu)\shuffle \Delta(x_v \fv) \\
&= \left(1 \otimes x_u \fu+\sum a_u \otimes (b_u\shuffle \fu)+\sum (a_u \triangleright a_\fu) \otimes (b_u\shuffle b_\fu) \right) \notag \\
&\shuffle \left(1 \otimes x_v \fv+\sum a_v \otimes (b_v\shuffle \fv)+\sum (a_v \triangleright a_\fv) \otimes (b_v\shuffle b_\fv)\right) \notag \\
&= 1 \otimes (x_u \fu\shuffle x_v \fv) \notag \\
&+ \sum a_v \otimes ((x_u \fu)\shuffle b_v\shuffle \fv)+\sum (a_v \triangleright a_\fv) \otimes ((x_u \fu)\shuffle b_v\shuffle b_\fv) \notag \\
&+ \sum a_u \otimes (b_u\shuffle \fu\shuffle (x_v \fv))+\sum (a_u\shuffle a_v) \otimes (b_u\shuffle b_v\shuffle \fu\shuffle \fv) \notag \\
&+\sum (a_u\shuffle (a_v \triangleright a_\fv)) \otimes (b_u\shuffle \fu\shuffle b_v\shuffle b_\fv) \notag \\
&+ \sum (a_u \triangleright a_\fu) \otimes (b_u\shuffle b_\fu\shuffle (x_v \fv))+ \sum ((a_u \triangleright a_\fu)\shuffle a_v) \otimes (b_u\shuffle b_\fu\shuffle b_v\shuffle \fv) \notag \\
&+\sum ((a_u \triangleright a_\fu)\shuffle (a_v \triangleright a_\fv)) \otimes (b_u\shuffle b_\fu\shuffle b_v\shuffle b_\fv). \notag
\end{align}

Plugging the equations \eqref{eq:term 1b}, \eqref{eq:term 2b}, \eqref{eq:term 3b}, \eqref{eq:term 4b}, \eqref{eq:term 5b} into \eqref{eq:compatibility b} yields
\begin{align*}
\Delta(x_u \fu \shuffle x_v \fv)-\Delta(x_u \fu)\shuffle \Delta(x_v \fv) =S_0-S_1-S_2-S_\fu-S_\fv+S_3+S_4.
\end{align*}
Here the sums $S_i$ with $0 \leq i \leq 4$ and $S_\fu$, $S_\fv$ are given as follows: 
\begin{align*}
S_0&= 1\otimes x_v(x_u \fu\shuffle \fv) + 1 \otimes x_u(\fu\shuffle (x_v \fv))+1 \otimes x_{u+v}(\fu\shuffle \fv) \\
&+ \sum_{0<j<u+v} \Delta^j_{u,v} 1 \otimes x_{u+v-j} (x_j\shuffle \fu\shuffle \fv)-1 \otimes ((x_u \fu)\shuffle (x_v \fv)). \\
S_1&=\sum (a_u\shuffle a_v) \otimes (b_u\shuffle b_v\shuffle \fu\shuffle \fv)-\sum (a_u \triangleright a_v) \otimes (b_u\shuffle b_v\shuffle \fu\shuffle \fv) \\
&-\sum (a_v \triangleright a_u) \otimes (b_u\shuffle b_v\shuffle \fu\shuffle \fv). \\
S_2&=\sum ((a_u \triangleright a_\fu)\shuffle (a_v \triangleright a_\fv)) \otimes (b_u\shuffle b_\fu\shuffle b_v\shuffle b_\fv) \\
&-\sum (a_u \triangleright (a_\fu\shuffle (a_v \triangleright a_\fv))) \otimes (b_u\shuffle b_\fu\shuffle b_v\shuffle b_\fv) \\
&-\sum (a_v \triangleright ((a_u \triangleright a_\fu)\shuffle a_\fv) \otimes (b_u\shuffle b_\fu\shuffle b_v\shuffle b_\fv). \\
S_3&= \sum a_{u+v} \otimes (b_{u+v}\shuffle \fu\shuffle \fv) \\
&+\sum_{0<j<u+v} \Delta^j_{u,v} \sum a_{u+v-j} \otimes (b_{u+v-j}\shuffle x_j\shuffle \fu\shuffle \fv) \\
&+ \sum_{0<j<u+v} \Delta^j_{u,v} \sum (a_{u+v-j} \triangleright a_j) \otimes (b_{u+v-j}\shuffle b_j\shuffle \fu\shuffle \fv). \\
S_4&= \sum (a_{u+v} \triangleright a_{\fu\shuffle \fv}) \otimes (b_{u+v}\shuffle b_{\fu\shuffle \fv}) \\
&+\sum_{0<j<u+v} \Delta^j_{u,v} \sum (a_{u+v-j} \triangleright a_{\fu\shuffle \fv}) \otimes (b_{u+v-j}\shuffle x_j\shuffle b_{\fu\shuffle \fv}) \\
&+ \sum_{0<j<u+v} \Delta^j_{u,v} \sum (a_{u+v-j} \triangleright (a_j\shuffle a_{\fu\shuffle \fv})) \otimes (b_{u+v-j}\shuffle b_j\shuffle b_{\fu\shuffle \fv}),
\end{align*}
and
\begin{align*}
S_\fu&=\sum (a_u\shuffle (a_v \triangleright a_\fv)) \otimes (b_u\shuffle b_v\shuffle \fu\shuffle b_\fv)-\sum (a_u \triangleright (a_v \triangleright a_\fv)) \otimes (b_u\shuffle b_v\shuffle \fu\shuffle b_\fv) \\
&-\sum (a_v \triangleright (a_u\shuffle a_\fv)) \otimes (b_u\shuffle b_v\shuffle \fu\shuffle b_\fv), \\
S_\fv&=\sum ((a_u \triangleright a_\fu)\shuffle a_v) \otimes (b_u\shuffle b_v\shuffle b_\fu\shuffle \fv)-\sum (a_u \triangleright (a_\fu\shuffle  a_v)) \otimes (b_u\shuffle b_v\shuffle b_\fu\shuffle \fv) \\
&-\sum (a_v \triangleright (a_u \triangleright a_\fu)) \otimes (b_u\shuffle b_v\shuffle b_\fu\shuffle \fv).
\end{align*}

We claim that
\begin{enumerate}
\item $S_0=0$.
\item $S_1-S_3=0$.
\item $S_2+S_\fu+S_\fv-S_4=0$.
\end{enumerate}

We now prove the previous claim. For Part (1), we want to show that $S_0=0$. In fact, it follows immediately from \eqref{eq: expansion 2}.

For Part (2), we will show that 
\begin{align*}
S_1=S_3=\sum (a_u \diamond a_v)  \otimes (b_u\shuffle b_v\shuffle \fu\shuffle \fv).
\end{align*}
In fact, Lemma \ref{lem: triangle formulas} implies 
\begin{align*}
S_1&=\sum (a_u\shuffle a_v) \otimes (b_u\shuffle b_v\shuffle \fu\shuffle \fv)-\sum (a_u \triangleright a_v) \otimes (b_u\shuffle b_v\shuffle \fu\shuffle \fv) \\
&-\sum (a_v \triangleright a_u) \otimes (b_u\shuffle b_v\shuffle \fu\shuffle \fv) \\
&= \sum (a_u\shuffle a_v- a_u \triangleright a_v-a_v \triangleright a_u) \otimes (b_u\shuffle b_v\shuffle \fu\shuffle \fv) \\
&= \sum (a_u \diamond a_v)  \otimes (b_u\shuffle b_v\shuffle \fu\shuffle \fv).
\end{align*}

Recall that by \eqref{eq: x_u x_v 2},
\begin{align*} 
& \sum a_{u+v} \otimes b_{u+v} + \sum_{0<j<u+v} \Delta^j_{u,v} \sum a_{u+v-j} \otimes (b_{u+v-j}\shuffle x_j) \\
&+ \sum_{0<j<u+v} \Delta^j_{u,v} \sum (a_{u+v-j} \triangleright a_j) \otimes (b_{u+v-j}\shuffle b_j) \\
&= \sum (a_u \diamond a_v) \otimes (b_u\shuffle b_v). 
\end{align*}

As a direct consequence, we obtain
\begin{align*}
S_3&= \sum a_{u+v} \otimes (b_{u+v}\shuffle \fu\shuffle \fv) \\
&+\sum_{0<j<u+v} \Delta^j_{u,v} \sum a_{u+v-j} \otimes (b_{u+v-j}\shuffle x_j\shuffle \fu\shuffle \fv) \\
&+ \sum_{0<j<u+v} \Delta^j_{u,v} \sum (a_{u+v-j} \triangleright a_j) \otimes (b_{u+v-j}\shuffle b_j\shuffle \fu\shuffle \fv). \\
&= \sum (a_u \diamond a_v) \otimes (b_u\shuffle b_v\shuffle \fu\shuffle \fv)
\end{align*}
as desired. We conclude that $S_1=S_3$ as claimed.

For Part (3), we will show that
\begin{align*}
S_2+S_\fu+S_\fv=S_4=\sum ((a_u \diamond a_v) \triangleright a_{\fu\shuffle \fv}) \otimes (b_u\shuffle b_v\shuffle b_{\fu\shuffle \fv}).
\end{align*}
In fact, we have
\begin{align*}
S_2&=\sum ((a_u \triangleright a_\fu)\shuffle (a_v \triangleright a_\fv)) \otimes (b_u\shuffle b_\fu\shuffle b_v\shuffle b_\fv) \\
&-\sum (a_u \triangleright (a_\fu\shuffle (a_v \triangleright a_\fv))) \otimes (b_u\shuffle b_\fu\shuffle b_v\shuffle b_\fv) \\
&-\sum (a_v \triangleright ((a_u \triangleright a_\fu)\shuffle a_\fv) \otimes (b_u\shuffle b_\fu\shuffle b_v\shuffle b_\fv) \\
&=\sum ((a_u \triangleright a_\fu)\shuffle (a_v \triangleright a_\fv)-a_u \triangleright (a_\fu\shuffle (a_v \triangleright a_\fv))-a_v \triangleright ((a_u \triangleright a_\fu)\shuffle a_\fv)) \\
&\quad \quad \otimes (b_u\shuffle b_\fu\shuffle b_v\shuffle b_\fv) \\
&=\sum ((a_u \triangleright a_\fu) \diamond (a_v \triangleright a_\fv)) \otimes (b_u\shuffle b_\fu\shuffle b_v\shuffle b_\fv).
\end{align*}
Here the last equality follows from Lemma \ref{lem: triangle formulas} and Proposition \ref{prop: key properties}. More precisely,
\begin{align*}
&(a_u \triangleright a_\fu)\shuffle (a_v \triangleright a_\fv) \\&=(a_u \triangleright a_\fu) \triangleright (a_v \triangleright a_\fv)+(a_v \triangleright a_\fv) \triangleright (a_u \triangleright a_\fu)+(a_u \triangleright a_\fu) \diamond (a_v \triangleright a_\fv) \\
&=a_u \triangleright (a_\fu\shuffle (a_v \triangleright a_\fv))+a_v \triangleright ((a_u \triangleright a_\fu)\shuffle a_\fv))+(a_u \triangleright a_\fu) \diamond (a_v \triangleright a_\fv).
\end{align*}

By Lemma \ref{lem: triangle formulas} and Proposition \ref{prop: key properties} again, 
\begin{align*}
S_\fu&=\sum (a_u\shuffle (a_v \triangleright a_\fv)) \otimes (b_u\shuffle b_v\shuffle \fu\shuffle b_\fv)-\sum (a_u \triangleright (a_v \triangleright a_\fv)) \otimes (b_u\shuffle b_v\shuffle \fu\shuffle b_\fv) \\
&-\sum (a_v \triangleright (a_u\shuffle a_\fv)) \otimes (b_u\shuffle b_v\shuffle \fu\shuffle b_\fv) \\
&=\sum (a_u\shuffle (a_v \triangleright a_\fv)-a_u \triangleright (a_v \triangleright a_\fv)-a_v \triangleright (a_u\shuffle a_\fv)) \otimes (b_u\shuffle b_v\shuffle \fu\shuffle b_\fv) \\
&=\sum (a_u \diamond (a_v \triangleright a_\fv)) \otimes (b_u\shuffle b_v\shuffle \fu\shuffle b_\fv),
\end{align*}
and
\begin{align*}
S_\fv&=\sum ((a_u \triangleright a_\fu)\shuffle a_v) \otimes (b_u\shuffle b_v\shuffle b_\fu\shuffle \fv)-\sum (a_u \triangleright (a_\fu\shuffle  a_v)) \otimes (b_u\shuffle b_v\shuffle b_\fu\shuffle \fv) \\
&-\sum (a_v \triangleright (a_u \triangleright a_\fu)) \otimes (b_u\shuffle b_v\shuffle b_\fu\shuffle \fv) \\
&=\sum ((a_u \triangleright a_\fu) \diamond a_v) \otimes (b_u\shuffle b_v\shuffle b_\fu\shuffle \fv).
\end{align*}

Combining the previous formulas for $S_2$, $S_\fu$, $S_\fv$ with \eqref{eq: fu fv} yields
\begin{align*}
S_2+S_\fu+S_\fv=\sum ((a_u \diamond a_v) \triangleright a_{\fu\shuffle \fv}) \otimes (b_u\shuffle b_v\shuffle b_{\fu\shuffle \fv}).
\end{align*}

We now consider the term $S_4$. We have
\begin{align*}
S_4&= \sum (a_{u+v} \triangleright a_{\fu\shuffle \fv}) \otimes (b_{u+v}\shuffle b_{\fu\shuffle \fv}) \\
&+\sum_{0<j<u+v} \Delta^j_{u,v} \sum (a_{u+v-j} \triangleright a_{\fu\shuffle \fv}) \otimes (b_{u+v-j}\shuffle x_j\shuffle b_{\fu\shuffle \fv}) \\
&+ \sum_{0<j<u+v} \Delta^j_{u,v} \sum (a_{u+v-j} \triangleright (a_j\shuffle a_{\fu\shuffle \fv})) \otimes (b_{u+v-j}\shuffle b_j\shuffle b_{\fu\shuffle \fv}) \\
&= \sum (a_{u+v} \triangleright a_{\fu\shuffle \fv}) \otimes (b_{u+v}\shuffle b_{\fu\shuffle \fv}) \\
&+\sum_{0<j<u+v} \Delta^j_{u,v} \sum (a_{u+v-j} \triangleright a_{\fu\shuffle \fv}) \otimes (b_{u+v-j}\shuffle x_j\shuffle b_{\fu\shuffle \fv}) \\
&+ \sum_{0<j<u+v} \Delta^j_{u,v} \sum ((a_{u+v-j} \triangleright a_j) \triangleright a_{\fu\shuffle \fv}) \otimes (b_{u+v-j}\shuffle b_j\shuffle b_{\fu\shuffle \fv}) \\
&= \sum ((a_u \diamond a_v) \triangleright a_{\fu\shuffle \fv}) \otimes (b_u\shuffle b_v\shuffle b_{\fu\shuffle \fv}).
\end{align*}
Here the second equality follows from Proposition \ref{prop: key properties}. The third one is a direct consequence of \eqref{eq: x_u x_v 2}.
We then have proved that $S_2+S_\fu+S_\fv=S_4$ as claimed.

To conclude, we see immediately that 
\begin{align*}
\Delta(x_u \fu \shuffle x_v \fv)-\Delta(x_u \fu)\shuffle \Delta(x_v \fv)&=S_0-S_1-S_2-S_\fu-S_\fv+S_3+S_4 \\
&=S_0-(S_1-S_3)-(S_2+S_\fu+S_\fv-S_4) \\
&=0.
\end{align*}
The proof of Proposition \ref{prop: compatibility step 2} is finished.
\end{proof}

\subsubsection{Step 3} 
By Propositions \ref{prop: compatibility step 1} and \ref{prop: compatibility step 2}, for all words $\fu,\fv$ such that $\depth(\fu)+\depth(\fv)>2$ and $w(\fu)+w(\fv)=w$, we have proved  
	\[ \Delta(\fu\shuffle \fv)=\Delta(\fu)\shuffle \Delta(\fv). \]
To finish the proof of Theorem \ref{thm: compatibility} we prove the remaining case where both $\fu$ and $\fv$ have depth 1. 

\begin{proposition} \label{prop: compatibility step 3}
Let $u,v \in \bN$ such that $u+v=w$. Then 
	\[ \Delta(x_u\shuffle x_v)=\Delta(x_u)\shuffle \Delta(x_v). \]
\end{proposition}

\begin{proof}
By the definition of the coproduct $\Delta$, for all $n \in \bN$, 
\begin{equation} \label{eq: x_1}
\Delta(x_1\shuffle x_n)=\Delta(x_1)\shuffle \Delta(x_n).
\end{equation}
It follows that Proposition \ref{prop: compatibility step 3} holds if either $u=1$ or $v=1$. We now suppose that $u,v \geq 2$.

We claim that
\begin{equation} \label{eq: x_1 x_u-1 x_v}
\Delta(x_1\shuffle x_{u-1}\shuffle x_v)=\Delta(x_1)\shuffle \Delta(x_{u-1})\shuffle \Delta(x_v).
\end{equation}
In fact, we write
	\[ x_{u-1}\shuffle x_v=x_{u+v-1}+x_{u-1}x_v+x_vx_{u-1}+\sum_{0<j<u+v-1} \Delta^j_{u-1,v} x_{u+v-1-j} x_j. \]
Thus
\begin{align*}
& \Delta(x_1\shuffle x_{u-1}\shuffle x_v) \\
&= \Delta(x_1\shuffle (x_{u-1}\shuffle x_v)) \\
&=\Delta \left(x_1\shuffle \left(x_{u+v-1}+x_{u-1}x_v+x_vx_{u-1}+\sum_{0<j<u+v-1} \Delta^j_{u-1,v} x_{u+v-1-j} x_j\right)\right) \\
&=\Delta (x_1\shuffle x_{u+v-1})+\Delta(x_1\shuffle x_{u-1}x_v)+\Delta(x_1\shuffle x_vx_{u-1}) \\
&+\sum_{0<j<u+v-1} \Delta^j_{u-1,v} \Delta(x_1\shuffle x_{u+v-1-j} x_j) \\
&=\Delta (x_1)\shuffle \Delta(x_{u+v-1})+\Delta(x_1)\shuffle \Delta(x_{u-1}x_v) \\
&+\Delta(x_1)\shuffle \Delta(x_vx_{u-1})+\sum_{0<j<u+v-1} \Delta^j_{u-1,v} \Delta(x_1)\shuffle \Delta(x_{u+v-1-j} x_j).
\end{align*}
Here the last equality follows from \eqref{eq: x_1} and Proposition \ref{prop: compatibility step 1}. It implies
\begin{align*}
& \Delta(x_1\shuffle x_{u-1}\shuffle x_v) \\
&=\Delta (x_1)\shuffle \Delta(x_{u+v-1})+\Delta(x_1)\shuffle \Delta(x_{u-1}x_v) \\
&+\Delta(x_1)\shuffle \Delta(x_vx_{u-1})+\sum_{0<j<u+v-1} \Delta^j_{u-1,v} \Delta(x_1)\shuffle \Delta(x_{u+v-1-j} x_j) \\
&=\Delta (x_1)\shuffle \Delta\left(x_{u+v-1}+x_{u-1}x_v+x_vx_{u-1}+\sum_{0<j<u+v-1} \Delta^j_{u-1,v} x_{u+v-1-j} x_j \right) \\
&=\Delta (x_1)\shuffle \Delta(x_{u-1}\shuffle x_v).
\end{align*}
As $u+v-1=w-1<w$, $\Delta(x_{u-1}\shuffle x_v)=\Delta(x_{u-1})\shuffle \Delta(x_v)$, which implies the claim.

Now we express both sides of \eqref{eq: x_1 x_u-1 x_v} by a different way. First, the LHS of \eqref{eq: x_1 x_u-1 x_v} equals
\begin{align*}
&\Delta(x_1\shuffle x_{u-1}\shuffle x_v) \\
&=\Delta((x_1\shuffle x_{u-1})\shuffle x_v) \\
&=\Delta \left(\left(x_u+x_1x_{u-1}+x_{u-1}x_1+\sum_{0<j<u} \Delta^j_{1,u-1} x_{u-j} x_j \right)\shuffle x_v \right) \\
&=\Delta(x_u\shuffle x_v)+\Delta(x_1x_{u-1}\shuffle x_v)+\Delta(x_{u-1}x_1\shuffle x_v)+\sum_{0<j<u} \Delta^j_{1,u-1} \Delta(x_{u-j} x_j\shuffle x_v) \\
&=\Delta(x_u\shuffle x_v)+\Delta(x_1 x_{u-1})\shuffle \Delta(x_v)+\Delta(x_{u-1} x_1)\shuffle \Delta(x_v) \\
&+\sum_{0<j<u} \Delta^j_{1,u-1} \Delta(x_{u-j} x_j)\shuffle \Delta(x_v). 
\end{align*}
The last equality holds by Proposition \ref{prop: compatibility step 1}.

Next, as $\Delta(x_1)\shuffle \Delta(x_{u-1})=\Delta(x_1\shuffle x_{u-1})$, the RHS of \eqref{eq: x_1 x_u-1 x_v} equals
\begin{align*}
& \Delta(x_1)\shuffle \Delta(x_{u-1})\shuffle \Delta(x_v) \\
&=(\Delta(x_1)\shuffle \Delta(x_{u-1}))\shuffle \Delta(x_v) \\
&=\Delta(x_1\shuffle x_{u-1})\shuffle \Delta(x_v) \\
&=\Delta \left(x_u+x_1x_{u-1}+x_{u-1}x_1+\sum_{0<j<u} \Delta^j_{1,u-1} x_{u-j} x_j \right)\shuffle \Delta(x_v) \\
&=\Delta(x_u)\shuffle \Delta(x_v)+\Delta(x_1 x_{u-1})\shuffle \Delta(x_v)+\Delta(x_{u-1} x_1)\shuffle \Delta(x_v) \\
&+\sum_{0<j<u} \Delta^j_{1,u-1} \Delta(x_{u-j} x_j)\shuffle \Delta(x_v). 
\end{align*}

Putting all together, we deduce
	\[ \Delta(x_u\shuffle x_v)=\Delta(x_u)\shuffle \Delta(x_v) \]
as desired.
\end{proof}

\subsection{Coassociativity of the coproduct} ${}$\par

In this section we prove the coassociativity of the coproduct $\Delta$.

\begin{theorem} \label{thm: coassociativity}
Let $\fu \in \langle \Sigma \rangle$. Then we have
	\[ (\Id \otimes \Delta) \Delta(\fu)=(\Delta \otimes \Id)\Delta(\fu). \]
\end{theorem}

The rest of this section is devoted to a proof of Theorem \ref{thm: coassociativity}. The proof is by induction on the total weight $w=w(\fa)+w(\fb)$.

For $w=0$ and $w=1$ we see that Theorem \ref{thm: coassociativity} holds. Let $w \in \bN$ with $w \geq 2$ and we suppose that for all $\fu \in \langle \Sigma \rangle$ such that $w(\fu)<w$, we have
	\[ (\Id \otimes \Delta) \Delta(\fu)=(\Delta \otimes \Id)\Delta(\fu). \]

We now show that for all $\fu \in \langle \Sigma \rangle$ with $w(\fu)=w$,
	\[ (\Id \otimes \Delta) \Delta(\fu)=(\Delta \otimes \Id)\Delta(\fu). \]
The proof will be divided into several steps.

\subsubsection{Step 1} 
We first prove the following proposition. 
\begin{proposition} \label{prop: coassociativity step 1}
For all words $\fu, \fv \in \langle \Sigma \rangle$ with $w(\fu)+w(\fv)=w$, we have
	\[ (\Id \otimes \Delta) \Delta(\fu\shuffle \fv)=(\Delta \otimes \Id)\Delta(\fu\shuffle \fv). \]
\end{proposition}

\begin{proof}
We see that if $\fu=1$ or $\fv=1$, then the proposition holds. We can suppose that $\fu \neq 1$ and $\fv \neq 1$. If we write 
\begin{align*}
\Delta(\fu)&=\sum a_\fu \otimes b_\fu, \\
\Delta(\fv)&=\sum a_\fv \otimes b_\fv,
\end{align*}
then by the compatibility, we get
\begin{align*}
(\Id \otimes \Delta) \Delta(\fu\shuffle \fv)&=(\Id \otimes \Delta)(\Delta(\fu)\shuffle \Delta(\fv)) \\
&= (\Id \otimes \Delta)\left(\sum (a_\fu\shuffle a_\fv) \otimes (b_\fu\shuffle b_\fv) \right) \\
&= \sum (a_\fu\shuffle a_\fv) \otimes \Delta(b_\fu\shuffle b_\fv) \\
&= \sum (a_\fu\shuffle a_\fv) \otimes (\Delta(b_\fu)\shuffle \Delta(b_\fv)) \\
&= \left(\sum a_\fu \otimes \Delta(b_\fu)\right)\shuffle  \left(\sum a_\fv \otimes \Delta(b_\fv)\right) \\
&=(\Id \otimes \Delta) \Delta(\fu) \shuffle  (\Id \otimes \Delta) \Delta(\fv).
\end{align*}
Similarly, we get
\begin{align*}
(\Delta \otimes \Id) \Delta(\fu\shuffle \fv)&=(\Delta \otimes \Id) (\Delta(\fu)\shuffle \Delta(\fv)) \\
&=(\Delta \otimes \Id) \Delta(\fu) \shuffle  (\Delta \otimes \Id) \Delta(\fv).
\end{align*}

Since $w(\fu)<w$ and $w(\fv)<w$, the induction hypothesis implies 
\begin{align*}
(\Id \otimes \Delta) \Delta(\fu)&=(\Delta \otimes \Id)\Delta(\fu), \\
(\Id \otimes \Delta) \Delta(\fv)&=(\Delta \otimes \Id)\Delta(\fv).
\end{align*}
Putting all together, we deduce
\begin{align*}
(\Id \otimes \Delta) \Delta(\fu\shuffle \fv)&=(\Id \otimes \Delta) \Delta(\fu) \shuffle  (\Id \otimes \Delta) \Delta(\fv) \\
&=(\Delta \otimes \Id)\Delta(\fu)\shuffle (\Delta \otimes \Id)\Delta(\fv) \\
&=(\Delta \otimes \Id) \Delta(\fu\shuffle \fv).
\end{align*}
The proof is finished.
\end{proof}

\subsubsection{Step 2} 
We define the operator $\triangleright$ for tensors as follows:
\begin{align*}
\left(\sum \frak a_1 \otimes \frak b_1\right) \triangleright \left(\sum \frak a_2 \otimes \frak b_2\right):&=\sum (\frak a_1 \triangleright \frak a_2) \otimes (\frak b_1\shuffle \frak b_2), \\ 
\left(\sum \frak a_1 \otimes \frak b_1 \otimes \frak c_1\right) \triangleright \left(\sum \frak a_2 \otimes \frak b_2 \otimes \frak c_2 \right):&=\sum (\frak a_1 \triangleright \frak a_2) \otimes (\frak b_1\shuffle \frak b_2) \otimes (\frak c_1\shuffle \frak c_2).
\end{align*} 
We prove the following lemma which will be useful in the sequel.
\begin{lemma} \label{lem: delta without factor 1}
Let $\fu,\fv \in \langle \Sigma \rangle$ with $\fu \neq 1$. Then 
	\[ (\Delta(\fu)-1 \otimes \fu) \triangleright \Delta(\fv)=\Delta(\fu \triangleright \fv)-1 \otimes (\fu \triangleright \fv). \]
\end{lemma}

\begin{proof}
Suppose that $\depth(\fu)=1$, says $\fu=x_u$. Then $\fu \triangleright \fv=x_u \fv$ and we have to show that $(\Delta(x_u)-1 \otimes x_u) \triangleright \Delta(\fv)=\Delta(x_u \fv)-1 \otimes (x_u \fv)$. We write
\begin{align*}
\Delta(x_u)=1 \otimes x_u + \sum a_u \otimes b_u, \quad \Delta(\fv)= \sum a_\fv \otimes b_\fv.
\end{align*}
Recall that $\Delta(x_u\fv)=1 \otimes (x_u \fv)+\sum (a_u \triangleright a_\fv) \otimes (b_u\shuffle b_\fv)$. Thus
\begin{align*}
\Delta(x_u\fv)-1 \otimes (x_u \fv)&=\sum (a_u \triangleright a_\fv) \otimes (b_u\shuffle b_\fv) \\
&=(\Delta(\fu)-1 \otimes \fu) \triangleright \Delta(\fv)
\end{align*}
as desired.

We now suppose that $\depth(\fu)>1$. We write $\fu=x_u \fu'$. Thus $\fu \triangleright \fv=x_u(\fu'\shuffle \fv)$ and we have to show that
	\[ (\Delta(x_u \fu')-1 \otimes x_u \fu') \triangleright \Delta(\fv)=\Delta(x_u(\fu'\shuffle \fv))-1 \otimes x_u(\fu'\shuffle \fv). \]
We put 
\begin{align*}
\Delta(x_u)&=1 \otimes x_u+ \sum a_u \otimes b_u, \\
\Delta(\fu')&=\sum a_{\fu'} \otimes b_{\fu'}, \\
\Delta(\fv)&=\sum a_\fv \otimes b_\fv.
\end{align*}
It follows that 
\begin{equation} \label{eq: words}
\Delta(x_u \fu')-1 \otimes x_u \fu'=\sum (a_u \triangleright a_{\fu'}) \otimes (b_u\shuffle b_{\fu'}).
\end{equation}

By Theorem \ref{thm: compatibility}, we have
	\[ \Delta(\fu'\shuffle \fv)=\Delta(\fu')\shuffle \Delta(\fv)=\sum (a_{\fu'}\shuffle a_\fv) \otimes (b_{\fu'}\shuffle b_\fv). \]
Thus
\begin{align*}
\Delta(x_u(\fu'\shuffle \fv))-1 \otimes x_u(\fu'\shuffle \fv) &=\sum (a_u \triangleright (a_{\fu'}\shuffle a_\fv)) \otimes (b_u\shuffle b_{\fu'}\shuffle b_\fv) \\
&=\sum ((a_u \triangleright a_{\fu'}) \triangleright a_\fv)) \otimes (b_u\shuffle b_{\fu'}\shuffle b_\fv) \\
&=(\Delta(x_u \fu')-1 \otimes x_u \fu') \triangleright \Delta(\fv).
\end{align*}
The second equality holds by Proposition \ref{prop: key properties} and the last one follows from \eqref{eq: words}. We finish the proof.
\end{proof}

Next we prove the following proposition.
\begin{proposition} \label{prop: coassociativity step 2}
For all words $x_u \fu \in \langle \Sigma \rangle$ with $u \in \bN$, $\depth(\fu) \geq 1$ and $w(x_u \fu)=w$, we have
 	\[ (\Id \otimes \Delta) \Delta(x_u \fu)=(\Delta \otimes \Id)\Delta(x_u \fu). \]
\end{proposition}

\begin{proof}
We put 
\begin{align*}
\Delta(x_u)&=1 \otimes x_u + \sum a_u \otimes b_u, \\
\Delta(\fu)&=\sum a_\fu \otimes b_\fu.
\end{align*}
In particular,
	\[ \Delta(x_u \fu)=1 \otimes x_u\fu + \sum (a_u \triangleright a_\fu) \otimes (b_u\shuffle b_\fu). \]
It follows that
\begin{align*}
&(\Id \otimes \Delta)\Delta(x_u \fu) \\
&=(\Id \otimes \Delta) \left(1 \otimes x_u\fu + \sum (a_u \triangleright a_\fu) \otimes (b_u\shuffle b_\fu)\right) \\
&= 1 \otimes \Delta(x_u\fu) + \sum (a_u \triangleright a_\fu) \otimes \Delta(b_u\shuffle b_\fu) \\
&=1 \otimes \Delta(x_u\fu) + \sum (a_u \triangleright a_\fu) \otimes (\Delta(b_u)\shuffle \Delta(b_\fu)).
\end{align*}
The last equality follows from the compatibility proved in Theorem \ref{thm: compatibility}. 

Next, we have
\begin{align*}
&(\Delta \otimes \Id)\Delta(x_u \fu) \\
&=(\Delta \otimes \Id) \left(1 \otimes x_u\fu + \sum (a_u \triangleright a_\fu) \otimes (b_u\shuffle b_\fu)\right) \\
&= 1 \otimes 1 \otimes x_u\fu + \sum \Delta(a_u \triangleright a_\fu) \otimes (b_u\shuffle b_\fu).
\end{align*}
Thus we have to show that
\begin{align} \label{eq: coassociativity}
& 1 \otimes \Delta(x_u\fu) + \sum (a_u \triangleright a_\fu) \otimes (\Delta(b_u)\shuffle \Delta(b_\fu)) \\
&=1 \otimes 1 \otimes x_u\fu + \sum \Delta(a_u \triangleright a_\fu) \otimes (b_u\shuffle b_\fu). \notag
\end{align}

As $w(x_u)<w$, the induction hypothesis implies
 	\[ (\Id \otimes \Delta) \Delta(x_u)=(\Delta \otimes \Id)\Delta(x_u). \]
We write down the expressions of both sides. The LHS equals
\begin{align*}
(\Id \otimes \Delta) \Delta(x_u)&=(\Id \otimes \Delta)\left(1 \otimes x_u+\sum a_u \otimes b_u\right) \\
&=1 \otimes \Delta(x_u)+\sum a_u \otimes \Delta(b_u),
\end{align*}
and the RHS equals
\begin{align*}
(\Delta \otimes \Id) \Delta(x_u)&=(\Delta \otimes \Id)\left(1 \otimes x_u+\sum a_u \otimes b_u\right) \\
&=1 \otimes 1 \otimes x_u+\sum \Delta(a_u) \otimes b_u.
\end{align*}
As $(\Id \otimes \Delta) \Delta(x_u)=(\Delta \otimes \Id)\Delta(x_u)$, we get;
\begin{equation*} 
1 \otimes \Delta(x_u)+\sum a_u \otimes \Delta(b_u)=1 \otimes 1 \otimes x_u+\sum \Delta(a_u) \otimes b_u.
\end{equation*}
We use Lemma \ref{lem: factor 1} and cancel the terms of the form $1 \otimes \frak a \otimes \frak b$ on both sides to get
\begin{equation} \label{eq: coasso u}
\sum a_u \otimes \Delta(b_u)=\sum (\Delta(a_u)-1 \otimes a_u) \otimes b_u.
\end{equation}

Recall that $\Delta(\fu)=\sum a_\fu \otimes b_\fu$. By similar calculations, the equality $(\Id \otimes \Delta) \Delta(\fu)=(\Delta \otimes \Id)\Delta(\fu)$ implies
\begin{equation} \label{eq: coasso fu}
\sum a_\fu \otimes \Delta(b_\fu)=\sum \Delta(a_\fu) \otimes b_\fu.
\end{equation}

Combining \eqref{eq: coasso u} and \eqref{eq: coasso fu} yields
\begin{align*}
&\left(\sum a_u \otimes \Delta(b_u)\right) \triangleright \left(\sum a_\fu \otimes \Delta(b_\fu)\right) \\
&=\left(\sum (\Delta(a_u)-1 \otimes a_u) \otimes b_u\right) \triangleright \left(\sum \Delta(a_\fu) \otimes b_\fu\right).
\end{align*}
It follows that
\begin{align*}
&\sum (a_u \triangleright a_\fu) \otimes (\Delta(b_u)\shuffle \Delta(b_\fu)) \\
&=\sum \left((\Delta(a_u)-1 \otimes a_u) \triangleright \Delta(a_\fu)\right) \otimes (b_u\shuffle b_\fu) \\
&=\sum \left(\Delta(a_u \triangleright a_\fu)-1 \otimes (a_u \triangleright a_\fu)\right) \otimes (b_u\shuffle b_\fu) \\
&=\sum \Delta(a_u \triangleright a_\fu) \otimes (b_u\shuffle b_\fu)- \sum 1 \otimes (a_u \triangleright a_\fu) \otimes (b_u\shuffle b_\fu).
\end{align*}
The second equality holds by Lemma \ref{lem: delta without factor 1}.

As a consequence, the equality \eqref{eq: coassociativity} follows immediately from the equality
	\[ \Delta(x_u \fu)=1 \otimes x_u \fu+\sum (a_u \triangleright a_\fu) \otimes (b_u\shuffle b_\fu). \]
\end{proof}

\subsubsection{Step 3} 
By Proposition \ref{prop: coassociativity step 2}, for all words $\fu$ such that $\depth(\fu)>1$ and $w(\fu)=w$, we have proved  
 	\[ (\Id \otimes \Delta) \Delta(\fu)=(\Delta \otimes \Id)\Delta(\fu). \]
To finish the proof of Theorem \ref{thm: coassociativity} we prove the remaining case where $\fu=x_w$. 

\begin{proposition}
We have
 	\[ (\Id \otimes \Delta) \Delta(x_w)=(\Delta \otimes \Id)\Delta(x_w). \]
\end{proposition}

\begin{proof}
By Proposition \ref{prop: coassociativity step 1},
\begin{equation*} 
(\Id \otimes \Delta) \Delta(x_1\shuffle x_{w-1})=(\Id \otimes \Delta) \Delta(x_1\shuffle x_{w-1}).
\end{equation*}

Now we express both sides of the above equality by using
	\[ x_1\shuffle x_{w-1}=x_w+x_1x_{w-1}+x_{w-1}x_1+\sum_{0<j<w} \Delta^j_{1,w-1} x_{w-j} x_j. \]
First, the LHS equals
\begin{align*}
&(\Id \otimes \Delta) \Delta(x_1\shuffle x_{w-1}) \\
&=(\Id \otimes \Delta) \Delta\left(x_w+x_1x_{w-1}+x_{w-1}x_1+\sum_{0<j<w} \Delta^j_{1,w-1} x_{w-j} x_j\right) \\
&=(\Id \otimes \Delta) \Delta(x_w)+(\Id \otimes \Delta) \Delta(x_1x_{w-1})+(\Id \otimes \Delta) \Delta(x_{w-1}x_1) \\
&+\sum_{0<j<w} \Delta^j_{1,w-1} (\Id \otimes \Delta) \Delta(x_{w-j} x_j),
\end{align*}
and the RHS equals
\begin{align*}
&(\Delta \otimes \Id) \Delta(x_1\shuffle x_{w-1}) \\
&=(\Delta \otimes \Id) \Delta\left(x_w+x_1x_{w-1}+x_{w-1}x_1+\sum_{0<j<w} \Delta^j_{1,w-1} x_{w-j} x_j\right) \\
&=(\Delta \otimes \Id) \Delta(x_w)+(\Delta \otimes \Id) \Delta(x_1x_{w-1})+(\Delta \otimes \Id)\Delta(x_{w-1}x_1) \\
&+\sum_{0<j<w} \Delta^j_{1,w-1} (\Delta \otimes \Id) \Delta(x_{w-j} x_j).
\end{align*}

Putting all together and using Propositions \ref{prop: coassociativity step 1} and \ref{prop: coassociativity step 2}, we deduce that
 	\[ (\Id \otimes \Delta) \Delta(x_w)=(\Delta \otimes \Id)\Delta(x_w). \]
\end{proof}

\subsection{Hopf algebra structure} ${}$\par

The counit $\epsilon:\frak C \to \Fq$ is defined as follows: $\epsilon(1)=1$ and $\epsilon(\fu)=0$ otherwise. By induction on weight, we can check that $\Delta$ preserves the grading. So $(\frak C,\shuffle,u,\Delta,\epsilon)$ is a connected graded bialgebra. By Proposition \ref{prop: graded Hopf algebras}, we get

\begin{theorem} \label{thm: Hopf algebra for shuffle product}
The connected graded bialgebra $(\frak C,\shuffle,u,\Delta,\epsilon)$ is a connected graded Hopf algebra over $\Fq$.
\end{theorem}

We also note that the Hopf shuffle algebra is of finite type (see Definition \ref{defn: graded Hopf algebra}).

\subsection{Numerical verification} ${}$\par
We end this section by presenting some numerical experiments for the shuffle algebra in positive characteristic. We mention that these calculations have been crucial for us during this project.

We have written codes which can calculate the operations on $\mathfrak{C}$ and verified Proposition~\ref{prop: associative}, Theorem~\ref{thm: coassociativity} and Theorem~\ref{thm: compatibility} in numerous cases as below, ``extending" the verification which were done in \cite[\S 3.2.3]{Shi18} The machine used for the computation is MacBook Pro (15-inch, 2018), with 2.6 GHz 6-core Intel Core i7 CPU and 16GB of memory. 
\begin{description}
    \item[Associativity] \par For $q=2,3,4, 5, 7$, $\left(S_d(a)S_d(b)\right) S_d(c) = S_d(a)\left(S_d(b) S_d(c)\right)$ for all depth 1 tuples $a, b, c$ with weight $<q^3$. 
    \begin{itemize}
        \item The running time for each of the cases when $q=2, 3, 4, 5, 7$ are less than 1 second, 16 seconds, 170 seconds, 107 minutes, and 118 hours, respectively.
    \end{itemize}
    \item[Coassociativity] Coassociativity holds for some initial cases. Precisely, the coassociativity for depth one word $x_n$ for $1\le n < a_q$ and all words with weight $<w_q$ were verified within execution time $t_{q}$ and $t'_{q}$ seconds respectively, for following $q$'s:
    \begin{table}[!htbp]
\begin{tabular}{c|cccccccc}
\hline
$q$           & $2$          & $3$         & $5$         & $7$           \\
$(a_q, t_q)$  & $(33,362)$   & $(62,432)$  & $(122,335)$ & $(182, 362)$ \\
$(w_q, t_q')$ & $(15, 2038)$ & $(14, 386)$ & $(18, 578)$ & $(18, 236)$   \\
\hline 
$q$           & $8$         & $9$          & $11$         & $13$         \\
$(a_q, t_q)$  & $(300,222)$ & $(302, 183)$ & $(300, 191)$ & $(382, 283)$ \\
$(w_q, t_q')$ & $(18, 207)$ & $(18, 199)$  & $(18, 199)$  & $(18, 218)$  \\
\hline
\end{tabular}
\end{table}
    \item[Compatibility] Compatibility holds for some initial cases. Precisely, we verified that $\Delta(w_1 \shuffle  w_2) = \Delta(w_1)\Delta(w_2)$ holds for all words $w_1, w_2 \in \langle \Sigma \rangle$ with $\operatorname{weight}(w_1)+ \operatorname{weight}{w_2} \le 12$ within $t_q$ seconds of execution time respectively, where $t_2 = 1965$, $t_3 = 1414$, $t_4 = 450$, $t_5 = 825$, $t_7 = 800$, $t_8 = 365$, $t_9 = 506$, $t_{11} = 793$, $t_{13} = 789$.  
    Further, when $q=9$, we verified for all words $w_1, w_2 \in \langle \Sigma \rangle$ with $\operatorname{weight}(w_1)+ \operatorname{weight}(w_2) \le 13$ within 2450 seconds of execution time.
\end{description}


\section{Coproduct of depth one} \label{sec: depth one}

This section aims to prove Shi's conjecture on a Hopf algebra structure of the shuffle algebra (see Theorem \ref{thm: comparison with Shi's coproduct}). To do so we study coproduct for words of depth one and deduce that the coproduct $\Delta$ coincides with that introduced by Shi (see Proposition \ref{prop: comparison with Shi's coproduct}). Explicit formulas for coproduct of such words are given in many cases in \S \ref{sec: explicit formula for small weights} and the Appendix \ref{sec: numerical experiments}.

\subsection{Bracket operators} ${}$\par

In the next section we will give a formula for $\Delta(x_n)$ for all $n \in \N$ (see Prop \ref{prop: formula delta xn}). To do so, we need some preparatory results.

\begin{lemma} \label{lem: delta j<n}
Let $n$ be a natural number. We have
\begin{enumerate}
    \item for all $j < n$, 
\begin{equation*}
    \Delta^j_{1,n} = \begin{cases}  1 & \quad \text{if } (q - 1) \mid j \\
0 & \quad \text{otherwise}.
\end{cases}
\end{equation*}

\item $\Delta^n_{1,n} = 0$.
\end{enumerate}
\end{lemma}

\begin{proof}
The result is straightforward from the definition of $\Delta^j_{1,n}$.
\end{proof}

\begin{remark} \label{rmk: delta j<n}
It follows from Lemma \ref{lem: delta j<n} that for all $j,n,m \in \N$ with $j < n < m$, we may identify $\Delta^j_{1,n} = \Delta^{j}_{1,m}$.
\end{remark}

We recall that $\langle \Sigma \rangle$ is the set of all words over $\Sigma = \{x_n\}_{n \in \N}$. Let $\fa = x_{i_1}\dotsb x_{i_r} $ be a non-empty word in $\langle \Sigma \rangle$. We define the \textit{bracket operator} by the following formula:
\begin{equation} \label{eq: bracket operator}
    [\fa] := (-1)^{r}\Delta^{i_1}_{1,w(\fa) + 1}\dotsb\Delta^{i_r}_{1,w(\fa) + 1} x_{i_1} \shuffle \dotsb \shuffle x_{i_r}.
\end{equation}
As a matter of convention, we also agree that $[1] = 1$. The following lemmas will be useful.
\begin{lemma} \label{lem: prod of bracket}
Let $\fb$ and $\fc$ be two words in $\langle \Sigma \rangle$. Then we have
\begin{equation*}
    [\fb] \shuffle [\fc] = [\fb\fc],
\end{equation*}
where $\fb\fc$ is the concatenation product of $\fb$ and $\fc$.
\end{lemma}
\begin{proof}
The result holds trivially if $\fb = 1$ or $\fc = 1$. We thus assume that $\fb = x_{i_1}\dotsb x_{i_r} $ and $\fc = x_{j_1}\dotsb x_{j_s} $. Then it follows from the definition of the bracket operator and Remark \ref{rmk: delta j<n} that 
\begin{align*}
     [\fb] &\shuffle [\fc] \\
    &= (-1)^{r+s} \Delta^{i_1}_{1,w(\fb)+1} \dotsb \Delta^{i_r}_{1,w(\fb)+1}\Delta^{j_1}_{1,w(\fc)+1} \dotsb \Delta^{j_s}_{1,w(\fc)+1}(x_{i_1}\shuffle \cdots \shuffle x_{i_r}) \shuffle (x_{j_1} \shuffle \cdots \shuffle x_{j_s}) \\
    &= (-1)^{r+s} \Delta^{i_1}_{1,w(\fb\fc)+1} \dotsb \Delta^{i_r}_{1,w(\fb\fc)+1}\Delta^{j_1}_{1,w(\fb\fc)+1} \dotsb \Delta^{j_s}_{1,w(\fb\fc)+1}x_{i_1}\shuffle \cdots \shuffle x_{i_r} \shuffle x_{j_1} \shuffle \cdots \shuffle x_{j_s} \\
    &= [\fb\fc].
\end{align*}
This proves the lemma.
\end{proof}

\begin{lemma} \label{lem: bracket zero}
Let $\fa$ be a word in $\langle \Sigma \rangle$. If $(q-1) \nmid w(\fa)$, then $[\fa] = 0$.
\end{lemma}

\begin{proof}
We may assume that $\fa = x_{i_1}\dotsb x_{i_r}$, so that $w(\fa) = i_1 + \cdots + i_r$. If  $(q-1) \nmid w(\fa)$, then there exists an index $i_k$ for $1 \leq k \leq r$ such that $(q-1) \nmid i_k$. It follows from Lemma \ref{lem: delta j<n} that $\Delta^{i_k}_{1,w(\fa) +1} = 0$, and hence $[\fa] =0$. This proves the lemma.
\end{proof}

For the convenience of computation, we introduce the following result.

\begin{lemma} \label{lem: delta diamond}
Let $\mathfrak{u}$ and $\mathfrak{v}$ be two non-empty words in $\langle \Sigma \rangle$. Suppose that
\begin{align*}
    \Delta(\mathfrak{u}) &= 1 \otimes \mathfrak{u} + \sum \mathfrak{u}_{(1)} \otimes \mathfrak{u}_{(2)},\\
    \Delta(\mathfrak{v}) &= 1 \otimes \mathfrak{v} + \sum \mathfrak{v}_{(1)} \otimes \mathfrak{v}_{(2)}.
\end{align*}
Then 
\begin{equation*}
    \Delta(\mathfrak{u} \diamond \mathfrak{v}) = 1 \otimes (\mathfrak{u} \diamond \mathfrak{v}) + \sum (\mathfrak{u}_{(1)} \diamond \mathfrak{v}_{(1)}) \otimes (\mathfrak{u}_{(2)} \shuffle \mathfrak{v}_{(2)}).
\end{equation*}
\end{lemma}

\begin{proof}
From the compatibility, we have 
\begin{align*}
    \Delta(\mathfrak{u} \shuffle \mathfrak{v}) &= \Delta(\mathfrak{u}) \shuffle \Delta(\mathfrak{v})\\
    &= 1 \otimes (\mathfrak{u} \shuffle \mathfrak{v}) + \sum \mathfrak{u}_{(1)} \otimes (\mathfrak{u}_{(2)} \shuffle \mathfrak{v}) + \sum \mathfrak{v}_{(1)} \otimes (\mathfrak{u} \shuffle \mathfrak{v}_{(2)}) \\
    & + \sum (\mathfrak{u}_{(1)} \shuffle \mathfrak{v}_{(1)}) \otimes (\mathfrak{u}_{(2)} \shuffle \mathfrak{v}_{(2)}).
\end{align*}
On the other hand, it follows from Lemma \ref{lem: delta without factor 1} that
\begin{align*}
    \Delta(\mathfrak{u} \triangleright \mathfrak{v}) &= 1 \otimes (\mathfrak{u} \triangleright \mathfrak{v}) + \sum \mathfrak{u}_{(1)} \otimes (\mathfrak{u}_{(2)} \shuffle \mathfrak{v}) + \sum (\mathfrak{u}_{(1)} \triangleright \mathfrak{v}_{(1)}) \otimes (\mathfrak{u}_{(2)} \shuffle \mathfrak{v}_{(2)}),\\
    \Delta(\mathfrak{v} \triangleright \mathfrak{u}) &= 1 \otimes (\mathfrak{v} \triangleright \mathfrak{u}) + \sum \mathfrak{v}_{(1)} \otimes (\mathfrak{u} \shuffle \mathfrak{v}_{(2)}) + \sum (\mathfrak{v}_{(1)} \triangleright \mathfrak{u}_{(1)}) \otimes (\mathfrak{u}_{(2)} \shuffle \mathfrak{v}_{(2)}).
\end{align*}
Thus 
\begin{align*}
    \Delta(\mathfrak{u} \diamond \mathfrak{v}) &= \Delta(\mathfrak{u} \shuffle \mathfrak{v} - \mathfrak{u} \triangleright \mathfrak{v} - \mathfrak{v} \triangleright \mathfrak{u})\\
    &= \Delta(\mathfrak{u} \shuffle \mathfrak{v}) - \Delta(\mathfrak{u} \triangleright \mathfrak{v}) - \Delta(\mathfrak{v} \triangleright \mathfrak{u})\\
    &= 1 \otimes (\mathfrak{u} \shuffle \mathfrak{v} - \mathfrak{u} \triangleright \mathfrak{v} - \mathfrak{v} \triangleright \mathfrak{u}) \\
    &+ \sum (\mathfrak{u}_{(1)} \shuffle \mathfrak{v}_{(1)} - \mathfrak{u}_{(1)} \triangleright \mathfrak{v}_{(1)} - \mathfrak{v}_{(1)} \triangleright \mathfrak{u}_{(1)}) \otimes (\mathfrak{u}_{(2)} \shuffle \mathfrak{v}_{(2)})\\
    &= 1 \otimes (\mathfrak{u} \diamond \mathfrak{v}) + \sum (\mathfrak{u}_{(1)} \diamond \mathfrak{v}_{(1)}) \otimes (\mathfrak{u}_{(2)} \shuffle \mathfrak{v}_{(2)}).
\end{align*}
This proves the lemma.
\end{proof}

\subsection{A formula for the coproduct of depth one}  ${}$\par

The first result of this section reads as follows.
\begin{proposition} \label{prop: formula delta xn}
For all $n \in \N$, we have
\begin{equation*}
    \Delta(x_n) = 1 \otimes x_n + \sum \limits_{\substack{r \in \N, \fa \in \langle \Sigma \rangle\\ r + w(\fa) = n}} \binom{r + \depth(\fa) - 2}{\depth(\fa)} x_r \otimes [\fa].
\end{equation*}
Here we recall that $[\fa]$ is given as in \eqref{eq: bracket operator}.
\end{proposition}

The rest of this section is devoted to a proof of Proposition \ref{prop: formula delta xn}. We proceed the proof by induction on $n$. For $n = 1$, we have
\begin{equation*}
    \Delta(x_1) = 1 \otimes x_1 + \binom{-1}{0} x_1 \otimes [1] = 1 \otimes x_1 + x_1 \otimes 1,
\end{equation*}
which proves the base step. We assume that Proposition \ref{prop: formula delta xn} holds for all $n \leq m$ with $m \in \N$ and $m \geq 1$. We need to show that Proposition \ref{prop: formula delta xn} holds for $n = m + 1$. Indeed, from the induction hypothesis, we have
\begin{equation*}
    \Delta(x_m) = 1 \otimes x_m + \sum \limits_{\substack{r \in \N, \fa \in \langle \Sigma \rangle\\ r + w(\fa) = m}} \binom{r + \depth(\fa) - 2}{\depth(\fa)} x_r \otimes [\fa],
\end{equation*}
hence it follows from Lemma \ref{lem: delta diamond} that 
\begin{align} \label{eq: delta x_m+1 (1)}
    \Delta(x_1 \diamond x_m) &= 1 \otimes(x_1 \diamond x_m) + \sum \limits_{\substack{r \in \N, \fa \in \langle \Sigma \rangle\\ r + w(\fa) = m}} \binom{r + \depth(\fa) - 2}{\depth(\fa)} (x_1 \diamond x_r) \otimes [\fa] \\ \notag
    &= 1 \otimes x_{m+1} + \sum \limits_{i+j = m + 1} \Delta^j_{1,m} 1 \otimes x_jx_j\\ \notag
    &+ \sum \limits_{\substack{r \in \N, \fa \in \langle \Sigma \rangle\\ r + w(\fa) = m}} \binom{r + \depth(\fa) - 2}{\depth(\fa)} x_{r+1} \otimes [\fa]\\ \notag
    &+ \sum \limits_{\substack{r \in \N, \fa \in \langle \Sigma \rangle\\ r + w(\fa) = m}} \binom{r + \depth(\fa) - 2}{\depth(\fa)} \sum \limits_{h+k = r + 1} \Delta^k_{1,r}  x_{h}x_k \otimes [\fa].
\end{align}
On the other hand, we have
\begin{equation} \label{eq: delta x_m+1 (2)}
    \Delta(x_1 \diamond x_m) = \Delta(x_{m+1} + \sum \limits_{i+j = m + 1} \Delta^j_{1,m} x_ix_j) = \Delta(x_{m+1}) + \sum \limits_{i+j = m + 1} \Delta^j_{1,m} \Delta(x_ix_j).
\end{equation}
From the induction hypothesis, it follows that for all $i,j \in \N$ such that $i + j = m+1$,
\begin{align*}
    \Delta(x_i) = 1 \otimes x_i + \sum \limits_{\substack{s \in \N, \fb \in \langle \Sigma \rangle\\ s + w(\fb) = i}} \binom{s + \depth(\fb) - 2}{\depth(\fb)} x_s \otimes [\fb],\\
    \Delta(x_j) = 1 \otimes x_j + \sum \limits_{\substack{t \in \N, \fc \in \langle \Sigma \rangle\\ t + w(\fc) = j}} \binom{t + \depth(\fc) - 2}{\depth(\fc)} x_t \otimes [\fc],
\end{align*}
hence
\begin{align}  \label{eq: delta x_m+1 (3)}
    \Delta(x_ix_j) &= 1 \otimes x_ix_j + \sum \limits_{\substack{s \in \N, \fb \in \langle \Sigma \rangle\\ s + w(\fb) = i}} \binom{s + \depth(\fb) - 2}{\depth(\fb)} x_s \otimes ([\fb] \shuffle x_j)\\ \notag
    &+ \sum \limits_{\substack{s,t \in \N; \fb,\fc \in \langle \Sigma \rangle\\ s + w(\fb) = i \\ t + w(\fc) = j}} \binom{s + \depth(\fb) - 2}{\depth(\fb)}\binom{t + \depth(\fc) - 2}{\depth(\fc)} x_sx_t \otimes ([\fb] \shuffle [\fc]).
\end{align}
From \eqref{eq: delta x_m+1 (1)}, \eqref{eq: delta x_m+1 (2)} and \eqref{eq: delta x_m+1 (3)}, we have
\begin{align*}
    \Delta(x_{m+1}) = \Delta(x_1 \diamond x_m) - \sum \limits_{i+j = m + 1} \Delta^j_{1,m} \Delta(x_ix_j) = 1 \otimes x_{m+1} + S_1 + S_2,
\end{align*}
where
\begin{align*}
    S_1 &= \sum \limits_{\substack{r \in \N, \fa \in \langle \Sigma \rangle\\ r + w(\fa) = m}} \binom{r + \depth(\fa) - 2}{\depth(\fa)} x_{r+1} \otimes [\fa] \\
    &- \sum \limits_{i+j = m + 1} \Delta^j_{1,m} \sum \limits_{\substack{s \in \N, \fb \in \langle \Sigma \rangle\\ s + w(\fb) = i}} \binom{s + \depth(\fb) - 2}{\depth(\fb)} x_s \otimes ([\fb] \shuffle x_j),\\
    S_2 &= \sum \limits_{\substack{r \in \N, \fa \in \langle \Sigma \rangle\\ r + w(\fa) = m}} \binom{r + \depth(\fa) - 2}{\depth(\fa)} \sum \limits_{h+k = r + 1} \Delta^k_{1,r}  x_{h}x_k \otimes [\fa] \\
    &- \sum \limits_{i+j = m + 1} \Delta^j_{1,m} \sum \limits_{\substack{s,t \in \N; \fb,\fc \in \langle \Sigma \rangle\\ s + w(\fb) = i \\ t + w(\fc) = j}} \binom{s + \depth(\fb) - 2}{\depth(\fb)}\binom{t + \depth(\fc) - 2}{\depth(\fc)} x_sx_t \otimes ([\fb] \shuffle [\fc]).
\end{align*}

We next compute the sums $S_1$ and $S_2$ as follows. 
\subsubsection{The sum $S_1$} We first note that if $j <m$ then it follows from Remark \ref{rmk: delta j<n} and Lemma \ref{lem: prod of bracket} that $(-\Delta^j_{1,m})([\fb] \shuffle x_j) = [\fb] \shuffle [x_j] = [\fb x_j]$. Thus we have
\begin{align*}
    S_1 &= \sum \limits_{\substack{r \in \N, \fa \in \langle \Sigma \rangle\\ r + w(\fa) = m}} \binom{r + \depth(\fa) - 2}{\depth(\fa)} x_{r+1} \otimes [\fa] \\
    &- \sum \limits_{i+j = m + 1} \Delta^j_{1,m} \sum \limits_{\substack{s \in \N, \fb \in \langle \Sigma \rangle\\ s + w(\fb) = i}} \binom{s + \depth(\fb) - 2}{\depth(\fb)} x_s \otimes ([\fb] \shuffle x_j)\\
    &= \sum \limits_{r=2}^{m+1} \sum \limits_{\substack{\fa \in \langle \Sigma \rangle\\ r + w(\fa) = m +1}} \binom{r - 1 + \depth(\fa) - 2}{\depth(\fa)} x_{r} \otimes [\fa] \\
    &+ \sum \limits_{s=1}^{m}  \sum \limits_{\substack{j \in \N, \fb \in \langle \Sigma \rangle\\ s+ w(\fb) + j= m+1}} \binom{s + \depth(\fb) - 2}{\depth(\fb)} x_s \otimes (-\Delta^j_{1,m})([\fb] \shuffle x_j)\\
    &= \binom{m-2}{0} x_{m+1} \otimes [1] + \sum \limits_{r=2}^{m} \sum \limits_{\substack{\fa \in \langle \Sigma \rangle\\ r + w(\fa) = m +1}} \binom{r - 1 + \depth(\fa) - 2}{\depth(\fa)} x_{r} \otimes [\fa]\\
    &+ \sum \limits_{s=2}^{m}  \sum \limits_{\substack{j \in \N, \fb \in \langle \Sigma \rangle\\ s+ w(\fb) + j= m+1}} \binom{s + \depth(\fb) - 2}{\depth(\fb)} x_s \otimes (-\Delta^j_{1,m})([\fb] \shuffle x_j) \\
    &+ \sum \limits_{\substack{j \in \N, \fb \in \langle \Sigma \rangle\\ w(\fb) + j= m}} \binom{\depth(\fb) - 1}{\depth(\fb)} x_1 \otimes (-\Delta^j_{1,m})([\fb] \shuffle x_j)\\
    &= x_{m+1} \otimes 1 + \Bigg[\sum \limits_{r=2}^{m} \sum \limits_{\substack{\fa \in \langle \Sigma \rangle\\ r + w(\fa) = m +1}} \binom{r + \depth(\fa) - 3}{\depth(\fa)} x_{r} \otimes [\fa] \\
    &+ \sum \limits_{s=2}^{m}  \sum \limits_{\substack{j \in \N, \fb \in \langle \Sigma \rangle\\ s+ w(\fb) + j= m+1}} \binom{s + \depth(\fb) - 2}{\depth(\fb)} x_s \otimes [\fb x_j]\Bigg]
    \\ 
    &+ \sum \limits_{\substack{j \in \N, \fb \in \langle \Sigma \rangle\\ w(\fb) + j= m}} \binom{\depth(\fb) - 1}{\depth(\fb)} x_1 \otimes (-\Delta^j_{1,m})([\fb] \shuffle x_j)\\
    &= x_{m+1} \otimes 1 + \sum \limits_{r=2}^{m} \sum \limits_{\substack{\fa \in \langle \Sigma \rangle\\ r + w(\fa) = m +1}} \Bigg[\binom{r + \depth(\fa) - 3}{\depth(\fa)} + \binom{r + \depth(\fa) - 3}{\depth(\fa)-1} \Bigg]x_{r} \otimes [\fa] 
    \\ 
    &+ \sum \limits_{\substack{j \in \N, \fb \in \langle \Sigma \rangle\\ w(\fb) + j= m}} \binom{\depth(\fb) - 1}{\depth(\fb)} x_1 \otimes (-\Delta^j_{1,m})([\fb] \shuffle x_j)\\
    &= x_{m+1} \otimes 1 + \sum \limits_{r=2}^{m} \sum \limits_{\substack{\fa \in \langle \Sigma \rangle\\ r + w(\fa) = m +1}} \binom{r + \depth(\fa) - 2}{\depth(\fa)}x_{r} \otimes [\fa]  \\
    &+ \sum \limits_{\substack{j \in \N, \fb \in \langle \Sigma \rangle\\ w(\fb) + j= m}} \binom{\depth(\fb) - 1}{\depth(\fb)} x_1 \otimes (-\Delta^j_{1,m})([\fb] \shuffle x_j).
\end{align*}

We claim that for all $j \in \N$ and for all $\fb \in \langle \Sigma \rangle$ with $j + w(\fb) = m$,
\begin{equation} \label{eq: S1 zero}
     \binom{\depth(\fb) - 1}{\depth(\fb)} x_1 \otimes (-\Delta^j_{1,m})([\fb] \shuffle x_j) = 0.
\end{equation}
Indeed, if $j = m$, then it follows from Lemma \ref{lem: delta j<n} that $\Delta^m_{1,m} =0$, hence \eqref{eq: S1 zero} holds. If $j <m$, then $w(\fb) \geq 1$, i.e., $\depth(\fb) \geq 1$, hence $\binom{\depth(\fb) - 1}{\depth(\fb)} = 0$, showing that \eqref{eq: S1 zero} holds.
This proves the claim. As a consequence, one may identify
\begin{equation*}
    \sum \limits_{\substack{j \in \N, \fb \in \langle \Sigma \rangle\\ w(\fb) + j= m}} \binom{\depth(\fb) - 1}{\depth(\fb)} x_1 \otimes (-\Delta^j_{1,m})([\fb] \shuffle x_j) =  \sum \limits_{\substack{\fa \in \langle \Sigma \rangle\\ w(\fa) = m}} \binom{\depth(\fa) - 1}{\depth(\fa)}x_{1} \otimes [\fa] = 0.
\end{equation*}
Thus
\begin{align*}
    S_1 &= x_{m+1} \otimes [1] + \sum \limits_{r=2}^{m} \sum \limits_{\substack{\fa \in \langle \Sigma \rangle\\ r + w(\fa) = m +1}} \binom{r + \depth(\fa) - 2}{\depth(\fa)}x_{r} \otimes [\fa]  \\
    &+ \sum \limits_{\substack{\fa \in \langle \Sigma \rangle\\ w(\fa) = m}} \binom{\depth(\fa) - 1}{\depth(\fa)}x_{1} \otimes [\fa]\\
    &= \sum \limits_{r=1}^{m+1} \sum \limits_{\substack{\fa \in \langle \Sigma \rangle\\ r + w(\fa) = m +1}} \binom{r + \depth(\fa) - 2}{\depth(\fa)}x_{r} \otimes [\fa]\\
    &= \sum \limits_{\substack{r \in \N, \fa \in \langle \Sigma \rangle\\ r + w(\fa) = m +1}} \binom{r + \depth(\fa) - 2}{\depth(\fa)}x_{r} \otimes [\fa].
\end{align*}

\subsubsection{The sum $S_2$} We claim that $S_2 = 0$. The following lemma will be useful.
\begin{lemma} \label{lem: comb indentity}
For positive integers $h,k,l$, we have;
\begin{equation*}
    \sum \limits_{\substack{i , j \geq 0 \\i  +j = l}} \binom{h + i}{i} \binom{k + j}{j} = \binom{h + k + l + 1}{l}.
\end{equation*}
\end{lemma}
\begin{proof}
For all  $m ,n \in \N$, we have
\begin{equation*}
    \sum \limits_{r = 0}^m \binom{n + r}{r} = \binom{n + m + 1}{m}.
\end{equation*}
We now proceed the proof by induction on $k$. The base step $k = 0$ follows from the above identity. We assume that Lemma \ref{lem: comb indentity} holds for $k \geq 0$. We need to show that Lemma \ref{lem: comb indentity} holds for $k + 1$. From the previous identity and the induction hypothesis, we have
\begin{align*}
    \sum \limits_{\substack{i , j \geq 0 \\i  +j = l}} \binom{h + i}{i} \binom{(k + 1) + j}{j} &= \sum \limits_{\substack{i , j \geq 0 \\i  +j = l}} \binom{h + i}{i} \binom{ k + j + 1}{j}\\
    &=  \sum \limits_{\substack{i , j \geq 0 \\i  +j = l}} \binom{h + i}{i} \sum \limits_{r =0}^j  \binom{k + r}{r} \\
    &= \sum \limits_{s = 0}^l \sum \limits_{\substack{i , r \geq 0 \\i  +r = s}} \binom{h + i}{i}\binom{k + r}{r}\\
    &= \sum \limits_{s = 0}^l \binom{h + k + s + 1}{s}\\
    &= \binom{h + k + l + 2}{l}.
\end{align*}
This proves the lemma.
\end{proof}
It follows from Lemma \ref{lem: prod of bracket} that $[\fb] \shuffle [\fc] = [\fb\fc]$, hence
 \begin{align*}
     S_2&= \sum \limits_{\substack{r \in \N, \fa \in \langle \Sigma \rangle\\ r + w(\fa) = m}} \binom{r + \depth(\fa) - 2}{\depth(\fa)} \sum \limits_{h+k = r + 1} \Delta^k_{1,r}  x_{h}x_k \otimes [\fa] \\
    &- \sum \limits_{i+j = m + 1} \Delta^j_{1,m} \sum \limits_{\substack{s,t \in \N; \fb,\fc \in \langle \Sigma \rangle\\ s + w(\fb) = i \\ t + w(\fc) = j}} \binom{s + \depth(\fb) - 2}{\depth(\fb)}\binom{t + \depth(\fc) - 2}{\depth(\fc)} x_sx_t \otimes [\fb \fc]\\
    &= \sum \limits_{\substack{h,k \in \N, \fa \in \langle \Sigma \rangle\\ h + k + w(\fa) = m + 1}} \binom{h +k - 1 + \depth(\fa) - 2}{\depth(\fa)} \Delta^k_{1,h + k  - 1}  x_{h}x_k \otimes [\fa] \\
    &-  \sum \limits_{\substack{s,t \in \N; \fb,\fc \in \langle \Sigma \rangle\\ s + t + w(\fb) +w(\fc) = m + 1}} \binom{s + \depth(\fb) - 2}{\depth(\fb)}\binom{t + \depth(\fc) - 2}{\depth(\fc)} \Delta^{t + w(\fc)}_{1,m} x_sx_t \otimes [\fb \fc]\\
    &= \sum \limits_{\substack{h,k \in \N, \fa \in \langle \Sigma \rangle\\ h + k + w(\fa) = m + 1}} \Bigg[  \binom{h +k + \depth(\fa) - 3}{\depth(\fa)} \Delta^k_{1,h + k  - 1} \\
    &- \sum \limits_{\substack{ \fb,\fc \in \langle \Sigma \rangle\\ \fb \fc = \fa}} \binom{h + \depth(\fb) - 2}{\depth(\fb)}\binom{k + \depth(\fc) - 2}{\depth(\fc)}\Delta^{k + w(\fc)}_{1,m} \Bigg]  x_{h}x_k \otimes [\fa] 
 \end{align*}
 We will prove that for all $h, k \in \N$ and for all $\fa \in \langle \Sigma \rangle$ with $h + k + w(\fa) = m + 1$,
 \begin{equation} \label{eq: S2 zero}
     \Bigg[  \binom{h +k + \depth(\fa) - 3}{\depth(\fa)} \Delta^k_{1,h + k  - 1} - \sum \limits_{\substack{ \fb,\fc \in \langle \Sigma \rangle\\ \fb \fc = \fa}} \binom{h + \depth(\fb) - 2}{\depth(\fb)}\binom{k + \depth(\fc) - 2}{\depth(\fc)}\Delta^{k + w(\fc)}_{1,m} \Bigg]  x_{h}x_k \otimes [\fa] = 0.
 \end{equation}
If this is the case then $S_2 = 0$. We divide into three cases:\\
 
\noindent \textbf{Case 1: $ h = 1$}.
The LHS of \eqref{eq: S2 zero} equals
\begin{equation*}
    \Bigg[  \binom{k + \depth(\fa) - 2}{\depth(\fa)} \Delta^k_{1,k} - \sum \limits_{\substack{ \fb,\fc \in \langle \Sigma \rangle\\ \fb \fc = \fa}} \binom{\depth(\fb) - 1}{\depth(\fb)}\binom{k + \depth(\fc) - 2}{\depth(\fc)}\Delta^{k + w(\fc)}_{1,m} \Bigg]  x_{1}x_k \otimes [\fa].
\end{equation*}
It follows from Lemma \ref{lem: delta j<n} that $\Delta^k_{1,k} =0$. Moreover, if $\fb \ne 1$, then $\depth(\fb) \geq 1$, hence $\binom{\depth(\fb) - 1}{\depth(\fb)} = 0$. If $\fb = 1$, then $\fc = \fa$ hence $k + w(\fc) = k + w(\fa) = m$, showing that $\Delta^{k + w(\fc)}_{1,m} = \Delta^m_{1,m} = 0$. So \eqref{eq: S2 zero} holds in this case.\\

\noindent\textbf{Case 2: $ h \geq 2, k = 1$}. The LHS of \eqref{eq: S2 zero} equals
\begin{equation*}
    \Bigg[  \binom{h + \depth(\fa) - 2}{\depth(\fa)} \Delta^1_{1,h} - \sum \limits_{\substack{ \fb,\fc \in \langle \Sigma \rangle\\ \fb \fc = \fa}} \binom{h +\depth(\fb) - 2}{\depth(\fb)}\binom{\depth(\fc) - 1}{\depth(\fc)}\Delta^{1 + w(\fc)}_{1,m} \Bigg]  x_{h}x_1 \otimes [\fa].
\end{equation*}
Since $ 1 < h \leq m$, it follows from Remark \ref{rmk: delta j<n} that $\Delta^1_{1,h} = \Delta^1_{1,m}$. Moreover, if $\fc \ne 1$, then $\depth(\fc) \geq 1$, hence $\binom{\depth(\fc) - 1}{\depth(\fc)} = 0$. Then the LHS of \eqref{eq: S2 zero} equals
\begin{equation*}
    \Bigg[  \binom{h + \depth(\fa) - 2}{\depth(\fa)} \Delta^1_{1,m} - \binom{h +\depth(\fa) - 2}{\depth(\fa)}\binom{- 1}{0}\Delta^{1}_{1,m} \Bigg]  x_{h}x_1 \otimes [\fa] = 0.
\end{equation*}
So \eqref{eq: S2 zero} holds in this case.\\

\noindent\textbf{Case 3: $ h \geq 2, k \geq 2$} Since $k < h + k - 1 \leq m$, it follows from Remark \ref{rmk: delta j<n} that $\Delta^k_{1,h+k - 1} = \Delta^k_{1,m}$. Moreover, we claim that 
\begin{equation} \label{eq: S2 case 3}
    \Delta^{k+w(\fc)}_{1,m} x_hx_k \otimes [\fa] = \Delta^{k}_{1,m} x_hx_k \otimes [\fa].
\end{equation}
Indeed, note that $k + w(\fc) < m$. If $(q-1) \mid w(\fc)$, then it follows from Lemma $\ref{lem: delta j<n}$ that $\Delta^{k+w(\fc)}_{1,m} = \Delta^{k}_{1,m}$, hence \eqref{eq: S2 case 3} holds. If $(q-1) \nmid w(\fc)$, then it follows from Lemma \ref{lem: bracket zero} that $[\fc] = 0$, hence $[\fa] = 0$, showing that \eqref{eq: S2 case 3} holds. Thus the LHS of \eqref{eq: S2 zero} becomes
\begin{equation*}
    \Bigg[  \binom{h +k + \depth(\fa) - 3}{\depth(\fa)}  - \sum \limits_{\substack{ \fb,\fc \in \langle \Sigma \rangle\\ \fb \fc = \fa}} \binom{h + \depth(\fb) - 2}{\depth(\fb)}\binom{k + \depth(\fc) - 2}{\depth(\fc)} \Bigg] \Delta^{k}_{1,m} x_{h}x_k \otimes [\fa].
\end{equation*}
It follows from Lemma \ref{lem: comb indentity} that 
\begin{equation*}
    \binom{h +k + \depth(\fa) - 3}{\depth(\fa)}  = \sum \limits_{\substack{ \fb,\fc \in \langle \Sigma \rangle\\ \fb \fc = \fa}} \binom{h + \depth(\fb) - 2}{\depth(\fb)}\binom{k + \depth(\fc) - 2}{\depth(\fc)},
\end{equation*}
hence \eqref{eq: S2 zero} holds in this case.\\

From the above computations, we conclude that 
\begin{equation*}
    \Delta(x_{m+1}) = 1 \otimes x_{m+1} + S_1 + S_2 = 1 \otimes x_{m+1} + \sum \limits_{\substack{r \in \N, \fa \in \langle \Sigma \rangle\\ r + w(\fa) = m + 1}} \binom{r + \depth(\fa) - 2}{\depth(\fa)} x_r \otimes [\fa].
\end{equation*}
This proves Proposition \ref{prop: formula delta xn}.


\subsection{Comparison with Shi's coproduct} ${}$\par

In \cite{Shi18}, Shi defined another coproduct 
	\[ \Delta_1: \frak C \to \frak C \otimes \frak C. \]
using the concatenation rather than $\triangleright$ on recursive steps for words with depth $>1$. More precisely, we define it on  $\langle \Sigma \rangle$ by induction on weight and extend by $\Fq$-linearity to $\frak C$. First, we set
\begin{align*}
\Delta_1(1)&:=1 \otimes 1, \\
\Delta_1(x_1)&:=1 \otimes x_1 + x_1 \otimes 1.
\end{align*}
Let $w \in \bN$ and we suppose that we have defined $\Delta_1(\fv)$ for all words $\fv$ of weight $w(\fv)<w$. We now give a formula for $\Delta_1(\fu)$ for all words $\fu$ with $w(\fu)=w$. For such a word $\fu$ with $\depth(\fu)>1$, we put $\fu=x_u \fv$ with $w(\fv)<w$. Since $x_u$ and $\fv$ are both of weight less than $w$, we have already defined
\begin{align*}
\Delta_1(x_u)&:=1 \otimes x_u + \sum a_u \otimes b_u, \\
\Delta_1(\fv)&:= \sum a_\fv \otimes b_\fv.
\end{align*}
Then we set
\begin{align*}
\Delta_1(\fu):=1 \otimes \fu + \sum (a_u a_\fv) \otimes (b_u \shuffle  b_\fv).
\end{align*}
Our last task is to define $\Delta_1(x_w)$. We know that
	\[ x_1\shuffle x_{w-1}=x_w+x_1x_{w-1}+x_{w-1}x_1+\sum_{0<j<w} \Delta^j_{1,w-1} x_{w-j} x_j \]
where all the words $x_{w-j} x_j$ have weight $w$ and depth $2$ and all $\Delta^j_{1,w-1}$ belong to $\Fq$. Therefore, we set
\begin{equation*} 
\Delta_1(x_w):=\Delta_1(x_1) \shuffle  \Delta_1(x_{w-1})-\Delta_1(x_1x_{w-1})-\Delta_1(x_{w-1}x_1)-\sum_{0<j<w} \Delta^j_{1,w-1} \Delta_1(x_{w-j} x_j).
\end{equation*}

As an application of Proposition \ref{prop: formula delta xn}, we prove:
\begin{proposition} \label{prop: comparison with Shi's coproduct}
For all words $\fu \in \langle \Sigma \rangle$, we have
	\[ \Delta(\fu)=\Delta_1(\fu). \]
\end{proposition}

\begin{proof}
The proof is by induction on the weight $w$. We have to show that for all words $\fu$ of weight $w$, 
	\[ \Delta(\fu)=\Delta_1(\fu). \]
We denote this claim by $H_w$.

For $w=0$ and $w=1$, we are done as $\Delta(1)=\Delta_1(1)=1 \otimes 1$ and $\Delta(x_1)=\Delta_1(x_1)=1 \otimes x_1 + x_1 \otimes 1$. Suppose that for all words $\fu$ with $w(\fu)<w$, we have $\Delta(\fu)=\Delta_1(\fu)$. We will show that the claim $H_w$ holds.

Let $\fu$ be a word of weight $w$. Suppose that the depth of $\fu$ is at least 2. Then we put $\fu=x_u \fv$ with $w(\fv)<w$. By induction, we know that $\Delta(\fv)=\Delta_1(\fv)$. If $\Delta(x_u) = \sum \mathfrak{a}_u \otimes \mathfrak{b}_u$, then Proposition \ref{prop: formula delta xn} implies that $\depth(\mathfrak{a_n})\le1$. Thus $\Delta(\fu)=\Delta_1(\fu)$.

To conclude, we have to check the claim for $u=x_w$. By induction, for all $i<w$, $\Delta(x_i)=\Delta_1(x_i)$. It follows that
\begin{align*}
& \Delta(x_w) \\
&=\Delta(x_1) \shuffle  \Delta(x_{w-1})-\Delta(x_1x_{w-1})-\Delta(x_{w-1}x_1)-\sum_{0<j<w} \Delta^j_{1,w-1} \Delta(x_{w-j} x_j) \\
&=\Delta_1(x_1) \shuffle  \Delta_1(x_{w-1})-\Delta_1(x_1x_{w-1})-\Delta_1(x_{w-1}x_1)-\sum_{0<j<w} \Delta^j_{1,w-1} \Delta_1(x_{w-j} x_j) \\
&=\Delta_1(x_w).
\end{align*}
Thus we have proved $H_w$. The proposition follows.
\end{proof}

In particular, we get
\begin{theorem} \label{thm: comparison with Shi's coproduct}
Conjecture 3.2.11 in \cite{Shi18} holds.
\end{theorem}

\subsection{Auxiliary results} ${}$\par

In this section we prove some auxiliary results that will be useful in the sequel.

We define 
\begin{align*}
x_{s_1}x_{s_2}\dots x_{s_r} \star q &:= x_{qs_1} x_{qs_2}\dots x_{q s_r}, \\
(\mathfrak {u}\otimes \mathfrak {v}) \star q &:= (\mathfrak {u}\star q) \otimes (\mathfrak {v}\star q)
\end{align*}
for $\mathfrak {u}, \mathfrak {v}\in \langle \Sigma \rangle$, and extend it to be $\F_p$-linear. We define $1 \star q = 1$ when $1$ is the empty word.

We recall the Lucas's theorem \cite{Gra97}: \[\binom{a_0+ a_1 p +a_2 p^2+ \dots +a_k p^k}{b_0 + b_1 p + b_2 p^2+ \dots +b_k p^k} \equiv \binom{a_0}{b_0} \binom{a_1}{b_1} \binom{a_2}{b_2}\dots \binom{a_k}{b_k} \pmod{p}, \] where $p$ is a prime and $0\le a_i, b_i<p$.
\begin{lemma}\label{lemma: modular properties on chen delta} The following modular equations hold.
    \begin{enumerate}[$(1)$]
        \item $\binom{i-1}{a-1}\equiv 0 \pmod{p}$ when $q\nmid i$, $q\mid a$. $\Delta_{a, b}^i \equiv 0\pmod p$ follows when $q\mid a, b$.
        \item $\binom{pj-1}{pa-1} \equiv \binom{j-1}{a-1} \pmod{p}$ and $\Delta_{pa, pb}^{pj} \equiv \Delta_{a, b}^j \pmod{p}$ for $0<j<a+b$.
    \end{enumerate}
\end{lemma}
\begin{proof} Let $q = p^k$. 
    
    To prove (1), let $i = Aq +B$ with $A, B \in \mathbb{N}$, $0<B<q$. We write \[i-1 = Aq + (B-1) =(\beta_0 + \beta_1 p + \dots + \beta_{k-1}p^{k-1}) +  \sum_{s=k}^{r}\alpha_s p^s,\]  and \[ a-1 = \left( (p-1)+ (p-1)p + \dots + (p-1)p^{k-1}\right) + \sum_{s=k}^{r}\alpha'_s p^s\]with $0 \le \alpha_s, \alpha'_s,\beta_s<p$. By Lucas's theorem,
    \begin{align*}
        \binom{i-1}{a-1} &\equiv \binom{\beta_0}{p-1}\binom{\beta_1}{p-1}\cdots \binom{\beta_{k-1}}{p-1} \cdot \binom{\alpha_1}{\alpha'_1}\binom{\alpha_2}{\alpha'_2}\cdots \equiv 0 \pmod{p},
    \end{align*}
    since at least one of $\beta_s < p-1$. (Note that $B<q$.)

    For (2), $\binom{pj-1}{pa-1} \equiv \binom{j-1}{a-1} \pmod{p}$ is verified by similar routine calculations with Lucas's theorem.
\end{proof}
\begin{lemma} \label{lemma: star product raised by q}
For all words $\frak u, \frak v \in \frak C$, we have
$(\mathfrak {u}\shuffle  \mathfrak{v})\star q = (\mathfrak{u}\star q)\shuffle  (\mathfrak{v} \star q)$. Equivalently,
if $\mathfrak{u} \shuffle  \mathfrak{v} = \sum \mathfrak{a}_i$, then $(\mathfrak{u}\star q) \shuffle  (\mathfrak{v}\star q) = \sum (\mathfrak{a}_i\star q)$.
\end{lemma}
\begin{proof}
    Let $\mathfrak{u} = x_u$, $\mathfrak{v} = x_v$ be depth one words. We have \begin{align*}
        x_u\shuffle x_v &= x_{u+v}+ x_u+x_v + \sum_{0<i<u+v}\Delta^{i}_{u, v}x_{u+v-i}x_i, \\
        x_{qu}\shuffle x_{qv} &= x_{q(u+v)}+ x_{qu}+x_{qv} + \sum_{0<i<q(u+v)}\Delta^{i}_{qu, qv}
        x_{q(u+v)-i}x_i.
    \end{align*}
    By Lemma~\ref{lemma: modular properties on chen delta}, 
    \begin{align*}
        \sum_{0<i<q(u+v)}\Delta^{i}_{qu, qv}
        x_{q(u+v)-i}x_i & = \sum_{q\mid i, \ 0<i<q(u+v)}\Delta^{i}_{qu, qv}x_{q(u+v)-i}x_i\\
        & = 
        \sum_{0<j<u+v}\Delta^{qj}_{qu, qv}x_{q(u+v-j)}x_{qj}\\
        &=\sum_{0<j<u+v}\Delta^{qj}_{u, v}x_{q(u+v-j)}x_{qj}.
    \end{align*}
    Thus the lemma is proved for depth 1 words $\mathfrak{u}, \mathfrak{v}$. By the induction on $\depth(\mathfrak u) + \depth(v)$, we can conclude the general result. 
\end{proof}

We recall that
$\Delta(1) = 1\otimes 1$, 
$\Delta(x_1) = 1\otimes x_1 + x_1 \otimes 1$ by definition, and 
$\Delta(\fu) \in \frak C \otimes \frak C$ is defined for other all words $\fu$. 

\begin{proposition} \label{proposition: Hopf deltas depth 1 <= q} 
For all integers $1\le n \le q$, 
\begin{equation} \label{eq: Hopf deltas depth 1 <= q} 
    \Delta(x_n) = 1\otimes x_n  + x_n \otimes 1.
\end{equation}
\end{proposition}

\begin{proof}
    \eqref{eq: Hopf deltas depth 1 <= q} holds when $n=1$. Let $2\le n \le q$, and assume \eqref{eq: Hopf deltas depth 1 <= q} holds for $1, 2,\dots, n-1$. Then, 
    \begin{align*}
        x_1\shuffle x_{n-1} =x_1x_{n-1} +x_{n-1}x_{1} + x_n + \sum_{\substack{q-1 \mid j \\ 0<j< n}}\Delta^j_{1,n-1}x_{n-j}x_{j}. 
    \end{align*}
    The $\sum$ term here is empty sum for $n<q$, and when $n=q$, $j=q-1$ term is the only possible index. $\Delta_{1, q-1}^{q-1} = \binom{q-2}{q-2} + (-1)^q \binom{q-2}{q-2}$ vanishes for all $q$, so 
    \begin{align*}
        x_1\shuffle  x_{n-1} =x_{1}x_{n-1} +x_{n-1}x_{1} + x_n,
    \end{align*}
    that is, $\Delta(x_1)\Delta(x_{n-1}) = \Delta(x_{1}x_{n-1}) + \Delta(x_{n-1}x_1) + \Delta(x_n)$. 

    By calculation and the induction hypothesis, we have 
    \begin{align*}
        \Delta(x_1)\Delta(x_{n-1}) &= (1\otimes x_1 + x_1\otimes 1)(1\otimes x_{n-1} + x_{n-1}\otimes 1)\\ 
        & =1\otimes \left(x_1\shuffle x_{n-1}\right) + x_{n-1}\otimes x_1 + x_1\otimes x_{n-1} + \left( x_1\shuffle x_{n-1}\right) \otimes 1,\\
        \Delta(x_1x_{n-1}) & = e\otimes x_1x_{n-1} + x_1\otimes x_{n-1} + x_1 x_{n-1}\otimes 1, \\
        \Delta(x_{n-1}x_1) & = 1\otimes x_{n-1}x_1 + x_{n-1}\otimes x_1 + x_{n-1}x_1\otimes 1
    \end{align*}
    Recall $x_1\shuffle  x_{n-1} =x_{1}x_{n-1} +x_{n-1}x_{1} + x_n$ to have
    \begin{align*}
        \Delta(x_n) = 1\otimes x_n + x_n\otimes 1.
    \end{align*}
\end{proof}

\begin{proposition}\label{proposition: Hopf coproduct raised by q}
    For all $n \in \bN$, we have
   		\[ \Delta (x_{qn}) = \Delta(x_n)\star q \] for all $n\in \mathbb{N}$. This can be stated as follows:
    letting $\Delta(x_n) = e\otimes x_n + \sum a_n \otimes b_n$ we get $\Delta(x_{qn}) = e\otimes x_{qn} + \sum_i (a_n\star q) \otimes (b_n \star q)$.
\end{proposition}

\begin{proof}
The statement holds for $n=1$ by Proposition~\ref{proposition: Hopf deltas depth 1 <= q}. 

    Assume that the statement holds for all $n\le w$. Say $\Delta(x_w) = 1\otimes x_w + \sum a_w \otimes b_w$. Since 
    \begin{align*}
        x_w\shuffle  x_1 &= x_{w+1} + x_w x_1 + x_1x_w + \sum_{0 < j < w+1}\Delta_{1, w}^j x_{w+1-j}x_{j}, \\
        \Delta(x_w x_1)& = 1\otimes x_w x_1 + \sum (a_{w}\triangleright x_1) \otimes b_w+ \sum a_{w}\otimes (b_{w}\shuffle  x_1), \\ 
        \Delta(x_1 x_w)& = 1\otimes x_1 x_w + x_1 \otimes x_w + \sum (x_1 \triangleright a_{w}) \otimes b_w, 
    \end{align*} 
    we have 
    \begin{align*}
        & \Delta(x_{w+1}) \\
        & = 1\otimes (x_1 \shuffle  x_w) + x_1 \otimes x_w + \sum a_w \otimes (b_w \shuffle  x_1) + \sum (a_w \shuffle  x_1)\otimes b_w \\
        & - 1 \otimes x_w x_1 - \sum (a_w \triangleright x_1)\otimes b_w - \sum a_{w}\otimes (b_{w}\shuffle  x_1)\\ 
        & - 1\otimes x_1 x_w - x_1 \otimes x_w - \sum (x_1 \triangleright a_{w}) \otimes b_w\\ 
        & -  \sum_{0 < j < w+1}\Delta_{1, w}^j \Delta(x_{w+1-j}x_{j})
    \end{align*}

    By the induction hypothesis, $\Delta (x_{qw}) =  1\otimes x_{qw} + \sum (a_w\star q) \otimes (b_w\star q)$. Note that $\Delta (x_q) = 1 \otimes x_q + x_q \otimes 1$ by Proposition~\ref{proposition: Hopf deltas depth 1 <= q} and since we proved the compatibility for $\Delta$,  
    \begin{align*}
        x_{qw}\shuffle  x_q &= x_{q(w+1)} + x_{qw} x_q + x_q x_{qw} + \sum_{0 < j < q(w+1)}\Delta_{q, qw}^j x_{qw+q-j}x_{j}, \\
        \Delta(x_{qw} x_q)& = 1\otimes x_{qw} x_q + \sum ((a_{w}\star q)\triangleright x_q) \otimes (b_w\star q)+ \sum (a_{w}\star q)\otimes ((b_{w}\star q)\shuffle  x_q), \\ 
        \Delta(x_q x_{qw})& = 1\otimes x_q x_{qw} + x_q \otimes x_{qw} + \sum (x_q \triangleright (a_{w}\star q)) \otimes (b_w\star q), 
    \end{align*} 
    we have
    \begin{align*}
        & \Delta(x_{qw+q}) \\
        & = 1\otimes (x_q \shuffle  x_{qw}) + x_q \otimes x_{qw} + \sum (a_w\star q)\otimes ((b_w\star q) \shuffle  x_q) + \sum ((a_w\star q) \shuffle  x_q)\otimes (b_w\star q) \\
        & -1\otimes x_{qw} x_q - \sum ((a_{w}\star q)\triangleright x_q) \otimes (b_w\star q)- \sum (a_{w}\star q)\otimes ((b_{w}\star q)\shuffle  x_q) \\ 
        & - 1\otimes x_q x_{qw} - x_q \otimes x_{qw} - \sum (x_q \triangleright (a_{w}\star q)) \otimes (b_w\star q) \\ 
        & - \sum_{0 < j < q(w+1)}\Delta_{q, qw}^j \Delta(x_{qw+q-j}x_{j}).
    \end{align*}
    With Lemma~\ref{lemma: star product raised by q}, we are done if $\Delta((x_{w+1-i}x_i)\star q) = \Delta(x_{w+1-i}x_i)\star q$ for $i \le w$, which can be shown by applying the induction hypothesis. 
\end{proof}

\subsection{Explicit formula for $\Delta(x_n)$ with $1\le n \le q^2$} ${}$\par \label{sec: explicit formula for small weights}
 
In this section we give an explicit formula for $\Delta(x_n)$ with $1\le n \le q^2$. It could be obtained as an application of Proposition \ref{prop: formula delta xn}. We will give below another way to do calculations.

For $1 \leq n \leq q$, an explicit formula for $\Delta(x_n)$ is given in Proposition \ref{proposition: Hopf deltas depth 1 <= q}:
	\[ \Delta(x_n) = 1\otimes x_n  + x_n \otimes 1. \]
We note that a direct consequence of Propositions \ref{proposition: Hopf deltas depth 1 <= q} and  \ref{proposition: Hopf coproduct raised by q} implies 	
	\[ \Delta(x_{aq^r}) = 1\otimes x_{a q^r} + x_{a q^r}\otimes 1 \] 
for all $1\le a < q$. By Theorem \ref{thm: compatibility}, we have an algorithm to calculate $\Delta(x_{n})$:
\begin{itemize}
\item Write $n = n_0 + n_1q + \dots + n_r q^r$, with $0\le n_i < q$. Then $\Delta(x_{n_i q^i}) = 1\otimes x_{n_i q^i} + x_{n_i q^i} \otimes 1$ by Proposition~\ref{proposition: Hopf deltas depth 1 <= q}.

\item Use compatibility result to calculate $\Delta(x_n)$, by calculating $\Delta(x_{n_0})\shuffle \Delta(x_{n_1q})\shuffle  \dots \shuffle  \Delta(x_{n_r q^r}) = \Delta(x_{n_0}\shuffle  x_{n_1 q}\shuffle  \dots \shuffle  x_{n_r q^r})$.
\end{itemize}

\begin{lemma}\label{lemma: Hopf delta for (q-1) multiples}
    For $2\le k \le q$, we have 
    \[\Delta(x_{k(q-1)}) = 1 \otimes x_{k(q-1)} + x_{k(q-1)} \otimes 1 + \sum_{i=1}^{k-1} \binom{k}{i} x_{i(q-1)} \otimes x_{(k-i)(q-1)}.\]
\end{lemma}

\begin{proof}

    The lemma follows from Lemma \ref{lemma: star product of q-1 multiples} and the compatibility equation (see Theorem \ref{thm: compatibility}) \[\Delta(x_{(k-1)(q-1)}\shuffle  x_{q-1}) = \Delta(x_{(k-1)(q-1)})\shuffle \Delta(x_{q-1}).\] 
\end{proof}

\begin{lemma} \label{lemma: star product of q-1 multiples}
For all $a,b \in \bN$ with $a+b\le q$, we have 
\begin{equation}x_{a(q-1)}\shuffle  x_{b(q-1)} = \underbrace{x_{q-1}\shuffle  \dots \shuffle  x_{q-1}}_{a+b} = x_{(a+b)(q-1)}.
\label{eqn: star product of q-1 multiples}
\end{equation}
\end{lemma}

\begin{proof}
It suffices to prove the following claim: for $2\le k \le q$, we have $x_{(k-1)(q-1)}\shuffle  x_{q-1} = x_{k(q-1)}.$ In fact, with Lucas's theorem, one can verify that $\Delta_{q-1,k(q-1)}^{k(q-1)} =-2$ when $k=1$, and $\Delta_{q-1,k(q-1)}^{q-1} \equiv \Delta_{q-1, k(q-1)}^{k(q-1)}\equiv  -1 \pmod{p}$ and  $\Delta_{q-1,k(q-1)}^{i(q-1)} \equiv  0\pmod p$ when $k>1$. Thus the claim follows from the definition of $\shuffle $. 
\end{proof}

\begin{remark} \label{remark: q^2-1 Hopf delta}
    Note that the Lemma~\ref{lemma: Hopf delta for (q-1) multiples} does not answer for $\Delta(x_{q^2-1})$. By similar calculation for $x_{q(q-1)}\shuffle x_{q-1}$, all lines except $\Delta_{q-1, q(q-1)}^{q(q-1)} =0$ are parallel, which yields
    \[ x_{q(q-1)} \shuffle  x_{q-1} = x_{q^2-1} + x_{q-1}x_{q(q-1)}.\] 
    Similar calculation gives \[ \Delta(x_{q^2-1}) = 1 \otimes x_{q^2-1} + x_{q^2-1}\otimes 1 + x_{q(q-1)}\otimes x_{q-1}.\]
\end{remark}

\begin{lemma}\label{lemma: Aq star B(q-1)} 
    Let $1\le A, B \le q-1$ with $A + B \le q$. Then $$x_{Aq}\shuffle  x_{B(q-1)} = x_{Aq + B(q-1)}+ x_{B(q-1)}x_{Aq}.$$
\end{lemma}
\begin{proof}
    Let $ n = Aq + B(q-1) = (A+B)(q-1) +A$. Since $1< A \le q-1$, 
    \begin{align*}x_{Aq}\shuffle x_{B(q-1)} & = x_{n} + x_{Aq}x_{B(q-1)} + x_{B(q-1)} x_{Aq}+  S_1 + S_2
    \end{align*}
    where 
    \begin{align*}
        S_1 & = 
        (-1)^{A-1}\sum_{i = 1}^{A+B} \binom{i(q-1)-1 }{Aq-1} x_{n - i(q-1)} x_{i(q-1)}, \\ 
        S_2 & = 
        -\sum_{i = 1}^{A+B} \binom{i(q-1)-1}{B(q-1)-1} x_{n - i(q-1)} x_{i(q-1)}. 
    \end{align*}
    From Lemma~\ref{lemma: modular properties on chen delta}, $\binom{i(q-1)}{Aq}$ vanishes only when $q \mid i$, which is possible only when $A+B = q$. In this case, the only nonzero summand is 
    \begin{align*}
        (-1)^{A-1}\binom{q(q-1)-1}{Aq-1} \equiv (-1)^{A-1}\binom{q-2}{A-1} \equiv A \pmod{p},
    \end{align*}
    thus 
    $
        S_1  = \begin{cases}
            0, & (A+B<q)\\ 
            A x_{n-q(q-1)}x_{q(q-1)}. & (A+B=q)
        \end{cases}
    $

    For $S_2$, when $i <q$, 
    \begin{align*}
        \binom{(i-1)q + q - i - 1}{(B-1)q+(q-B-1)} \equiv \binom{i-1}{B-1}\binom{q-i-1}{q-B-1}, \pmod p
    \end{align*}
    is nonzero only when $i = B$. Thus, if $A+B \le q-1$, $S_2 = - x_{Aq}x_{B(q-1)}$. When $A+B=q$, the coefficient for the summand corresponding to $i=q$ is 
    \begin{align*}
        \binom{(q-2)q+(q-1)}{(B-1)q + (q-B-1)}\equiv \binom{q-2}{B-1}\binom{q-1}{q-B-1}\equiv -B\pmod{p}, 
    \end{align*}
    so 
    $
        S_2 = \begin{cases}
            -x_{Aq}x_{B(q-1)}, & (A+B<q)\\
            -x_{Aq}x_{B(q-1)} + B x_{n-q(q-1)}x_{q(q-1)}, & (A+B =q)
        \end{cases}
    $
    which yields \[ S_1 + S_2 = -x_{Aq}x_{B(q-1)}\] in any case. Thus we have \[ x_{Aq}\shuffle x_{B(q-1)}  = x_{Aq + B(q-1)} + x_{B(q-1)}x_{Aq}\] as we desired.
\end{proof}
\begin{proposition}\label{proposition: Hopf deltas depth 1 <= q^2}  When $q+1 \le n \le q^2$, with some $1\le k < q$ such that $kq+1 \le n \le (k+1)q$, we have;
        \begin{align*}
            \Delta(x_n) & = 1\otimes x_n + x_n \otimes 1  -n x_{n-q+1}\otimes x_{q-1}  + \frac{n(n+1)}{2} x_{n-2q+2}\otimes x_{2q-2}  - \dots \\
            & + (-1)^k \binom{n+k-1}{k} x_{n-kq+k}\otimes x_{kq-k} \\ 
            & = 1\otimes x_n + x_n \otimes 1 + 
            \sum_{i =1}^k (-1)^i \binom{n-1 + i}{i} x_{n- i(q-1)}\otimes x_{i(q-1)}.
        \end{align*}
\end{proposition} 
\begin{proof}Note that $q\mid n$ and $(q-1)\mid n$ cases are already treated above. Note that  Lemma~\ref{lemma: Hopf delta for (q-1) multiples} coincides with the statement as follows: when $n = r(q-1)$ with $2\le r \le q$, then 
    \begin{align*}
        & (-1)^i\binom{rq-r-1+i}{i} = (-1)^i\binom{(r-1)q+(q-r-1+i)}{i}\\ 
        & \equiv (-1)^i \binom{r-1}{0}\binom{q-r-1+i}{i}  = \frac{(-1)^i}{i!}\cdot \frac{(q-(r-i+1))(q-(r-i+2))\dots 1}{(q-(r+1))(q-(r+2))\dots 1}\\
        & \equiv \frac{(-1)^i}{i!}\cdot (-1)^{q-(r-i+1) + (q-r-1)}\frac{((r-i+1)-q)((r-i+2)-q)\dots ((q-1)-q)}{((r+1)-q)((r+2)-q)\dots ((q-1)-q)}\\
        & = \frac{r!}{i!(r-i)!} = \binom{r}{i}\pmod{p},
    \end{align*}
    by Lucas's theorem and calculations. This kind of calculation will be used often.

    We show the $k=1$ case for the initial step. The statement for $k=1$ is for $q+1\le n \le 2q$, but it also holds for $n=q$, since $-q x_1x_{q-1}$ term vanishes over $\mathbb{F}_p$. Based on this initial case, we let $q+1\le n < 2q$, and assume that the statement holds for all $q, q+1, \dots, n-1$. Then,
     \begin{align*} x_1 \shuffle  x_{n-1} & = x_1x_{n-1}+ x_{n-1}x_1+ x_n + \sum_{\substack{q-1 \mid j \\ j< n}}\Delta^j_{1,n-1}x_{n-j}x_j.
    \end{align*}
    Note that $j=q-1$ yields a summand. There is another summand for $j=2q-2$ only when $n=2q-1$.
    
    Since $q+1 \le n$, $\Delta_{1, n-1}^{q-1} = 1$. Also since $\Delta_{1, (2q-1)-1}^{2q-2} = 1 + (-1)^{2q-1}=0$,
    we have
    \begin{align*} & x_1 \shuffle  x_{n-1} = 
           x_1 x_{n-1}+x_{n-1}x_1+ x_{n} + x_{n-q+1}x_{q-1},
   \end{align*}
thus
    \begin{align*}
        \Delta(x_1)\shuffle \Delta(x_{n-1}) = \Delta(x_1 x_{n-1}) + \Delta(x_{n-1} x_1) + \Delta(x_{n-q+1}x_{q-1}) + \Delta(x_n), 
    \end{align*} and a routine calculation with the induction hypothesis yields the result. 

    We divide the remaining cases into two subcases: when $n=kq+b = k(q-1)+(k+b)$ with $1\le k, b \le q-1$, (1) $2\le k+b < q-1$ and (2) $q\le k+b$.

    Define a temporal function $\Delta'$ to be $\Delta' (\mathfrak{u}) = \Delta(\mathfrak{u}) - 1 \otimes \mathfrak{u} - \mathfrak{u} \otimes 1$. 

    (1) Assume that the statement holds for all $kq + b$ with $k+b <q-1$ and $k < \kappa$, for some $\kappa\ge2$. Let $b\in \mathbb{N}$ with $ \kappa + b < q-1$, and let $n = \kappa q + b$. We have
    \begin{align*}
        x_{\kappa q} \shuffle  x_b & = x_{\kappa q + b} + x_{\kappa q}x_b + x_bx_{\kappa q} + \sum_{i=1}^{\kappa } \Delta_{\kappa q, b}^{i(q-1)} x_{n - i(q-1)}x_{i(q-1)}, \\ 
        \Delta (x_{\kappa q}) & = 1\otimes x_{\kappa q} + x_{\kappa q} \otimes 1,\\ 
        \Delta( x_b) & = 1 \otimes x_b + x_b \otimes 1 ,\\ 
        \Delta(x_{\kappa q} x_b) &= 1 \otimes x_{\kappa q} x_b + x_{\kappa q} \otimes x_b + x_{\kappa q} x_b \otimes 1 ,\\ 
        \Delta (x_b x_{\kappa q})& = 1 \otimes x_b x_{\kappa q} + x_b \otimes x_{\kappa q} + x_b x_{\kappa q} \otimes 1, 
    \end{align*}
    so 
    \begin{align*}
        & \Delta(x_{\kappa q} \shuffle  x_b) \\
        &= \Delta (x_{\kappa q+b}) +  1 \otimes x_{\kappa q} x_b + x_{\kappa q} \otimes x_b + x_{\kappa q} x_b \otimes 1  + 1 \otimes x_b x_{\kappa q} + x_b \otimes x_{\kappa q} + x_b x_{\kappa q} \otimes 1 \\ 
        & + \sum_{i=1}^{\kappa } \Delta_{\kappa q, b}^{i(q-1)} \Delta (x_{n - i(q-1)}x_{i(q-1)})\\
        & = x_{\kappa q}\otimes x_b + x_b \otimes x_{\kappa q} \\ 
        & + 1\otimes x_{\kappa q + b} + 1\otimes x_{\kappa q}x_b + 1\otimes x_bx_{\kappa q} + \sum_{i=1}^{\kappa } \Delta_{\kappa q, b}^{i(q-1)} 1\otimes x_{n - i(q-1)}x_{i(q-1)} \\ 
        & + x_{\kappa q + b}\otimes 1 + x_{\kappa q}x_b\otimes 1 + x_bx_{\kappa q}\otimes 1 + \sum_{i=1}^{\kappa } \Delta_{\kappa q, b}^{i(q-1)} x_{n - i(q-1)}x_{i(q-1)}\otimes 1, 
    \end{align*}
    thus
    \begin{align*}
        \Delta (x_{\kappa q+b})  = 1\otimes x_{\kappa q + b} + x_{\kappa q + b}\otimes 1 - \sum_{i=1}^{\kappa } \Delta_{\kappa q, b}^{i(q-1)} \Delta' (x_{n - i(q-1)}x_{i(q-1)}),
    \end{align*}
    for $1\le i \le \kappa$, $i(q-1)<\kappa q$. Also since $b<q$, Lucas's theorem and some calculations yield $\binom{iq-i-1}{b-1} \equiv  \binom{q-1-i}{b-1} \equiv (-1)^{b-1} \binom{b+i-1}{b-1}
     \pmod p$, so 
    \begin{align}
        \Delta (x_{\kappa q+b})  = 1\otimes x_{\kappa q + b} + x_{\kappa q + b}\otimes 1 
        -\sum_{i=1}^{\kappa} \binom{b+i-1}{b-1}\Delta' (x_{n - i(q-1)}x_{i(q-1)}).
        \label{eqn: Delta(n) up to q*q, case 1, delta formula}
    \end{align}
    Let $a_i = n-i(q-1) = (\kappa -i )q + (b+i)$ and $b_i = i(q-1)$. Note that $b+i \le b+ \kappa <q-1$, so by induction hypothesis we have 
    \begin{align*}
        \Delta(x_{a_i})  = & \ 1\otimes x_{a_i} + x_{a_i} \otimes 1 + \sum_{j=1}^{\kappa -i } (-1)^j\binom{b+i+j-1}{j} x_{a_i -j(q-1)} \otimes x_{j(q-1)}, \\ 
        \Delta(x_{b_i})  = & \ 1 \otimes x_{b_i} + x_{b_i} \otimes 1 + \sum_{s =1}^{i-1} \binom{i}{s} x_{s(q-1)}\otimes x_{(i-s)(q-1)},\\
        \Delta(x_{a_i}x_{b_i})  = & \ 1\otimes x_{a_i}x_{b_i} +  x_{a_i}\otimes x_{b_i} + x_{a_i}x_{b_i} \otimes 1  \\
        & + \sum_{s =1}^{i-1} \binom{i}{s} x_{a_i}x_{s(q-1)}\otimes x_{(i-s)(q-1)}  \\
        &  + \sum_{j=1}^{\kappa -i} (-1)^j\binom{b+i+j-1}{j} x_{a_i -j(q-1)} \otimes x_{j(q-1)}\shuffle  x_{b_i}\\ 
        & + \sum_{j=1}^{\kappa -i} (-1)^j\binom{b+i+j-1}{j} x_{a_i -j(q-1)}x_{b_i}  \otimes x_{j(q-1)}  \\
        & + \sum_{j=1}^{\kappa -i} \sum_{s=1}^{i-1} (-1)^j\binom{b+i+j-1}{j}\binom{i}{s}
         x_{a_i -j(q-1)} x_{s(q-1)}\otimes x_{j(q-1)} \shuffle  x_{(i-s)(q-1)}\\
        = & \  1\otimes x_{a_i}x_{b_i} +  x_{a_i}\otimes x_{b_i} + x_{a_i}x_{b_i} \otimes 1 \\
        & + \sum_{s =1}^{i-1} \binom{i}{s} x_{a_i}x_{s(q-1)}\otimes x_{(i-s)(q-1)} \\
        &  + \sum_{j=1}^{\kappa -i} (-1)^j\binom{b+i+j-1}{j} x_{a_i -j(q-1)} \otimes x_{(i+j)(q-1)}\\ 
        & + \sum_{j=1}^{\kappa -i} (-1)^j\binom{b+i+j-1}{j} x_{a_i -j(q-1)}x_{b_i}  \otimes x_{j(q-1)} \\
        & + \sum_{j=1}^{\kappa -i} \sum_{s=1}^{i-1} (-1)^j\binom{b+i+j-1}{j}\binom{i}{s}
          x_{a_i -j(q-1)} x_{s(q-1)}\otimes x_{(i+j-s)(q-1)} 
    \end{align*}
    by \eqref{eqn: star product of q-1 multiples}. By manipulating indices combining sums,
    \begin{align*}
        \Delta'(x_{a_i}x_{b_i}) & = x_{a_i}\otimes x_{b_i}  + \sum_{s =1}^{i-1} \binom{i}{s} x_{a_i}x_{(i-s)(q-1)}\otimes x_{s(q-1)} + \\
        &  + \sum_{j=i+1}^{\kappa} (-1)^{j-i}\binom{b+j-1}{j-i} x_{a_i +(i-j)(q-1)} \otimes x_{j(q-1)}\\ 
        & + \sum_{j=1}^{\kappa -i} \sum_{s=0}^{\infty} \mathbf{1}_{[0, i)}(s) (-1)^j \binom{b+i+j-1}{j}\binom{i}{s}
          x_{a_i -j(q-1)} x_{(i-s)(q-1)}\otimes x_{(j+s)(q-1)} 
    \end{align*}
    where $\mathbf{1}_{[0, i)}(s) =1$ when $0\le s <i$, otherwise 0.
    So recalling the previous calculation \eqref{eqn: Delta(n) up to q*q, case 1, delta formula},
    \begin{align*}
        \Delta' (x_{\kappa q+b}) & = - \sum_{i=1}^{\kappa} \binom{b+i-1}{b-1}\Delta' (x_{a_i}x_{b_i})\\ 
        & = S_1 + S_2 + S_3, \end{align*}
    where 
    \begin{align*}
        S_1 &= - \sum_{i=1}^{\kappa}\binom{b+i-1}{b-1} x_{a_i}\otimes x_{b_i}, \\
        S_2 &= - \sum_{i=2}^{\kappa }(-1)^i \sum_{j=1}^{i-1}(-1)^{j} \binom{b+j-1}{b-1}\binom{b+i-1}{i-j} x_{a_i} \otimes x_{b_i},\\ 
        S_3 & = 
        -\sum_{i=1}^{\kappa}
        \sum_{j=0}^{\kappa -i} \sum_{s=0}^{\infty} 
        \mathbf{1}_{[0, i)}(s) (-1)^j \binom{b+i-1}{b-1} 
        \binom{b+i+j-1}{j}\binom{i}{s}
          x_{a_i -b_j} x_{b_{i-s}}\otimes x_{b_{j+s}} 
    \end{align*}
    We show that $S_1 + S_2$ yields the desired equation, and $S_3$ vanishes.

    One can see that sum for $i=1$ in $S_1$ is $-b x_{n-q+1} \otimes x_{q-1}$.
    
    For the remaining terms in $S_1+S_2$, 
    we can prove 
    \[- \binom{b+i-1}{b-1} - (-1)^i \sum_{j=1}^{i-1} (-1)^j \binom{b+j-1}{b-1} \binom{b+i-1}{i-j} = (-1)^i \binom{b+i-1}{i}
        \]
    for $i\ge2$ and $1\le b <q$ from the equation
    \begin{align*}
        \binom{b+j-1}{b-1}\binom{b+i-1}{i-j}& = 
         \binom{i}{j}\binom{b+i-1}{b-1}
    \end{align*}
    and the binomial equation $\sum_{j=1}^{i-1} (-1)^j \binom{i}{j}  = -1 -(-1)^i$. 

    Thus it is enough to show that $S_3 =0$. We manipulate indices as $u=j+s$, and then $v = i+j$, to have 
    \begin{align*}
        & -S_3 \\
        & = \sum_{u=0}^\infty \sum_{i=1}^\kappa \sum_{j=0}^{\kappa-i}
        \mathbf{1}_{[0, i)}(u-j) (-1)^j \binom{b+i-1}{b-1} 
        \binom{b+i+j-1}{j}\binom{i}{u-j}
          x_{a_i -b_j} x_{b_{i+j-u}}\otimes x_{b_u}\\ 
        & = \sum_{u=1}^{\kappa} \sum_{v=1}^\kappa \sum_{i=1}^{v} 
          \mathbf{1}_{[0, i)}(u-v+i) (-1)^{v-i} \binom{b+i-1}{b-1} 
          \binom{b+v-1}{v-i}\binom{i}{v-u}
            x_{a_i -b_{v-i}} x_{b_{v-u}}\otimes x_{b_u}.
    \end{align*}
    The coefficient for $x_{a_i-(v-i)(q-1)} x_{(v-u)(q-1)}\otimes x_{u(q-1)}$ of $-S_3$ vanishes when $u\ge v$ due to $\mathbf{1}_{[0,i)}(u-v+i) =0$ factor.
    Otherwise, it is
    \begin{align*}
        &  \sum_{i=1}^{v} 
        \mathbf{1}_{[0, i)}(u-v+i) (-1)^{v-i} \binom{b+i-1}{b-1} 
        \binom{b+v-1}{v-i}\binom{i}{v-u}\\  
        &  =\sum_{i=v-u}^{v} 
         (-1)^{v-i} \binom{b+i-1}{b-1} 
        \binom{b+v-1}{v-i}\binom{i}{v-u}\\ 
        &  =\sum_{i=0}^{u} 
         (-1)^{u-i} \binom{b+i+v-u-1}{b-1} 
        \binom{b+v-1}{u-i}\binom{i+v-u}{v-u}\\ 
        &  =\frac{(-1)^u(b+v-1)!}{u!(b-1)!(v-u)!}\sum_{i=0}^{u} 
         (-1)^{i}
        \frac{u!}{i!(u-i)!}=0.
    \end{align*}
    Thus we are done. 

    \
    
    (2) 
    Let $n = \kappa q + b$, and suppose that $q-1 <\kappa + b$. We have that $n = \kappa q + b = (\kappa +1) (q-1) + (\kappa + b - q + 1)$ with $0 < \kappa + b - q+ 1 < q$. We exclude the $n = q^2-1$ case, i.e. $\kappa=b=q-1$ case, since it is already treated earlier, so that $0<\kappa + b - q+1 \le q-2$. 
    
    Let $A = \kappa +b-q+1$, $B = q-b$. Then $n = Aq + B(q-1)$, $1 \le A \le q-2$, $1 \le B \le q-1$, and $A+B = \kappa +1$. What we need to prove turns out to be 
    \begin{equation} \Delta(x_n) = 1\otimes x_n + x_n \otimes 1 + \sum_{i=1}^{B}\binom{B}{i} x_{n - i(q-1)}\otimes x_{i(q-1)}
    \label{eqn: Delta(n) up to q*q, case 2, in terms of A and B}
    \end{equation} in terms of $A$ and $B$, by Lucas's theorem and direct calculations.

    From Lemma~\ref{lemma: Aq star B(q-1)}, 
    \begin{align*}
        x_{Aq} \shuffle  x_{B(q-1)} &= x_n + x_{B(q-1)}x_{Aq}
    \end{align*}
    for all $A$, $B$ with $A+B = \kappa+1 \le q$. 

    Now we can show that $\Delta(x_{Aq} \shuffle  x_{B(q-1)}) = \Delta(x_{Aq}) \shuffle  \Delta(x_{B(q-1)})$. From 
    \begin{align*}
        \Delta(x_{B(q-1)})& = 1 \otimes x_{B(q-1)} + x_{B(q-1)}\otimes1 + 
        \sum_{j=1}^{B-1} \binom{B}{j} x_{j(q-1)}\otimes x_{(B-j)(q-1)}, \\ 
        \Delta(x_{Aq}) & = 1\otimes x_{Aq} + x_{Aq} \otimes 1,
    \end{align*}
    we have 
    \begin{align*}
        & \Delta(x_{B(q-1)}x_{Aq} ) \\
        & = 1\otimes x_{B(q-1)}x_{Aq} + 
        x_{B(q-1)}\otimes x_{Aq} + x_{B(q-1)}x_{Aq} \otimes 1 \\ 
        & +  \sum_{j=1}^{B-1} \binom{B}{j}  x_{j(q-1)}\otimes x_{Aq}\shuffle x_{(B-j)(q-1)} +  \sum_{j=1}^{B-1} \binom{B}{j}  x_{j(q-1)}x_{Aq} \otimes x_{(B-j)(q-1)}\\ 
        &= 1\otimes x_{B(q-1)}x_{Aq} + 
        x_{B(q-1)}\otimes x_{Aq} + x_{B(q-1)}x_{Aq} \otimes 1 \\ 
        & +  \sum_{j=1}^{B-1} \binom{B}{j}  x_{j(q-1)}\otimes \left( x_{Aq +(B-j)(q-1)} + x_{(B-j)(q-1)}x_{Aq}\right)\\ 
        & +  \sum_{j=1}^{B-1} \binom{B}{j}  x_{j(q-1)}x_{Aq} \otimes x_{(B-j)(q-1)},
    \end{align*}
    thus 
    \begin{align*}
        & \Delta(x_{Aq} \shuffle  x_{B(q-1)}) \\
        & = \Delta(n) + 1\otimes x_{B(q-1)}x_{Aq} + 
        x_{B(q-1)}\otimes x_{Aq} + x_{B(q-1)}x_{Aq} \otimes 1 \\ 
        & +  \sum_{j=1}^{B-1} \binom{B}{j}  x_{j(q-1)}\otimes \left( x_{Aq +(B-j)(q-1)} + x_{(B-j)(q-1)}x_{Aq}\right)\\ 
        & +  \sum_{j=1}^{B-1} \binom{B}{j}  x_{j(q-1)}x_{Aq} \otimes x_{(B-j)(q-1)}\\ 
        & = \Delta(x_{Aq})\shuffle  \Delta(x_{B(q-1)})\\ 
        & = x_{Aq} \otimes x_{B(q-1)} + x_{Aq}\shuffle  _{B(q-1)}\otimes1 + 
        \sum_{j=1}^{B-1} \binom{B}{j} x_{Aq}\shuffle  x_{j(q-1)}\otimes x_{(B-j)(q-1)}\\ 
        & + 1 \otimes x_{B(q-1)}\shuffle  x_{Aq} + x_{B(q-1)}\otimes x_{Aq} + 
        \sum_{j=1}^{B-1} \binom{B}{j} x_{j(q-1)}\otimes x_{(B-j)(q-1)}\shuffle  x_{Aq}, 
    \end{align*}
    so routine calculation yields  
    \begin{align*}
        \Delta(x_n) & = 1 \otimes x_n + x_n\otimes 1 + x_{Aq} \otimes x_{B(q-1)}   
         + \sum_{j=1}^{B-1} \binom{B}{j} x_{n - (B-j)(q-1)}\otimes x_{(B-j)(q-1)}\\ 
        & = 1 \otimes x_n + x_n\otimes 1 + \sum_{j=1}^{B} \binom{B}{j} x_{n -j(q-1)}\otimes x_{j (q-1)},
    \end{align*}
    which is \eqref{eqn: Delta(n) up to q*q, case 2, in terms of A and B}, thus we are done.
\end{proof}

We mention that some explicit formulas for $\Delta(x_n)$ with $n>q^2$ can be found in the appendix (see Appendix \ref{sec: numerical experiments}).


\section{Stuffle Hopf algebra and stuffle map in positive characteristic} \label{sec: stuffle algebra}

In this section we define the stuffle algebra and the stuffle map in positive characteristic. The stuffle algebra is easy to define. However, to define the stuffle map we make use of a deep connection between the $K$-vector space spanned by the MZV's in positive characteristic and that spanned by the multiple polylogarithms in positive characteristic proved in \cite{IKLNDP22, ND21} (see Theorem \ref{thm:bridge}).

\subsection{The stuffle algebra in positive characteristic} ${}$\par

We recall that the composition space $\frak C$ is introduced in \S \ref{sec: composition space}. We define the stuffle product in the same way as that of $(\frak h^1,*)$ as given in \S \ref{sec: Hoffman algebra}. More precisely,
\begin{align*} 
   * \colon \mathfrak{C} \times \mathfrak{C} \longrightarrow \mathfrak{C}
\end{align*}
by setting $1 *  \mathfrak{a} = \mathfrak{a} * 1 = \mathfrak{a}$ and
\begin{align*}
    \fa *  \fb &= \ x_{a}(\fa_- *  \fb) + x_{b}(\fa * \fb_-) + x_{a+b} (\fa_- * \fb_-)
\end{align*}
for any words $\fa,\fb \in \langle \Sigma \rangle$. We call $*$ the \textit{stuffle product}.

\begin{proposition}
With the above notation, $(\frak C,*)$ is a commutative $\Fq$-algebra.
\end{proposition}

\begin{proof}
The proof follows the same line as that of $(\frak h^1,*)$ (see \S \ref{sec: Hoffman algebra}).
\end{proof}

\subsection{Hopf algebra structure} ${}$\par

We now define a coproduct $\Delta_*:\frak C \otimes \frak C \to \frak C$ and a counit $\epsilon:\frak C \to \Fq$ by
	\[ \Delta_*(\fu)= \sum_{\fa\fb=\fu} \fa \otimes \fb \]
and for any words $\fu \in \langle \Sigma \rangle$,
\begin{align*}
\epsilon(\fu)=\begin{cases} 1 \quad \text{if } \fu=1, \\ 0 \quad \text{otherwise}. \end{cases}
\end{align*}

By Theorem \ref{theorem: Hopf algebras for qp}, we get:
\begin{theorem} \label{thm: Hopf algebra for stuffle product}
The stuffle algebra $(\frak C,*,u,\Delta_*,\epsilon)$ is a connected graded Hopf algebra of finite type over $\Fq$.
\end{theorem}

\subsection{The stuffle map in positive characteristic} ${}$\par \label{sec: stuffle map}

We put $\ell_0 := 1$ and $\ell_d := \prod^d_{i=1}(\theta - \theta^{q^i})$ for all $d \in \mathbb{N}$. Letting $\mathfrak s = (s_1 , \dots, s_n) \in \mathbb{N}^n$, for $d \in \mathbb{Z}$, we define analogues of power sums by
\begin{equation*}
        \Si_d(\mathfrak s) = \sum\limits_{d=d_1> \dots > d_n\geq 0} \dfrac{1}{\ell_{d_1}^{s_1} \dots \ell_{d_n}^{s_n}} \in K,
\end{equation*}
and 
\begin{equation*}
        \Si_{<d}(\mathfrak s) = \sum\limits_{d>d_1> \dots > d_n\geq 0} \dfrac{1 }{\ell_{d_1}^{s_1} \dots \ell_{d_n}^{s_n}} \in K.
\end{equation*}
Thus
\begin{align*}
\Si_{<d}(\fs) =\sum_{i=0}^{d-1} \Si_i(\fs), \quad \Si_{d}(\fs) =\Si_d(s_1) \Si_{<d}(\fs_-)=\Si_d(s_1) \Si_{<d}(s_2,\dots,s_n).
\end{align*}
Here by convention we define empty sums to be $0$ and empty products to be $1$. In particular, $\Si_{<d}$ of the empty tuple is equal to $1$.

Then we define the Carlitz multiple polygarithm (CMPL for short) as follows
\begin{equation*}
    \Li(\fs)  = \sum \limits_{d \geq 0} \Si_d (\fs)  = \sum\limits_{d_1> \dots > d_n\geq 0} \dfrac{1}{\ell_{d_1}^{s_1} \dots \ell_{d_n}^{s_n}}   \in K_{\infty}.
\end{equation*}
We agree also that $\Li(\emptyset)  = 1$. We call $\depth(\fs) = n$ the depth, $w(\fs) = s_1 + \dots + s_n$ the weight of $\Li(\fs)$.
 
\begin{lemma} 
For all $\fs$ as above such that $s_i \leq q$ for all $i$, we have 
\begin{equation*}
   S_d(\fs)  =  \Si_d(\fs)  \quad \text{for all } d \in \mathbb{Z}.
\end{equation*}
Therefore, 
\begin{equation*}
    \zeta_A (\fs)  = \Li(\fs) .
\end{equation*}
\end{lemma}

\begin{proof}
See \cite[Lemma 1.1]{IKLNDP22}.
\end{proof}

Combining the above lemma and some extensions of results in \cite{ND21} we get a deep connection between the $K$-vector space spanned by the MZV's in positive characteristic and that spanned by the multiple polylogarithms in positive characteristic.
\begin{theorem} \label{thm:bridge}
The $K$-vector space $\mathcal{Z}_w$ of MZV's of weight $w$ and the $K$-vector space $\mathcal{L}_w$ of CMPL's of weight $w$ are the same.
\end{theorem}

\begin{proof}
See \cite[Theorem 4.3]{IKLNDP22}.
\end{proof}
For all $d \in \bZ$, we define two $\mathbb{F}_q$-linear maps
\begin{equation*}
    \Si_{<d} \colon \frak C \rightarrow K_{\infty} \quad \text{and} \quad \Li \colon \frak C \rightarrow K_{\infty},
\end{equation*}
which map the empty word to the element $1 \in K_{\infty}$, and map any word $x_{s_1} \dotsc x_{s_n}$ to $\Si_{<d}(s_1, \dotsc ,s_n)$ and $\Li(s_1, \dotsc, s_n)$, respectively. We have the following result:
\begin{proposition}\label{lastProp}
For all words $\fa, \fb \in \frak C$ and for all $d \in \bZ$ we have
\begin{align*}
\Si_{<d}(\fa * \fb) &=\Si_{<d}(\fa) \, \Si_{<d}(\fb), \\
\Li(\fa *\fb) &=\Li(\fa) \, \Li(\fb).
\end{align*}
\end{proposition}

\begin{proof}
We leave the proof to the reader.
\end{proof}

Combining Theorem \ref{thm:bridge} and Proposition~\ref{lastProp} yields the $K$-linear map \[ Z_*:\frak C \otimes_{\Fq} K \to \mathcal Z, \] which sends a word $\fa \in \frak C$ to $\Li(\fa)$, is a homomorphism of $K$-algebras, and is called the stuffle map in positive characteristic.


\appendix

\section{Numerical experiments} \label{sec: numerical experiments}

In this appendix we write down some explicit formulas for $\Delta(x_n)$ for $q^2<n \le q^3 + q^2$ and $q=3, 5$. 
\subsection{The case $q=3$} 
$\phantom{1}$
{
\footnotesize
\begin{align*}
    \Delta(x_{10})&=1\otimes x_{10}+x_{2}\otimes x_{2}x_{6}+2x_{4}\otimes x_{6}+x_{6}\otimes x_{4}+2x_{8}\otimes x_{2}+x_{10}\otimes 1\\
\Delta(x_{11})&=1\otimes x_{11}+2x_{3}\otimes x_{2}x_{6}+2x_{5}\otimes x_{6}+x_{9}\otimes x_{2}+x_{11}\otimes 1\\
\Delta(x_{12})&=1\otimes x_{12}+2x_{6}\otimes x_{6}+x_{12}\otimes 1\\
\Delta(x_{13})&=1\otimes x_{13}+2x_{3}\otimes x_{2}x_{8}+x_{3}\otimes x_{4}x_{6}+2x_{5}\otimes x_{2}x_{6}+x_{5}\otimes x_{8}+x_{7}\otimes x_{6}\\&+x_{9}\otimes x_{4}+2x_{11}\otimes x_{2}+x_{13}\otimes 1\\
\Delta(x_{14})&=1\otimes x_{14}+x_{6}\otimes x_{2}x_{6}+2x_{6}\otimes x_{8}+x_{8}\otimes x_{6}+x_{12}\otimes x_{2}+x_{14}\otimes 1\\
\Delta(x_{15})&=1\otimes x_{15}+x_{9}\otimes x_{6}+x_{15}\otimes 1\\
\Delta(x_{16})&=1\otimes x_{16}+x_{6}\otimes x_{2}x_{8}+2x_{6}\otimes x_{4}x_{6}+2x_{6}\otimes x_{10}+2x_{8}\otimes x_{8}+x_{12}\otimes x_{4}
+2x_{14}\otimes x_{2}+x_{16}\otimes 1\\
\Delta(x_{17})&=1\otimes x_{17}+x_{9}\otimes x_{8}+x_{15}\otimes x_{2}+x_{17}\otimes 1\\
\Delta(x_{18})&=1\otimes x_{18}+x_{18}\otimes 1\\
\Delta(x_{19})&=1\otimes x_{19}+x_{3}\otimes x_{2}x_{14}+x_{3}\otimes x_{4}x_{12}+2x_{5}\otimes x_{2}x_{12}+x_{5}\otimes x_{6}x_{8}+2x_{5}\otimes x_{8}x_{6}\\&+x_{7}\otimes x_{12}+x_{9}\otimes x_{10}+x_{11}\otimes x_{2}x_{6}+2x_{13}\otimes x_{6}+x_{15}\otimes x_{4}+2x_{17}\otimes x_{2}
+x_{19}\otimes 1\\
\Delta(x_{20})&=1\otimes x_{20}+x_{6}\otimes x_{2}x_{12}+2x_{6}\otimes x_{6}x_{8}+x_{6}\otimes x_{8}x_{6}+x_{8}\otimes x_{12}+2x_{12}\otimes x_{2}x_{6}\\&+2x_{14}\otimes x_{6}+x_{18}\otimes x_{2}+x_{20}\otimes 1\\
\Delta(x_{21})&=1\otimes x_{21}+x_{9}\otimes x_{12}+2x_{15}\otimes x_{6}+x_{21}\otimes 1\\
\Delta(x_{22})&=1\otimes x_{22}+x_{6}\otimes x_{2}x_{14}+2x_{6}\otimes x_{4}x_{12}+x_{6}\otimes x_{8}x_{8}+2x_{6}\otimes x_{10}x_{6}+2x_{8}\otimes x_{14}
+2x_{12}\otimes x_{2}x_{8}
\\&
+x_{12}\otimes x_{4}x_{6}+2x_{14}\otimes x_{2}x_{6}+x_{14}\otimes x_{8}+x_{16}\otimes x_{6}+x_{18}\otimes x_{4}
+2x_{20}\otimes x_{2}+x_{22}\otimes 1\\
\Delta(x_{23})&=1\otimes x_{23}+x_{9}\otimes x_{14}+x_{15}\otimes x_{2}x_{6}+2x_{15}\otimes x_{8}+x_{17}\otimes x_{6}+x_{21}\otimes x_{2}
+x_{23}\otimes 1\\
\Delta(x_{24})&=1\otimes x_{24}+x_{18}\otimes x_{6}+x_{24}\otimes 1\\
\Delta(x_{25})&=1\otimes x_{25}+x_{9}\otimes x_{16}+x_{15}\otimes x_{2}x_{8}+2x_{15}\otimes x_{4}x_{6}+2x_{15}\otimes x_{10}+2x_{17}\otimes x_{8}\\&+x_{21}\otimes x_{4}+2x_{23}\otimes x_{2}+x_{25}\otimes 1\\
\Delta(x_{26})&=1\otimes x_{26}+x_{18}\otimes x_{8}+x_{24}\otimes x_{2}+x_{26}\otimes 1\\
\Delta(x_{27})&=1\otimes x_{27}+x_{27}\otimes 1\\
\Delta(x_{28})&=1\otimes x_{28}+2x_{2}\otimes x_{2}x_{6}x_{18}+x_{4}\otimes x_{6}x_{18}+x_{6}\otimes x_{2}x_{20}+x_{6}\otimes x_{4}x_{18}+2x_{6}\otimes x_{6}x_{16}\\&+2x_{6}\otimes x_{8}x_{14}+2x_{6}\otimes x_{10}x_{12}+x_{8}\otimes x_{2}x_{18}+2x_{8}\otimes x_{6}x_{14}+x_{8}\otimes x_{8}x_{12}+2x_{8}\otimes x_{12}x_{8}\\&+x_{8}\otimes x_{14}x_{6}+2x_{10}\otimes x_{18}+x_{12}\otimes x_{2}x_{14}+x_{12}\otimes x_{4}x_{12}+2x_{14}\otimes x_{2}x_{12}+x_{14}\otimes x_{6}x_{8}\\&+2x_{14}\otimes x_{8}x_{6}+x_{16}\otimes x_{12}+x_{18}\otimes x_{10}+x_{20}\otimes x_{2}x_{6}+2x_{22}\otimes x_{6}+x_{24}\otimes x_{4}\\&+2x_{26}\otimes x_{2}+x_{28}\otimes 1\\
\Delta(x_{29})&=1\otimes x_{29}+x_{3}\otimes x_{2}x_{6}x_{18}+x_{5}\otimes x_{6}x_{18}+2x_{9}\otimes x_{2}x_{18}+x_{9}\otimes x_{6}x_{14}+2x_{9}\otimes x_{8}x_{12}\\&+x_{9}\otimes x_{12}x_{8}+2x_{9}\otimes x_{14}x_{6}+2x_{11}\otimes x_{18}+x_{15}\otimes x_{2}x_{12}+2x_{15}\otimes x_{6}x_{8}+x_{15}\otimes x_{8}x_{6}\\&+x_{17}\otimes x_{12}+2x_{21}\otimes x_{2}x_{6}+2x_{23}\otimes x_{6}+x_{27}\otimes x_{2}+x_{29}\otimes 1\\
\Delta(x_{30})&=1\otimes x_{30}+x_{6}\otimes x_{6}x_{18}+2x_{12}\otimes x_{18}+x_{18}\otimes x_{12}+2x_{24}\otimes x_{6}+x_{30}\otimes 1\\
\Delta(x_{31})&=1\otimes x_{31}+x_{3}\otimes x_{2}x_{8}x_{18}+2x_{3}\otimes x_{4}x_{6}x_{18}+x_{5}\otimes x_{2}x_{6}x_{18}+2x_{5}\otimes x_{8}x_{18}+2x_{7}\otimes x_{6}x_{18}\\&+2x_{9}\otimes x_{4}x_{18}+2x_{9}\otimes x_{6}x_{16}+x_{9}\otimes x_{8}x_{14}+2x_{9}\otimes x_{14}x_{8}+x_{9}\otimes x_{16}x_{6}+x_{11}\otimes x_{2}x_{18}\\&+2x_{11}\otimes x_{6}x_{14}+x_{11}\otimes x_{8}x_{12}+2x_{11}\otimes x_{12}x_{8}+x_{11}\otimes x_{14}x_{6}+x_{11}\otimes x_{20}+2x_{13}\otimes x_{18}\\&+x_{15}\otimes x_{2}x_{14}+2x_{15}\otimes x_{4}x_{12}+x_{15}\otimes x_{8}x_{8}+2x_{15}\otimes x_{10}x_{6}+2x_{17}\otimes x_{14}+2x_{21}\otimes x_{2}x_{8}\\&+x_{21}\otimes x_{4}x_{6}+2x_{23}\otimes x_{2}x_{6}+x_{23}\otimes x_{8}+x_{25}\otimes x_{6}+x_{27}\otimes x_{4}+2x_{29}\otimes x_{2}
+x_{31}\otimes 1\\
\Delta(x_{32})&=1\otimes x_{32}+2x_{6}\otimes x_{2}x_{6}x_{18}+x_{6}\otimes x_{8}x_{18}+2x_{8}\otimes x_{6}x_{18}+2x_{12}\otimes x_{2}x_{18}+x_{12}\otimes x_{6}x_{14}\\&+2x_{12}\otimes x_{8}x_{12}+x_{12}\otimes x_{12}x_{8}+2x_{12}\otimes x_{14}x_{6}+2x_{12}\otimes x_{20}+2x_{14}\otimes x_{18}+x_{18}\otimes x_{14}\\&+x_{24}\otimes x_{2}x_{6}+2x_{24}\otimes x_{8}+x_{26}\otimes x_{6}+x_{30}\otimes x_{2}+x_{32}\otimes 1\\
\Delta(x_{33})&=1\otimes x_{33}+2x_{9}\otimes x_{6}x_{18}+2x_{15}\otimes x_{18}+x_{27}\otimes x_{6}+x_{33}\otimes 1\\
\Delta(x_{34})&=1\otimes x_{34}+2x_{6}\otimes x_{2}x_{8}x_{18}+x_{6}\otimes x_{2}x_{26}+x_{6}\otimes x_{4}x_{6}x_{18}+x_{6}\otimes x_{4}x_{24}+x_{6}\otimes x_{10}x_{18}\\&+x_{8}\otimes x_{8}x_{18}+2x_{12}\otimes x_{4}x_{18}+2x_{12}\otimes x_{6}x_{16}+x_{12}\otimes x_{8}x_{14}+2x_{12}\otimes x_{14}x_{8}+x_{12}\otimes x_{16}x_{6}\\&+2x_{12}\otimes x_{22}+x_{14}\otimes x_{2}x_{18}+2x_{14}\otimes x_{6}x_{14}+x_{14}\otimes x_{8}x_{12}+2x_{14}\otimes x_{12}x_{8}+x_{14}\otimes x_{14}x_{6}\\&+x_{14}\otimes x_{20}+2x_{16}\otimes x_{18}+x_{18}\otimes x_{16}+x_{24}\otimes x_{2}x_{8}+2x_{24}\otimes x_{4}x_{6}+2x_{24}\otimes x_{10}\\&+2x_{26}\otimes x_{8}+x_{30}\otimes x_{4}+2x_{32}\otimes x_{2}+x_{34}\otimes 1\\
\Delta(x_{35})&=1\otimes x_{35}+2x_{9}\otimes x_{8}x_{18}+2x_{15}\otimes x_{2}x_{18}+x_{15}\otimes x_{6}x_{14}+2x_{15}\otimes x_{8}x_{12}+x_{15}\otimes x_{12}x_{8}\\&+2x_{15}\otimes x_{14}x_{6}+2x_{15}\otimes x_{20}+2x_{17}\otimes x_{18}+x_{27}\otimes x_{8}+x_{33}\otimes x_{2}+x_{35}\otimes 1\\
\Delta(x_{36})&=1\otimes x_{36}+2x_{18}\otimes x_{18}+x_{36}\otimes 1
\end{align*}
}
\subsection{The case $q=5$} $\phantom{1}$
{\footnotesize

}


\end{document}